\documentclass[envcountsame]{livre1}
\usepackage{etex}
\usepackage{morefloats} %
\usepackage[doc,show]{optional} 
\usepackage{array}
\usepackage{srcltx}
\usepackage{xcolor}
\usepackage{soul}
\usepackage{pdfpages}

\usepackage[
  letterpaper,
  top=3.0cm,
  bottom=3.0cm,
  left=3.0cm,
  right=3.0cm
]{geometry}

\definecolor{labelkey}{rgb}{0,0.08,0.45}
\definecolor{refkey}{rgb}{0,0.6,0.0}
\definecolor{Brown}{rgb}{0.45,0.0,0.05}
\definecolor{nido}{rgb}{1.0,0.2,0.0}
\definecolor{dgreen}{rgb}{0.00,0.40,0.00}
\definecolor{dblue}{rgb}{0,0.08,0.45}
\PassOptionsToPackage{normalem}{ulem}
\usepackage{ulem}
\usepackage{epsfig}
\usepackage[usenames]{pstricks}
\usepackage{pstricks-add}
\usepackage{pst-grad}
\usepackage{pst-plot} 
\usepackage{pst-eucl} 
\usepackage{multicol}
\usepackage{graphicx}
\usepackage[tbtags]{amsmath}
\usepackage{amssymb}
\usepackage{pifont}
\usepackage{calc}
\usepackage{ntheorem}
\usepackage{rotating}
\usepackage{euscript}
\usepackage{epic,eepic}
\usepackage{makeidx}
\usepackage{index}   
\usepackage{exscale,relsize} 
\usepackage{enumitem}
\usepackage{fancyhdr}
\usepackage{cancel}
\RequirePackage[colorlinks,hyperindex]{hyperref}
\hypersetup{linktocpage=true,citecolor=dblue,linkcolor=dgreen} 

\usepackage[capitalize, nameinlink]{cleveref} 

\crefdefaultlabelformat{#2\textup{#1}#3}
\crefrangeformat{enumi}{#3\textup{#1}#4\textup{\textendash}#5\textup{#2}#6}
\crefname{equation}{}{}

\crefname{chapter}{Chapter}{chapters}
\crefname{figure}{Figure}{Figures}

\crefdefaultlabelformat{#2\textup{#1}#3}
\crefname{enumi}{}{}
\crefname{equation}{}{}
\crefname{page}{p.}{pp.}
\crefname{item}{}{}
\crefname{figure}{Figure}{Figures}
\crefname{theorem}{Theorem}{Theorems}
\crefname{condition}{Condition}{Conditions}
\crefname{lemma}{Lemma}{Lemmas}
\crefname{proposition}{Proposition}{Propositions}
\crefname{corollary}{Corollary}{Corollaries}
\crefname{definition}{Definition}{Definitions}
\crefname{fact}{Fact}{Facts}
\crefname{example}{Example}{Examples}
\crefname{remark}{Remark}{Remarks}
\crefname{problem}{Problem}{Problems}
\crefname{question}{Question}{Questions}
\crefname{exercise}{Exercise}{Exercises}

\makeatletter

\let\if@envcntshowhiercnt=\iftrue
\renewtheorem{lemma}[theorem]{Lemma}
\renewtheorem{fact}[theorem]{Fact}
\renewtheorem{corollary}[theorem]{Corollary}
\renewtheorem{proposition}[theorem]{Proposition}

\theoremstyle{plain}{\theorembodyfont{\rmfamily}%
\renewtheorem{conjecture}[theorem]{Conjecture}}
\theoremstyle{plain}{\theorembodyfont{\rmfamily}%
\renewtheorem{example}[theorem]{Example}}
\theoremstyle{plain}{\theorembodyfont{\rmfamily}%
\renewtheorem{remark}[theorem]{Remark}}
\theoremstyle{plain}{\theorembodyfont{\rmfamily}%
}
\theoremstyle{plain}{\theorembodyfont{\rmfamily}%
\renewtheorem{definition}[theorem]{Definition}}
\theoremstyle{plain}{\theorembodyfont{\rmfamily}%
\renewtheorem{problem}[theorem]{Problem}}
\theoremstyle{plain}{\theorembodyfont{\rmfamily}%
\renewtheorem{definition}[theorem]{Definition}}
\theoremstyle{plain}{\theorembodyfont{\rmfamily}%
\renewtheorem{exercise}{Exercise}[chapter]}

\g@addto@macro\normalsize{%
	\setlength\abovedisplayskip{6pt}
	\setlength\belowdisplayskip{6pt}
	\setlength\abovedisplayshortskip{6pt}
	\setlength\belowdisplayshortskip{6pt}
}

\let\orgdescriptionlabel\descriptionlabel
\renewcommand*{\descriptionlabel}[1]{%
	\let\orglabel\label
	\let\label\@gobble
	\phantomsection
	\edef\@currentlabel{#1}%
	\let\label\orglabel
	\orgdescriptionlabel{#1}%
}

\newcommand{\hhbcom}[1]{{\opt{hhb}{\marginpar{\color{blue}hhb}%
{\color{blue}{{\rm\bfseries [}{\rm #1}{\rm\bfseries ]}}}}}}

\newcommand{\ww}{\ensuremath{\sqrt{5}}}
\newcommand{\thalb}{\ensuremath{\tfrac{1}{2}}}

\newcommand{\menge}[2]{\big\{{#1}~\big |~{#2}\big\}}

\newcommand{\emp}{\ensuremath{\varnothing}}

\newcommand{\ii}{\ensuremath{\mathrm i}}

\newcommand{\QQ}{\ensuremath{\mathbb Q}}
\newcommand{\RR}{\ensuremath{\mathbb R}}
\newcommand{\Rel}{{\ensuremath{\mathsf R}}}
\newcommand{\notRel}{{\not\Rel\,}}
\newcommand{\FF}{\ensuremath{\mathbb F}}
\newcommand{\SR}{\ensuremath{\mathbb S}}
\newcommand{\CC}{\ensuremath{\mathbb C}}
\newcommand{\ZZ}{\ensuremath{\mathbb Z}}

\newcommand{\RP}{\ensuremath{{\mathbb R}_{+}}}

\newcommand{\RPP}{\ensuremath{{\mathbb R}_{++}}}

\newcommand{\NN}{\ensuremath{\mathbb N}}

\newcommand{\zeroun}{\ensuremath{\left]0,1\right[}}

\newcommand{\gr}{\ensuremath{\operatorname{gra}}}
\newcommand{\ran}{\ensuremath{\operatorname{ran}}}

\newcommand{\Id}{\ensuremath{{\operatorname{Id}}}}

\newcommand{\minf}{\ensuremath{- \infty}}
\newcommand{\pinf}{\ensuremath{+ \infty}}
\newcommand{\exi}{\ensuremath{\exists\,}}

\newcommand{\nnn}{\ensuremath{{n\in{\mathbb N}}}}
\newcommand{\kkk}{\ensuremath{{k\in{\mathbb N}}}}

{\begin{list}{}{%
\settowidth{\labelwidth}{\textrm{#1~}}%
\setlength{\leftmargin}{\labelwidth+\labelsep}}}
{\end{list}}

\newcommand*{\QEDW}{\hfill \ensuremath{\square}}%

\smartqed     

\renewenvironment{solution}
  {\noindent{\it Solution}. \ignorespaces}
  {\vspace{-6mm}\hfill\;\;\;
   \begin{flushright}\QEDW\end{flushright}\vskip 2mm}

\renewcommand\theenumii{(\alph{enumii})}

\allowdisplaybreaks 
\makeindex

\begin{document}
\thispagestyle{empty}
\mbox{~}\vskip 32mm
{\bfseries

{\huge
\noindent
%
%
}

\vskip 12mm

{\Huge 

\noindent
Mathematical Proof\\[-10mm]

%
}}

\vskip 24mm


\hskip 23mm\begin{turn}{26}%
\hfill{\huge\mbox{${(x+y)^n = \displaystyle \sum_{k=0}^{n} {n \choose k}
x^{n-k}y^k}$}}
\end{turn} 

\bigskip 

\bigskip 

These Course Notes provide an introduction to mathematical proofs 
for undergraduate students transitioning from computational calculus
to abstract mathematics. Topics include propositional logic,
proof techniques, mathematical induction, fields, sets and relations,
sequences and series, completeness of the real numbers, cardinality,
and related foundational material. Numerous examples and exercises 
(with complete solutions) are included. 
The notes are designed for a one-semester proof course.

\bigskip 

Version 1 (initial release, typeset in March 2026) 


\vfill 

{\textcopyright\ 
\larger \textbf{2026 Heinz H.\ Bauschke} 
}

\bigskip 

{
This work is licensed under the
Creative Commons Attribution–NonCommercial–ShareAlike 4.0 International License.

You are free to share and adapt the material for non-commercial purposes,
provided appropriate credit is given and derivative works are distributed
under the same license.

\url{https://creativecommons.org/licenses/by-nc-sa/4.0/}
}

\bigskip

\pagenumbering{roman}

\thispagestyle{empty}

%
%
%
%
%

\section*{Preface}

The calendar description of the \textbf{MATH 220 Mathematical Proof}
states
\begin{quotation}
\noindent
Sets and functions; induction; cardinality; properties of the real numbers; sequences, series, and limits. Logic, structure, style, and clarity of proofs emphasized throughout. 
\end{quotation}

This material is classical;
consequently, there is a large number of text books
on this subject already available; see, e.g., 
\cite{Beardon}
\cite{BG}, \cite{Bloch}, 
\cite{CPZ}, \cite{Cunning}, \cite{Forster}, \cite{Galovich}, 
\cite{Gerstein}, 
\cite{Howie}, 
\cite{Lay}, \cite{Lipschutz}, \cite{SES}, \cite{Solow}
for some useful related references. 

These notes were greatly 
inspired by the basic flow of Forster's \cite{Forster} --- a book that strongly 
influenced me as an undergraduate student.

This set of lecture notes helps with
the transition from Calculus (which is largely mechanical) to more
advanced abstract mathematics.
The amount of topics covered is not excessive;
indeed, the material in these lecture notes can be covered in one term.

Please note that the UBC library provides access to many books, some of
which (in particular those published by Springer) 
you can download while on a campus network, including: 
\begin{itemize}
\item Matthias Beck and Ross Geoghegan: \emph{The Art of Proof}
\item Ethan Bloch: \emph{Proofs and Fundamentals}
\item Daniel W.\ Cunningham: \emph{A Logical Introduction to Proof}
\item Larry J.\ Gerstein: \emph{Introduction to Mathematical Structures and
Proofs}
\item John M.\ Howie: \emph{Real Analysis}
\end{itemize} 

I hope that you will enjoy these lecture notes,
and I welcome feedback, suggestions, and typos. 

I thank deeply the sharp and kind students who already contributed by squashing
typos and suggesting improvements. 
A special thank you to Alexandra Gretchko who gave me a long list of useful
comments in Fall 2013!
Thanks to Nicholas Kayban for useful comments in Fall 2017!
I am grateful to Jakob Thoms who found several pertinent typos in Fall 2018!
Last but not least, I thank Liangjin Yao, Walaa Moursi, Minh
Bui, Hui Ouyang, and Shambhavi Singh who
have been wonderful Teaching Assistants for
this course, and who suggested several exercises and spotted various typos!

\vspace{\baselineskip}
\begin{flushright}\noindent

\hfill {\it Heinz Bauschke, University of British Columbia Okanagan}\\
Kelowna, 2026\hfill~\\
\end{flushright}

\tableofcontents
\newpage
\thispagestyle{empty}
\pagenumbering{arabic}
\setcounter{page}{0}

\chapter{Background}

\label{cha:background}

\section{Logical Connectives}

\begin{definition}
A \textbf{statement}\index{Statement} is a sentence that can be
classified as either true (T) or false (F). 
\end{definition}

\begin{example}\label{ex:s}
The following are statements:
\begin{enumerate}
\item 
\label{ex:si}
$1+0=1$.
\item 
\label{ex:sii}
Every integer is a prime number.
\item 
\label{ex:gold}
\textbf{(Goldbach)} Every even integer greater than $2$ is the sum
of two prime numbers.
\end{enumerate}
\end{example}

Here \ref{ex:si} is evidently true, while
\ref{ex:sii} is false (e.g., because $6=2\cdot 3$).
It is clear that \ref{ex:gold} is either true or false, and
hence \ref{ex:gold} is a statement; however, it is presently unknown
whether \ref{ex:gold} is true.\index{Goldbach conjecture} 

Statements can be negated, or combined with ``and'' and ``or''. 

\begin{definition}[not, and, or, if, if and only if] 
\label{d:notandor}
Let $p$ and $q$ be statements. Then:
\begin{enumerate}
\item \textbf{``not $p$'':} This is written as $\lnot p$. 
\item \textbf{``$p$ and $q$''}: This is written as $p\land q$.
\item \textbf{``$p$ or $q$''}: This is written as $p\lor q$. 
\item \textbf{``$p$ implies $q$''}: This is written as $p\Rightarrow
q$ and also expressed as ``If $p$, then $q$''. 
\item \textbf{``$p$ if and only if $q$''}: This is written as $p\Leftrightarrow
q$ and it means ``$p\Rightarrow q$ and $q\Rightarrow p$''. 
\end{enumerate}
\end{definition}

\begin{example}
Suppose $p$ represents ``$5$ is an odd integer'' and 
$q$ represents ``$6$ is a prime number''. 
Then $p$ is true while $q$ is false.
Furthermore:
\begin{enumerate}
\item
$\lnot p$ means ``$5$ is not an odd integer'', i.e., ``$5$ is even'', which
is false.
\item
$\lnot q$ means ``$6$ is not a prime number'', which is true.
\item
$p\land q$ means ``$5$ is an odd integer and $6$ is a prime number'',
which is false. 
\item
$p\lor q$ means ``$5$ is an odd integer or $6$ is a prime number'',
which is true. 
\item
$p\Rightarrow q$ means ``If $5$ is an odd integer, then $6$ is a prime
number'', which is false.
\item
$q\Rightarrow p$ means ``If $6$ is a prime number, then $5$ is an odd
integer'', which is true. 
\item 
$q\Leftrightarrow p$ means ``$5$ is an odd integer if and only if 
$6$ is a prime number'', which is false.
\end{enumerate}
\end{example}

\section{Truth Tables and More}
\index{Truth table}
It is illustrative to understand the logical connective through truth
tables: 

\begin{center}
\newcolumntype{C}{>{\centering\arraybackslash}m{1.0cm}<{}}
\begin{tabular}{|C|C|C|C|C|C|C|C|}
\hline
$p$ & $q$ & $\lnot p$ & $\lnot q$ & $p\land q$ & $p\lor q$ &
$p\Rightarrow q$ & $p\Leftrightarrow q$\\
\hline\hline
T & T & F & F & T & T & T & T \\
T & F & F & T & F & T & F & F \\
F & T & T & F & F & T & T & F \\
F & F & T & T & F & F & T & T \\
\hline
\end{tabular}
\end{center}

The statements ``$p\Rightarrow q$'' and
``$(\lnot p)\lor q$'' have the same truth tables, hence they are 
\textbf{logically equivalent}. 
If a statement is always true, then one has a \textbf{tautology};
if, at the other extreme, it is always false, then we have a
\textbf{contradiction}. 
\index{logically equivalent}
\index{Tautology}
\index{Contradiction}

\begin{definition}
Let $p$ and $q$ be statements. Then:
\begin{enumerate}
\item 
``$q\Rightarrow p$'' is the \textbf{converse} of ``$p\Rightarrow
q$''.\index{Converse (of an implication)}
\item 
``$(\lnot p)\Rightarrow (\lnot q)$'' is the \textbf{inverse} 
of ``$p\Rightarrow q$''.\index{Inverse (of an implication)}
\item 
``$(\lnot q)\Rightarrow (\lnot p)$'' is the \textbf{contrapositive} 
of ``$p\Rightarrow q$''.\index{Contrapositive (of an implication)}
\end{enumerate}
\end{definition}

Of course, one may combine more than two statements which gives rise
to larger truth tables to cover all the cases. 

Some statements concern phrases involving ``for each'',
``for every'', ``for all'', ``there exists'', etc.
For instance, the negation of ``Every integer is a prime number'' is
``There exists an integer that is not a prime number'' (e.g.,
$10=(2)(5)$, which serves as a \textbf{counterexample});
the answer ``Every integer is not a prime number'' is \emph{wrong}. 

Finally, we point out that \texttt{WolframAlpha} does very nicely
truth tables for you; try, e.g., \texttt{"truth table for p implies q"}. 

\section{Proof Techniques for Mathematical Results}

Mathematical results are often labelled
Theorem, Proposition, Lemma, Corollary, or Example.
A Theorem is usually an important mathematical result,
while a Proposition or a Lemma is not quite so important.
A Corollary is an easy consequence of another mathematical result.
An Example illustrates a more abstract theorem. 

In the following chapters we shall see many proofs.
Suppose we wish to prove a statement such as $p\Rightarrow q$.
Of course, if $p$ is false, then this statement is true (check the truth table). 
It thus remains to assume that $p$ is true and derive at the truth of
$q$. 

If one is lucky, one can reason from $p$ to $q$ directly. 
This route leads to the most insightful proofs as one clearly sees 
the path to the conclusion $q$. 

If this direct reasoning fails, one common approach is to break up the
hypothesis $p$ into several cases and then establish the conclusion
$q$ for each case. \index{Proof by cases} We think of this as a 
\emph{divide-and-conquer approach}. Many harder proofs follow this pattern. 

Sometimes, it helps to consider proving the contrapositive. 
\index{Proof by contrapositive}

Finally, if everything fails, then it is advised to argue by contradiction:
Assume that $p$ is true and $q$ is false, and derive from this a false
statement (an absurdity). While this does not shed much insight into
why $p\Rightarrow q$ is true, it is a very powerful tool. 
\index{Proof by contradiction}

Statements involving positive integers are often proved using the
principle of mathematical induction (see Chapter~\ref{cha:Induction}). 

We conclude with some simple proofs illustrating some of the
techniques just mentioned. 

\begin{proposition}[a direct proof]
\label{p:120604a}
If $n$ is an even integer, then $n^2$ is also even.
\end{proposition}
\begin{proof}
Write $n=2k$, where $k$ is an integer --- this is what
it means for $n$ to be even!
Then $n^2 = (2k)^2 = 4k^2 = 2(2k^2)$ is also even because
$2k^2$ is evidently an integer. 
\end{proof}

\begin{proposition}[a proof by cases] 
If $n$ is an integer that is not divisible by $3$, then
the remainder of $n^2$ after division by $3$ is $+1$.
\end{proposition}
\begin{proof}
After division by $3$, the integer $n$ has remainder either $1$ or
$2$.

\emph{Case 1:} $n = 3k+1$, where $k$ is another integer.\\
Then $n^2 = (3k+1)^2 = (3k)^2 + 2(3k)(1) + 1 = 9k^2 + 6k + 1
= 3(3k^2 +2k) +1$, as claimed. 

\emph{Case 2:} $n = 3k+2$, where $k$ is another integer.\\
Then $n^2 = (3k+2)^2 = (3k)^2 + 2(3k)(2) + 4 = 9k^2 + 12k +4 
= 3(3k^2 +4k+1) +1$, as claimed. 
\end{proof}

\begin{proposition}[a proof by contradiction]
If $n$ is an integer such that $n^2$ is odd, then $n$ is odd. 
\end{proposition}
\begin{proof}
Assume to the contrary that $n$ is even.
By Proposition~\ref{p:120604a}, $n^2$ is even, which is absurd
because $n^2$ is assumed to be odd.
\end{proof}

To show that a mathematical statement is \emph{false},
one provides a \emph{counterexample}. 
For instance, a counterexample to the statement
``Every prime number is odd'', is the number $2$ 
because $2$ is a prime number that is \emph{not odd}.

\section*{Exercises}\markright{Exercises}
\addcontentsline{toc}{section}{Exercises}
\setcounter{theorem}{0}

\begin{exercise}
Let $p$ and $q$ be statements.
Compute the truth table for $(\lnot p)\lor q$ and compare to the truth
table of $p\Rightarrow q$; deduce therefore that these last two
statements are logically equivalent. 
\end{exercise}
\opt{show}{
\begin{solution}
Indeed, 
\begin{center}
\newcolumntype{C}{>{\centering\arraybackslash}m{2.0cm}<{}}
\begin{tabular}{|C|C|C|C|C|}
\hline
$p$ & $q$ & $\lnot p$ & $(\lnot p)\lor q$ & $p\Rightarrow q$\\
\hline\hline
T & T & F & T & T \\
T & F & F & F & F \\
F & T & T & T & T \\
F & F & T & T & T \\
\hline
\end{tabular}
\end{center}
and the logical equivalence follows because the last two columns coincide. 
\end{solution}
}

\begin{exercise}
Write down (i) a tautology statement,
and (ii) a contradiction statement. 
\end{exercise}
\opt{show}{
\begin{solution}
Let $p$ be any statement. 
(i): 
Then $p \lor (\lnot p)$ is certainly
always T, indeed (arguing by cases): 
If $p$ is T, then 
$p \lor (\lnot p) = T \lor (\lnot T) = T \lor F = T$.
And if $p$ is F, then
$p \lor (\lnot p) = F \lor (\lnot F) = F \lor T = T$.
A concrete instance is ``every integer is even or odd''. 
(ii): Similarly to (i), consider
$p \land (\lnot p)$, which is always F;
concretely, ``every integer is even and odd''. 
\end{solution}
}

\begin{exercise}[contrapositive]
\label{exo:contrapositivetruth}
Let $p$ and $q$ be statements.
Compute the truth table for $(\lnot q)\Rightarrow (\lnot p)$ and compare to the truth
table of $p\Rightarrow q$; deduce therefore that these latter two
statements are logically equivalent. 
\end{exercise}
\begin{solution}
Indeed, 
\begin{center}
\newcolumntype{C}{>{\centering\arraybackslash}m{1.5cm}<{}}
\begin{tabular}{|C|C|C|C|C|C|}
\hline
$p$ & $q$ & $\lnot p$ & $\lnot q$ & $\lnot q \Rightarrow \lnot p$ & $p
\Rightarrow q$\\
\hline\hline
T & T & F & F & T & T\\
T & F & F & T & F & F\\
F & T & T & F & T & T\\
F & F & T & T & T & T\\
\hline
\end{tabular}
\end{center}
and the logical equivalence follows because the last two column coincide.
\end{solution}

\begin{exercise}[proof by contradiction]
Let $p$ and $q$ be statements.
Compute the truth table for $(p \land (\lnot q))\Rightarrow F$ 
and compare to the truth
table of $p\Rightarrow q$; deduce therefore that these latter two
statements are logically equivalent. 
\end{exercise}
\begin{solution}
Indeed, 
\begin{center}
\newcolumntype{C}{>{\centering\arraybackslash}m{1.6cm}<{}}
\begin{tabular}{|C|C|C|C|C|C|}
\hline
$p$ & $q$ & $\lnot q$ & $p \land (\lnot q)$ & \shortstack[c]{$(p \land (\lnot q))$ \\ $\Rightarrow$ F} & $p \Rightarrow q$\\
\hline\hline
T & T & F & F & T & T\\
T & F & T & T & F & F\\
F & T & F & F & T & T\\
F & F & T & F & T & T\\
\hline
\end{tabular}
\end{center}
and the logical equivalence follows because the last two column coincide.
\end{solution}

\begin{exercise}[proof by cases]
Let $p$, $q$, and $r$ be statements.
Compute the truth tables for 
$(p \lor q) \Rightarrow r$
and $(p \Rightarrow r)\land (q\Rightarrow r)$;
deduce therefore that these latter two
statements are logically equivalent. 
\end{exercise}
\begin{solution}
We show the tables using \texttt{WolframAlpha}.
\begin{center}
\includegraphics[width=1.0\textwidth]{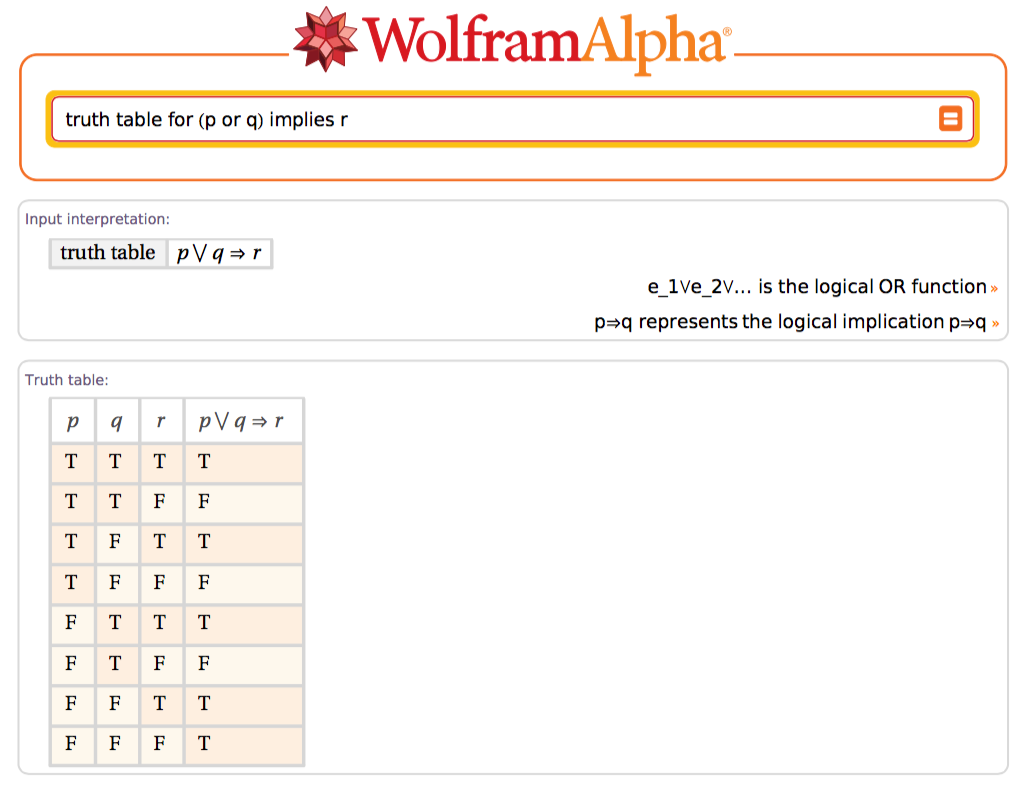}
\end{center}
\begin{center}
\includegraphics[width=1.0\textwidth]{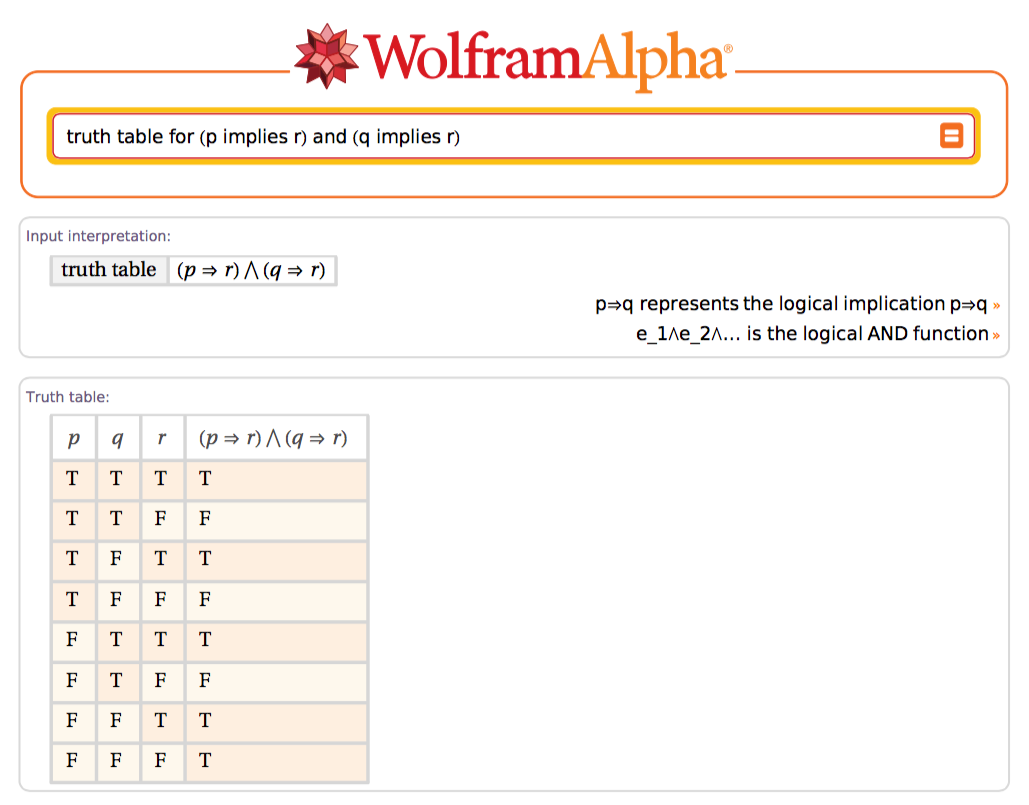}
\end{center}
\end{solution}


\begin{exercise}
Show that if $n$ is an integer, then
the remainder of $n^2$ after division by $4$ is
either $0$ or $1$. 
\emph{Hint:} Argue by cases.
\end{exercise}
\begin{solution}





\emph{Case~1:} $n$ is even, say $n=2k$, where $k$ is another
integer.\\
Then
$n^2 = (2k)^2 = 4k^2$ has remainder $0$ when dividing by $4$.

\emph{Case~2:} $n$ is odd, say $n=2k+1$, where $k$ is another
integer.\\
Then
$n^2 = (2k+1)^2 = (2k)^2+2(2k)+1^2 = 4k^2+4k+1 = 4(k^2+k)+1$ 
has remainder $1$ when dividing by $4$.
\end{solution}

\begin{exercise}
Below we have written down a proof of a statement.
Determine what the statement is. 
\begin{quotation}
\begin{proof}
Write $m=2k$ and $n=2l$, where $k$ and $l$ are integers.
Then $m+n = 2k+2l=2(k+l)$ is also even.
\end{proof}
\end{quotation}
\end{exercise}
\begin{solution}
The statement was: ``The sum of two even integers is even.''
\end{solution}

\begin{exercise}
Show that the sum of an odd and an even integer is odd.
\end{exercise}
\begin{solution}
Write $m=2k$ and $n=2l+1$, where $k$ and $l$ are integers
so that $m$ represents the even and $n$ the odd integer.
Then $m+n = 2k+2l+1=2(k+l)+1$ is indeed odd.
\end{solution}

\begin{exercise}
\label{exo:140818a}
Provide two examples where
$p\Rightarrow q$ is true but the
inverse $\lnot p \Rightarrow \lnot q$ is false.
\end{exercise}
\begin{solution}
First example:
Let $p$ be the statement
``The function of $f$ is differentiable.''
and  let $q$ be the statement 
``The function of $f$ is continuous.''
From Calculus~I, we know that $p\Rightarrow q$ is T. 
The inverse, however, is false:
E.g., when $f = |\cdot|$ is the absolute value,
then $f$ is continuous (so $q$ is T, so $\lnot q$ is F)
and $f$ is not differentiable (so $p$ is F and $\lnot p$ is T).
Then $(\lnot p)\Rightarrow (\lnot q) = (\lnot(\lnot p))\lor (\lnot q)
= (\lnot T)\lor F = F \lor F = F$.
Or put differently, the inverse of $q$ is the contrapositive of
$q\Rightarrow p$, and we know that continuity does not imply
differentiability. 

Second example:
Let $p$ be the statement
``$x\geq 0$''
and  let $q$ be the statement 
``$x^2\geq 0$''.  
Since $x^2\geq 0$ does not imply $x\geq 0$ (consider $x=-1$),
we get another example.

More generally, you can take any true implication $p\Rightarrow q$,
where $q\Rightarrow p$ is false. 
\end{solution}

Suppose we are interested in proving the implication
$p\Rightarrow q$. 
One may start with $q$, manipulate, and end up with $p$.
In this case, we have shown that $q\Rightarrow p$.
If we are lucky, all steps are \emph{reversible} in which
case one has proven $p\Leftrightarrow q$ and in particular 
$p\Rightarrow q$. This common technique is also called
\textbf{working backwards}. 

\begin{exercise}
In a round robin tournament with $n$ players $P_1,\ldots,P_n$,
where $n\in\{2,3,\ldots\}$, each player plays exactly one game
with each of the other players. Assume that each game ends in a
win or loss (draws cannot occur). 
Denote by $W_i$ and $L_i$ be the number of games won and lost, 
respectively, by player $P_i$. Show that
\begin{equation*}
W_1^2 + W_2^2 + \cdots + W_n^2 = L_1^2 + L_2^2 + \cdots + L_n^2.
\end{equation*}
\emph{Hint:} Work backwards to derive a statement that is true.
If all your steps are reversible, you have found a proof.
Write down a clean proof that does not work backwards. 
\end{exercise}
\opt{show}{
\begin{solution}
We work backwards. 
We will use $\Sigma$ notation for the sums, and simply write
$\Sigma$ to indicate a summation from $i=1$ to $i=n$. 
Suppose that 
\begin{equation*}
\sum W_i^2 = \sum L_i^2.
\end{equation*}
Then, since $W_i+L_i = n-1$, 
\begin{align*}
&\qquad \sum (W_i^2 - L_i^2) = 0\\
&\Rightarrow \sum (W_i+L_i)(W_i-L_i) = 0 \\
&\Rightarrow \sum (n-1)(W_i-L_i) = 0 \\
&\Rightarrow \sum (W_i-L_i) = 0 \\
&\Rightarrow \sum W_i = \sum L_i.
\end{align*}
Now the last equation is true since the total number of won games is
the same as the total number of lost games.
Moreover, all implications are \emph{reversible}. 
Hence, we are led to a proof by working backwards. 

\texttt{Clean Proof:}
Observe that each player plays $n-1$ games so that
$W_i+L_i = n-1$. Also observe that 
the total number of won games is
the same as the total number of lost games: $\sum W_i = \sum
L_i$. Hence
$\sum (W_i-L_i)=0$.
Thus $0 = (n-1)\sum(W_i-L_i) = 
\sum(n-1)(W_i-L_i) = \sum(W_i+L_i)(W_i-L_i)
=\sum(W_i^2-L_i^2) = (\sum W_i^2) - (\sum L_i^2)$.
Therefore, $\sum W_i^2 = \sum L_i^2$.
\end{solution}
}

\begin{exercise}
Show that there is no function $f$ from the
set $\{0,1\}$ to the set $\{0,1\}$ such that\footnote{Functions
of this type play a role in \emph{Quantum Information Theory}.}
for every $x\in\{0,1\}$, we have
$(f\circ f)(x) = f(f(x)) = 1-x$.
\emph{Hint:} Argue by contradiction. 
\end{exercise}
\opt{show}{
\begin{solution}
We argue by contradiction.
Suppose to the contrary that a function
$f$ has the property that 
\begin{equation}
\label{e:160826a}
f\big(f(x)\big) = 1- x
\end{equation}
for every $x\in\{0,1\}$. 
There are only two conceivable alternatives:

\emph{Alternative~1:} $f(0)=0$.\\
On the one hand,
$f(f(0))=f(0)=0$.
On the other hand,
$f(f(0))=1-0 = 1$ by \eqref{e:160826a}.
Altogether, we get the absurdity that $0 = f(f(0))=1$. 

\emph{Alternative~2:} $f(0)=1$.\\
Then, by \eqref{e:160826a},
$1 = 1-0 = f(f(0)) = f(1)$ and so $f(1)=1$. 
Applying $f$, we get $f(f(1))=f(1)=1$.
On the other hand, using again \eqref{e:160826a},
$f(f(1))=1-1=0$.
Altogether, we get the absurdity that $f(f(1))$ is both $0$ and
$1$. 

To conclude, we obtained  contradictions for all alternatives.
Therefore, no such function $f$ exists. 
\end{solution}
}

\begin{exercise}
\label{exo:160826b}
Let $f\colon [0,1]\to[0,1]$ be such that
$f(f(x))=1-x$. 
Show that 
(i) $f(1/2)=1/2$;
(ii) $0<f(0)<1$; and
(iii) $0<f(1)<1$. 
\end{exercise}
\opt{show}{
\begin{solution}
(i):
We have $f(f(1/2))=1-1/2 = 1/2$.
Hence $f(1/2)=f(f(f(1/2))) = 1-f(1/2)$
and thus
$2 f(1/2)=1$, i.e., $f(1/2)=1/2$. 

(ii):
If $f(0)=0$, then $0=f(0)=f(f(0))=1-0=1$ which is absurd. 
Hence $0<f(0)$. 
Assume to the contrary that $f(0)=1$.
Then $1-0 = f(f(0))=f(1)$ and thus $f(1)=1$.
We have $f(0)=1=f(1)$.
Applying $f$ yields
$1-0=f(f(0)) = f(f(1))=1-1=0$ which is absurd.
Hence $f(0)<1$. 

(iii): This is similar to the proof of (ii):
If $f(1)=1$, then $1=f(1)=f(f(1))=1-1=0$ which is absurd.
Hence $f(1)<1$.
Assume to the contrary that $f(1)=0$.
Then $0=1-1 = f(f(1))=f(0)$ and so $f(0)=0$.
Hence $f(0)=0=f(1)$.
Applying $f$ yields
$1-0 = f(f(0))=f(f(1))=1-1$ which is absurd. 
Hence $f(1)>0$.
\end{solution}
}

\begin{exercise}[NAND]
\label{exo:nand}
Let $p$ and $q$ be statements.
Compute the truth table for $\lnot(p \land q)$,
which is known as \textbf{negative-AND} or \textbf{NAND}
in Electrical Engineering. 
Compare to the truth
table of $(\lnot p)\lor(\lnot q)$.
Are the truth tables the same?
\end{exercise}
\opt{show}{
\begin{solution}
Indeed, computing the truth tables 
\begin{center}
\newcolumntype{C}{>{\centering\arraybackslash}m{1.6cm}<{}}
\begin{tabular}{|C|C|C|C|C|C|}
\hline
$p$ & $q$ & $\lnot p$ & $\lnot q$ & 
$\lnot(p \land q)$ & $(\lnot p)\lor(\lnot q)$ \\
\hline\hline
T & T & F & F & F & F\\
T & F & F & T & T & T\\
F & T & T & F & T & T\\
F & F & T & T & T & T\\
\hline
\end{tabular}
\end{center}
we see they are the same because the last two columns coincide.
\end{solution}
}

\begin{exercise}[XOR]
\label{exo:xor}
Let $p$ and $q$ be statements.
Compute the truth table for $(p\land \lnot q)\lor(q\land \lnot p)$,
which is known as \textbf{exclusive-OR} or \textbf{XOR}
in Electrical Engineering and Computer Science. 
How would you express this operation in plain English?
\end{exercise}
\opt{show}{
\begin{solution}
Indeed, computing the truth tables 
\begin{center}
\newcolumntype{C}{>{\centering\arraybackslash}m{1.6cm}<{}}
\begin{tabular}{|C|C|C|C|C|C|C|}
\hline
$p$ & $q$ & $\lnot p$ & $\lnot q$ & 
$p \land \lnot q$ & $q \land \lnot p$ & $p$ XOR $q$ \\
\hline\hline
T & T & F & F & F & F & F \\
T & F & F & T & T & F & T \\
F & T & T & F & F & T & T \\
F & F & T & T & F & F & F \\
\hline
\end{tabular}
\end{center}
In plain English, we would say
``Either $p$ or $q$ but not both''. 
\end{solution}
}

\begin{exercise}[XOR swap]
\label{exo:xorswap}
Let $p$ and $q$ be variables that take on values T (TRUE) or F (FALSE).
Now imagine you do the following operations in this order:
1) $p := p\; \text{XOR}\; q$;
2) $q := p\; \text{XOR}\; q$;
3) $p := p\; \text{XOR}\; q$,
where ``XOR'' is the exclusive-OR operator discussed in Exercise~\ref{exo:xor}. 
(So $p$ gets overwritten twice, and $q$ gets overwritten once.)
Show that executing these 3 operations provides a way
of swapping the contents of the variables $p$ and $q$,
\emph{without the need of a temporary intermediate variable}. 
(This trick was shown
to me by the late Dr.~Alan Paeth, a former colleague in Computer Science.)
You can prove this either by considering cases 
or by considering a truth table.
\end{exercise}
\begin{solution}
We work through the following table column-wise, from left-to-right,
and we write ``X'' for ``XOR'' for brevity:
\begin{center}
\newcolumntype{C}{>{\centering\arraybackslash}m{2.2cm}<{}}
\begin{tabular}{|C|C|C|C|C|}
\hline
$p$ & $q$ & $p := p$ XOR $q$ & $q := p$ XOR $q$ & $p := p$ XOR $q$ \\
\hline\hline
T & T & F & T & T \\
T & F & T & T & F \\
F & T & T & F & T \\
F & F & F & F & F \\
\hline
\end{tabular}
\end{center}
So we see that applying 3 times XOR in this fashion
amounts to swapping the contents of $p$ and $q$. 
\end{solution}

\begin{exercise} 
Let us revisit Exercise~\ref{exo:200713wrongproof1} which 
concerns the statement
\begin{equation*}
\text{If $a$ and $b$ are integers such that $ab$ is even, 
then $a$ or $b$ is even.}
\end{equation*}
Prove this statement by proving the \emph{contrapositive}.
\end{exercise}
\begin{solution}
Let $p$ be the statement ``$a$ and $b$ are integers such that $ab$ is even''
and let $q$ be the statement 
``$a$ and $b$ are integers and at least one of them is even''.
Our job is to prove that $p\Rightarrow q$.

Note that 
$\lnot p$ is the statement ``$a$ and $b$ are integers such that $ab$ is odd'' 
and 
$\lnot q$ is the statement 
``$a$ and $b$ are integers which are both odd''.

Proving our result by contrapositive means we 
prove the logically equivalent 
$\lnot q \Rightarrow \lnot p$, i.e., 
\begin{equation*}
\text{If $a$ and $b$ are integers which are both odd,
then $ab$ is odd.}
\end{equation*}
Let's do this!
Assume that $a$ and $b$ are both odd,
say $a=2k+1$ and $b=2l+1$ for two integers $k$ and $l$.
Then 
\begin{align*}
ab &= (2k+1)(2l+1)
= 4kl+2k+2l+1
= 2(2kl+k+l)+1
\end{align*}
is indeed odd because $2kl+k+l$ is an integer. 
\end{solution}

\begin{exercise} 
Let us revisit Exercise~\ref{exo:200713wrongproof1} which 
concerns the statement
\begin{equation*}
\text{If $a$ and $b$ are integers such that $ab$ is even, 
then $a$ or $b$ is even.}
\end{equation*}
Prove this statement by a \emph{proof by contradiction}.
\end{exercise}
\begin{solution}
Let $p$ be the statement ``$a$ and $b$ are integers such that $ab$ is even''
and let $q$ be the statement 
``$a$ and $b$ are integers and at least one of them is even''.
Our job is to prove that $p\Rightarrow q$.

Note that 
$\lnot q$ is the statement 
``$a$ and $b$ are integers which are both odd''.

Proving our result by contradiction means we 
assume $p$ and $\lnot q$, and then we aim for a contradiction: 
\begin{align*}
&\text{Assume that $a$ and $b$ are integers such that 
$ab$ is even,}
\text{and $a$ and $b$ both odd.}\\
&\text{Find a contradiction.}
\end{align*}
Let's do this!
Assume that $a$ and $b$ are both odd,
say $a=2k+1$ and $b=2l+1$,
and that $ab = 2n$ is even, 
for three integers $k,l,n$. 
Then $ab = a\cdot b$ is just 
\begin{equation*}
2n = (2k+1)(2l+1).
\end{equation*}
Expanding the right-hand side yields 
\begin{equation*}
2n = 4kl+2k+2l+1.
\end{equation*}
Dividing by $2$ gives 
\begin{equation*}
n = 2kl+k+l+\tfrac{1}{2}
\end{equation*}
which is absurd because 
on the LHS we have an integer $n$,
while on the RHS we have an integer, 
$2kl+k+l$, to which we add $\tfrac{1}{2}$,
so the RHS is not an integer!
This contradiction proves the result.
\end{solution}
    
\begin{exercise} 
Let us revisit Exercise~\ref{exo:200713wrongproof1} which 
concerns the statement
\begin{equation*}
\text{If $a$ and $b$ are integers such that $ab$ is even, 
then $a$ or $b$ is even.}
\end{equation*}
Prove this statement by a \emph{proof by cases}.
Do two cases: (i) $a$ is even; (ii) $a$ is odd. 
\end{exercise}
\begin{solution}
We assume that $a$ and $b$ are integers such that 
$ab$ is even, say 
\begin{equation*}
    ab = 2n,
\end{equation*}
for some integer $n$.

\emph{Case~1:} $a$ is even.\\
Then we are done! (Almost feels like cheating but it isn't! 
The hard work is in \emph{Case~2}.)

\emph{Case~2} $a$ is odd, say $a = 2k+1$, for some integer $k$.\\
Then we write $ab=a\cdot b$:
\begin{equation*}
2n = (2k+1)b.
\end{equation*}
Hence $2n = 2kb+b$; equivalently,
$2n-2kb = b$ and also equivalently, 
\begin{equation*}
b = 2(n-kb).
\end{equation*}
Because $n-kb$ is an integer, the number $b$ is even and we are done!
\end{solution}

\begin{exercise} 
Prove the following statement
\begin{equation*}
\text{If $n$ is an integer, then $n^2+n+1$ is odd.}
\end{equation*}
by considering two cases (i) $n$ is even; (ii) $n$ is odd. 
\end{exercise}
\begin{solution}
\emph{Case~1:} $n$ is even, say $n=2k$ for some integer $k$.\\
Then 
\begin{align*}
n^2 + n +1 
&=(2k)^2 + (2k) + 1 
=4k^2 + 2k+1
= 2(2k^2+k)+1
\end{align*}
is odd because $2k^2+k$ is an integer. 

\emph{Case~2} $n$ is odd, say $n = 2k+1$, for some integer $k$.\\
Then 
\begin{align*}
n^2 + n +1 
&=(2k+1)^2 + (2k+1) + 1 
=(4k^2+4k+1) + (2k+1) + 1\\
&= \big(2(2k^2+2k)+1\big) + (2k + 1) + 1
=  2(2k^2+3k+1) + 1
\end{align*}
is odd because $2k^2+3k+1$ is an integer. 
\end{solution}

\begin{exercise} 
Each of the following statements is \emph{false}.
To show this, provide a counterexample for each statement: 
\begin{enumerate}
    \item ``If $n$ is an even integer, then $n^2$ is odd.''
    \item ``If $n$ is an integer that is odd, then $n$ is divisible by $3$.''
    \item ``If $p$ is an odd prime number, then $p+2$ is also a prime number.''
    \item ``If $p$ is an odd prime number, then $p+2$ is not a prime number.''
\end{enumerate}
\end{exercise}
\begin{solution}
(i): $n=2$ is even but $n^2=2^2=4$ is still even.
(In fact, any even $n$ is a counterexample.)
(ii): $n=5$. 
(iii): $p=7$. 
(iv): $p=3$. 
\end{solution}

\begin{exercise}[YOU be the marker!]
\label{exo:200713wrongproof1}
Consider the following statement:
\begin{equation}
\text{If $a$ and $b$ are integers such that $ab$ is even, 
then $a$ or $b$ is even.}
\end{equation}
Explain why the following ``Proof'' does not work: 
\begin{quotation}
\emph{Step~1}:
Suppose that $a$ is even, say $a=2k$, where $k$ is some integer.\\
(The case when $b$ is even is analogous.)\\
\emph{Step~2}:
Then $ab = (2k)b = 2(kb)$ is even because $kb$ is an integer
and we are done. 
\end{quotation}
Which result does the above wrong proof actually prove?
\end{exercise}
\begin{solution}
Let $p$ be the statement ``$a$ and $b$ are integers such that $ab$ is even''
and let $q$ be the statement 
``$a$ and $b$ are integers and at least one of them is even''.
Our job is to prove that $p\Rightarrow q$.

The above ``proof'' does not prove this as it starts with the desired 
conclusion, namely $q$, and deduces the desired hypothesis, namely $p$.
In other words, the wrong ``proof'' is actually a correct proof
for the statement:
\begin{equation*}
\text{If $a$ and $b$ are integers such that $a$ or $b$ is even, 
then $ab$ is even.}
\end{equation*}
\end{solution}

\begin{exercise}[TRUE or FALSE?]
Each of the following statements is either T or F.
Mark each statement either T or F, and explain briefly why.
\begin{enumerate}
\item
$p\Rightarrow q$ is logically equivalent to $\lnot p \Rightarrow \lnot q$.
\item
$p\Rightarrow q$ is logically equivalent to $\lnot q \Rightarrow \lnot p$.
\item
$p\Rightarrow q$ is logically equivalent to $q \Rightarrow p$.
\end{enumerate}
\end{exercise}
\begin{solution}
(i): F. See Exercise~\ref{exo:140818a} for counterexamples. 
(ii): T. See Exercise~\ref{exo:contrapositivetruth}.  
(iii): F. See Exercise~\ref{exo:140818a} for counterexamples. 
\end{solution}

\begin{exercise}[TRUE or FALSE?]
Each of the following statements is either true or false. 
If it is true, then briefly explain why. 
If it is false, then provide a counterexample (do not ``correct'' 
the wrong statement).
\begin{enumerate}
    \item ``If $a$ is an even integer, then $a^2$ is even.''
    \item ``If $p$ is prime, then $p^2+1$ is a prime.''
    \item ``If $a$ and $b$ are odd integers, then $ab$ is also odd.''
    \item ``If $p$ is a prime number, then the set 
    $\{p+1,p+2,p+3,p+4,p+5\}$ also contains a prime number.''
\end{enumerate}
\end{exercise}
\begin{solution}
(i): TRUE: because $a=2k$ $\Rightarrow$ $a^2 = (2k)^2 = 4k^2 = 2(2k^2)$.

(ii): FALSE: E.g., $p=3$ is prime but $p^2+1=10$ is not. 

(iii): TRUE: [$a=2k+1$ and $b=2l+1$] $\Rightarrow$ 
$ab = (2k+1)(2l+1) = 4kl+2k+2l+1 = 2(2kl+k+l)+1$. 

(iv): FALSE: $p=23$ is prime, but the set 
$\{24,25,26,27,28\}$ contains no primes. 
\end{solution} 
\chapter{The Principle of Mathematical Induction}
\label{cha:Induction}

\index{Principle of Mathematical Induction}
\index{Induction}

\section{Statement and Examples}

The {Principle of Mathematical Induction} is one of the most important
proof techniques in Mathematics. It is used to verify the truth of
statements $S(n)$ for \emph{all} integers $n$ greater than or equal to some
smallest integer of interest, say $n_0$. 
This is remarkable because it is conceptionally hard to check infinitely
many statements.

\begin{fact}[Principle of Mathematical Induction]
\label{f:induction}
Let $n_0$ be an integer and let $S(n)$ be a statement, 
formulated  for all integers $n\geq n_0$. 
Suppose we are able to verify the following.
\begin{enumerate}
\item \textbf{Base Case}: $S(n_0)$ is true.
\item \textbf{Inductive Step}: If $n$ is an integer $\geq n_0$ and
$S(n)$ is true, then $S(n+1)$ is also true.
\end{enumerate}
Then the statement $S(n)$ is true for \emph{all} integers $n\geq n_0$.
\end{fact}

\hhbcom{Here in class one can talk about dominos,
ladders, etc.}

In practical applications of Fact~\ref{f:induction}, the smallest integer
of interest, $n_0$, is often either $0$ or $1$.
Let us denote the set of all integers $\{0,1,-1,2,-2,3,-3,\ldots\}$ 
\label{def:ZZ}
by $\ZZ$, 
and the set of nonnegative
integers $\{0,1,2,\ldots\}$ \label{def:NN} by $\NN$:
\begin{equation}
\ZZ = \{0,1,-1,2,-2,3,-1,\ldots\}
\;\;\text{and}\;\;
\NN = \{0,1,2,3,\ldots\}.
\end{equation}

We start with a classical example.

\begin{theorem}
\label{t:Gauss}
For every integer $n\geq 1$, one has
\begin{equation}
\label{e:100321:a}
1+2+\cdots + n = \frac{n(n+1)}{2}.
\end{equation}
\end{theorem}
\begin{proof}
We prove this result by using the principle of mathematical induction,
with $n_0 = 1$. The statement $S(n)$ is \eqref{e:100321:a}. 

\textbf{Base Case}: When $n=1$,
the claimed identity \eqref{e:100321:a} reads
\begin{equation}
1 = \frac{1(1+1)}{2},
\end{equation}
which is obviously true.

\textbf{Inductive Step}:
Now suppose that $n\in \ZZ$, that $n\geq 1$, and that 
$S(n)$ is true, i.e., that 
\eqref{e:100321:a} is true for this $n$.
Our job is to show that $S(n+1)$ is true, i.e.,
that \eqref{e:100321:a} is true when $n$ is replaced by
$n+1$. That is, we must prove that
\begin{equation}
\label{e:100321:b}
1+2+\cdots + n + (n+1) \stackrel{?}{=} \frac{(n+1)(n+2)}{2}.
\end{equation}
This seems even harder than verifying $S(n)$ but remember that
we assume that $S(n)$ is true, so we can put it to good use!
Indeed, using the assumption that $S(n)$ is true --- sometimes this is
referred to as \textbf{using the induction hypothesis} --- we see that
\begin{subequations}
\begin{align}
1+2+\cdots + n + (n+1) &= \big( 1+2+\cdots + n\big) + (n+1)\\
&= \frac{n(n+1)}{2} + (n+1)\\
&= (n+1)\left( \frac{n}{2} + 1\right)\\
&= \frac{(n+1)(n+2)}{2}.
\end{align}
\end{subequations}
Hence \eqref{e:100321:b} is verified, i.e., $S(n+1)$ is true.
Therefore, by the Principle of Mathematical Induction
(Fact~\ref{f:induction}), the formula \eqref{e:100321:a} is true for
every integer $n\geq 1$.
\end{proof}

\begin{remark}
\label{r:Gauss}
Theorem~\ref{t:Gauss} is related to a well known 
story involving the young
Carl Friedrich Gauss, \index{Gauss, C.F.} 
who amazed his Elementary School teacher when he
almost instantaneously added the numbers from $1$ to $100$ in his head and
announced the result to be $5050$. 
While it is not likely that Gauss used
Theorem~\ref{t:Gauss}, he could have 
combined the first and last, the second and second-to-last etc.\ terms, 
so that 
\begin{subequations}
\begin{align}
1+ 2+ \cdots + 100 &= (1+100) + (2+99) + \cdots + (49+52) + (50+51)\\
&= 50\cdot 101 = 5050.
\end{align}
\end{subequations}
\end{remark}

The well-known \textbf{sigma notation}, 
\label{def:sigma-notation}
encountered already in Calculus~II,
allows us to write statements such as the one arising in
Theorem~\ref{t:Gauss} more concisely:
Let $m$ and $n$ be in $\ZZ$ such that $m \leq n$.
Let $\alpha_k$ be a real number, 
for every integer $k$ such that $m\leq k\leq n$.
One then defines\footnote{``$:=$'' means that the left side is
defined by the right side --- as in \texttt{Maple}!}
\begin{equation}
\sum_{k=m}^{n} \alpha_k := \alpha_m + \alpha_{m+1} + \cdots + \alpha_n,
\end{equation}
which we can also write as 
\begin{equation}
\sum_{k=m}^{n} \alpha_k = \sum_{k=m+1}^{n+1} \alpha_{k-1} 
= \sum_{k=m-1}^{n-1} \alpha_{k+1}.
\end{equation}
It is convenient to set
\begin{equation}
\sum_{k=m}^{n} \alpha_k := 0, \quad\text{when $m>n$;}
\end{equation}
this is the so-called \textbf{empty sum convention}. 
\index{Empty sum convention} 
Note that we can now rewrite (and slightly 
extend\footnote{We now also cover the case
when $n=0$, in which case the identity turns into $0=0$.})
Theorem~\ref{t:Gauss} as 
\begin{equation}
\sum_{k=1}^{n} k = \frac{n(n+1)}{2}, \quad\text{for every $n\in\NN$.}
\end{equation}

Let us now consider the related problem of adding the first few \emph{odd}
integers:
\begin{subequations}
\begin{align}
1 &= 1,\\
1+3 &= 4,\\
1+3+5 &=9,\\
1+3+5+7 & = 16.
\end{align}
\end{subequations}
We recognize that the sums as \emph{squares}:
$1=1^2$, $4=2^2$, $9=3^2$, $16=4^2$, and are motivated
to prove the following result.

\begin{theorem}
\label{t:odd}
For every integer $n\geq 1$, one has
\begin{equation}
\label{e:100322:a}
\sum_{k=1}^{n} (2k-1) = n^2.
\end{equation}
\end{theorem}
\begin{proof}
We prove this result by mathematical induction, with $n_0=1$ and
where the statement $S(n)$ is \eqref{e:100322:a}. 

\textbf{Base Case}: When $n=1$, \eqref{e:100322:a} turns into
\begin{equation}
\sum_{k=1}^{1} (2k-1) = (2\cdot 1-1) = 1 = 1^2,
\end{equation}
which is clearly true.

\textbf{Inductive Step}: Now suppose that $n\in\NN$ is such that
$n\geq 1$ and $S(n)$ is true, i.e., \eqref{e:100322:a} holds for this $n$.
Our goal is to show that $S(n+1)$ is true, i.e., that 
\eqref{e:100322:a} holds when $n$ is replaced by $n+1$. That is,
we must prove that
\begin{equation}
\label{e:100322:b} 
\sum_{k=1}^{n+1} (2k-1) \stackrel{?}{=} (n+1)^2.
\end{equation}
Indeed, using the \textbf{induction hypothesis} that $S(n)$ is true in 
\eqref{e:100322:c}, we see that
\begin{subequations}
\begin{align}
\sum_{k=1}^{n+1} (2k-1) &= \left(\sum_{k=1}^{n} (2k-1)\right) + 
\big(2(n+1)-1\big)\\
&= n^2 + (2n+1)\label{e:100322:c}\\
&= (n+1)^2.
\end{align}
\end{subequations}
Hence \eqref{e:100322:b} is verified, i.e., $S(n+1)$ is true.
By the Principle of Mathematical Induction, the formula \eqref{e:100322:a}
is true for every integer $n\geq 1$. 
\end{proof}

It is important to note that Mathematical Induction is a great tool
for verifying results that we believe to be true --- it does not help us to
discover these results: in both Theorem~\ref{t:Gauss} and
Theorem~\ref{t:odd}, we had to have a good grip in what the answer is
before we were able to start the verification. 

Mathematical Induction is not only useful for proving identities --- it is
also useful for verifying inequalities or other mathematical statements
such as the next result which is rooted in Number Theory.

Let $a$ and $b$ be in $\ZZ$. Recall that \textbf{$a$ divides $b$}, written
as $a|b$, if $a\neq 0$ and $b/a\in\ZZ$ so that $b = k\cdot a$, where
$k\in\ZZ$. Thus, $4|12$ but $4$ does not divide
$13$ since $13/4 = 3\tfrac{1}{4}\notin\ZZ$.

\begin{theorem}
\label{t:divide}
For every integer $n\geq 0$, one has
\begin{equation}
\label{e:100322:d}
3 | (7^n-4^n).
\end{equation}
\end{theorem}
\begin{proof}
Again, we prove this result by mathematical induction, with $n_0=0$ and
where the statement $S(n)$ is \eqref{e:100322:d}. 

\textbf{Base Case}:
When $n=0$, \eqref{e:100322:d} turns into $3|(7^0-4^0)$, i.e., into
$3|0$ which is obviously true.

\textbf{Inductive Step}:
Now suppose that $n\in\ZZ$ is such that $n\geq 0$ and $S(n)$ is true, i.e.,
\eqref{e:100322:d} holds. 
Our job is to show that $S(n+1)$ is true, i.e.,
\begin{equation}
\label{e:100322:e}
3 \stackrel{?}{|} (7^{n+1}-4^{n+1}).
\end{equation}
Now since $S(n)$ is true, we know there exists $k\in\ZZ$ such that
$7^n-4^n = 3\cdot k$. 
Thus, using this \textbf{induction hypothesis} in 
\eqref{e:100322:f}, we see that
\begin{subequations}
\begin{align}
7^{n+1}-4^{n+1} &= 7^{n+1} - 4\cdot 7^n + 4(7^n - 4^n)\\
&= 7^n(7-4) + 4(3\cdot k)\label{e:100322:f}\\
&= 3\big(7^n+4\cdot k\big).
\end{align}
\end{subequations}
Since $7^n+4k\in\ZZ$, it is clear that \eqref{e:100322:e} is true.
Thus, by the principle of mathematical induction, we see that
\eqref{e:100322:d} is true for every integer $n\geq 0$. 
\end{proof}

\begin{theorem}[Geometric Sum]
\index{Geometric sum} 
\label{t:geosum}
Let $x\neq 1$ and let $\nnn$.
Then
\begin{equation}
\label{e:geosum}
\sum_{k=0}^{n}x^k = \frac{1-x^{n+1}}{1-x}.
\end{equation}
\end{theorem}
\begin{proof}
By Mathematical Induction.

\textbf{Base Case}: $n=0$.
Then \eqref{e:geosum} clearly holds since each side equals $1$.

\textbf{Inductive Step}: Assume that $\nnn$ is such that
\eqref{e:geosum} holds. We must show that \eqref{e:geosum} also holds with
$n$ replaced by $n+1$. 
Indeed, using the \textbf{induction hypothesis} in \eqref{e:100406:e}, 
we obtain 
\begin{subequations}
\begin{align}
\sum_{k=0}^{n+1} x^k &= \left(\sum_{k=0}^n x^k\right) + x^{n+1}\\
&= \frac{1-x^{n+1}}{1-x} + x^{n+1} \label{e:100406:e}\\
&= \frac{1-x^{n+1}+x^{n+1}-x x^{n+1}}{1-x} =
\frac{1-x^{(n+1)+1}}{1-x},
\end{align}
\end{subequations}
as required.
\end{proof}

\section{Factorials and the Binomial Theorem}

\begin{definition}[Products and Factorial Function]
Let $m$ and $n$ be in $\ZZ$ such that $m\leq n$, and let $\alpha_k$ be a
real number, for every integer $k$ such that $m\leq k\leq n$. 
Analogously to the sigma notation for sums, we define
\label{def:Pi-notation} 
\begin{equation}
\prod_{k=m}^{n} \alpha_k = \alpha_m\cdot\alpha_{m+1}\cdots
\alpha_n.
\end{equation}
When $m>n$, we adopt the \textbf{empty product convention}
\index{Empty product convention} 
$\prod_{k=m}^{n} \alpha_k := 1$,
which is motivated by the fact that multiplication by $1$ does not change
anything, just like addition of $0$. 
The product of the first $n$ strictly positive integers is called
\textbf{``$n$ factorial''} and denoted by $n!$:
\label{def:factorial}
\index{Factorial} 
\begin{equation}
n! := \prod_{k=1}^{n} k = 1\cdot 2 \cdots n.
\end{equation}
\end{definition}

We thus have $0!=1$, $1!=1$, $2!=2$, $3!=6$, $4!=24$, \ldots,
$(n+1)! = (n+1)\cdot n!$. 
The factorial function occurs frequently when dealing with combinatorial
questions such as the following: In how many different ways can we arrange
the 3 distinct letters $A,B,C$? Answer: $6 = 3!$, namely:
$ABC, ACB, BAC, BCA, CAB, CBA$. 
These 3-letter words which utilize each letter
exactly once are called \textbf{permutations} of $\{A,B,C\}$. 
More generally, we are now able to prove the following result. 

\begin{theorem}[Permutations Theorem]
\index{Permutations Theorem}
\label{t:factorial}
Let $n$ be a strictly positive integer and let 
$A=\{a_1,\ldots,a_n\}$ be a set containing $n$ distinct elements. 
Then the number of all permutations of these $n$ distinct elements is $n!$.
\end{theorem}
\begin{proof}
By Mathematical Induction.

\textbf{Base Case}: $n=1$ indeed allows only one permutation, as claimed. 

\textbf{Inductive Step}: 
For $k\in\{1,\ldots,n+1\}$, 
denote the set of permutations of $\{a_1,\ldots,a_n,a_{n+1}\}$ 
with $a_k$ at the last position by $B_k$. 
Invoking the induction hypothesis, we see that 
each $B_k$ contains exactly $n!$ permutations. 
Thus the total number of permutations is
$(n+1)\cdot n! = (n+1)!$. 
\end{proof}

\begin{definition}
[Binomial Coefficients]
\label{d:choose}
\index{Binomial coefficients} 
Given $n\in\NN$ and $k\in\ZZ$, 
the \emph{binomial coefficient} 
\textbf{``$n$ choose $k$''} is defined by 
\begin{equation}
{n \choose k} := 
\begin{cases} \displaystyle 
\prod_{i=1}^k \frac{n-i+1}{i} = \frac{n(n-1)\cdots
(n-k+1)}{1\cdot 2\cdots k}, &\text{if $k\geq 0$;}\\[+4mm]
0, &\text{if $k<0$.}
\end{cases}
\end{equation}
\end{definition}

\begin{lemma}
\label{l:choose}
Let $n\in\NN$ and let $k\in\ZZ$. Then the following hold.
\begin{enumerate}
\item 
\label{l:choose:i}
$\displaystyle {n \choose 0}= {n\choose n}= 1$ and
$\displaystyle {n \choose 1} = n$.\\[+2mm]
\item
\label{l:choose:ii}
$\displaystyle {n \choose k} = 
\frac{n!}{k!(n-k)!} = {n \choose n-k}$~~~ if $0\leq k\leq n$.\\[+2mm]
\item
\label{l:choose:iii}
$\displaystyle {n \choose k} = 
0 = {n \choose n-k}$~~~ if $k<0$ or $k>n$.\\[+2mm]
\item
\label{l:choose:iv}
$\displaystyle {n \choose k} = 
{n-1 \choose k-1} + {n-1 \choose k}$~~~ if $n\geq 1$. 
\end{enumerate}
\end{lemma}
\begin{proof}
Identities \ref{l:choose:i}--\ref{l:choose:iii} 
follow readily from the definitions. 

\ref{l:choose:iv}: 
If $k\leq 0$ or $k\geq n$, then the identity is easily verified directly. 
Thus assume that $1\leq k\leq n-1$. Using \ref{l:choose:ii}, we obtain
\begin{subequations}
\begin{align}
{n \choose k} &= \frac{n!}{k!(n-k)!} 
= \frac{\big(k+(n-k)\big)(n-1)!}{k!(n-k)!} \\
&= \frac{k(n-1)!}{k!(n-k)!} + \frac{(n-k)(n-1)!}{k!(n-k)!}
= \frac{(n-1)!}{(k-1)!(n-k)!} + \frac{(n-1)!}{k!(n-k-1)!}\\
&= {n-1 \choose k-1} + {n-1 \choose k},
\end{align}
\end{subequations}
as required.
\end{proof}

\begin{remark}[Pascal's Triangle]
\label{r:pascal}
Lemma~\ref{l:choose}\ref{l:choose:iv} has the beautiful
visualization as \emph{\textbf{Pascal's Triangle}}
$$\begin{tabular}{ccccccccc}
&&&&1&&&&\\
&&&1&&1&&&\\
&&1&&2&&1&&\\
&1&&3&&3&&1&\\
1&&4&&6&&4&&1\end{tabular}$$
\index{Pascal's Triangle} 
where each number on the boundary is $1$, and each number in the interior
can be found by adding the two nearest numbers in the previous row.
\end{remark}

Expanding and simplifying yields the identities
\begin{subequations}
\begin{align}
(x+y)^0 &=1,\\
(x+y)^1 &= x+y = 1\cdot x + 1\cdot y,\\
(x+y)^2 &= x^2 + 2xy + y^2 = 1\cdot x^2 + 2\cdot xy+ 1\cdot y^2,\\
(x+y)^3 &= 1\cdot x^3 + 3\cdot x^2y + 3\cdot xy^2 + 1\cdot y^3,
\end{align}
\end{subequations}
whose coefficients bear a striking resemblance to 
Pascal's Triangle (Remark~\ref{r:pascal}). 
This is not a coincidence --- this result is 
known as the \emph{Binomial Theorem} which dates back to Euclid
to the 4th century BC. 

\begin{theorem}[Binomial Theorem]
\label{t:binomial}
Let $x$ and $y$ be real numbers, and let $n\in\NN$.
Then \index{Binomial Theorem} 
\begin{equation}
\label{e:100406:a}
(x+y)^n = \sum_{k=0}^{n} {n \choose k} x^{n-k}y^k.
\end{equation}
\end{theorem}
\begin{proof}
By Mathematical Induction.

\textbf{Base Case}:
When $n=0$, \eqref{e:100406:a} turns into
\begin{equation}
(x+y)^0 = 1 = {0\choose 0} x^0y^0.
\end{equation}

\textbf{Inductive Step}:
Now assume that $n\in\NN$ is such that \eqref{e:100406:a} holds.
We need to check that \eqref{e:100406:a} is true when $n$ is replaced by
$n+1$. 
Using the \textbf{induction hypothesis} in \eqref{e:100406:b}, 
the fact that
${n \choose n+1} = {n \choose -1} = 0$ in \eqref{e:100406:c}, 
and Lemma~\ref{l:choose}\ref{l:choose:iv} in \eqref{e:100406:d}, 
we obtain 
\begin{subequations}
\begin{align}
(x+y)^{n+1} &= (x+y)^{n}(x+y) = (x+y)^n x + (x+y)^n y\\
&= \sum_{k=0}^{n}{n\choose k}x^{n+1-k}y^{k} + 
\sum_{k=0}^{n}{n\choose k}x^{n-k}y^{k+1} \label{e:100406:b}\\
&= \sum_{k=0}^{n+1}{n\choose k}x^{n+1-k}y^{k} + 
\sum_{k=0}^{n+1}{n\choose k-1}x^{n-k+1}y^{k} \label{e:100406:c}\\
&= \sum_{k=0}^{n+1}{n+1\choose k}x^{n+1-k}y^{k},
\label{e:100406:d}
\end{align}
\end{subequations}
as required. 
\end{proof}

Setting $x=1$ and $y=\pm 1$ in Theorem~\ref{t:binomial}
yields the following result.

\begin{corollary} 
\label{c:binomial}
Let $n\in\{1,2,3,\ldots\}$. Then
\begin{equation}
\label{e:2^n}
\sum_{k=0}^{n} {n\choose k} = 2^n
\end{equation}
and
\begin{equation}
\label{e:bino-1}
\sum_{k=0}^{n} (-1)^k{n\choose k} = 0.
\end{equation}
\end{corollary}

\begin{theorem}[Combinations Theorem]
\label{t:combinations}
\index{Combinations Theorem}
Let $n$ be a strictly positive integer and
let $A_n = \{a_1,\ldots,a_n\}$ be a set containing $n$ distinct elements.
Let $k\in\{0,1,\ldots,n\}$. Then the number of all subsets
(``\textbf{combinations}'') of $A_n$ of size
$k$ is ${n \choose k}$. 
\end{theorem}
\begin{proof}
By Mathematical Induction.

\textbf{Base Case}: $n=1$.
There is exactly $1={1\choose 0}$ subset of size $0$, namely $\varnothing$,
and there is exactly $1={1\choose 1}$ subset of size $1$, namely $\{a_1\}$
itself.

\textbf{Inductive Step}: We assume that the result is true for 
$A_n = \{a_1,\ldots,a_n\}$, where $n\in\NN$ and $n\geq 1$.
The set under investigation 
$A_{n+1} = \{a_1,\ldots,a_n,a_{n+1}\}$ contains $n+1$ distinct elements. 
Let $k\in\{0,1,\ldots,n,n+1\}$.
If $k=0$ or $k=n+1$, then the result is clearly true.
Thus, assume that $1\leq k \leq n$. 
Let $B$ be a subset of $A_{n+1}$ containing $k$ elements.
Then either $a_{n+1}\in B$ or $a_{n+1}\notin B$.
In the former case, $B=C\cup \{a_{n+1}\}$, where $C$ contains
$k-1$ elements from $A_n$, which shows that there are ${n\choose k-1}$ 
choices for $B$ in this case;
in the latter case, $B$ is a subset of $A_n$ for which there are
${n\choose k}$ choices. (Here we used the induction hypothesis twice.)
In view of Lemma~\ref{l:choose}\ref{l:choose:iv}, there are altogether
${n\choose k-1} + {n\choose k} = {n+1 \choose k}$ choices for $B$,
as claimed. 
\end{proof}

\opt{sec}{\newpage}

\section*{Exercises}\markright{Exercises}
\addcontentsline{toc}{section}{Exercises}
\setcounter{theorem}{0}

\begin{exercise}
\label{exo:100906:a}
Prove that, for every integer $n\geq 1$, one has
\begin{equation}
\label{e:100906:a}
1^2+2^2+\cdots + n^2 = \frac{n(n+1)(2n+1)}{6}.
\end{equation}
\end{exercise}
\begin{solution}
We prove this result by using the Principle of Mathematical Induction,
with $n_0 = 1$. The statement $S(n)$ is \eqref{e:100906:a}. 

\textbf{Base Case}: When $n=1$,
the claimed identity \eqref{e:100906:a} reads
\begin{equation*}
1^2 = \frac{1(1+1)(2\cdot 1+1)}{6},
\end{equation*}
which simplifies to $1=1$, 
which is obviously true.

\textbf{Inductive Step}:
Now suppose that $n\in \NN$, that $n\geq 1$, and that 
$S(n)$ is true, i.e., that 
\eqref{e:100906:a} is true for this $n$.
Our job is to show that $S(n+1)$ is true, i.e.,
that \eqref{e:100906:a} is true when $n$ is replaced by
$n+1$. That is, we must prove that
\begin{equation}
\label{e:100906:b}
1^2+2^2+\cdots + n^2 + (n+1)^2 \stackrel{?}{=} \frac{(n+1)(n+2)(2n+3)}{6}.
\end{equation}
Using the \textbf{induction hypothesis}, we see that
\begin{align*}
1^2+2^2+\cdots + n^2 + (n+1)^2 &= \big( 1^2+2^2+\cdots + n^2\big) +
(n+1)^2\\
&= \frac{n(n+1)(2n+1)}{6} + (n+1)^2\\
&= (n+1)\left( \frac{n(2n+1)}{6} + (n+1)\right)\\
&= (n+1)\left( \frac{2n^2+n+6n+6}{6} \right)\\
&= \frac{n+1}{6}\big( (n+2)(2n+3) \big).
\end{align*}
Hence \eqref{e:100906:b} is verified, i.e., $S(n+1)$ is true.
Therefore, by the Principle of Mathematical Induction
(Fact~\ref{f:induction}), the formula \eqref{e:100906:a} is true for
every integer $n\geq 1$.
\end{solution}

\begin{exercise}
\label{exo:150914a}
Prove that, for every integer $n\geq 1$, one has
\begin{equation}
\label{e:150914:a}
\frac{1}{2} + \frac{2}{2^2} + \frac{3}{2^3} + \cdots +
\frac{n}{2^n} = 2 - \frac{n+2}{2^n}. 
\end{equation}
\end{exercise}
\begin{solution}
We prove this result by using the Principle of Mathematical Induction,
with $n_0 = 1$. The statement $S(n)$ is \eqref{e:150914:a}. 

\textbf{Base Case}: When $n=1$,
the claimed identity \eqref{e:150914:a} reads
\begin{equation*}
\frac{1}{2} = 2 - \frac{1+2}{2^1}
\end{equation*}
which simplifies to $1/2=1/2$, 
which is obviously true.

\textbf{Inductive Step}:
Now suppose that $n\in \NN$, that $n\geq 1$, and that 
$S(n)$ is true, i.e., that 
\eqref{e:150914:a} is true for this $n$.
Our job is to show that $S(n+1)$ is true, i.e.,
that \eqref{e:150914:a} is true when $n$ is replaced by
$n+1$. That is, we must prove that
\begin{equation}
\label{e:150914:b}
\frac{1}{2} + \frac{2}{2^2} + \frac{3}{2^3} + \cdots +
\frac{n}{2^n} + \frac{n+1}{2^{n+1}} \stackrel{?}{=}  2 -
\frac{(n+1)+2}{2^{n+1}}.
\end{equation}
Using the \textbf{induction hypothesis}, we see that
\begin{align*}
\frac{1}{2} + \frac{2}{2^2} + \frac{3}{2^3} + \cdots +
\frac{n}{2^n} + \frac{n+1}{2^{n+1}} &= 
\Big(\frac{1}{2} + \frac{2}{2^2} + \frac{3}{2^3} + \cdots +
\frac{n}{2^n}\Big) + \frac{n+1}{2^{n+1}}\\
&=  2 - \frac{n+2}{2^n} + \frac{n+1}{2^{n+1}}\\
&=  2 - \frac{2(n+2)-(n+1)}{2^{n+1}}\\
&=  2 - \frac{n+3}{2^{n+1}}\\
&=  2 - \frac{(n+1)+2}{2^{n+1}}.
\end{align*}
Hence \eqref{e:150914:b} is verified, i.e., $S(n+1)$ is true.
Therefore, by the Principle of Mathematical Induction
(Fact~\ref{f:induction}), the formula \eqref{e:150914:a} is true for
every integer $n\geq 1$.
\end{solution}

\begin{exercise} 
Find a formula for the sum of the first $n$ strictly positive 
\emph{even} integers
(\`a la Theorem~\ref{t:odd}) and prove your formula using 
Mathematical Induction. 
\end{exercise}
\begin{solution}
For every integer $n\geq 1$, one has
\begin{equation}
\label{e:100906:c}
\sum_{k=1}^{n} (2k) = n(n+1).
\end{equation}
We prove this result by mathematical induction, with $n_0=1$ and
where the statement $S(n)$ is \eqref{e:100906:c}. 

\textbf{Base Case}: When $n=1$, \eqref{e:100906:c} turns into
\begin{equation*}
\sum_{k=1}^{1} (2k) = (2\cdot 1) = 2 = 1(2) = 1(1+1),
\end{equation*}
which is obviously true.

\textbf{Inductive Step}: Now suppose that $n\in\NN$ is such that
$n\geq 1$ and $S(n)$ is true, i.e., \eqref{e:100906:c} holds for this $n$.
Our goal is to show that $S(n+1)$ is true, i.e., that 
\eqref{e:100906:c} holds when $n$ is replaced by $n+1$. 
That is, we must prove that
\begin{equation}
\label{e:100906:d} 
\sum_{k=1}^{n+1} (2k) \stackrel{?}{=} (n+1)(n+2).
\end{equation}
Indeed, using the \textbf{induction hypothesis} that $S(n)$ is true in 
\eqref{e:100906:e}, we see that
\begin{subequations}
\begin{align}
\sum_{k=1}^{n+1} (2k) &= \left(\sum_{k=1}^{n} (2k)\right) + 
\big(2(n+1)\big)\\
&= n(n+1) + (2n+2)\label{e:100906:e}\\
&= (n+1)(n+2).
\end{align}
\end{subequations}
Hence \eqref{e:100906:d} is verified, i.e., $S(n+1)$ is true.
By the Principle of Mathematical Induction, the formula
\eqref{e:100906:c}
is true for every integer $n\geq 1$. 
\end{solution}

\begin{exercise}
You are standing in line at the Grand 10 Movie Theatre, to get tickets
for the movie \emph{Avengers: Infinity Matters}. 
The line consists of at least two people, 
the first of which is an odd, and the
last of which is an even. 
(In this question, we call people odd if their birthyears are odd; 
otherwise, they're even.)
Using mathematical induction, show that 
somewhere along that line there is an even 
directly standing behind an odd.
\end{exercise}
\begin{solution}
We use the principle of mathematical induction,
with $n_0=2$. The statement is, for an integer $n\geq 2$, 
\begin{quotation}
\noindent
$S(n)$: In a line of $n$ people (first an odd , last an even),
there is somewhere an even standing directly behind an odd.
\end{quotation}

\textbf{Base Case:} When $n=2$, $S(n)$ is certainly true since the
queue consists of two people only, an odd followed by an even.

\textbf{Inductive Step:} Now suppose that $n\in\NN$ such that $n\geq
2$ and that $S(n)$ is true. Our job is to show that $S(n+1)$ is true.
So consider a line of $n+1$ people, say $p_1,\ldots,p_n,p_{n+1}$, 
where $p_1$ is an odd and
$p_{n+1}$ is an even. Since $n\geq 2$, it follows that $n+1\geq
3$. Now consider the first $n$ people, i.e.,
$p_1,\ldots,p_n$.
If $p_n$ is an odd, then $S(n+1)$ is true since
$p_{n+1}$, who is an even, stands behind $p_{n}$.
Otherwise, $p_n$ is an even. But then, using
the \textbf{induction hypothesis}, somewhere in the line
$p_1,\ldots,p_n$ must an even stand right behind an odd.
\end{solution}

\begin{exercise}
Prove that, for every integer $n\geq 0$, 
\begin{equation}
\label{e:100913:a}
5 | (9^n-4^n).
\end{equation}
\end{exercise}
\begin{solution}
This result is very similar to Theorem~\ref{t:divide}.
We argue by induction, with $n_0=0$ and where
the statement $S(n)$ is \eqref{e:100913:a}.

\textbf{Base Case}:
When $n=0$, \eqref{e:100913:a} turns into $5|(9^0-4^0)$, i.e., into
$5|0$ which is obviously true.

\textbf{Inductive Step}:
Now suppose that $n\in\NN$ is such that $S(n)$ is true, i.e.,
\eqref{e:100913:a} holds. 
Our job is to show that $S(n+1)$ is true, i.e.,
\begin{equation}
\label{e:100913:b}
5 \stackrel{?}{|} (9^{n+1}-4^{n+1}).
\end{equation}
Now since $S(n)$ is true, we know there exists $k\in\ZZ$ such that
$9^n-4^n = 5\cdot k$. 
Thus, using this \textbf{induction hypothesis} in 
\eqref{e:100913:c}, we see that
\begin{subequations}
\begin{align}
9^{n+1}-4^{n+1} &= 9^{n+1} - 4\cdot 9^n + 4(9^n - 4^n)\\
&= 9^n(9-4) + 4(5\cdot k)\label{e:100913:c}\\
&= 5\big(9^n+4\cdot k\big).
\end{align}
\end{subequations}
Since $9^n+4k\in\ZZ$, it is clear that \eqref{e:100913:b} is true.
Thus, by the principle of mathematical induction, we see that
\eqref{e:100913:a} is true for every integer $n\geq 0$. 
\end{solution}

\opt{sec}{
\begin{exercise} 
\label{exo:bern}
For $n\in\NN$ and $k\in\{0,1,2\ldots,n\}$ we define
\begin{equation}
P_k^n(x) := {n \choose k} x^k (1-x)^{n-k}. 
\end{equation}
Now let $f\colon[0,1]\to\RR$ be a function.
The \emph{Bernstein polynomials} $B_nf$, where $n\in\NN$, 
associated with $f$ are \index{Bernstein polynomials} 
\begin{equation}
(B_nf)(x) := \sum_{k=0}^{n} f\big(\tfrac{k}{n}\big) P_k^n(x) =
\sum_{k=0}^{n} f\big(\tfrac{k}{n}\big)
{n \choose k}
x^k(1-x)^{n-k}.
\end{equation}
Let $f\colon [0,1]\to\RR$, $g\colon [0,1]\to\RR$,
$\alpha\in\RR$, $\beta\in\RR$, and $n\in\NN$. 
Show the following.
\begin{enumerate}
\item 
\label{exo:berni}
$B_n(\alpha f+\beta g) = \alpha B_nf+\beta B_n g$.
\item 
\label{exo:bernii}
If $f(x)\leq g(x)$ for all $x\in[0,1]$,
then $(B_nf)(x)\leq (B_ng)(x)$ for all $x\in[0,1]$.
\item 
\label{exo:berniii}
If $f(x)=1$ for every $x\in[0,1]$, then
$(B_nf)(x)=1$ for every $x\in[0,1]$. 
\end{enumerate}
\end{exercise}
\opt{hide}{
\begin{proof}
\hhbcom{to do...}
\hhbcom{Could also do some more fun things from Davidson-Donzig}
\end{proof}
}
}

\begin{exercise}
Explain what is wrong with the following ``proof'' by induction 
of the statement (for $n\geq 1$) $S(n)$:
``Any collection of $n$ rubber ducks consists of rubber ducks of the
same colour.''
\begin{quotation}
\noindent
``\emph{Proof.}
Clearly $S(1)$ is true.
Now suppose that $S(n)$ is true for some $n\geq 1$.
Consider a collection $M$ of $n+1$ rubber ducks.
Take one rubber duck, call it $x$.
The induction hypothesis applied to $M\smallsetminus\{x\}$ shows
that the remaining $n$ rubber ducks have all the same colour, say $c$.
Now return $x$ to the set and remove a different rubber duck $y$.
Again, the remaining rubber ducks $M\smallsetminus\{y\}$, 
one of which is $x$, must all have
the same colour $c$. Therefore, all $n+1$ rubber ducks have the same
colour $c$ and the result is proved by the principle of
mathematical induction.''
\end{quotation}
\end{exercise}
\enlargethispage{2\baselineskip}
\begin{solution}
The proof looks perfect but is flawed: the
inductive step breaks down when going from $n=1$ to $n+1=2$
rubber ducks: Consider a collection of two rubber ducks, one of which
is yellow and the other red. There is no reason that when $y$ is taken
out, the remaining collection should have the same colour as before.
\end{solution}

\begin{exercise}
Consider the statement $S(n)$,
\begin{equation}
\label{e:190916a}
\sum_{k=1}^n k = \pi + \frac{n(n+1)}{2}. 
\end{equation}
Clearly $S(n)$ is false for every $n\geq 1$ since
the sum on the left side of \eqref{e:190916a} 
is actually an integer but the right side is
not.
However, show that the actual inductive step is true.
(This shows the importance of the base case in proofs by
induction!)
\end{exercise}
\begin{solution}
Suppose $S(n)$ is true for some $n\geq 1$.
Then
\begin{align*}
\sum_{k=1}^{n+1}k &= \left(\sum_{k=1}^n k\right) + n+1\\
&= \pi + \frac{n(n+1)}{2} + n+1 \quad \text{(since $S(n)$ is
true)}\\
&= \pi + \frac{n(n+1)+2(n+1)}{2}\\
&=\pi + \frac{(n+1)(n+2)}{2}\\
&=\pi + \frac{(n+1)\big((n+1)+1\big)}{2}\\
\end{align*}
which shows that $S(n+1)$ holds as well.
\end{solution}

\begin{exercise}
Determine a simple formula for the sum
$3+3^2 + \cdots +3^n$, where $n\geq 1$ is an integer.
Provide a proof of your formula. 
\end{exercise}
\begin{solution}
Using \eqref{e:geosum}, we obtain
\begin{align*}
3+3^2+\cdots + 3^n &= -1 + \sum_{k=0}^{n} 3^k
= -1 + \frac{1-3^{n+1}}{1-3}
= \frac{-2}{2} + \frac{3^{n+1}-1}{2}
= \frac{3^{n+1}-3}{2}.
\end{align*}
\end{solution}

\begin{exercise}
Explain why all binomial coefficients --- which are defined as
quotients --- are actually integers.
\end{exercise}
\begin{solution}
The statement is $S(n)$, where $n\in\NN$: 
``${n \choose k}$ is an integer for every $k\in\ZZ$.''

\textbf{Base Case:} When $n=0$, $S(n)$ is true since,
by definition (or by Lemma~\ref{l:choose}\ref{l:choose:i}) 
${0 \choose k}$ is either $0$ or $1$ and thus an integer.

\textbf{Inductive Step:} 
Now suppose $S(n)$ is true for some integer $n\geq 0$.
We need to show that $S(n+1)$ is true, i.e.,
${n+1\choose k}$ is an integer for all $k\in\ZZ$.
But this is clear using the induction hypothesis and
Lemma~\ref{l:choose}\ref{l:choose:iv}.

Therefore, by the Principle of Mathematical Induction, all binomial
coefficients are integers.

\emph{Alternatively}, one may refer to the Combinations~Theorem 
(Theorem~\ref{t:combinations}) the proof of which is similar to
the above. 
\end{solution}

\begin{exercise}[Lotto 6/49] 
Determine the number of all combinations of winning numbers
for Lotto 6/49. \index{Lotto 6/49} 
\end{exercise}
\begin{solution}
By Theorem~\ref{t:combinations}, 
there are
\begin{equation*}
{49 \choose 6} = \frac{49\cdot 48\cdot 47\cdot 46\cdot 45\cdot 44}{1\cdot
2\cdot 3\cdot 4\cdot 5\cdot 6} = 13,983,816
\end{equation*}
combinations of size $6$ drawn from the set $\{1,2,\ldots,49\}$.
Thus, the probability of guessing the winning numbers is approximately
$1$ to 14 million. 
\end{solution}

\begin{exercise}
Let $n\in\{1,2,\ldots\}$. 
Prove that
\begin{equation*}
\sum_{k \text{~even}} {n\choose k} = \sum_{k \text{~odd}} {n \choose
k}. 
\end{equation*}
\end{exercise}
\begin{solution}
From \eqref{e:100406:a} with $x=1$ and $y=-1$, we obtain
\begin{align*}
0 &=0^n = (1-1)^n = \sum_{k=0}^{n} {n \choose k} 1^{n-k}(-1)^k 
= \sum_{k=0}^{n} {n \choose k} (-1)^k\\
&= \sum_{k\in\{0,1,\ldots,n\} \text{~and even}}  {n \choose k} (-1)^k
+ \sum_{k\in\{0,1,\ldots,n\} \text{~and odd }}  {n \choose k} (-1)^k\\
&= \sum_{k\in\{0,1,\ldots,n\} \text{~and even}}  {n \choose k} 
+ \sum_{k\in\{0,1,\ldots,n\} \text{~and odd }}  {n \choose k} (-1)
\end{align*}
Hence
\begin{equation}
\label{e:120918a}
\sum_{k\in\{0,1,\ldots,n\} \text{~and even}}  {n \choose k} 
= \sum_{k\in\{0,1,\ldots,n\} \text{~and odd}}  {n \choose k}. 
\end{equation}
On the other hand, 
by Lemma~\ref{l:choose}\ref{l:choose:iii},
\begin{equation}
\label{e:120918b}
{n\choose k} = 0,\quad\text{if $k<0$ or $k>n$.}
\end{equation}
Therefore, using \eqref{e:120918b}, we see that 
\eqref{e:120918a} can be stated as
\begin{equation*}
\sum_{\text{$k$ even}}  {n \choose k} 
= \sum_{\text{$k$ odd}}  {n \choose k}. 
\end{equation*}
This completes the proof.
\end{solution}

\begin{exercise}
Prove that, for every $n\in\NN$,
\begin{equation}
3^n = \sum_{k=0}^{n} {n \choose k} 2^k.
\end{equation}
\end{exercise}
\begin{solution}
This follows from the Binomial Theorem
(Theorem~\ref{t:binomial}) by setting $x=1$ and $y=2$.
\end{solution}

\begin{exercise}
The polynomial $x^3+y^3$ in two variables $x$ and $y$
factors as follows: $x^3+y^3 = (x+y)(x^2-xy+y^2)$.
In fact, prove that for every $n\in\NN$, 
$x^{2n+1}+y^{2n+1}$ is divisible by $x+y$.
\end{exercise}
\begin{solution}
Consider the statement $S(n)$ for $n\in\NN$ stating that
``The polynomial $x^{2n+1}+y^{2n+1}$ is divisible by $x+y$.''. 
We argue by mathematical induction.

\textbf{Base Case}: $n=0$. Then $S(0)$ turns into 
``$x^{2\cdot 0+1} + y^{2\cdot 0 +1} =x+y$ is divisible by $x+y$'',
which is certainly true.

\textbf{Inductive Step}: 
Assume the statement is true for some $n\geq 0$. 
Then there exists a polynomial $q$ in two variables $x,y$ such that
\begin{equation}
x^{2n+1} + y^{2n+1} = q\cdot (x+y).
\end{equation}
Using the \textbf{induction hypothesis}, we see that
\begin{align*}
x^{2(n+1)+1} + y^{2(n+1)+1} &= x^2x^{2n+1} + y^2y^{2n+1} \\
&= x^2\big(x^{2n+1}+y^{2n+1}\big) + (y^2-x^2)y^{2n+1}\\
&= x^2\cdot q\cdot (x+y) + (y+x)(y-x)y^{2n+1}\\
&= (x+y)\big(x^2\cdot q + (y-x)y^{2n+1}\big).
\end{align*}
Hence $S(n+1)$ is true as well. 
By the Principle of Mathematical Induction, the result is true for every
$n\in\NN$. 
\end{solution}

\begin{exercise}
Prove that, for every $\nnn$,
\begin{equation*}
10^n = \sum_{k=0}^n {n \choose k} 9^{n-k}.
\end{equation*}
\end{exercise}
\begin{solution}
This follows from the Binomial Theorem
(Theorem~\ref{t:binomial}) by setting $x=9$ and $y=1$.
\end{solution}

\begin{exercise}[a general calculus product rule]
Recall the calculus rule
\begin{equation*}
(f\cdot g)' = f'\cdot g + f\cdot g',
\end{equation*}
which we assume to be true.
Show that if $n\geq 2$ and $f_1,\ldots,f_n$ are differentiable functions,
then
\begin{equation*}
\big(f_1f_2\cdots f_n\big)' = f_1'f_2f_3\cdots f_n + f_1f_2'f_3\cdots f_n +
\cdots + f_1\cdots f_{n-1}f_n'.
\end{equation*}
\end{exercise}
\begin{solution}
We prove this by induction on $n$. The base case $n=2$ is clear by the well known
product rule.
Now assume the result is true for some integer $n\geq 2$.
Consider $n+1$ functions $f_1,\cdots, f_{n+1}$ and their product, which
we write as $(f_1\cdots f_n)\cdot f_{n+1}$. Then, by using the induction
hypothesis and the case $n=2$, we get
\begin{align*}
(f_1\cdots f_{n+1})' &= \big( (f_1\cdots f_n)f_{n+1}\big)' 
=
(f_1\cdots f_n)'f_{n+1} + (f_1\cdots f_n)f'_{n+1}
\\
&=  (f_1'f_2\cdots f_n + \cdots + f_1\cdots f_{n-1}f_n')f_{n+1} + 
(f_1\cdots f_n)f'_{n+1}\\
&=f_1'f_2\cdots f_{n+1} + \cdots + f_1\cdots f_nf_{n+1}',
\end{align*}
as required. 
By the Principle of Mathematical Induction, the result is true for every
integer $n\geq 2$. 
\opt{sol}{Could also do Leibniz rule for $n$th derivative of $fg$.}
\end{solution}

\begin{exercise}[more fun with calculus]
Consider the function
\begin{equation*}
f(x) = \frac{1}{x}.
\end{equation*}
Compute a few derivatives of $f$ and guess a formula for $f^{(n)}$,
the $n$th derivative of $f$. Then prove your formula by induction. 
\end{exercise}
\begin{solution}
We have $f(x) = f^{(0)}(x) = 1/x = x^{-1}$,
$f'(x) = f^{(1)}(x) = (x^{-1})' = (-1)x^{-2} = -x^{-2}$, 
$f''(x) = f^{(2)}(x) = (-x^{-2})' = (-2)(-1)x^{-3} = 2x^{-3}$, etc.
This suggests the formula
\begin{equation*}
f^{(n)}(x) = \frac{(-1)^n n!}{x^{n+1}}, \quad\text{for $n\geq 0$.}
\end{equation*}
The base case, $n=0$, is clear since $f(x) = 1/x = (-1)^{0} 0! /x^{0+1}$.
Now assume the result is true for some $n\geq 0$.
Then
\begin{align*}
f^{(n+1)}(x) &= \frac{d}{dx} f^{(n)}(x)
= \frac{d}{dx}\frac{(-1)^n n!}{x^{n+1}}
= (-1)^n n! \frac{d}{dx} x^{-n-1}\\
&= (-1)^n n! (-n-1)x^{-n-2}
= (-1)^{n+1} n!(n+1)x^{-((n+1)+1)}\\
&= \frac{(-1)^{n+1} (n+1)!}{x^{(n+1)+1}},
\end{align*}
as required.
By the Principle of Mathematical Induction, the result is true for every
$n\in\NN$. 
\end{solution}

\begin{exercise}[even more fun with calculus]
Consider the function
\begin{equation*}
f(x) = \frac{1}{1-x^2},
\end{equation*}
and denote its $n$th derivative by 
$f^{(n)}$. Show that 
\begin{equation*}
f^{(n)}(x) = \frac{n!}{2}\left(\frac{1}{(1-x)^{n+1}} +
\frac{(-1)^n}{(1+x)^{n+1}}\right), \quad \text{for $n\geq 0$.}
\end{equation*}
\end{exercise}
\begin{solution}
The base case, $n=0$, is clear since 
\begin{align*}
f^{(0)}(x) &= f(x) = \frac{1}{1-x^2} = 
\frac{1}{2}\left( \frac{1}{1-x}+\frac{1}{1+x}\right)\\
&= \frac{0!}{2}\left( \frac{1}{(1-x)^{0+1}} +
\frac{(-1)^0}{(1+x)^{0+1}}\right). 
\end{align*}
Now assume the result is true for some $n\geq 0$.
Then
\begin{align*}
f^{(n+1)}(x) &= \frac{d}{dx} f^{(n)}(x)
= \frac{d}{dx}\frac{n!}{2}\left(\frac{1}{(1-x)^{n+1}} +
\frac{(-1)^n}{(1+x)^{n+1}}\right) \\
&= \frac{n!}{2} \left( \frac{d}{dx}\left( (1-x)^{-n-1}\right) +
(-1)^n\frac{d}{dx}\left((1+x)^{-n-1}\right)\right)\\
&= \frac{n!}{2}\left( (-n-1)(1-x)^{-n-2}(-1) + (-1)^n(-n-1)(1+x)^{-n-2}\right)\\
&= \frac{n!}{2}(n+1)\left( (1-x)^{-((n+1)+1)} +
(-1)^{n+1}(1+x)^{-((n+1)+1)} \right) \\
&= \frac{(n+1)!}{2} \left( \frac{1}{(1-x)^{(n+1)+1}} +
\frac{(-1)^{n+1}}{(1+x)^{(n+1)+1}}\right),
\end{align*}
as required.
By the Principle of Mathematical Induction, the result is true for every
$n\in\NN$. 
\end{solution}

\begin{exercise}[The Twelve Days of Christmas] \
\begin{enumerate}
\item Locate the lyrics of the carol ``The Twelve Days of
Christmas''.
\item How many gifts in total did your true love send to you?
\emph{Hint:} The answer is \underline{not} $1+2+\cdots + 12= 78$ because except
for the last gift each 
gift is received more than once. 
\item Suppose the carol was called 
``The $n$ days of Christmas'', where $n\in\NN$.
How many gifts did your true love send to you in total?
\emph{Hint:} You may use \eqref{e:100906:a}. 
\end{enumerate}
\index{Twelve Days of Christmas} 
\end{exercise}
\begin{solution}
See, e.g.,
\begin{center}
{\url{https://www.poetryfoundation.org/poems/42913/the-twelve-days-of-christmas}}
\end{center}
Indeed, using Theorem~\ref{t:Gauss} 
and Exercise~\ref{exo:100906:a}, we obtain
\begin{align*}
\sum_{k=1}^{n} (1+\cdots + k) 
&= \sum_{k=1}^{n}\frac{k(k+1)}{2} 
= \thalb\sum_{k=1}^{n} k + \thalb\sum_{k=1}^n k^2\\
&= \frac{n(n+1)}{4} + \frac{n(n+1)(2n+1)}{12}\\
&= \frac{n(n+1)}{12}\big(3+2n+1\big)
= \frac{n(n+1)}{12}(2n+4)\\ 
&= \frac{n(n+1)(n+2)}{6}.
\end{align*}
When $n=12$, we get $(12)(13)(14)/6 = 364$. 
\end{solution}

\begin{exercise}
Prove that for every integer $n\geq 8$, 
there exist $\alpha,\beta$ in $\NN$ such that 
$n=\alpha\cdot 3 + \beta \cdot 5$. 
\emph{Hint:}
These are actually three proofs by induction, with a common inductive step!
\end{exercise}
\begin{solution}
Denote the statement by $S(n)$.
Note that:
\begin{itemize}
\item $S(8)$ is true since $8 = 1\cdot 3 + 1\cdot 5$.
\item $S(9)$ is true since $9 = 3\cdot 3 + 0\cdot 5$.
\item $S(10)$ is true since $10 = 0\cdot 3 + 2\cdot 5$.
\end{itemize}
These are the three \textbf{base cases}. 
We now prove that $S(8+3n)$, $S(9+3n)$, and $S(10+3n)$ are true;
luckily, the \textbf{inductive step} is identical for all three!

So suppose that $k\geq 8$ and $S(k)$ is true.
It suffices to show that $S(k+3)$ is true.
Write $k = \alpha\cdot 3+\beta\cdot 5$,
where $\alpha$ and $\beta$ are in $\NN$. 
Then $k+3=(\alpha+1)\cdot 3 +\beta\cdot 5$ and $S(k+3)$ is true.
Therefore, 
the result holds by the Principle of Mathematical Induction.
\opt{hhb}{See induction book page 602 for more.}
\end{solution}

\begin{exercise}
Prove that for every integer $n\geq 4$, 
there exist $\alpha,\beta$ in $\NN$ such that 
$n=\alpha\cdot 2 + \beta \cdot 5$. 
\emph{Hint:}
These are actually two proofs by induction, with a common inductive step!
\end{exercise}
\begin{solution}
Denote the statement by $S(n)$.
Note that:
\begin{itemize}
\item $S(4)$ is true since $4 = 2\cdot 2 + 0\cdot 5$.
\item $S(5)$ is true since $5 = 0\cdot 2 + 1\cdot 5$.
\end{itemize}
These are the two \textbf{base cases}. 
We now prove that $S(4+2n)$ and $S(5+2n)$ are true;
luckily, the \textbf{inductive step} is identical for both!

So suppose that $k\geq 4$ and $S(k)$ is true.
It suffices to show that $S(k+2)$ is true.
Write $k = \alpha\cdot 2+\beta\cdot 5$,
where $\alpha$ and $\beta$ are in $\NN$. 
Then $k+2=(\alpha+1)\cdot 2 +\beta\cdot 5$ and $S(k+2)$ is true.
Therefore, 
the result holds by the Principle of Mathematical Induction.
\opt{hhb}{See induction book page 602 for more.}
\end{solution}

\begin{exercise}
Consider the number 5. \emph{Taking order into account},
we can write 5 in six different ways as a sum of 3 
strictly positive integers:
\begin{equation*}
5 = 1 + 1 +3 = 1 + 3 + 1 = 3+1+1 = 1 + 2 + 2 = 2 + 1 + 2 = 2 + 2 +1.
\end{equation*}
In how many ways can we write $n$, where $n\geq 3$,
as a sum of 3 strictly positive integers, taking order into
account?
\end{exercise}
\begin{solution}
Write 
\begin{equation*}
n = 1 + 1 + \cdots + 1,
\end{equation*}
and observe that we have $n-1$ plus signs. Any choice of two
gives rise to an ordered sum. So the answer is ${n-1 \choose 2}$ by
the Combinations Theorem~\ref{t:combinations}.
\end{solution}

\begin{exercise}
Show that 
\begin{equation*}
s(n) := \frac{n^5}{5} + \frac{n^4}{2} + \frac{n^3}{3} - \frac{n}{30}
\end{equation*}
is an integer for every $n\in\NN$.
\end{exercise}
\begin{solution}
We prove this by induction. The result is clear for the base
case, $n=0$, since $s(0)$ is obviously equal to $0$.

Now assume that $s(n)$ is an integer, for some $\nnn$.
Then, using Pascal's Triangle, 
\begin{align*}
s(n+1) &=
\frac{(n+1)^5}{5} + \frac{(n+1)^4}{2} + \frac{(n+1)^3}{3} -
\frac{n+1}{30}\\
&= \frac{n^5 + 5n^4 + 10n^3 + 10n^2 + 5n+1}{5}\\
&\qquad + \frac{n^4 + 4n^3 + 6n^2 + 4n+ 1}{2}\\
&\qquad + \frac{n^3 + 3n^2 + 3n+1}{3} - \frac{n+1}{30}\\
&= \left(\frac{n^5}{5} + \frac{n^4}{2} + \frac{n^3}{3} -
\frac{n}{30} \right)\\
&\qquad + \left(n^4 + 2n^3 + 2n^3 + 2n^2 + 3n^2 + n^2 + n + 2n +n
\right)\\
&\qquad + \left(\tfrac{1}{5} + \tfrac{1}{2} + \tfrac{1}{3}
-\tfrac{1}{30}\right)\\
&= s(n) + n^4 + 4n^3 + 6n^2 + 4n + 1,
\end{align*}
which is clearly an integer. 
Thus, by the Principle of Mathematical Induction, the result is
proven. 
\end{solution}

\begin{exercise}[The Principle of Strong Induction]
\index{strong induction}
\label{exo:strongind}
The \emph{Principle of Strong Induction} states the following:
\begin{quotation}
Let $n_0$ be an integer and let $T(n)$ be a statement, 
formulated  for all integers $n\geq n_0$. 
Suppose we are able to verify the following.
\begin{enumerate}
\item \textbf{Base Case}: $T(n_0)$ is true.
\item \textbf{Inductive Step}: If $n$ is an integer $\geq n_0$ and
$T(n_0)$, $T(n_0+1)$,\ldots, $T(n)$ are true, then $T(n+1)$ is also true.
\end{enumerate}
Then the statement $T(n)$ is true for \emph{all} integers $n\geq n_0$.
\end{quotation}
Prove the Principle of Strong Induction using the Principle of (ordinary)
Induction.
\end{exercise}
\begin{solution}
For every integer $n\geq n_0$, define the statement 
\begin{center}$S(n)$ by ``$T(n_0)$ and
$T(n_0+1)$ and \ldots\ and $T(n)$''. 
\end{center}
We shall prove that $S(n)$ is true for all integers $n\geq n_0$ by regular induction.

The base case is clear by (i).

Now assume that $S(n)$ is true for some $n\geq n_0$. 
Then, by (ii), $T(n+1)$ is true. Hence $S(n+1)$ is true as well.

By the Principle of Mathematical Induction, $S(n)$ is true for all $n\geq
n_0$.

Therefore, $T(n)$ is also true for all $n\geq n_0$. 
\end{solution}

\begin{exercise}[factorization into prime numbers]
\label{exo:primefactors}
Recall a number $p\in\{2,3,4,\ldots\}$ is \emph{prime}
if $p$ is only divisible (without remainder) by $\pm 1$ and $\pm
p$. Use the principle of strong induction to show that
every integer $n\geq 2$ can be written as a product of prime
numbers.
\emph{Aside comment:} You will see in Number Theory that the
factorization is ``unique''.
\index{Prime number}
\end{exercise}
\begin{solution}
Clearly $n=2$ can written as a product of primes since $2$ is
prime.
Now assume the statement holds for all integers
$2,3,\ldots,n$. 
Consider the number $n+1$.
If $n+1$ is prime, then we are done.
So assume $n+1$ is not prime.
Then $n+1 = ab$, where $2\leq a \leq n$ and
$2\leq b\leq n$. By the hypothesis, we can factor
both $a$ and $b$ into primes, say
$a=p_1p_2\cdots p_k$ and $b= q_1q_2\cdots q_l$.
Then $n+1=ab = p_1p_2\cdots p_kq_1q_2\cdots q_l$
is indeed a product of primes.
Therefore, by the Principle of Strong Induction, we're done. 
\end{solution}

\begin{exercise}[infinitely many primes]
Show that there are infinitely many primes.
\emph{Hint:} Argue by contradiction and assume there are only
finitely many primes $p_1,\ldots,p_k$.
Now consider $n:=p_1p_2\cdots p_k+1$.
\end{exercise}
\begin{solution}
Since $n$ divided by any of the $p_i$ leaves a remainder of $1$,
it is clear that $n$ is not divisible by any of the numbers
$p_1,\ldots,p_k$. 
On the other hand, by the previous exercise,
there exists a prime number $q$ dividing $n$.
Clearly, $q\notin\{p_1,\ldots,p_k\}$ and we have found
a new prime number not in our list of all prime numbers,
which is absurd.
\end{solution}

\begin{exercise}
\label{exo:walaa2}
Let $\nnn$. 
Using mathematical induction and 
Theorem~\ref{t:Gauss}, show that
\begin{equation*}
1^3 + 2^3 + \cdots + n^3 = (1+2+\cdots +n)^2.
\end{equation*}
\end{exercise}
\begin{solution}
The base case, $n=0$, is obviously true since we obtain $0=0$ by
the empty sum convention.

Now assume the identity holds for some $\nnn$.
Then, using the inductive hypothesis and Theorem~\ref{t:Gauss} (twice), 
we obtain
\begin{align*}
1^3 + 2^3 + \cdots + (n+1)^3
&= \big(1^3 + 2^3+\cdots + n^3) + (n+1)^3\\
&= (1+2 + \cdots + n)^2 + (n+1)^3\\
& = \left(\frac{n(n+1)}{2}\right)^2 + (n+1)^3 \\
& = \frac{(n+1)^2}{4}\big( n^2 + 4(n+1)\big)\\
&=  \frac{(n+1)^2}{4}(n+2)^2\\
&= \left( \frac{(n+1)(n+2)}{2}\right)^2\\
&= \big( 1+ 2+ \cdots + (n+1)\big)^2,
\end{align*}
as required.
Therefore, the entire result is proven by the 
Principle of Mathematical Induction. 
\end{solution}

\begin{exercise}[fun with trig calculus]
\label{exo:walaa4}
Show that
\begin{equation*}
\frac{d^n}{dx^n}\big(\sin x\big) = \sin\big(x+\tfrac{n}{2}\pi\big)
\end{equation*}
for every $\nnn$.
\end{exercise}
\begin{solution}
The result is obvious for $n=0$.

Now assume the formula holds for some $n\in\NN$.
Then, using the inductive hypothesis,
we obtain
\begin{align*}
\frac{d^{n+1}}{dx^{n+1}}\big(\sin x\big) 
&= \frac{d}{dx} \left(\frac{d^{n}}{dx^{n}}\big(\sin x\big)\right)\\
&= \frac{d}{dx} \sin\big(x+\tfrac{n}{2}\pi\big)\\
&= \cos\big(x+\tfrac{n}{2}\pi\big).
\end{align*}
On the other hand, 
since $\sin(y+\pi/2) = \cos(y)$, we see that
\begin{equation*}
\cos\big(x+\tfrac{n}{2}\pi\big) = 
\sin\big(x+\tfrac{n+1}{2}\pi\big).
\end{equation*}
Altogether, we have verified the formula for $n+1$.
Therefore, by the Principle of Mathematical Induction, the identity is
proven.
\end{solution}

\begin{exercise}
\label{exo:wasser2}
Let $n$ be an integer such that $n\geq 1$. Show that
\begin{equation}
\label{e:wasser2:a}
\frac{1}{2!} + \frac{2}{3!}+\cdots + \frac{n}{(n+1)!} = 1 -
\frac{1}{(n+1)!}. 
\end{equation}
\end{exercise}
\begin{solution}
We prove this result by using the Principle of Mathematical Induction,
with $n_0 = 1$. The statement $S(n)$ is \eqref{e:wasser2:a}. 

\textbf{Base Case}: When $n=1$,
the claimed identity \eqref{e:wasser2:a} reads
\begin{equation}
\frac{1}{2!} = 1 - \frac{1}{(1+1)!}
\end{equation}
which simplifies to $1/2=1/2$, 
which is obviously true.

\textbf{Inductive Step}:
Now suppose that $n\in \NN$, that $n\geq 1$, and that 
$S(n)$ is true, i.e., that 
\eqref{e:wasser2:a} is true for this $n$.
Our job is to show that $S(n+1)$ is true, i.e.,
that \eqref{e:wasser2:a} is true when $n$ is replaced by
$n+1$. That is, we must prove that
\begin{equation}
\label{e:wasser2:b}
\frac{1}{2!} + \frac{2}{3!}+\cdots + \frac{n}{(n+1)!} 
+ \frac{n+1}{((n+1)+1)!}
\stackrel{?}{=}
1 - \frac{1}{((n+1)+1)!}. 
\end{equation}
Using the \textbf{induction hypothesis}, we see that
\begin{align*}
\frac{1}{2!} + \frac{2}{3!}+\cdots + \frac{n}{(n+1)!} 
+ \frac{n+1}{(n+1)!} 
&= 1- \frac{1}{(n+1)!}+ \frac{n+1}{((n+1)+1)!} \\
&= 1 - \frac{(n+2)-(n+1)}{(n+2)!}\\
&= 1 - \frac{1}{(n+2)!}\\
&= 1 - \frac{1}{((n+1)+1)!}
\end{align*}
Hence \eqref{e:wasser2:b} is verified, i.e., $S(n+1)$ is true.
Therefore, by the Principle of Mathematical Induction
(Fact~\ref{f:induction}), the formula \eqref{e:wasser2:a} is true for
every integer $n\geq 1$.
\end{solution}

\begin{exercise}[triominos]
An ``L-shaped'' union of 3 unit squares looking like 
\begin{center}
\includegraphics[height=1cm]{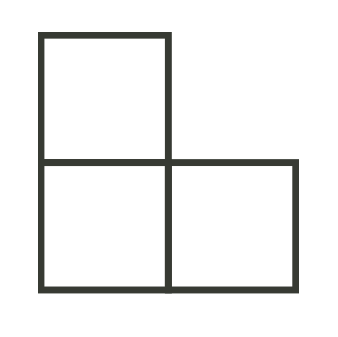}
\end{center}
is called a \emph{triomino}.
Let $n$ be an integer such that $n\geq 1$.
Show that a $2^n \times 2^n$ grid of unit squares, with one
square removed, can always be covered (without overlapping) 
by triominos.
\end{exercise}
\begin{solution}
This is a charming proof by induction on $n$.

The base case, $n=1$, is obvious since removing
1 square from a $2\times 2$ grid leaves us precisely with a
triomino. 

Now assume the result is true for $n\geq 1$ and consider
a $2^{n+1}\times 2^{n+1}$ grid with one black square removed.
Divide the grid into four $2^n\times 2^n$ grids and place
a triomino as shown in the following picture:
\begin{center}
\includegraphics[height=4cm]{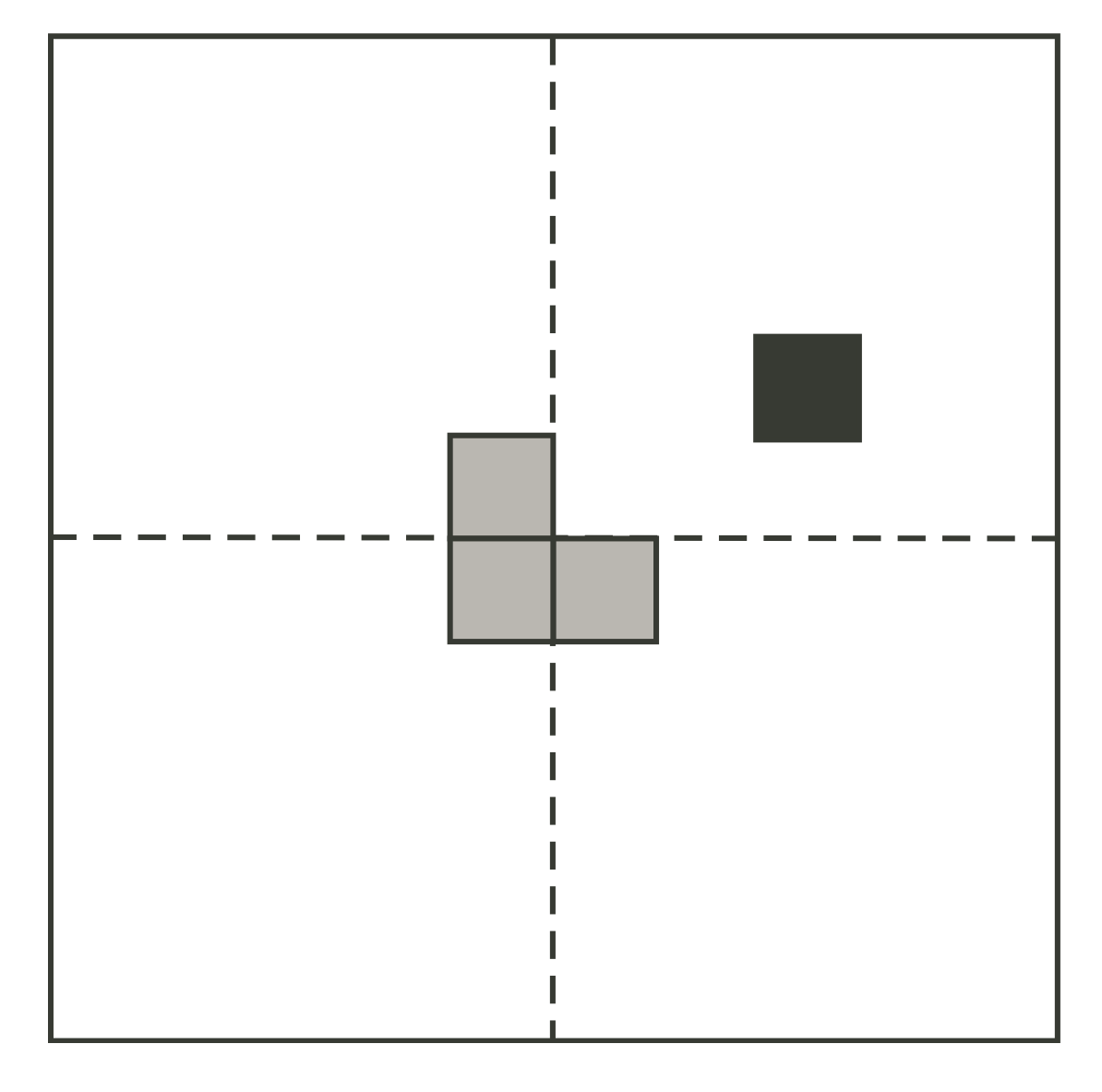}
\end{center}
By induction hypothesis, the four smaller grids can be covered
with triominos which completes the proof. 
\end{solution}

\begin{exercise}[Gauss revisited]
In Theorem~\ref{t:Gauss}, we proved by induction that 
for every integer $n\geq 1$, one has
\begin{equation}
\label{e:200714a}
1+2+\cdots + n = \frac{n(n+1)}{2}.
\end{equation}
In Remark~\ref{r:Gauss},
we pointed out an alternative approach to this formula.
Use Gauss's approach to prove \eqref{e:200714a} by 
(i) considering first case when $n$ is even;
and then (ii) considering the case when $n$ is odd and using (i).
\end{exercise}
\begin{proof}
(i): Suppose $n$ is even, say $n=2k$ for the integer $k=n/2$.
Then there are $k$ pairs which we group as follows:
\begin{align*}
1+2+\cdots+n 
&= 
(1+n)+\big(2+(n-1)\big)+\cdots + \big(k+(k-1)\big)\\
&=k(n+1)
=\frac{n}{2}(n+1)
\end{align*}
as claimed. 

(ii): Now suppose that $n$ is odd. 
Then $n-1$ is even, so by (i), we have 
\begin{equation*}
1+2+\cdots + (n-1) = \frac{(n-1)\big((n-1)+1\big)}{2}
= \frac{(n-1)n}{2}.
\end{equation*}
It follows that 
\begin{align*}
1+2+\cdots + n 
&= \big(1+2+\cdots + (n-1))+ n \\
&= \frac{(n-1)n}{2} + n 
= \frac{(n-1)n + 2n}{2}\\
&= \frac{n^2+n}{2}
= \frac{n(n+1)}{2}
\end{align*}
as required. 
\end{proof}

\begin{exercise}
Prove that, for every integer $n\geq 0$, 
\begin{equation}
\label{e:200720:a}
7 | (9^n-2^n).
\end{equation}
\end{exercise}
\begin{solution}
This result is very similar to Theorem~\ref{t:divide}.
We argue by induction, with $n_0=0$ and where
the statement $S(n)$ is \eqref{e:200720:a}.

\textbf{Base Case}:
When $n=0$, \eqref{e:200720:a} turns into $7|(9^0-2^0)$, i.e., into
$7|0$ which is obviously true.

\textbf{Inductive Step}:
Now suppose that $n\in\NN$ is such that $n\geq 0$ and $S(n)$ is true, i.e.,
\eqref{e:200720:a} holds. Our job is to show that $S(n+1)$ is true, i.e.,
\begin{equation}
\label{e:200720:b}
7 \stackrel{?}{|} (9^{n+1}-2^{n+1}).
\end{equation}
Now since $S(n)$ is true, we know there exists $k\in\ZZ$ such that
$9^n-2^n = 7\cdot k$. 
Thus, using this \textbf{induction hypothesis} in 
\eqref{e:100913:c}, we see that
\begin{subequations}
\begin{align}
9^{n+1}-2^{n+1} &= 9^{n+1} - 2\cdot 9^n + 2(9^n - 2^n)\\
&= 9^n(9-2) + 2(7\cdot k)\label{e:200720:c}\\
&= 7\big(9^n+2\cdot k\big).
\end{align}
\end{subequations}
Since $9^n+2k\in\ZZ$, it is clear that \eqref{e:200720:b} is true.
Thus, by the Principle of Mathematical Induction, we see that
\eqref{e:200720:a} is true for every integer $n\geq 0$. 
\end{solution}

\begin{exercise}
Find a simple formula for the sum
\begin{equation}
\label{e:200720:d}
\tfrac{1}{4^n}+\tfrac{1}{4^{n-1}}+\cdots + \tfrac{1}{4}+1+4+4^2 + \cdots +4^n,
\end{equation}
where $n\geq 1$ is an integer.
Provide a proof of your formula. 
\emph{Hint:} \eqref{e:geosum}.
\end{exercise}
\begin{solution}
Using \eqref{e:geosum}, we obtain
\begin{align*}
\tfrac{1}{4^n}+\tfrac{1}{4^{n-1}}+\cdots + \tfrac{1}{4}+1+4+4^2 + \cdots +4^n 
&= \frac{1}{4^n}\big(1+4 + 4^2 + \cdots + 4^{2n} \big)\\
&= \frac{1}{4^n}\cdot \frac{1-4^{2n+1}}{1-4}
= \frac{4^{2n+1}-1}{3\cdot 4^n}.
\end{align*}
\end{solution}

\begin{exercise}
\label{exo:200721a}
Use \emph{strong} induction to prove that 
$13^n$ can be written as a sum of two squares of integers
for every integer $n\geq 1$.\\
\emph{Hint:} Do this directly for $n=1$ and $n=2$.
Then assume the result is true for some $n\geq 2$ and
use the inductive hypothesis for $n-1$ to find the decomposition for $n+1$.
\end{exercise}
\begin{solution}
Let $S(n)$ be the statement in question. 

\textbf{Base Cases}:
When $n=1$, we have 
\begin{equation*}
    13^1 = 13 = 4+9 = 2^2 + 3^2.
\end{equation*}
And when $n=2$, we have
\begin{equation*}
    13^2 = 169 = 25+144 = 5^2+12^2.
\end{equation*}
(This can be found by brute force: 
Check $\sqrt{169-1^2}$, $\sqrt{169-2^2}$ etc.)
So the result holds for $n=1$ and $n=2$.

\textbf{Inductive Step}:
Now suppose that $n\in\NN$ is such that $n\geq 2$ and $S(1),S(2),\ldots,S(n)$ are 
true. In particular, $S(n-1)$ is true, say 
\begin{equation*}
13^{n-1} = a^2+b^2,
\end{equation*}
for some integers $a$ and $b$.
Note that we use here $n-1$, not the usual $n$,
to do the inductive step to $n+1$. Indeed,
\begin{align*}
13^{n+1}
&= 13^2\cdot 13^{n-1}
= 13^2(a^2+b^2)
=(13a)^2+(13b)^2
\end{align*}
and the result holds for $n+1$, i.e., $S(n+1)$ is true.
Thus, by the principle of strong mathematical induction, we see that
the statement is true for every positive integer $n$.
\end{solution}

\begin{exercise}[YOU be the marker!] 
Consider the following statement 
\begin{equation*}
\text{``Every positive integer is greater than $42$.''}
\end{equation*}
This is clearly not true (consider $1$ for instance).
Here is a ``proof'' that you are presented with:
\begin{quotation}
Suppose $n$ is a positive integer.\\
We assume the statement is true for this $n$, so $n>42$.\\
Then $n+1 > 42+1 = 43>42$.\\
Hence the statement is true for $n+1$.\\
Therefore, by the principle of mathematical induction, \\
the statement is true.
\end{quotation}
Why is this proof wrong?
\end{exercise}
\begin{solution}
This proof has no base case --- and the base case $n=1$ is not true!
The inductive step is done fine. 
(It is a proof that all integers bigger than $42$ are bigger than $42$.)
\end{solution}

\begin{exercise}[YOU be the marker!]
Consider the following statement 
\begin{equation*}
\text{``Every integer $n\geq 0$ satisfies $3|(7^n-4^n)$.''}
\end{equation*}
Here is a ``proof'' that you are presented with:
\begin{quotation}
Denote the statement by $S(n)$.\\
We prove this by induction on $n$. \\
The base case $n=0$ is true because $3|(7^0-4^0)$.\\
Now suppose $n\geq 0$ and $S(n)$ is true.\\
Because $n$ is arbitrary, the result is true for all $n$ and we are done.
\end{quotation}
Why is this proof wrong?
\end{exercise}
\begin{solution}
The base case is correct.
But the inductive step requires the result to be true for
some $n\geq 0$ and requires us to prove that the result 
is then true for $n+1$. This is the whole point of induction!
\end{solution}

\begin{exercise}[TRUE or FALSE?]
Mark each of the following statements as either true or false. 
Briefly justify your answer.
\begin{enumerate}
\item ``When doing a proof by induction, we must prove the base case 
and the inductive step.''
\item ``A proof by induction is a good technique to try to prove 
statements about real numbers.''
\item ``A proof by induction is a great approach to discover new theorems.''
\item ``There are statements that can be proved with 
strong induction but not with regular induction.''
\end{enumerate}
\end{exercise}
\begin{solution}
(i): TRUE: This is what is required in a proof by induction. 

(ii): FALSE: Unfortunately not. The problem is that there is no smallest 
real number $n_0$ and given $n$, there is not ``next larger'' real number $n+1$.

(iii): FALSE: You can use induction to try to prove results that are true 
--- you need to know what you want to prove \emph{in advance}!

(iv): TRUE: See, e.g., Exercise~\ref{exo:primefactors}.
(If you know how to factor $n=24$, say, then this does not help
you one bit to factor $n+1=25$.)
\end{solution} 
\chapter{Fields}

\label{cha:fields}

In this chapter, we study properties of fields.
Fields are sets sharing essential properties with
real numbers which we take for granted and which we denote by $\RR$.
We assume there are two operations defined for the real numbers, namely
\textbf{addition}
\begin{equation} 
+\colon \RR\times\RR\to \RR \colon (x,y)\mapsto x+y
\end{equation}
and \textbf{multiplication}
\begin{equation} 
\cdot\colon \RR\times\RR\to \RR \colon (x,y)\mapsto x\cdot y.
\end{equation}
One also writes $xy$ instead of $x\cdot y$. 
These two operations are assumed to satisfy certain properties
called the axioms of the field. 
\index{Real numbers}

\section{Axioms of Addition}

\index{Axioms of Addition}

We start with properties defined for addition.

\begin{definition}[Axioms of Addition] \ 
\begin{enumerate}
\item[\textbf{A1}]
\textbf{(associativity)}
For all $x,y,z$ in $\RR$, we have: $(x+y)+z = x+(y+z)$.
\item[\textbf{A2}]
\textbf{(commutativity)}
For all $x,y$ in $\RR$, we have: $x+y=y+x$.
\item[\textbf{A3}]
\textbf{(existence of zero)}
There exists a number $0\in\RR$ such that for every $x\in\RR$, we have
 $x+0=x$. 
\item[\textbf{A4}]
\textbf{(existence of the negative)}
For every $x\in\RR$, there exists a number $-x\in\RR$ such that
$x+(-x)=0$. 
\end{enumerate}
\end{definition}

\begin{proposition} The following hold.
\label{p:may26:1}
\begin{enumerate}
\item 
\label{p:may26:1i}
The number $0$ is unique.
\item 
\label{p:may26:1ii}
For every $x\in\RR$, $-x$ (i.e., the negative of $x$) is unique.
\item 
\label{p:may26:1iii}
$-0 = 0$.
\item 
\label{p:may26:1iv}
For every $x\in\RR$, we have: $-(-x)=x$.
\item 
\label{p:may26:1v}
For all $x,y$ in $\RR$, we have: $-(x+y)=-x+(-y)$.
\end{enumerate}
\end{proposition}
\begin{proof}
\ref{p:may26:1i}:
Suppose there is another (possibly different) 
real number $0'\in\RR$ that acts like a zero, i.e., 
for every $x\in\RR$, $x+0'=x$; in particular (set $x=0$),
$0+0'=0$.
On the other hand, using \textbf{A3} with $x=0'$, we have
$0'+0=0'$.
Altogether, using \textbf{A2}, we deduce that
$0=0+0' = 0'+0=0'$, i.e., that $0=0'$. 

\ref{p:may26:1ii}:
Let $x\in\RR$ and suppose that $y$ is another (possibly different
from $-x$)
real number such that 
\begin{equation}
\label{e:may26:a}
x+y=0.
\end{equation}
Therefore, 
\begin{align*}
y &= y+0 \quad\text{[by \textbf{A3}]}\\
&= 0+y \quad\text{[by \textbf{A2}]}\\
&= \big(x+(-x)\big)+y \quad\text{[by \textbf{A4}]}\\
&= \big((-x)+x\big) + y \quad\text{[by \textbf{A2}]}\\
&= (-x)+(x + y) \quad\text{[by \textbf{A1}]}\\
&= (-x)+0 \quad\text{[by \eqref{e:may26:a}]}\\
&= -x \quad\text{[by \textbf{A3}],}
\end{align*}
as claimed. 

\ref{p:may26:1iii}:
On the one hand, by \textbf{A4} (with $x=0$), 
we have $0 + (-0) = 0$.
On the other hand, using \textbf{A3} (with $x=0$), we see that
$0+0 = 0$. 
Altogether, since $-0$ is unique by \ref{p:may26:1ii},
we deduce that $-0=0$. 

\ref{p:may26:1iv}:
Let $x\in\RR$. 
On the one hand,
$(-x) + (-(-x)) = 0$ by \textbf{A4} (applied to $-x$).
On the other hand, using \textbf{A2} and \textbf{A4}, 
we see that $(-x) + x = x+(-x) = 0$. 
Altogether, since $-(-x)$ is unique by
\ref{p:may26:1ii}, it follows that $x = -(-x)$. 

\ref{p:may26:1v}: 
Let $x$ and $y$ be in $\RR$. Then
\begin{align*}
(x+y) + \big( -x + (-y)\big) &= x + \Big( y + \big( -x+(-y)\big)\Big)
\quad\text{[by \textbf{A1}]}\\
&= x + \Big( y + \big( (-y)+(-x)\big)\Big)
\quad\text{[by \textbf{A2}]}\\
&= x + \Big( \big(y + (-y)\big)+(-x))\Big)
\quad\text{[by \textbf{A1}]}\\
&= x + \big( 0 +(-x)\big)
\quad\text{[by \textbf{A4}]}\\
&= x + \big( (-x) + 0\big)
\quad\text{[by \textbf{A2}]}\\
&= x + (-x) 
\quad\text{[by \textbf{A3}]}\\
&= 0
\quad \text{[by \textbf{A4}]}.
\end{align*}
In view of the uniqueness of the negative (apply \ref{p:may26:1ii} 
to $x+y$), we deduce that $-(x+y)=-x+(-y)$. 
\end{proof}

Given $x$ and $y$ in $\RR$, we define \textbf{subtraction} by 
\begin{equation}
x - y := x + (-y).
\end{equation}
We then have, e.g., that 
$x- x = x+(-x) = (-(-x))+ (-x) = 0$. 
Subtraction is useful for solving equations involving addition.

\begin{proposition}
\label{p:may26:2}
Let $a$ and $b$ be in $\RR$. Then the equation 
\begin{equation}
\label{e:may26:z}
a+x=b 
\end{equation}
has a unique solution, namely $x=b-a$. 
\end{proposition}
\begin{proof}
Exercise~\ref{exo:may26:1}.
\end{proof}

\section{Axioms of Multiplicaton}

\index{Axioms of Multiplication}
\index{Axiom of Distributivity}

\begin{definition}[Axioms of Multiplication and Distributivity] \ 
\begin{enumerate}
\item[\textbf{M1}]
\textbf{(associativity)}
For all $x,y,z$ in $\RR$, we have: $(xy)z = x(yz)$.
\item[\textbf{M2}]
\textbf{(commutativity)}
For all $x,y$ in $\RR$, we have: $xy=yx$.
\item[\textbf{M3}]
\textbf{(existence of one)}
There exists a number $1\in\RR\smallsetminus\{0\}$ 
such that for every $x\in\RR$, we have
 $x\cdot 1=x$. 
\item[\textbf{M4}]
\textbf{(existence of the inverse)}
For every $x\in\RR\smallsetminus\{0\}$, 
there exists a number $x^{-1}\in\RR$ such that
$x x^{-1}=1$. 
\item[\textbf{D}]
\textbf{(distributivity)}
For all $x,y,z$ in $\RR$, we have: $x(y+z) = xy+xz$. 
\end{enumerate}
\end{definition}

Completely analogous to Proposition~\ref{p:may26:1} and 
Proposition~\ref{p:may26:2}, we obtain the following 
two results.

\begin{proposition}
\label{p:100908:2}
The following hold.
\begin{enumerate}
\item The number $1$ is unique.
\item For every $x\in\RR\smallsetminus\{0\}$, $x^{-1}$ (i.e., the inverse
of $x$) is unique.
\item $1^{-1} = 1$.
\item 
\label{p:100908:2iv}
For every $x\in\RR\smallsetminus\{0\}$, we have:
$(x^{-1})^{-1} =x $.
\item 
\label{p:100908:2v}
For all $x,y$ in $\RR\smallsetminus\{0\}$, we have:
$(xy)^{-1} = x^{-1}\cdot y^{-1}$. 
\end{enumerate}
\end{proposition}

Given $x\in\RR$ and $y\in\RR\smallsetminus\{0\}$, we define 
\textbf{division} by
\begin{equation}
\frac{x}{y} := x/y := xy^{-1}.
\end{equation}
Note that for every $y\in\RR\smallsetminus\{0\}$, one has
$y/y = yy^{-1} = 1$. Division is important for solving equations involving
multiplication.

\begin{proposition}
\label{p:may26:3}
Let $a\in\RR\smallsetminus\{0\}$ and $b\in\RR$. 
Then the equation
\begin{equation}
\label{e:may26:y}
a\cdot x = b
\end{equation}
has a unique solution, namely $x=b/a$.
\end{proposition}
\begin{proof}
Exercise~\ref{exo:may26:2}.
\end{proof}

\begin{proposition}
\label{p:may26:4}
The following hold.
\begin{enumerate}
\item
\label{p:may26:4i}
For all $x,y,z$ in $\RR$, we have:
$(x+y)z = xz+yz$.
\item 
\label{p:may26:4ii}
For every $x\in\RR$, we have $x\cdot 0 = 0$.
\item 
\label{p:may26:4iii}
For all $x,y$ in $\RR$ we have:
$xy = 0$ $\Leftrightarrow$
\emph{[}$x=0$ or $y=0$\emph{]}.
\item
\label{p:may26:4iv}
For every $x\in\RR$, we have:
$-x = (-1)x$.
\item
\label{p:may26:4v}
For all $x,y$ in $\RR$, we have:
$(-x)(-y)=xy$.
\end{enumerate}
\end{proposition}
\begin{proof}
Let $x,y,z$ be in $\RR$.

\ref{p:may26:4i}:
Using \textbf{M2} (repeatedly) and \textbf{D}, we see that
$(x+y)z = z(x+y) = zx + zy = xz + yz$. 

\ref{p:may26:4ii}:
Since $0+0=0$, it follows from \textbf{D} that
$x\cdot 0 + x\cdot 0 = x\cdot(0+0) = x\cdot 0$.
Subtracting $x\cdot 0$ from both sides (i.e.,
adding $-(x\cdot 0)$) yields $x\cdot 0 = 0$. 

\ref{p:may26:4iii}:
``$\Leftarrow$'': If $x=0$, then $xy = yx = y\cdot 0 = 0$ by
\ref{p:may26:4ii}, and similarly when $y=0$.
``$\Rightarrow$'': Now assume that $xy=0$.
If $x=0$, then we are done; thus, we assume that $x\neq 0$.
Multiply from the left by $x^{-1}$, 
and use \ref{p:may26:4ii}, \textbf{M1}, \textbf{M2}, \textbf{M4},
\textbf{M2}, and \textbf{M3}, to conclude that
$0 = x^{-1}0 = x^{-1}(xy) = (x^{-1}x)y = (xx^{-1})y = 1\cdot y = y\cdot 1 =
y$. Hence $y=0$. 

\ref{p:may26:4iv}:
By \textbf{M3}, \textbf{M2}, \ref{p:may26:4i}, \textbf{M2},
\ref{p:may26:4ii},  we obtain
$x+(-1)x = x\cdot 1 + (-1)x = 
1\cdot x + (-1)\cdot x = (1-1)\cdot x = 0\cdot x = x\cdot 0 = 0$.
Hence, because $-x$ is unique 
(Proposition~\ref{p:may26:1}\ref{p:may26:1ii}), 
we deduce that $-x = (-1)x$.  

\ref{p:may26:4v}:
Exercise~\ref{exo:100908:1}. 
\end{proof}

\section{Laws of Associativity, Commutativity, and Distributivity}

While addition is technically defined only for two terms, we extend it to
more than two terms by the following reduction:
\begin{equation}
x_1 + x_2 + \cdots + x_n := \Big( \cdots \big( (x_1+x_2)+x_3\big)+\cdots\Big)
+ x_n.
\end{equation}
Note that the expression on the right is actually independent of the way we
place the parentheses by \textbf{A1}. 
Analogous comments apply to multiplication and \textbf{M1}:
\begin{equation}
x_1 \cdot x_2 \cdot \cdots \cdot x_n := \Big( \cdots \big( (x_1\cdot x_2)\cdot
x_3\big)\cdot \,\cdots\Big)
\cdot x_n.
\end{equation}

Of course, in view of \textbf{A2} and \textbf{M2}, the \emph{order}
in which we add or multiply is not important. More formally,
let $(i_1,\ldots,i_n)$ be a permutation of $(1,\ldots,n)$. 
(E.g., $(4,2,1,3)$ is a permutation of $(1,2,3,4)$.)
Then
\begin{equation}
x_1+x_2 + \cdots + x_n = x_{i_1} + x_{i_2} + \cdots + x_{i_n}
\end{equation}
and
\begin{equation}
x_1\cdot x_2 \cdot \,\cdots\, \cdot x_n = x_{i_1} \cdot x_{i_2}
\cdot\,\cdots \,\cdot  x_{i_n}.
\end{equation}
(Which reflects that $x_1+x_2+x_3+x_4 = x_4 + x_2 + x_1+x_3$, etc.)

We can now state and prove a useful rule for manipulating double sums
(sometimes jokingly called the \emph{Fundamental Theorem of Accounting}
meaning that if you sum up all numbers in a table, it should not matter if
you sum along rows or along columns):
\begin{equation}
\label{e:100908:a}
\sum_{i=1}^{m}\sum_{j=1}^n x_{i,j} = \sum_{j=1}^n \sum_{i=1}^m x_{i,j}.
\end{equation}
Indeed, the left side of \eqref{e:100908:a} is
\begin{align*}
\sum_{i=1}^{m}\sum_{j=1}^n x_{i,j} &= 
\left(\sum_{j=1}^n x_{1,j}\right) + \cdots + \left(\sum_{j=1}^n
x_{m,j}\right)\\
&= \big(x_{1,1} + \cdots + x_{1,n}\big) \\
&\quad + \cdots\\
&\quad + \big(x_{m,1} + \cdots + x_{m,n}\big).
\end{align*}
On the other hand, the right side of \eqref{e:100908:a} is
\begin{align*}
\sum_{j=1}^n \sum_{i=1}^m x_{i,j} &= 
\left(  \sum_{i=1}^m x_{i,1}\right)
+ \cdots +
\left(  \sum_{i=1}^m x_{i,n}\right)\\
&= \big( x_{1,1} + \cdots + x_{m,1}\big)\\
&\quad + \cdots\\
&\quad + \big(x_{1,n} + \cdots + x_{m,n}\big). 
\end{align*}
Altogether, we see that both double sums have exactly the same terms, just
in a different order.
This verifies \eqref{e:100908:a}. 

Another useful identity is
\begin{equation}
\left( \sum_{i=1}^{m} x_i\right)\cdot
\left(\sum_{j=1}^{n}y_j\right)
= \sum_{i=1}^{m}\sum_{j=1}^{n} x_iy_j,
\end{equation}
which follows repeated application of
\textbf{D} and Proposition~\ref{p:may26:4}\ref{p:may26:4i}.

\section{Powers and Exponents}
\begin{definition}
Let $x\in\RR$ and let $n\in\{1,2,\ldots\}$. We define
\begin{equation}
x^0 := 1,
\end{equation}
\begin{equation}
x^n := \underbrace{x\cdot x \cdot\,\cdots\cdot x}_{\text{$n$ terms}},
\end{equation}
and 
\begin{equation}
x^{-n} := (x^{-1})^n,\quad\text{if $x\neq 0$.}
\end{equation}
\end{definition}

One may prove the following rules.

\begin{proposition}
\label{p:100908:1}
Let $x$ and $y$ be in $\RR$, and let $m$ and $n$ be in $\ZZ$.
Then the following hold.
\begin{enumerate}
\item 
$x^{n}x^{m} = x^{n+m}$.
\item
$(x^n)^m = x^{n\cdot m}$.
\item 
\label{p:100908:1iii}
$(xy)^n = x^n\cdot y^n$.
\end{enumerate}
(It is assumed here that $x$ or $y$ is nonzero as needed when 
negative exponents are encountered.)
\end{proposition}

\section{General Fields}

\label{sec:fields}

\index{Field}

\begin{definition}[field]
Let $\FF$ be a set along with two operations
\begin{equation} 
+\colon \FF\times\FF\to \FF \colon (x,y)\mapsto x+y
\end{equation}
and 
\begin{equation} 
\cdot\colon \FF\times\FF\to \FF \colon (x,y)\mapsto x\cdot y.
\end{equation}
such that 
\textbf{A1}--\textbf{A4},
\textbf{M1}--\textbf{M4},
and \textbf{D} hold
(with $\RR$ replaced by $\FF$, of course). 
Then $\FF$ is called a \textbf{field}.
\end{definition}

Note that all the rules derived in this chapter will hold because
we only used the axioms to derive these rules.

Important examples of fields are:
the real number $\RR$,
the rational numbers $\QQ$, 
and the so-called {complex numbers}:

\begin{definition}[complex numbers]
The set of \emph{complex numbers} is 
\begin{equation}
\mathbb{C} := \menge{x+\ii\cdot y}{x\in\RR,y\in\RR},
\quad \text{where} \quad
\ii := \sqrt{-1} 
\quad \text{so that} \quad
\ii^2 = -1.
\end{equation}
Here is how addition and multiplication are defined for two complex numbers
$z_1 = x_1+\ii y_1 $ and $z_2 = x_2 + \ii y_2$:
\begin{align}
\label{e:CCaddmult}
z_1+z_2 &:= (x_1+x_2)+\ii(y_1+y_2),\\
z_1\cdot z_2 &:= (x_1x_2-y_1y_2)+\ii(x_1y_2+x_2y_1).
\end{align}
The zero is $0+\ii 0$ and the one is $1+\ii 0$. 
Given $x+\ii y\in\CC$, the negative and inverse are
\begin{equation}
\label{e:CCinverses}
-(x+\ii y) := (-x) + \ii(-y)
\quad\text{and}\quad
(x+\ii y)^{-1} := \frac{x+\ii(-y)}{x^2 + y^2}.
\end{equation}
\end{definition}

The set of integers $\ZZ$ is \emph{not} a field:
While \textbf{A1}--\textbf{A4}
and \textbf{D} do hold, \textbf{M4} fails.
For instance, the integer $3$ does not have a multiplicative
inverse in $\ZZ$ because $1/3\notin\ZZ$. 
Similarly, the nonnegative integers $\NN$ do not form a field
(in fact, \textbf{A4} fails as well). 

A curious field is $\FF_2 = \{0,1\}$, where the operations are \emph{defined} by 
\hhbcom{A nice table would be great, time permitting}
\begin{subequations}
\label{e:F2add}
\begin{align}
0 + 0 := 0\\
1 + 0 := 1\\
0 + 1 := 1\\
1 + 1 := 0
\end{align}
\end{subequations}
and 
\begin{subequations}
\label{e:F2mult}
\begin{align}
0 \cdot 0 := 0\\
1 \cdot 0 := 0\\
0 \cdot 1 := 0\\
1 \cdot 1 := 1
\end{align}
\end{subequations}
By considering all possible case, we can directly verify that $\FF_2$ is
indeed a field. 
$\FF_2$ and other finite fields play an important role in Number Theory and
Algebra.
Note that in $\FF_2$, we have the rather curious looking identity $1+1=0$. 

In Algebra, it can be shown (with quite some work) that
each finite field is essentially unique and contains $p^n$
elements, where $p$ is a prime number and $n\in\{1,2,\ldots\}$. 

The next result is convenient for checking that subsets of $\RR$ are
fields.

\begin{proposition}
\label{p:subfield}
Let $\SR$ be a subset of $\RR$ such that the following hold:
\begin{enumerate}
\item 
$0\in\SR$.
\item
$1\in\SR$.
\item 
If $x$ and $y$ belong to $\SR$, then so does $x+y$.
\item 
If $x$ and $y$ belong to $\SR$, then so does $x\cdot y$.
\item 
If $x\in\SR$, then $-x\in\SR$.
\item 
If $x\in\SR\smallsetminus\{0\}$, then $x^{-1}\in\SR$.
\end{enumerate}
Then $\SR$ is also a field (with addition and multiplication
inherited from $\RR$).
\end{proposition}
\begin{proof}
Indeed, all axioms hold since they do hold for $\RR$.
The required assumptions are needed to guarantee that
all required objects exists in $\SR$ and that addition and multiplication
do not lead us outside $\SR$.
\end{proof} 

\section*{Exercises}\markright{Exercises}
\addcontentsline{toc}{section}{Exercises}
\setcounter{theorem}{0}

\begin{exercise}
\label{exo:may26:1} 
Verify Proposition~\ref{p:may26:2}.
(You need to check that (i) $x = b-a$ is a solution of 
\eqref{e:may26:z} and that (ii) there is no other solution.)
\end{exercise}
\begin{solution}
We need to show that (i) the equation has a solution and (ii) the solution
is unique.

(i): Using (in this order) the definition of subtraction,
\textbf{A2}, \textbf{A1}, \textbf{A4},
\textbf{A2}, and \textbf{A3},
we see that $a+(b-a) = a+\big(b+ (-a)\big) = a+\big( (-a)+b\big)
= \big(a + (-a)\big)+b = 0+b = b+0 = b$, so $b-a$ is a solution.

(ii): Now assume that $y$ is another (possibly different from $b-a$) 
solution, i.e.,
$a+y=b$. 
Adding $-a$ to both sides from the left, 
and also using \textbf{A1}, \textbf{A2}, \textbf{A4} and \textbf{A3}, 
we obtain 
$(-a)+(a+y) = (-a) + b$, i.e.,
$\big((-a)+a\big)+y = b+ (-a)$, i.e.,
$\big(a+(-a)\big)+y = b-a$, i.e.,
$0+y=y+0 = y =b-a$. 
\end{solution}

\begin{exercise}
Verify Proposition~\ref{p:100908:2}.
\end{exercise}
\begin{solution}
(i): Let $1'$ be a number that acts like a one.
Then $1\cdot 1' = 1$ and $1'\cdot 1 = 1'$ by \textbf{M3}.
In view of \textbf{M2}, we obtain
$1 = 1\cdot 1' = 1'\cdot 1 = 1'$. 

(ii):
Let $y$ be a number that acts like an inverse of $x$, where
$x\in \RR\smallsetminus\{0\}$.
Then 
\begin{equation}
\label{e:171011}
x\cdot y=1.
\end{equation}
Therefore, 
\begin{align*}
y &= y\cdot 1 \quad\text{[by \textbf{M3}]}\\
&= 1\cdot y \quad\text{[by \textbf{M2}]}\\
&= \big(x\cdot x^{-1}\big)\cdot y \quad\text{[by \textbf{M4}]}\\
&= \big(x^{-1}\cdot x\big) \cdot y \quad\text{[by \textbf{M2}]}\\
&= x^{-1}\cdot (x \cdot y) \quad\text{[by \textbf{M1}]}\\
&= x^{-1}\cdot 1 \quad\text{[by \eqref{e:171011}]}\\
&= x^{-1} \quad\text{[by \textbf{M3}],}
\end{align*}
as claimed. 

(iii): 
Since $1\cdot 1 =1$ by \textbf{M3}, 
it follows from (ii) that $1^{-1}=1$. 

(iv): 
Let $x\in\RR\smallsetminus\{0\}$.
By \textbf{M3} and \textbf{M2}, we have
$1=x\cdot x^{-1} = x^{-1}\cdot x$. 
On the other hand, $x^{-1}\neq 0$ and
$x^{-1}\cdot x = 1$; thus, by (ii),
$(x^{-1})^{-1}=x$. 

(v):
Since the inverse of $x\cdot y$ is unique,
it suffices to check that the candidate for the inverse does the
job. Indeed, 
\begin{align*}
(x\cdot y) \cdot \big( x^{-1}\cdot y^{-1}\big) &= x \cdot \Big(
y \cdot \big( x^{-1} \cdot y^{-1}\big)\Big)
\quad\text{[by \textbf{M1}]}\\
&= x \cdot \Big( y \cdot \big( y^{-1}\cdot x^{-1}\big)\Big)
\quad\text{[by \textbf{M2}]}\\
&= x \cdot \Big( \big(y \cdot y^{-1}\big)\cdot x^{-1}\Big)
\quad\text{[by \textbf{M1}]}\\
&= x \cdot \big( 1 \cdot x^{-}\big)
\quad\text{[by \textbf{M4}]}\\
&= x \cdot \big( x^{-1} \cdot 1\big)
\quad\text{[by \textbf{M2}]}\\
&= x \cdot x^{-1}
\quad\text{[by \textbf{M3}]}\\
&= 1
\quad \text{[by \textbf{M4}]}.
\end{align*}
\end{solution}

\begin{exercise}
\label{exo:may26:2} 
Verify Proposition~\ref{p:may26:3}. 
(You need to check that (i) $x = b/a$ is a solution 
of \eqref{e:may26:y} and that (ii) there is no other solution.)
What happens when $a=0$?
\end{exercise}
\begin{solution}
(i): Using (in this order) 
\textbf{M2}, \textbf{M1}, \textbf{M4},
\textbf{M2}, and \textbf{M3},
we see that $a\cdot b/a = a\cdot (b a^{-1}) = a\big(a^{-1}\cdot b\big)
= \big(a \cdot a^{-1}\big)\cdot b = 1\cdot b = b\cdot 1 = b$, so $b/a$ is a solution.

(ii): Now assume that $y$ is another (possibly different from $b/a$) solution, i.e.,
$ay=b$. Multiplying $a^{-1}$ to both sides from the left, 
and also using \textbf{M1}, \textbf{M2}, \textbf{M4} and \textbf{M3}, 
yields
$a^{-1}\cdot(ay) = a^{-1} \cdot b$, i.e.,
$\big(a^{-1}a\big)y = ba^{-1}$, i.e.,
$\big(aa^{-1}\big)y = b/a$, i.e.,
$1\cdot y=y\cdot 1 = y =b/a$. 

Now suppose that $a=0$. 
Then, 
using \textbf{M2} and Proposition~\ref{p:may26:4}\ref{p:may26:4ii}, 
$a\cdot x = 0\cdot x =x \cdot 0 =0$.

Thus if $b \neq 0$, 
then $a\cdot x= 0 \neq b$, so there is no solution.

And if $b=0$, then $a\cdot x = 0 = b$,
so every $x\in\RR$ is a solution. 
\end{solution}

\begin{exercise}
\label{exo:100908:1}
Verify Proposition~\ref{p:may26:4}\ref{p:may26:4v}.
\emph{Hint}: Use Proposition~\ref{p:may26:4}\ref{p:may26:4iv}. 
\end{exercise}
\begin{solution}
Indeed,
\begin{align*}
(-x)(-y) &= (-x)\big((-1)y\big) \quad\text{[by Proposition~\ref{p:may26:4}\ref{p:may26:4iv}]}\\
&= \big((-x)(-1)\big)y \quad\text{[by \textbf{M1}]}\\
&= \big((-1)(-x)\big)y \quad\text{[by \textbf{M2}]}\\
&= \big(-(-x)\big)y \quad\text{[by Proposition~\ref{p:may26:4}\ref{p:may26:4iv}]}\\
&= xy \quad\text{[by Proposition~\ref{p:may26:1}\ref{p:may26:1iv}].}
\end{align*}
The proof is complete. 
\end{solution}

\begin{exercise}
Prove the following:
For every $n\in\{2,3,4,\ldots\}$ and
real numbers $x_1,\ldots,x_n$ such that the product $x_1x_2\cdots x_n$
is equal to $0$, then at least one of the factors must be equal to $0$.
\end{exercise}
\begin{solution}
Denote the statement in question by $S(n)$.

\textbf{Base Case}: The statement is true for $n=2$ by 
the implication ``$\Rightarrow$'' of
Proposition~\ref{p:may26:4}\ref{p:may26:4iii}.

\textbf{Inductive Step}: Now let $n\in\{2,3,\ldots\}$ be such that
$S(n)$ is true. We must show that $S(n+1)$ is true. Thus assume
$x_1,\ldots,x_n,x_{n+1}$ are real numbers such that $x_1\cdots x_nx_{n+1} =
0$. Write this as $(x_1\cdots x_n)x_{n+1}=0$; here, we have a product with
two factors. Since the Base Case is true, we must have that
$x_1\cdots x_n=0$ or that $x_{n+1}=0$. 
In the latter case, we are done. In the former case, 
the inductive hypothesis yields that one of the numbers $x_1,\ldots,x_n$
must equal $0$, and we are done as well.
In either case, $S(n+1)$ is verified.

Therefore, the entire statement is proven using the Principle of
Mathematical Induction. 
\end{solution}

\begin{exercise}
Suppose that $x$ and $y$ are two nonzero real numbers. Show
that 
$\displaystyle \left(\frac{x}{y}\right)^{-1} = \frac{y}{x}$. 
\end{exercise}
\begin{solution}
By definition, $\tfrac{x}{y} = xy^{-1}$ and $\tfrac{y}{x} = yx^{-1}$,
which are both well defined because neither $x$ nor $y$ is equal to $0$.
The two quotients/products are nonzero by
Proposition~\ref{p:may26:4}\ref{p:may26:4iii}. 
Using Proposition~\ref{p:100908:2}\ref{p:100908:2v}, 
Proposition~\ref{p:100908:2}\ref{p:100908:2iv}, 
and \textbf{M2}, 
we have
\begin{equation*}
\left(\frac{x}{y}\right)^{-1} 
= \big(xy^{-1}\big)^{-1} 
= x^{-1}(y^{-1})^{-1} 
= x^{-1}y 
= yx^{-1}
= \frac{y}{x}, 
\end{equation*}
as required. 
\emph{Note}: Similarly you can prove all rules you have learnt
on manipulating fractions such as 
$\tfrac{a}{b}\big/\tfrac{x}{y} = \tfrac{ay}{bx}$, etc. 
\end{solution}

\begin{exercise}
Prove Proposition~\ref{p:100908:1}\ref{p:100908:1iii}.
\emph{Hint:} Consider two cases, $n\geq 0$ and $n<0$. 
In the former case, use induction. In the latter case, reduce to the former
case.
\end{exercise}
\begin{solution}
Assume first $n\geq 0$.
If $n=0$, then $(xy)^0 = 1$, $x^0=1$, and $y^0=1$ by definition and the
result is clear.
Now assume the result is true for some nonnegative integer $n\geq 0$.
Then
\begin{equation*}
(xy)^{n+1} = (xy)(xy)^n = (xy)\big(x^ny^n\big) = x\cdot y \cdot x^n\cdot
y^n = \big(x\cdot x^n\big)\cdot \big( y\cdot y^n\big)
= x^{n+1}y^{n+1}.
\end{equation*}
By the Principle of Mathematical Induction, the result holds for all
integers $n\geq 0$.

Now assume that $n<0$. Then $m := -n>0$ and 
by definition and using the already verified case
as well as Proposition~\ref{p:100908:2}\ref{p:100908:2v},
we obtain
\begin{align*}
x^n y^n &= x^{-m}y^{-m} = (x^{-1})^m (y^{-1})^m \\
&= \big( x^{-1}y^{-1}\big)^m = \big( (xy)^{-1}\big)^m = 
(xy)^{-m} = (xy)^n.
\end{align*}
(In this case, we must have that $x\neq 0$ and that $y\neq 0$.)
\end{solution}

\begin{exercise}
Check that $\FF_2$ is a field.
\end{exercise}
\begin{solution} 
OK, not much choice here --- we need to work through the nine axioms!

\textbf{A1}:
There are 8 cases to consider $x\in\{0,1\},y\in\{0,1\},z\in\{0,1\}$.
We shall use \eqref{e:F2add} repeatedly. 

\emph{Case~1:} $x=0$, $y=0$, $z=0$:\\
Then 
$(x+y)+z = (0+0)+0 = 0+0 = 0 = 0 + 0 = 0+(0+0) = x+(y+z)$.

\emph{Case~2:} $x=0$, $y=0$, $z=1$:\\
Then 
$(x+y)+z = (0+0)+1 = 0+1 = 1 = 0 + 1 = 0+(0+1) = x+(y+z)$. 

\emph{Case~3:} $x=0$, $y=1$, $z=0$:\\
Then 
$(x+y)+z = (0+1)+0 = 1+0 = 1 = 0 + 1 = 0+(1+0) = x+(y+z)$. 

\emph{Case~4:} $x=0$, $y=1$, $z=1$:\\
Then 
$(x+y)+z = (0+1)+1 = 1+1 = 0 = 0 + 0 = 0+(1+1) = x+(y+z)$. 

\emph{Case~5:} $x=1$, $y=0$, $z=0$:\\
Then 
$(x+y)+z = (1+0)+0 = 1+0 = 1 = 1 + 0 = 1+(0+0) = x+(y+z)$. 

\emph{Case~6:} $x=1$, $y=0$, $z=1$:\\
Then 
$(x+y)+z = (1+0)+1 = 1+1 = 0 = 1 + 1 = 1+(0+1) = x+(y+z)$. 

\emph{Case~7:} $x=1$, $y=1$, $z=0$:\\
Then 
$(x+y)+z = (1+1)+0 = 0+0 = 0 = 1 + 1 = 1+(1+0) = x+(y+z)$. 

\emph{Case~8:} $x=1$, $y=1$, $z=1$:\\
Then 
$(x+y)+z = (1+1)+1 = 0+1 = 1 = 1 + 0 = 1+(1+1) = x+(y+z)$. 

\textbf{A2}:
It is clear from \eqref{e:F2add} that this is true. 
(Or consider 4 cases: $x\in\{0,1\},y\in\{0,1\}$.)

\textbf{A3}:
It is clear from \eqref{e:F2add} that this is true. 
(Or consider 2 cases: $x\in\{0,1\}$.)

\textbf{A4}:
We define $-0 := 0$ and $-1 := 1$.
Then, by \eqref{e:F2add},
$0+(-0) = 0 + 0 = 0$ and 
$1+(-1) = 1 + 1 = 0$ as required.

\textbf{M1}:
There are 8 cases to consider $x\in\{0,1\},y\in\{0,1\},z\in\{0,1\}$.
We shall use \eqref{e:F2mult} repeatedly. 

\emph{Case~1:} $x=0$, $y=0$, $z=0$:\\
Then 
$(x\cdot y)\cdot z = (0\cdot 0)\cdot 0 = 0\cdot 0 = 0 = 0 \cdot 0 = 
0\cdot (0\cdot 0) = x\cdot (y\cdot z)$.

\emph{Case~2:} $x=0$, $y=0$, $z=1$:\\
Then 
$(x\cdot y)\cdot z = (0\cdot 0)\cdot 1 = 0\cdot 1 = 0 = 0 \cdot 0 
= 0\cdot (0\cdot 1) = x\cdot (y\cdot z)$. 

\emph{Case~3:} $x=0$, $y=1$, $z=0$:\\
Then 
$(x\cdot y)\cdot z = (0\cdot 1)\cdot 0 = 0\cdot 0 = 0 = 0 \cdot 0 = 
0\cdot (1\cdot 0) = x\cdot (y\cdot z)$. 

\emph{Case~4:} $x=0$, $y=1$, $z=1$:\\
Then 
$(x\cdot y)\cdot z = (0\cdot 1)\cdot 1 = 0\cdot 1 = 0 
= 0 \cdot  1 = 0\cdot (1\cdot 1) = x\cdot (y\cdot z)$. 

\emph{Case~5:} $x=1$, $y=0$, $z=0$:\\
Then 
$(x\cdot y)\cdot z = (1\cdot 0)\cdot 0 = 0\cdot 0 = 0 
= 1 \cdot 0 = 1\cdot (0\cdot 0) = x\cdot (y\cdot z)$. 

\emph{Case~6:} $x=1$, $y=0$, $z=1$:\\
Then 
$(x\cdot y)\cdot z = (1\cdot 0)\cdot 1 = 0\cdot 1 = 0 = 1 \cdot 0 = 
1\cdot (0\cdot 1) = x\cdot (y\cdot z)$. 

\emph{Case~7:} $x=1$, $y=1$, $z=0$:\\
Then 
$(x\cdot y)\cdot z = (1\cdot 1)\cdot 0 = 1\cdot 0 
= 0 = 1 \cdot 0 = 1\cdot (1\cdot 0) = x\cdot (y\cdot z)$. 

\emph{Case~8:} $x=1$, $y=1$, $z=1$:\\
Then 
$(x\cdot y)\cdot z = (1\cdot 1)\cdot 1 = 1 \cdot 1  
= 1 = 1 \cdot 1 = 1\cdot (1\cdot 1) = x\cdot (y\cdot z)$. 

\textbf{M2}:
It is clear from \eqref{e:F2mult} that this is true. 
(Or consider 4 cases: $x\in\{0,1\},y\in\{0,1\}$.)

\textbf{M3}:
It is clear from \eqref{e:F2mult} that this is true. 
(Or consider 2 cases: $x\in\{0,1\}$.)

\textbf{M4}:
We define $1^{-1} := 1$.
Then, by \eqref{e:F2mult},
$1\cdot 1^{-1} = 1 \cdot 1 = 1$ 
as required. 

\textbf{D}:
There are 8 cases to consider $x\in\{0,1\},y\in\{0,1\},z\in\{0,1\}$.
We shall use \eqref{e:F2add} and \eqref{e:F2mult} repeatedly. 

\emph{Case~1:} $x=0$, $y=0$, $z=0$:\\
Then 
$x(y+z) = 0\cdot(0+0) = 0\cdot 0 = 0 = 0 + 0 = 0\cdot 0 + 0 \cdot 0 = xy+xz$. 

\emph{Case~2:} $x=0$, $y=0$, $z=1$:\\
Then 
$x(y+z) = 0\cdot(0+1) = 0\cdot 1 = 0 = 0 + 0 = 0\cdot 0 + 0 \cdot 1 = xy+xz$. 

\emph{Case~3:} $x=0$, $y=1$, $z=0$:\\
Then 
$x(y+z) = 0\cdot(1+0) = 0\cdot 1 = 0 = 0 + 0 = 0\cdot 1 + 0 \cdot 0 = xy+xz$. 

\emph{Case~4:} $x=0$, $y=1$, $z=1$:\\
Then 
$x(y+z) = 0\cdot(1+1) = 0\cdot 0 = 0 = 0 + 0 = 0\cdot 1 + 0 \cdot 1 = xy+xz$. 

\emph{Case~5:} $x=1$, $y=0$, $z=0$:\\
Then 
$x(y+z) = 1\cdot(0+0) = 1\cdot 0 = 0 = 0 + 0 = 1\cdot 0 + 1 \cdot 0 = xy+xz$. 

\emph{Case~6:} $x=1$, $y=0$, $z=1$:\\
Then 
$x(y+z) = 1\cdot(0+1) = 1\cdot 1 = 1 = 0 + 1 = 1\cdot 0 + 1 \cdot 1 = xy+xz$. 

\emph{Case~7:} $x=1$, $y=1$, $z=0$:\\
Then 
$x(y+z) = 1\cdot(1+0) = 1\cdot 1 = 1 = 1 + 0 = 1\cdot 1 + 1 \cdot 0 = xy+xz$. 

\emph{Case~8:} $x=1$, $y=1$, $z=1$:\\
Then 
$x(y+z) = 1\cdot(1+1) = 1\cdot 0 = 0 = 1 + 1 = 1\cdot 1 + 1 \cdot 1 = xy+xz$. 
\end{solution}

\begin{exercise}
\label{exo:QQ}
Check that $\QQ$ is a field.
\emph{Hint:} Proposition~\ref{p:subfield}. 
\end{exercise}
\begin{solution}
Observe that $0 = 0/1$ and $1=1/1$ are rational numbers.
Let $x = a/b$ and $y=c/d$ be in $\QQ$,
where $a,b,c,d$ are integers, and $b,d$ are nonzero. 
Then $x+y = (ad + bc)/(bd) \in \QQ$
and $xy = (ac)/(bd)\in\QQ$. 
Moreover, $-x = (-a)/b \in \QQ$
and if $x\neq 0$, i.e., $a\neq 0$, then
$x^{-1} = b/a \in \QQ$ as well. 
The result thus follows from Proposition~\ref{p:subfield}. 
\end{solution}

\begin{exercise}
The complex numbers $\CC$ form a field\footnote{The complex numbers are crucial
in modern Physics and Engineering.}.
\begin{enumerate}
    \item Verify \textbf{A1}, \textbf{A2}, \textbf{A3}, and \textbf{A4}.
    \item Verify \textbf{M1}, \textbf{M2}, \textbf{M3}, and \textbf{M4}.
    \item Verify \textbf{D}. 
\end{enumerate}
\end{exercise}
\begin{solution} 
Let $z,z_1,z_2,z_3$ be in $\CC$, say 
$z=x +\ii y$, 
$z_1 = x_1+\ii y_1$, 
$z_2 = x_2+\ii y_2$, 
$z_3 = x_3+\ii y_3$, 
where 
$x,x_1,x_2,x_3,y,y_1,y_2,y_3$ are in $\RR$.

We now verify the nine axioms.
We will make use of \eqref{e:CCaddmult} 
and \eqref{e:CCinverses} repeatedly 
as well as the field axioms for the real numbers $\RR$: 

(i):
\textbf{A1}:
Indeed, using also \textbf{A1} for $\RR$, we get 
\begin{align*}
(z_1+z_2)+z_3
&= \big((x_1+\ii y_1) + (x_2+\ii y_2)\big) + (x_3+\ii y_3)
= \big((x_1+x_2) + \ii(y_1+y_2) \big) + (x_3+\ii y_3)\\
&= \big((x_1+x_2) + x_3 \big) +\ii\big((y_1+y_2)+y_3\big)
= \big(x_1+(x_2+x_3) \big) +\ii(\big(y_1+(y_2+y_3)\big)\\
&= (x_1+\ii y_1) + \big((x_2+x_3)+\ii(y_2+y_3)\big)
= (x_1+\ii y_1) + \big((x_2+\ii y_2)+(x_3+\ii y_3)\big)\\
&= z_1 + (z_2+z_3).
\end{align*}

\textbf{A2}:
Indeed, using also \textbf{A2} for $\RR$, we get 
\begin{align*}
z_1+z_2
&= (x_1+\ii y_1) + (x_2+\ii y_2)
= (x_1+x_2) + \ii(y_1+y_2)
= (x_2+x_1) + \ii(y_2+y_1)\\
&= (x_2+\ii y_2) + (x_1+\ii y_1)
= z_2+z_1.
\end{align*}

\textbf{A3}:
Indeed, using also \textbf{A3} for $\RR$, we get 
\begin{align*}
z+0 
&= (x+\ii y) + (0+\ii 0)
= (x+0) + \ii(y+0)
= x + \ii y
= z.
\end{align*}

\textbf{A4}:
Indeed, using also \textbf{A4} for $\RR$, we get 
\begin{align*}
z+ (-z)
&= (x+\ii y) + \big((-x)+\ii (-y)\big)
= \big(x+(-x)\big) + \ii\big(y+(-y)\big)
= 0 + \ii 0
= 0.
\end{align*}

(ii): \textbf{M1}:
Indeed, 
\begin{align*}
(z_1\cdot z_2)\cdot z_3 
&= \big((x_1+\ii y_1)\cdot (x_2+\ii y_2)\big)\cdot (x_3+\ii y_3)
= \big((x_1x_2-y_1y_2)+\ii(x_1y_2+x_2y_1) \big) \cdot (x_3+\ii y_3)\\
&= \big((x_1x_2-y_1y_2)x_3 - (x_1y_2+x_2y_1)y_3 \big) 
+ \ii\big((x_1x_2-y_1y_2)y_3+x_3(x_1y_2+x_2y_1)\big)\\
&= \big(x_1x_2x_3-y_1y_2x_3-x_1y_2y_3-y_1x_2y_3\big)
+\ii\big(x_1x_2y_3-y_1y_2y_3+x_1y_2x_3+y_1x_2x_3\big)\\
&= \big(x_1(x_2x_3-y_2y_3)-y_1(x_2y_3+x_3y_2) \big)
+ \ii\big(x_1(x_2y_3+x_3y_2)+ (x_2x_3-y_2y_3)y_1\big)\\
&= (x_1+\ii y_1)\big((x_2x_3-y_2y_3)+\ii(x_2y_3+x_3y_2)\big)
= (x_1+\ii y_1)\big((x_2+\ii y_2)\cdot (x_3+\ii y_3)\big)\\
&= z_1 \cdot (z_2\cdot z_3).
\end{align*}

\textbf{M2}:
Indeed, 
\begin{align*}
z_1\cdot z_2 
&= (x_1+\ii y_1)\cdot (x_2+\ii y_2)
= (x_1x_2-y_1y_2)+\ii(x_1y_2+x_2y_1)
= (x_2x_1-y_2y_1)+\ii(x_2y_1+x_1y_2)\\
&= (x_2+\ii y_2)\cdot (x_1+\ii y_1)
= z_2 \cdot z_1.
\end{align*}

\textbf{M3}:
Indeed, 
\begin{align*}
z \cdot 1
&= (x+\ii y)\cdot (1+\ii 0)
= (x\cdot 1 - y\cdot 0)+\ii(x\cdot 0 + 1\cdot y)
= (x-0)+\ii(0+y)
= x+\ii y
= z. 
\end{align*}

\textbf{M4}:
Indeed, assuming that $z\neq 0$, we have 
\begin{align*}
z \cdot z^{-1}
&= (x+\ii y)\cdot \Big(\frac{x}{x^2+y^2}+\ii\frac{-y}{x^2+y^2}\Big)\\
&= \Big(x\cdot \frac{x}{x^2+y^2}-y\frac{-y}{x^2+y^2}\Big) + 
\ii\Big(x\frac{-y}{x^2+y^2}+\frac{x}{x^2+y^2}y\Big)\\
&= \frac{x^2+(-y)^2}{x^2+y^2}
+ \ii\frac{x(-y)+xy}{x^2+y^2}
= 1+\ii 0
= 1.
\end{align*}

(iii): \textbf{D}:
Indeed, we have 
\begin{align*}
z_1(z_2+z_3)
&= (x_1+\ii y_1)\big((x_2+\ii y_2) + (x_3+\ii y_3)\big)
= (x_1+\ii y_1)\big((x_2+x_3)+\ii (y_2 +  y_3)\big)\\ 
&= \big(x_1(x_2+x_3)-y_1(y_2+y_3) \big)
+\ii\big(x_1(y_2+y_3)+(x_2+x_3)y_1 \big)\\ 
&= \big(x_1x_2+x_1x_3-y_1y_2-y_1y_3 \big)
+ \ii\big(x_1y_2+x_1y_3+x_2y_1+x_3y_1\big)\\
&= \big(x_1x_2-y_1y_2+x_1x_3-y_1y_3\big)
+\ii\big(x_1y_2+x_2y_1 + x_1y_3 + x_3y_1\big)\\
&= \big((x_1x_2-y_1y_2)+\ii(x_1y_2+x_2y_1) \big)
+ \big((x_1x_3-y_1y_3)+\ii(x_1y_3+x_3y_1) \big) \\
&= (x_1+\ii y_1)(x_2+\ii y_2)+ (x_1+\ii y_1)(x_3+\ii y_3)
= z_1z_2+z_1z_3.
\end{align*}
\end{solution}

\begin{exercise}
In this exercise, we assume that $\sqrt{2}\in\RR\smallsetminus\QQ$,
a fact we shall derive later (see Theorem~\ref{t:squareroot2}). 
We define
\begin{equation*}
\QQ(\sqrt{2}) := \menge{x+\sqrt{2}y}{x\in\QQ,\,y\in\QQ}.
\end{equation*}
Show that $\QQ(\sqrt{2})$ is a field\footnote{Fields of this type
arise in Number Theory and Algebra (Galois Theory).}.
\emph{Hint:} Proposition~\ref{p:subfield} and Exercise~\ref{exo:QQ}. 
\end{exercise}
\begin{solution}
First, $0 = 0+\sqrt{2}0$, $1 = 1 + \sqrt{2}0$,
and $\{0,1\}\subseteq \QQ$, hence $0$ and $1$ belong to 
$\QQ(\sqrt{2})$.

Now let $x+\sqrt{2}y$ and $u+\sqrt{2}v$ be in 
$\QQ(\sqrt{2})$, where $\{x,y,u,v\}\subseteq\QQ$. 
Then
\begin{equation*}
\big(x+\sqrt{2}y\big) + \big(u+\sqrt{2}v\big)
= (x+u)+\sqrt{2}(y+v)\in \QQ(\sqrt{2})
\end{equation*}
because $x+u \in \QQ$ and $y+v\in\QQ$. 
Similarly, since $\sqrt{2}^2 = 2$, 
\begin{equation*}
\big(x+\sqrt{2}y\big) \cdot \big(u+\sqrt{2}v\big)
= (xu + 2yv) + \sqrt{2}(xv+yu) \in  \QQ(\sqrt{2}).
\end{equation*}
Furthermore, $-(x+\sqrt{2}y) = (-x) + \sqrt{2}(-y)\in
 \QQ(\sqrt{2})$ since $\{-x,-y\}\subseteq\QQ$.
Finally, if $x+\sqrt{2}y\neq 0$, then
\begin{equation*}
\frac{1}{x+\sqrt{2}y} = 
\frac{x+\sqrt{2}(-y)}{x^2 - 2y^2} \in  \QQ(\sqrt{2})
\end{equation*}
because $\{x,-y,x^2-2y^2\}\subseteq\QQ$.
Note that we do \emph{not} divide by $0$:
$x^2\neq 2y^2$ because $\sqrt{2}\notin\QQ$. 
\end{solution}

\begin{exercise}
Let $f\colon \QQ\to\RR$ be a function such that 
for all $x$ and $y$ in $\QQ$, we have
$f(x+y) = f(x)+f(y)$.
Show that $f(x)=f(1)\cdot x$, for every $x\in\QQ$.
\end{exercise}
\begin{solution}
We do this by ``hill climbing'', a proof technique that
gradually builds up to the desired conclusion.

First, an easy induction shows that 
\begin{equation}
\label{e:blubbers}
f(a_1 + \cdots + a_n) = f(a_1) + \cdots + f(a_n).
\end{equation}

We now tackle the proof

\emph{Case 1: $x\in\{1,2,\ldots\}$}:\\
Clearly $f(1) = f(1)\cdot 1$. 
Also, $f(2) = f(1+1) = f(1)+f(1)= 2f(1)$.
Inductively, $f(n) = n f(1)$.

\emph{Case 2: $x\in\{0,-1,-2,\ldots\}$}:\\
We have $f(0) = f(0+0) = f(0) + f(0) = 2f(0)$
and thus $f(0) = 0 = 0f(1)$.
Next, if $n\in\NN$, then
$0 = f(0) = f(n+(-n)) = f(n) + f(-n)$ and
so $f(-n) = -f(n) = -n f(1)=(-n)f(1)$.

\emph{Case 3: $x\in\{1/1,1/2,1/3,\ldots\}$}:\\
Using \eqref{e:blubbers},
we obtain 
\begin{equation*}
f(1) = f\big(\tfrac{1}{n} + \cdots + \tfrac{1}{n}\big)
= f\big(\tfrac{1}{n}\big) + \cdots + 
f\big(\tfrac{1}{n}\big) = n f\big(\tfrac{1}{n}\big).
\end{equation*}
Hence $f(1/n) = (1/n) f(1)$. 

\emph{Case 4: $x\in\{-1/1,-1/2,-1/3,\ldots\}$}:\\
Similarly to the proof of Case~2.

\emph{Case 5: $x\in\QQ$}:\\
Write $x = m/n$, where $n$ is a nonzero integer.
Using \eqref{e:blubbers} and Case~4, we finally obtain
\begin{equation*}
f\big(\tfrac{m}{n}\big)
= m f\big(\tfrac{1}{n}\big) = \tfrac{m}{n}f(1).
\end{equation*}
\end{solution}

\begin{exercise}
Evaluate
$\displaystyle 
\sum_{j=0}^n \sum_{i=j}^n {n\choose i}{i\choose j}
$.
\end{exercise}
\begin{solution}
Observe that we sum over indices $(i,j)$ such that
$0\leq j\leq i\leq n$.
In the sum in question, we fix $j$ and sum over $i$ first.
Alternatively, and with the same overall result,
we can fix $i$ and sum over $j$, thus obtaining with
the help of the Binomial Theorem, 
\begin{align*}
\sum_{j=0}^n \sum_{i=j}^n {n\choose i}{i\choose j}
&= \sum_{i=0}^n \sum_{j=0}^{i} {n\choose i}{i\choose j}
= \sum_{i=0}^n {n\choose i} \sum_{j=0}^{i} {i\choose j}
= \sum_{i=0}^n {n\choose i} 2^i
= (1+2)^n 
= 3^n. 
\end{align*}
\end{solution}

\begin{exercise}
\label{exo:walaa3}
Let $x>0$ be a real number such that $x+\tfrac{1}{x}\in \NN$.
Show that for every $\nnn$, $x^n + \tfrac{1}{x^n}\in\NN$ as well. 
\emph{Hint:} Induction on $n$. 
\end{exercise}
\begin{solution}
We prove this by strong induction for. Clearly,
the result holds when $n=0$ since $x^0+1/x^0 = 1+1/1 = 2\in\NN$.
It also clearly holds when $n=1$, by hypothesis on $x$. 

Now suppose that $x^k + \tfrac{1}{x^k}\in\NN$ for every
$k\in\{0,1,\ldots,n\}$ and for some $n\geq 1$. 
Observe that
\begin{equation*}
x^{n+1} + \frac{1}{x^{n+1}}  = 
\underbrace{\underbrace{\left(x+\tfrac{1}{x}\right)}_{\in\NN} 
\underbrace{\left(x^n + \tfrac{1}{x^n}\right)}_{\in\NN}}_{\in\NN}
- \underbrace{\left( x^{n-1} + \tfrac{1}{x^{n-1}}\right)}_{\in\NN}. 
\end{equation*}
Hence $x^{n+1} + \frac{1}{x^{n+1}}\in\ZZ$. But since this quantity is
obviously positive (since $x>0$), it follows that
$x^{n+1} + \frac{1}{x^{n+1}}\in\{1,2,\ldots\}$, as required.

Therefore, by the Principle of Mathematical Induction, the result is
proven.
\end{solution}

\begin{exercise}
\label{exo:rerafu}
A \emph{real rational function} $f(x)$ is a quotient of polynomials
$$ 
f(x) = \frac{ a_m x^m+a_{m-1} x^{m-1} + \cdots + a_1x + a_0}{b_n x^n +
b_{n-1}x^{n-1} + \cdots + b_1x+b_0},
$$ where all coefficients are real, $m\in\NN$, $n\in\NN$, and $b_n\neq 0$.
\end{exercise}
The set of real rational functions form a field, commonly written
as $\RR(x)$, 
with 
$$ (f+g)(x) := f(x)+g(x) \quad\text{and}\quad (fg)(x) := f(x)\cdot g(x),$$
with $0$ and $1$ viewed as constant polynomials, 
with 
$$ 
(-f)(x) := \frac{-a_m x^m-a_{m-1} x^{m-1} - \cdots - a_1x - a_0}{b_n x^n +
b_{n-1}x^{n-1} + \cdots + b_1x+b_0}
$$ 
and with 
$$ (1/f)(x) := \frac{b_n x^n + b_{n-1}x^{n-1} + \cdots + b_1x+b_0}{ a_m
x^m+a_{m-1} x^{m-1} + \cdots + a_1x + a_0}$$
when $a_m\neq 0$.
\begin{enumerate}
\item Verify \textbf{A4}.
\item Verify \textbf{M4}.
\end{enumerate}
\begin{solution}
Write 
\begin{equation*}
f(x) = \frac{ a_m x^m+a_{m-1} x^{m-1} + \cdots + a_1x + a_0}{b_n x^n +
b_{n-1}x^{n-1} + \cdots + b_1x+b_0},
\end{equation*}
where $b_n\neq 0$. 

(i): \textbf{A4}:
Using the definition and \textbf{D}, we obtain 
\begin{align*}
(-f)(x) &= \frac{-a_m x^m-a_{m-1} x^{m-1} - \cdots - a_1x - a_0}{b_n x^n +
b_{n-1}x^{n-1} + \cdots + b_1x+b_0}\\
&= - \frac{ a_m x^m+a_{m-1} x^{m-1} + \cdots + a_1x + a_0}{b_n x^n +
b_{n-1}x^{n-1} + \cdots + b_1x+b_0}\\
&= -f(x).
\end{align*}
Hence $f(x) + (-f)(x) = f(x) - f(x) = 0$. 

(ii): \textbf{M4}: 
Indeed, 
using the definition, we have 
\begin{align*}
f(x)\cdot (1/f)(x)
&= 
\frac{ a_m x^m+ \cdots + a_0}{b_n x^n +
\cdots +b_0} 
\cdot 
\frac{b_n x^n + \cdots +b_0}{ a_m
x^m+ \cdots + a_0} 
= 1.
\end{align*}
\end{solution}

\begin{exercise}
\label{exo:walaa150923}
Suppose that $\FF =\{0,1,a,b\}$
is the finite field containing four (distinct) elements: $0,1,a,b$. 
Show the following:
\begin{enumerate}
\item
$ab=1$.
\item
$a^2=b$.
\end{enumerate}
\emph{Hint:} You could argue by cases and reach contradictions for other
conceivable outcomes.
\end{exercise}
\begin{solution}
Note that both $a$ and $b$ are different from $0$ and $1$. 

(i): 
\emph{Case~1:} $ab=0$.\\
By Proposition~\ref{p:may26:4}\ref{p:may26:4iii}, 
$a=0$ or $b=0$ which is absurd since both are nonzero.

\emph{Case~2:} $ab=a$.\\
Since $a\neq 0$, we can multiply by $a^{-1}$ to deduce that $b=1$
which is absurd. 

\emph{Case~3:} $ab=b$.\\
Since $b\neq 0$, we can multiply by $b^{-1}$ to deduce that $a=1$
which is absurd. 

\emph{Case~4:} $ab=1$.\\
Since the other conceivable possibilities were just brought to a
contradiction, this is the only remaining alternative and
therefore must hold since fields are closed under multiplication. 

(ii):
\emph{Case~1:} $a^2=0$.\\
By Proposition~\ref{p:may26:4}\ref{p:may26:4iii}, 
$a=0$ which is absurd.

\emph{Case~2:} $a^2=a$.\\
Since $a\neq 0$, we can multiply by $a^{-1}$ to deduce that $a=1$
which is absurd. 

\emph{Case~3:} $a^2=1$.\\
Then, using (i), $a^2=1=ab$ and hence (multiply by $a^{-1}$)
$a=b$ which is absurd. 

\emph{Case~4:} $a^2=b$.\\
Since the other conceivable possibilities were just brought to a
contradiction, this is the only remaining alternative and
therefore must hold since fields are closed under multiplication. 
\end{solution}

\begin{exercise}
\label{exo:heinz150923}
Suppose we impose the existence of 
an inverse for $0$ in \textbf{M4} not
just for every nonzero number. 
Show that this assumption leads to contradiction with
at least one of the results obtained in this chapter.
\end{exercise}
\begin{solution}
We suppose that $0^{-1}$ exists; thus,
$0\cdot 0^{-1} =1$.
On the other hand, 
Proposition~\ref{p:may26:4}\ref{p:may26:4ii} 
(with $x=0^{-1}$) would then yield
$0^{-1}\cdot 0 = 0$.
Altogether, $1=0$ which is absurd. 
\end{solution}

\begin{exercise}
\label{exo:heinz150923b}
Let $\FF$ be a finite field. 
For every $n\in\{1,2,3,\ldots\}$, 
we \emph{define} 
$n\odot 1 := \sum_{i=1}^n 1$.
(Note that $n$ may or may not lie in $\FF$,
but $n\odot 1$ must lie in $\FF$.)
Show the following:
\begin{enumerate}
\item[(i)]
There exists a number $c\in \{1,2,\ldots\}$ such that 
$c\odot 1 = 0$.
\item[(ii)]
Show that the smallest integer $c\geq 1$ such that $c\odot 1 = 0$
is a prime number. 
(The number $c$ is called the \emph{characteristic} of the field
$\FF$. 
\end{enumerate}
\end{exercise}
\begin{solution}
(i): 
Since $\FF$ contains only finitely many elements,
there must exist integers $1\leq m < n$ such that
$m\odot 1 = n\odot 1$. 
Subtracting $m\odot 1$, we learn that 
$(n-m)\odot 1 = 0$. 

(ii): 
Suppose to the contrary that $c$ is not a prime number,
say $c=p\cdot q$, where $p\geq 2$ and $q\geq 2$ are integers.
Then $0 = c\odot 1 = (pq)\odot 1 = (p\odot 1)(q\odot 1)$.
By Proposition~\ref{p:may26:4}\ref{p:may26:4iii},
$p\odot 1 = 0$ or $q\odot 1=0$.
Since $p<c$ and $q<c$, this contradicts the assumption
that $c$ is the smallest integer $\geq 1$ such that $c\odot 1 =
0$. 
\end{solution}

\begin{exercise}
\label{exo:151007a}
Let $\FF$ be a field,
and let $x,y,u,v$ be in $\FF$. Show that 
$$ (xu-yv)^2+(xv+yu)^2 = (x^2+y^2)(u^2+v^2).$$
\end{exercise}
\begin{solution}
This follows from a simple expansion: indeed,
since $(\alpha+\beta)^2 = (\alpha+\beta)(\alpha+\beta) =
(\alpha+\beta)\alpha+(\alpha+\beta)\beta =
\alpha\alpha+\beta\alpha+\alpha\beta+\beta\beta =
\alpha^2+\alpha\beta+\alpha\beta+\beta^2 =
\alpha^2+2\alpha\beta+\beta^2$, we obtain 
\begin{align*}
(xu-yv)^2+(xv+yu)^2 
&=
(xu)^2 + (-yv)^2 + 2(xu)(-yv) + (xv)^2 + (yu)^2+2(xv)(yu)\\
&= x^2u^2+y^2v^2 + x^2v^2+y^2u^2
= (x^2+y^2)(u^2+v^2),
\end{align*}
as claimed. 
\end{solution}

\begin{exercise}
In Exercise~\ref{exo:200721a}, we proved that 
$13^n$ can be written as a sum of two squares of integers
for every integer $n\geq 1$
by using \emph{strong} induction.
In this exercise, our job is to prove this result 
using \emph{regular} induction.
\emph{Hint:} Put Exercise~\ref{exo:151007a} to good use!
\end{exercise}
\begin{solution}
Denote the statement in question by $S(n)$.

\textbf{Base Case:} 
When $n=1$, we have 
\begin{equation*}
    13^1 = 13 = 4 + 9 = 2^2 + 3^2
\end{equation*}
and so the base case is verified. 

\textbf{Inductive Step:} 
Now suppose that $S(n)$ holds for some integer $n\geq 1$, say 
\begin{equation*}
13^n = a^n + b^n,
\end{equation*}
for some integers $a$ and $b$. 
Then 
\begin{equation*}
13^{n+1} = 13\cdot 13^n = \big(2^2+3^2\big)\big(a^2+b^2\big).
\end{equation*}
Set $x=2,y=3,u=a,v=b$.
Then, by Exercise~\ref{exo:151007a}, 
\begin{align*}
(2^2+3^2)(a^2+b^2)
&= (x^2+y^2)(u^2+v^2)
= (xu-yv)^2 + (xv+yu)^2\\
&=(2a-3b)^2 + (2b+3a)^2
\end{align*}
is the sum of two squares of integers 
because $2a-3b$ and $2b+3a$ are clearly also integers. 
Altogether,
\begin{equation*}
13^{n+1} = (\underbrace{2a-3b}_{\in\ZZ})^2 + (\underbrace{2b+3a}_{\in\ZZ})^2
\end{equation*}
is the sum of two squares of integers.
Hence $S(n+1)$ is true.

Therefore, by the Principle of Mathematical Induction, 
the result is true for every integer $n\geq 1$. 
\end{solution}

\begin{exercise}[YOU be the marker!] 
Consider the following statement 
\begin{equation*}
\text{``$0=1$''}
\end{equation*}
and the following ``proof'':
\begin{quotation}
On the one hand,
$0\cdot 0^{-1} = 0$ by Proposition~\ref{p:may26:4}\ref{p:may26:4iii}.\\
On the other hand,
$0\cdot 0^{-1} = 1$ by \textbf{M4}.\\
Altogether, 
$0=0\cdot 0^{-1} = 1$.
\end{quotation}
Why is this proof wrong?
\end{exercise}
\begin{solution}
$0^{-1}$ does not exist: 
\textbf{M4} only works for nonzero numbers. 
\end{solution}

\begin{exercise}[TRUE or FALSE?]
Mark each of the following statements as either true or false.  
Briefly justify your answer.
\begin{enumerate}
\item ``The set of integers, $\ZZ$, forms a field with the usual 
definitions of addition and multiplication.''
\item ``The set of rational numbers, $\QQ$, forms a field with 
the usual definitions of addition and multiplication.''
\item ``If $n\in\{1,2,\ldots\}$, then there exists a field 
containing precisely $n$ numbers.''
\item ``If $a\neq 0$ in a field, then the equation $ax=1$
always has a solution.''
\end{enumerate}
\end{exercise}
\begin{solution}
(i): FALSE: For instance, the number $2$ has no inverse. 

(ii): TRUE: This follows from Proposition~\ref{p:subfield}.

(iii): FALSE: For instance, there is no field containing 
exactly one number, because every field contains at least 
two numbers: $0$ and $1$ (and $1$ is assumed to be different from $0$).

(iv): TRUE: Proposition~\ref{p:may26:3}.
\end{solution} 
\chapter{Sets and Relations}
\label{cha:sets}

\section{Basic Concepts}

A \textbf{set}\index{Set}\index{Element} 
is a ``well defined'' collection of objects called
\textbf{elements}. It is customary to use curly brackets to denote sets.
For instance, the set comprising of the first 4 positive integers is
denoted by
\begin{equation}
\{1,2,3,4\}.
\end{equation}
If $A$ is a set, we write $x\in A$ if $x$ belongs to $A$; otherwise, we
write $x\notin A$.
Thus,
\begin{equation}
3\in \{1,2,3,4\} \quad\text{but}\quad
-7\notin \{1,2,3,4\}.
\end{equation}
Sets need not be restricted to numbers, we can also speak of the set of all
students in this class, or all cells in an organism, etc.

Often, sets are constructed from a given ``universe'' of potential
elements. For instance, we write
\begin{equation}
\menge{n\in\NN}{\text{$n$ is a prime number}} 
= \{2,3,5,7,11,\ldots\}
\end{equation}
to describe the set of all prime numbers;
another frequently used notation for this set is 
$\{n\in\NN\;:\;\text{$n$ is a prime number}\}$. 
These constructions allow for sets to be formed that contain
\emph{nothing}.
For instance, the set
\begin{equation}
S := \menge{n\in\NN}{\text{$n$ is a prime number between $14$ and $16$}} 
\end{equation}
contains no element because the numbers 14, 15, 16 are all composite:
$14=2\cdot 7$, $15 = 3\cdot 5$, and $16 = 2\cdot 8$.
Thus, $S=\{\,\}$, which is also and more commonly written as 
\begin{equation}
S = \varnothing;
\end{equation}
\label{l:emptyset}
in words: $S$ is the \textbf{empty set}.\index{Empty set}

Let $A$ and $B$ be two sets.
It may be that these sets are comparable in the sense that one is contained
in the other. For instance, the set of prime numbers is certainly contained
in the set of all nonnegative integers. Or the set of all men is
contained in the set of all human beings. Or the set of all vowels is
contained in the set of all letters of the alphabet. 

\begin{definition}[subset]\label{d:subset}
Let $A$ and $B$ be sets.\index{Subset}
If every element in $A$ belongs to $B$, then we say that
$A$ is a \textbf{subset} of $B$ and we write $A\subseteq B$.
If $A\subseteq B$ and $B\subseteq A$, then the sets $A$ and $B$ coincide, i.e.,
$A=B$.
If $A$ is not a subset of $B$, i.e., if there is an element $x\in A$ such
that $x\notin B$, then we write $A\not\subseteq B$. 
\end{definition}

We can now write $\NN\subseteq\ZZ\subseteq\QQ\subseteq\RR$ and
$\ZZ\not\subseteq\NN$. 
To verify that two sets $A$ and $B$ are the same, we check that $A\subseteq B$
and that $B\subseteq A$, i.e.,
\begin{equation}
A = B
\quad\Leftrightarrow\quad
[A\subseteq B \;\;\text{and}\;\; B\subseteq A].
\end{equation}
Note that $\varnothing\subseteq A$, which implies that the empty set is
unique.

\begin{definition}[union, intersection, complement]
\label{d:uic} 
Let $A$ and $B$ be sets.
Then\footnote{Great notation: 
(i) $x\in A\cup B$ $\Leftrightarrow$ [$x\in A$ $\lor$ $x\in B$];
(ii) $x\in A\cap B$ $\Leftrightarrow$ [$x\in A$ $\land$ $x\in B$].}
\begin{enumerate}
\item
$A\cup B := \menge{x}{\text{$x\in A$ or $x\in B$}}$ is the 
\textbf{union} of $A$ and $B$; \index{Union}
\item
$A\cap B := \menge{x}{\text{$x\in A$ and $x\in B$}}$ is the 
\textbf{intersection} of $A$ and $B$; \index{Intersection}
\item
$A\smallsetminus B := \menge{x\in A}{x\notin B}$ is the 
\textbf{complement} of $B$ in $A$. \index{Complement}
\end{enumerate}
If $A\cap B=\varnothing$, then $A$ and $B$ are said to be \textbf{disjoint}.
\end{definition}

Often, one works with subsets of a ``universal'' set $X$ (e.g., with subsets
of real numbers). If the universal set $X$ is understood and $A\subseteq
X$, \index{Universal set} 
then one also writes $A^c$ for $X\smallsetminus A$; the superscript ``$c$''
stands for ``complement''. 

These notions all generalize to more than 2 
--- even infinitely many ---
sets. For instance, let $\{A_i\}_{i\in I}$ be an
indexed family of sets.
Then their union is
\begin{equation}
\bigcup_{i\in I} A_i = \menge{x}{\text{$x\in A_i$ for some $i\in I$}},
\end{equation}
and their intersection is
\begin{equation}
\bigcap_{i\in I} A_i = \menge{x}{\text{$x\in A_i$ for every $i\in I$}}.
\end{equation}

\section{Laws for Sets and the Power Set}

The operations introduced in the previous section satisfy 
several nice rules which we collect now.

\begin{proposition}
\label{p:setlaws}
Let $X$ be a set, and let $A,B,C$ be subsets of $X$.
Then the following hold.
\begin{enumerate}
\item
\label{p:setlawsi}
(idempotent laws) $A \cup A=A$ and $A \cap A = A$.

\item
\label{p:setlawsii}
(associative laws)
$(A\cup B)\cup C = A \cup(B\cup C)$
and 
$(A\cap B)\cap C = A \cap(B\cap C)$.

\item
\label{p:setlawsiii}
(commutative laws)
$A\cup B = B \cup A$ and $A\cap B = B \cap A$.
\item
\label{p:setlawsiv}
(distributive laws)
$A\cup (B\cap C) = (A \cup B)\cap (A \cup C)$ 
and 
$A\cap (B\cup C) = (A \cap B)\cup (A \cap C)$. 
\item 
\label{p:setlawsv}
(identity laws)
$A \cup \varnothing = A$,
$A \cap X= A$,
$A\cup X = X$,
$A\cap \varnothing = \varnothing$.
\item 
\label{p:setlawsvi}
(involution law)
$(A^c)^c = A$.
\item 
\label{p:setlawsvii}
(complement laws)
$A\cup A^c = X$,
$A \cap A^c = \varnothing$,
$X^c = \varnothing$,
$\varnothing^c = X$.
\item
\label{p:setlawsviii}
(DeMorgan's laws)\index{DeMorgan's laws} 
$(A\cup B)^c = A^c \cap B^c$
and $(A \cap B)^c = A^c \cup B^c$.
\end{enumerate}
\end{proposition}
\begin{proof}
\ref{p:setlawsviii}:
Let $x\in X$.
Then 
\begin{subequations}
\begin{align}
x\in (A\cup B)^c 
&\Leftrightarrow
x\notin A\cup B\\
&\Leftrightarrow
[x\notin A \text{~and~} x\notin B]\\
&\Leftrightarrow
[x\in A^c \text{~and~} x\in B^c]\\
&\Leftrightarrow
x \in A^c \cap B^c.
\end{align}
\end{subequations}
The proof of $(A \cap B)^c = A^c \cup B^c$ is similar.
\end{proof}

Looking back at the beginning of the chapter,
we didn't fully define the meaning of a ``set''.
In fact, this is a rather subtle notion as was discovered
in 1901 by Bertrand Russell.\index{Russell, B.}

\begin{remark}[Russell's Paradox]
\label{r:Russell}\index{Russell's Paradox} 
Consider the set of all sets $X$ that do not contain themselves:
\begin{equation}
R := \menge{X}{X\notin X}. 
\end{equation}
Is $R$ an element of $R$ or not?
\hhbcom{Arguing by cases!}

\emph{Case~1}: Assume that $R\in R$.\\
Then, by definition of $R$, $R\notin R$, which is absurd.

\emph{Case~2}: Assume that $R\notin R$.\\
Then, by definition of $R$, $R\in R$, which again is absurd.

The resolution to this paradox is that $R$ is not a set. 
Thus, set theory requires some care to not
encounter such paradoxes. However, working scientists and
mathematicians essentially never encounter paradoxes of this form ---
these paradoxes are looked at by logicians.
\end{remark}

\begin{definition}
\label{d:powerset} \index{Power Set} 
Let $X$ be a set.
The collection of all subsets of $X$ is denoted by $\mathcal{P}(X)$
and called the \textbf{power set} of $X$.
\end{definition}

For instance, if $X=\{A,B,C\}$,
then the power set
\begin{equation}
\mathcal{P}(X) = 
\big\{\varnothing,
\{A\},\{B\},\{C\},
\{A,B\}, \{A,C\}, \{B,C\},
\{A,B,C\}\big\}
\end{equation}
contains precisely 8 elements.
This is not a coincidence as the next result shows.

\begin{theorem}
\label{t:powerset}
Let $X$ be a set that contains $n$ distinct elements, where
$n\in\NN$. 
Then the power set $\mathcal{P}(X)$ contains $2^n$ elements\footnote{
This motivates an alternative notation for the power set sometimes seen in
the literature, namely $2^X$.}. 
\end{theorem}
\begin{proof}
Using Theorem~\ref{t:combinations} and \eqref{e:2^n}, 
we see that 
\begin{subequations}
\begin{align}
\text{number of elements in $\mathcal{P}(X)$} &= 
\sum_{k=0}^{n} \text{number of subsets of $X$ of size $k$}\\
&= \sum_{k=0}^{n} {n \choose k}\\
&= 2^n,
\end{align}
\end{subequations}
as claimed.
Alternatively, one observes that every element of $X$ 
either belongs to a given subset or it does not, giving rise to $2^n$ possibilities.
\end{proof}

\section{Product Sets and Relations}

\label{sec:psar}

\begin{definition}[product set]\label{d:productset}
Let $A$ and $B$ be sets.\index{Product Set}\index{Cartesian Product
Set}
Then the \textbf{(Cartesian) product set} of $A$ and $B$ is
\begin{equation}
A \times B := \menge{(a,b)}{\text{$a\in A$ and $b\in B$}}.
\end{equation}
\end{definition}

Thus if $A = \{a,b,c\}$ and $B=\{1,2\}$, then
\begin{equation}
A\times B =
\big\{ (a,1),(a,2),(b,1),(b,2),(c,1),(c,2)\big\}.
\end{equation}
An important example is the \textbf{Euclidean plane}\index{Euclidean
plane}
\begin{equation}
\RR^2 = \RR \times \RR = \menge{(x,y)}{\text{$x\in\RR$ and
$y\in\RR$}};
\end{equation}
here, product set consists of all ordered pairs of real numbers. 
Note that the order is important: the point $(1,2)$ is not the same
as the point $(2,1)$. (Although for \emph{sets}, we do have $\{1,2\} =
\{2,1\}$.)
In the Euclidean plane, the product set $[0,1]\times[1,3]$ consists
precisely of all the points lying in the rectangle with corner
points $(0,1)$, $(0,3)$, $(1,1)$, and $(1,3)$.

\begin{definition}[relation]
\label{d:relation}
\index{Relation}
Let $A$ and $B$ be sets.
A \textbf{relation} from $A$ to $B$ is simply a subset of
$A\times B$.
Now let $\Rel$ be a relation from $A$ to $B$,
and let $a\in A$ and $b\in B$.
If $A=B$, we say that $\Rel$ is a relation on $A$. 
If $(a,b)\in \Rel$, we say $a$ is related (or, $\Rel$-related) to $b$ and we
write $a\Rel b$.
\end{definition}

Consider the case when $A=B=\RR$. If $\Rel = \menge{(x,y)\in\RR^2}{x=y}$,
then $\Rel$ is the diagonal in the Euclidean plane, and the meaning of
$x\Rel y$ is exactly $x = y$.
And if $f\colon \RR\to\RR$ is a function (say, the exponential
function seen in Calculus~I), then $\Rel =
\menge{(x,y)\in\RR^2}{y=f(x)} = \menge{(x,f(x))}{x\in\RR}$ is a
relation. Relations thus generalize functions; in fact, relations
play important roles in modern Analysis where one considers
``functions'' whose outputs are sets (e.g., ``derivatives'' of
functions such as the absolute value which is not differentiable at
$0$ in the classical sense). 

The following type of relation is important in many areas of
Mathematics; in particular in Number Theory and Algebra.

\begin{definition}[equivalence relation]
Let $\Rel$ be a relation on a set $A$.\index{Equivalence relation}
Then $\Rel$ is an \textbf{equivalence relation} if it satisfies
the following properties for all $a,b,c$ in $A$. \label{d:eqrel}
\begin{enumerate}
\item
(reflexivity)\index{Reflexivity (for equivalence relation)}
$a\Rel a$.
\item 
(symmetry)\index{Symmetry (for equivalence relation)}
If $a\Rel b$, then $b\Rel a$.
\item
(transitivity)\index{Transitivity (for equivalence relation)} 
If $a\Rel b$ and $b\Rel c$, then $a\Rel c$.
\end{enumerate}
\end{definition}

One sees also the notations $a\equiv b$ or $a\sim b$ instead of
$a\Rel b$, especially in the context of equivalence relations.
Note that $a\Rel b :\Leftrightarrow a=b$ defines an equivalence
relation on $A$, which motivates these notations.

\begin{example}[mod 5]
\label{ex:mod5}
For $a$ and $b$ in $\ZZ$, we say 
$a$ and $b$ are \emph{congruent modulo 5} and write 
\index{Congruent modulo}\index{mod} 
\begin{equation}
a\equiv b \pmod{5},
\end{equation}
if $5$ divides $b-a$, i.e., $(b-a)/5 \in\ZZ$. 
(Thus $7\equiv -13 \pmod{5}$ but $5\not\equiv 12\pmod{5}$.)
This defines an equivalence relation on $\ZZ$.
\end{example}
\begin{proof}
Let $a,b,c$ be in $\ZZ$.
For convenience, we abbreviate $a\equiv b\pmod{5}$ by $a\equiv b$.

``reflexivity'':
Since $a-a=0$ and $5|0$, we have $a\equiv a$.

``symmetry'':
Suppose that $a\equiv b$.
Then there exists $k\in\ZZ$ such that $b-a = 5\cdot k$.
Hence $a-b=5\cdot(-k)$. Since $-k\in\ZZ$, we have
$b\equiv a$.

``transitivity'':
Suppose that $a\equiv b$ and that $b\equiv c$.
Then there exist $k$ and $l$ in $\ZZ$ such that\footnote{A common
mistake is to not use a (possibly) different integer $l$ --- but there
is no reason that the quotients $(b-a)/5$ and $(c-b)/5$ should be
identical!}
\begin{equation}
b-a = 5\cdot k
\end{equation}
and 
\begin{equation}
c-b = 5\cdot l.
\end{equation}
Adding the last two equations yields
\begin{equation}
c-a = (b-a) + (c-b) = 5\cdot k + 5\cdot l = 5\cdot(k+l).
\end{equation}
Since $k+l\in\ZZ$, it follows that $a\equiv c$.
\end{proof}

A second inspection of the proof of Example~\ref{ex:mod5} shows that
there was nothing special about the number 5; in fact, the proof works
analogously if $5$ is replaced by any nonzero (typically positive)
integer. Let us record this important observation.

\begin{proposition}[mod $m$]
\label{p:modm}
Let $m$ be a positive integer and define a relation on $\ZZ$ by
\begin{equation}
a\equiv b\pmod{m} 
\quad :\Leftrightarrow
\quad
\frac{b-a}{m}\in\ZZ.
\end{equation}
Then one says that $a$ and $b$ are \textbf{congruent modulo $m$};
moreover, this relation is an equivalence relation.
\index{Congruent modulo}\index{mod}
\end{proposition}

One useful property of an equivalence relation defined on a set $A$ is
that the relation induces a ``partition'' of the set $A$ into subsets
of elements that are all related to each other. 

\begin{definition}[equivalence classes and quotient set]
\label{d:eqclass}
Let $A$ be a set and let $\Rel$ be an equivalence relation on $A$.
For each $a\in A$, denote by $[a]$ the set of elements related to $a$,
which is also called the \textbf{equivalence class} of
$a$,\index{Equivalence class} 
i.e.,\index{Equivalence class}\index{Quotient set}
\begin{equation}
[a] := \menge{b\in A}{a\Rel b}.
\end{equation}
The \textbf{quotient set} is the collection of all equivalence classes,
written as\index{Quotient set}
\begin{equation}
A/\Rel := \menge{[a]}{a\in A}. 
\end{equation}
\end{definition}

\begin{example}[mod 5 cont'd]
Let us revisit Example~\ref{ex:mod5}. 
Fix $a\in \ZZ$, and let $b\in \ZZ$.
Then $a\equiv b\pmod{5}$ 
$\Leftrightarrow$
$(b-a)/5\in \ZZ$
$\Leftrightarrow$
$b-a \in 5\cdot \ZZ = \{0,\pm 5,\pm 10,\pm 15,\ldots\}$
$\Leftrightarrow$
$b\in \{a,a\pm 5,a\pm 10,a\pm 15,\ldots\}$.
Hence
\begin{subequations}
\begin{align}
[0] &= \{0,\pm 5,\pm 10,\ldots\}=\{\ldots,-5,0,5,\ldots\} \\
[1] &= \{1,6,-4,11,-9,\ldots\} = \{\ldots,-4,1,6,\ldots\}\\
[2] &= \{\ldots,-3,2,7,\ldots\}\\
[3] &= \{\ldots,-2,3,8,\ldots\}\\
[4] &= \{\ldots,-1,4,9,\ldots\}
\end{align}
\end{subequations}
and thus
\begin{subequations}
\begin{align}
\cdots = [-5] = [0] &= [5] = \cdots\\
\cdots = [-4] = [1] &= [6] = \cdots\\
\cdots = [-3] = [2] &= [7] = \cdots\\
\cdots = [-2] = [3] &= [8] = \cdots\\
\cdots = [-1] = [4] &= [9] = \cdots.
\end{align}
\end{subequations}
Therefore, the quotient set
\begin{equation}
\ZZ/(\text{mod}\,5) = \big\{[0],[1],[2],[3],[4]\big\},
\end{equation}
which is also written as $\ZZ/(5\ZZ)$ or $\ZZ_5$, 
consists of 5 elements. 
Note how the original set $\ZZ$ is ``partitioned'' into 5 ``pieces''
$[0],[1],[2],[3],[4]$; 
these sets have no overlap, and they collectively cover $\ZZ$. 
\end{example}

Perhaps not too surprisingly, Proposition~\ref{p:modm} can be refined
just like the previous example extended Example~\ref{ex:mod5}:

\begin{example}[mod $m$ cont'd]
Let us revisit Proposition~\ref{p:modm}. Then the quotient set is 
\begin{equation}
\ZZ/(\text{mod}\,m) = \big\{[0],[1],[2],\ldots,[m-1]\big\},
\end{equation}
and it contains $m$ distinct elements. 
\end{example}

Furthermore, 
the partition property holds in fact true for any equivalence
relation.

\begin{theorem}
\label{t:partition}
Let $A$ be a set and let $\Rel$ be an equivalence relation on $A$.
Then the following hold.
\begin{enumerate}
\item 
\label{t:partitioni}
For every $a\in A$, we have $a\in [a]$.
\item
\label{t:partitionii}
For all $a$ and $b$ in $A$, we have
$[a]=[b]$ $\Leftrightarrow$ $a\Rel b$.
\item 
\label{t:partitioniii}
For all $a$ and $b$ in $A$, we have
$[a]\neq [b]$ $\Rightarrow$ $[a]\cap [b]=\varnothing$.
\end{enumerate}
\end{theorem}
\begin{proof}
Let $a,b,c$ be in $A$.

\ref{t:partitioni}:
Since $\Rel$ is reflexive, we have $a\Rel a$ and thus $a\in[a]$.

\ref{t:partitionii}:
``$\Rightarrow$'': 
By \ref{t:partitioni}, $b\in [b]=[a]$; hence $b\in [a]$, i.e., $a\Rel b$.

~~~~~``$\Leftarrow$'': 
Take $c\in[a]$. Then $a\Rel c$ and also $a\Rel b$.
By symmetry, $b\Rel a$ and hence, by transitivity, $b\Rel c$.
Hence $c\in[b]$. Thus, $[a]\subseteq [b]$ and an analogous
argument shows that $[b]\subseteq [a]$. Altogether, $[a]=[b]$.

\ref{t:partitioniii}:
We show the contrapositive.
So assume that $[a]\cap[b]\neq\varnothing$, say $c\in[a]\cap[b]$.
Then $a\Rel c$ and $b\Rel c$.
By symmetry, $c\Rel b$ and, by transitivity, $a\Rel b$.
Hence, by \ref{t:partitionii}, $[a]=[b]$.
\end{proof}

\section*{Exercises}\markright{Exercises}
\addcontentsline{toc}{section}{Exercises}
\setcounter{theorem}{0}

\begin{exercise}
Show that the empty set is unique. 
\end{exercise}
\begin{solution}
Recall that the empty set is a subset of every set.
So let's let $\varnothing$ and $\varnothing^\prime$ be
empty sets.
Since $\varnothing$ is an empty set and $\varnothing^\prime$ is a
set, we have $\varnothing\subseteq\varnothing^\prime$.
Analogously, 
since $\varnothing^\prime$ is an empty set and $\varnothing$ is a
set, we have $\varnothing^\prime\subseteq\varnothing$.
Altogether, $\varnothing=\varnothing^\prime$. 
\end{solution}

\begin{exercise}
\label{exo:260304a}
Let $A$ and $B$ be sets.
Prove that $A\cap B\subseteq A \subseteq A\cup B$.
\end{exercise}
\begin{solution}
``$A\cap B\subseteq A$'':
We have $x\in A\cap B$
$\Leftrightarrow$ 
[$x\in A$ and $x\in B$]
$\Rightarrow$
$x\in A$.

``$A\subseteq A\cup B$'':
We have  $x\in A$
$\Rightarrow$
[$x\in A$ or $x\in B$]
$\Leftrightarrow$
$x\in A\cup B$.
\end{solution}

\begin{exercise}
Let $A$ and $B$ be sets. Show that the following are equivalent.
\begin{enumerate}
\item $A\subseteq B$.
\item $A\cap B = A$.
\item $A\cup B = B$.
\end{enumerate}
\end{exercise}
\begin{solution}
The most efficient way to prove
an equivalence of several items is to prove implications
on a circular route, whenever this is possible. 

``(i)$\Rightarrow$(ii)'':
Suppose that $A\subseteq B$.

We clearly have $A\cap B\subseteq A$ 
(see Exercise~\ref{exo:260304a}).
Now take $x\in A$. Since $A\subseteq B$, we also
have $x\in B$. Combining, we deduce that $x\in A\cap B$.
Altogether, $A\cap B = A$.

``(ii)$\Rightarrow$(iii)'':
Suppose that $A\cap B = A$.
It is clear that $B\subseteq A \cup B$.
Now take $x\in A\cup B$. 
If $x\in A$, then, since $A=A\cap B$,
$x\in A\cap B$ and so $x\in B$.
If $x\in B$, then clearly $x\in B$.
In either case, $x\in B$.
Therefore, $A\cup B \subseteq B$.
Altogether, $A \cup B = B$.

``(iii)$\Rightarrow$(i)'':
Now assume that $A\cup B = B$.
Take $x\in A$. Since $A\cup B\subseteq B$,
we have $x\in B$. Thus, $A\subseteq B$.
\end{solution}

\begin{exercise}
Let $X$ be a (universal) set, and let $A$ and $B$ be subsets of $X$.
Prove that $(A \cap B)^c = A^c \cup B^c$.
\end{exercise}
\begin{solution}
Indeed, 
\begin{align*}
x\in (A\cap B)^c 
&\Leftrightarrow
x\notin A\cap B
\Leftrightarrow
[x\notin A \text{~or~} x\notin B]
\Leftrightarrow
[x\in A^c \text{~or~} x\in B^c]
\Leftrightarrow
x \in A^c \cup B^c.
\end{align*}
\end{solution}

\begin{exercise}
Let $X$ be a (universal) set, and let $A$ and $B$ be subsets of $X$.
Prove that $A\subseteq B$ if and only if $B^c \subseteq A^c$.
\end{exercise}
\begin{solution}
``$\Rightarrow$'':
Let $y\in B^c$. Then $y\in X\smallsetminus B$.
We claim that $y\in A^c$.
Suppose to the contrary that $y\notin A^c$.
Then $y\in (A^c)^c=A$.
The hypothesis implies that $y\in B$, which contradicts
the choice of $y$ (in $B^c$).

``$\Leftarrow$'':
Let $x\in A$. 
We claim that $x\in B$.
Suppose to the contrary that $x\notin B$,
i.e., $x\in B^c$.
The hypothesis implies that $x\in A^c$,
which is absurd because $x\in A$. 
\end{solution}

\begin{exercise}
\label{exo:A=B}
Suppose that $A,B$ are sets,
and that $X_1,\ldots,X_n$ are sets
such that $A\cup B \subseteq X_1\cup X_2\cup\cdots\cup X_n$.
Furthermore, suppose that
whenever $x$ belongs to some $X_i$, then
either ($x\in A$ and $x\in B$) or ($x\notin A$ and $x\notin B$).
Show that $A=B$.
\end{exercise}
\begin{solution}
Suppose to the contrary that $A\neq B$.
Then there exists $x\in (A\smallsetminus B)\cup (B\smallsetminus
A)$. Assume first that $x\in A\smallsetminus B$.
Since $A\smallsetminus B\subseteq A\cup B$, there exists $i$ such
that $x\in X_i$. By assumption either $x$ is in both $A$ and $B$
(impossible) or $x$ is neither in $A$ nor in $B$ (also
impossible), which is absurd. 
The case when $x\in B\smallsetminus A$ is dealt with similarly. 
Altogether, we have reached a contradiction; hence, we have proved the
result. 
\end{solution}

\begin{exercise}[symmetric difference]
\label{exo:symmdiff}
\index{symmetric difference}
The \emph{symmetric difference} of two sets
$A$ and $B$ is defined by 
\begin{equation*}
A\oplus B := (A\smallsetminus B)\cup(B\smallsetminus A).
\end{equation*}
Show the following for three arbitrary sets $A,B,C$:
\begin{enumerate}
\item
\label{exo:symmdiff0}
$A\oplus \varnothing = A$ and $A\oplus A = \varnothing$. 
\item
\label{exo:symmdiff1}
$A\oplus B = (A\cup B)\smallsetminus(A\cap B)$.
\item
\label{exo:symmdiff2}
$A\oplus(B\oplus C) = (A\oplus B)\oplus C$.
\item
\label{exo:symmdiff3}
$A\oplus B = B\oplus A$.
\item
\label{exo:symmdiff4}
$A\cap(B\oplus C) = (A\cap B)\oplus(A\cap C)$.
\item
\label{exo:symmdiff4+}
$(A\oplus B)\oplus (B\oplus C) = A\oplus C$.
\item
\label{exo:symmdiff5}
If $A\oplus B = A\oplus C$, then $B=C$.
\end{enumerate}
\end{exercise}
\begin{solution}
\ref{exo:symmdiff0}:
$A \oplus\varnothing =
(A\smallsetminus\varnothing)\cup(\varnothing\smallsetminus A)
= A\cup\varnothing = A$ and
$A \oplus A =
(A\smallsetminus A)\cup (A\smallsetminus A) =
\varnothing\cup\varnothing = \varnothing$. 

\ref{exo:symmdiff1}:
First let $x\in A\oplus B$.
Then $x\in A\smallsetminus B$ or
$x\in B\smallsetminus A$.
\emph{Case~1:} $x\in A\smallsetminus B$.
Then $x\in A\subseteq A\cup B$.
Moreover, $x\notin B$ and hence $x\notin A\cap B$
because $A\cap B\subseteq B$.
It follows that $x\in (A\cup B)\smallsetminus (A\cap B)$.
\emph{Case~2:} $x\in B\smallsetminus A$.
Then 
$x\in (A\cup B)\smallsetminus (A\cap B)$ similarly to the proof
of Case~1.
We have shown that
\begin{equation*}
A\oplus B \subseteq (A\cup B)\smallsetminus (A\cap B).
\end{equation*}
Now suppose that $x\in (A\cup B)\smallsetminus (A\cap B)$.
Then $x\in A\cup B$ and $x\notin A\cap B$.
\emph{Case~1:} $x\in A$. Since $x\notin A\cap B$, it must
be that $x\notin B$. Hence $x\in A\smallsetminus B$ and thus
$x\in A\oplus B$.
\emph{Case~2:} $x\in B$. We see similarly that $x\in A\oplus B$.
Altogether,
\begin{equation*}
(A\cup B)\smallsetminus (A\cap B) \subseteq A\oplus B,
\end{equation*}
and the proof of \ref{exo:symmdiff1} is complete. 

\ref{exo:symmdiff2}: We use Exercise~\ref{exo:A=B} to deal with
this and consider 8 possible cases.

\emph{Case~1:} $x\in A\cap B\cap C$.
Then $x\in A \smallsetminus(B\oplus C)$ and thus $x\in
A\oplus(B\oplus C)$.
Similarly, $x\in(A\oplus B)\oplus C$.

\emph{Case~2:} $x\in A^c\cap B\cap C$.
Since $x\in B\cap C$, $x\notin B\oplus C$.
Furthermore, $x\notin A$. Hence $x\notin A\oplus(B\oplus C)$.
Since $x\in A^c\cap B$, we have $x\in A\oplus B$.
Futhermore, $x\in C$. Hence $x\notin (A\oplus B)\oplus C$.

\emph{Case~3:} $x\in A\cap B^c\cap C$.
Since $x\in B^c \cap C$, $x\in B\oplus C$.
Now $x\in A$, hence $x\notin A\oplus(B\oplus C)$.
Since $x\in A\cap B^c$, $x\in A\oplus B$.
Now $x\in C$, hence $x\notin (A\oplus B)\oplus C$. 

\emph{Case~4:} $x\in A\cap B\cap C^c$.
Since $x\in B\cap C^c$, $x\in B\oplus C$.
Now $x\in A$, hence $x\notin A\oplus(B\oplus C)$.
Since $x\in A\cap B$, $x\notin A\oplus B$.
Since $x\notin C$, $x\notin (A\oplus B)\oplus C$. 

\emph{Case~5:} $x\in A\cap B^c\cap C^c$.
Since $x\in B^c\cap C^c$, $x\notin B\oplus C$.
Now $x\in A$, hence $x\in A\oplus(B\oplus C)$. 
Since $x\in A\cap B^c$, $x\in A\oplus B$.
Now $x\notin C$, hence $x\in(A\oplus B)\oplus C$.

\emph{Case~6:} $x\in A^c\cap B\cap C^c$.
Since $x\in B\cap C^c$, $x\in B\oplus C$.
Now $x\notin A$, hence $x\in A\oplus (B\oplus C)$.
Since $x\in A^c\cap B$, $x\in A\oplus B$.
Now $x\notin C$, hence $x\in (A\oplus B)\oplus C$. 

\emph{Case~7:} $x\in A^c\cap B^c\cap C$.
Since $x\in B^c\cap C$, $x\in B\oplus C$.
Now $x\notin A$, hence $x\in A\oplus(B\oplus C)$.
Since $x\in A^c\cap B^c$, $x\notin A\oplus B$.
Now $x\in C$, hence $x\in(A\oplus B)\oplus C$.

\emph{Case~8:} $x\in A^c\cap B^c\cap C^c$.
Since $x\in B^c\cap C^c$, $x\notin B\oplus C$.
Now $x\notin A$, hence $x\notin A\oplus(B\oplus C)$.
Since $x\in A^c\cap B^c$, $x\notin A\oplus B$.
Now $x\notin C$, hence $x\notin (A\oplus B)\oplus C$.

\ref{exo:symmdiff3}: This is obvious from the definition. 

\ref{exo:symmdiff4}:
We use Exercise~\ref{exo:A=B} to deal with
this and consider 8 possible cases.

\emph{Case~1:} $x\in A\cap B\cap C$.
Since $x\in B\cap C$, $x\notin B\oplus C$.
Hence $x\notin A\cap(B\oplus C)$. 
Since $x\in A\cap B$ and $x\in A\cap C$,
$x\notin(A\cap B)\oplus(A\cap C)$.

\emph{Case~2:} $x\in A^c\cap B\cap C$.
Since $x\notin A$, $x\notin A\cap(B\oplus C)$. 
Furthermore, again since $x\notin A$,
$x\notin A\cap B$ and $x\notin A\cap C$,
hence $x\notin (A\cap B)\oplus(A\cap C)$.

\emph{Case~3:} $x\in A\cap B^c\cap C$.
Since $x\in B^c\cap C$, $x\in B\oplus C$.
Now $x\in A$, hence $x\in A\cap(B\oplus C)$.
Since $x\in A\cap B^c$ and $x\in A\cap C$,
we have $x\notin A\cap B$ and $x\in A\cap C$,
hence $x\in (A\cap B)\oplus (A\cap C)$. 

\emph{Case~4:} $x\in A\cap B\cap C^c$.
Since $x\in A$ and $x\in B\cap C^c$,
we have $x\in A\cap(B\oplus C)$.
Since $x\in A\cap B$ and $x\notin A\cap C$,
we have $x\in (A\cap B)\oplus(A\cap C)$.

\emph{Case~5:} $x\in A\cap B^c\cap C^c$.
We have $x\in A$ and $x\notin B\oplus C$,
hence $x\notin A \cap(B\oplus C)$.
Now $x\notin A\cap B$ and $x\notin A\cap C$,
hence $x\notin(A\cap B)\oplus(A\cap C)$.

\emph{Case~6:} $x\in A^c\cap B\cap C^c$.
Since $x\notin A$, $x\notin A\cap(B\oplus C)$.
For the same reason, $x\notin A\cap B$ and $x\notin A\cap C$,
hence $x\notin (A\cap B)\oplus(A\cap C)$. 

\emph{Case~7:} $x\in A^c\cap B^c \cap C$.
Since $x\notin A$, $x\notin A\cap(B\oplus C)$.
For the same reason, $x\notin A\cap B$ and $x\notin A\cap C$,
hence $x\notin (A\cap B)\oplus(A\cap C)$. 

\emph{Case~8:} $x\in A^c\cap B^c \cap C^c$.
Since $x\notin A$, $x\notin A\cap(B\oplus C)$.
For the same reason, $x\notin A\cap B$ and $x\notin A\cap C$,
hence $x\notin (A\cap B)\oplus(A\cap C)$.

\ref{exo:symmdiff4+}:
Denote the LHS by $L$ and the RHS by $R$. 
We use Exercise~\ref{exo:A=B} to deal with
this and consider 8 possible cases.

\emph{Case~1:} $x\in A\cap B\cap C$.
Then $x\notin A\oplus B$ and $x\notin B\oplus C$,
hence $x\notin L$.
Also, $x\notin R$. 

\emph{Case~2:} $x\in A^c\cap B\cap C$.
Then $x\in A\oplus B$ and $x\notin B\cap C$,
hence $x\in L$. Also, $x\in R$. 

\emph{Case~3:} $x\in A\cap B^c\cap C$.
Then $x\in A\oplus B$ and $x\in B\oplus C$,
hence $x\notin L$. Also, $x\notin R$. 

\emph{Case~4:} $x\in A\cap B\cap C^c$.
Then $x\notin A\oplus B$ and $x\in B\oplus C$,
hence $x\in L$. Also, $x\in R$.

\emph{Case~5:} $x\in A\cap B^c\cap C^c$.
Then $x\in A\oplus B$ and $x\notin B\oplus C$,
hence $x\in L$. Also, $x\in R$.

\emph{Case~6:} $x\in A^c\cap B\cap C^c$.
Then $x\in A\oplus B$ and $x\in B\oplus C$,
hence $x\notin L$. Also, $x\notin R$.

\emph{Case~7:} $x\in A^c\cap B^c\cap C$.
Then $x\notin A\oplus B$ and $x\in B\oplus C$,
hence $x\in L$. Also, $x\in R$.

\emph{Case~8:} $x\in A^c\cap B^c\cap C^c$.
Then $x\notin A\oplus B$ and $x\notin B\oplus C$,
hence $x\notin L$. Also, $x\notin R$.

\ref{exo:symmdiff5}:
We show the contrapositive, i.e., $B\neq C\Rightarrow A\oplus
B\neq A\oplus C$.
Suppose that $B\neq C$.
Then $B\oplus C\neq\varnothing$,
say $x\in B\smallsetminus C$ (the case $C\smallsetminus B$ is
similar). 
\emph{Case~1:} $x\in A$.
Then $x\in A\cap B$, hence $x\notin A\oplus B$.
And $x\in A\cap C^c$, hence $x\in A\oplus C$.
\emph{Case~2:} $x\notin A$.
Then $x\in A^c\cap B$, hence $x\in A\oplus B$.
And $x\in A^c\cap C^c$, hence $x\notin A\oplus C$.

In either case, we conclude that $A\oplus B\neq A\oplus C$.
\end{solution}




\begin{exercise}
Set $A = \ZZ\times(\ZZ\smallsetminus\{0\})
= \menge{(x,y)\in\ZZ\times\ZZ}{y\neq 0}$
and define a relation $\Rel$ on $A$ by
\begin{equation}
(x,y)\Rel (u,v) \quad :\Leftrightarrow\quad
xv = yu.
\end{equation}
Show that $\Rel$ is an equivalence relation.
What is the ``meaning'' of this relation and
what can you identify the quotient set with?
\end{exercise}
\begin{solution} 
Let $(x,y),(u,v),(a,b)$ be in $A$. 

Reflexivity: Because $xy=yx$, it is clear that $(x,y)\Rel(x,y)$. 

Symmetry: Assume $(x,y)\Rel(u,v)$, i.e., $xv=yu$.
Then $uy=vx$ and hence $(u,v)\Rel(x,y)$.

Transitivity: Assume $(x,y)\Rel(u,v)$, i.e., $xv=yu$.
Also assume that $(u,v)\Rel(a,b)$, i.e., $ub=va$. 
Note that $y$, $v$, and $b$ are all nonzero.
Hence
$x/y = u/v$ and $u/v = a/b$.
Combining, we learn that
$x/y=a/b$. Thus $xb=ya$, i.e., 
$(x,y)\Rel(a,b)$.

We thus have shown that $\Rel$ is an equivalence relation!

We have $(x,y)\Rel(u,v)$ if and only if $x/y= u/v$.

Therefore,
the quotient set is essentially the set of rational numbers $\QQ$ since
So $[(x,y)]$ can be viewed as $x/y$ in $\QQ$; the equivalence class contains
all fractions giving the same rational number!
\end{solution}

\begin{exercise}
Construct relations $\Rel$ such that:
\begin{enumerate}
\item $\Rel$ is reflexive and symmetric, but not transitive;
\item $\Rel$ is reflexive and transitive, but not symmetric;
\item $\Rel$ is symmetric and transitive, but not reflexive.
\end{enumerate}
\end{exercise}
\begin{solution}

(i): $A := \{1,2,3\}$ and 
$\Rel := \{ (1,1),(2,2),(3,3),(1,2),(2,1),(2,3),(3,2)\}$.
Reflexivity and symmetry are clear by inspection.
Transitivity fails: $(1,2)\in\Rel$ and $(2,3)\in\Rel$ but 
$(1,3)\notin\Rel$. 

(ii): $A := \{1,2\}$ and
$\Rel := \{(1,1),(2,2),(1,2)\}$.
By inspection, $\Rel$ is reflexive and transitive, but not symmetric:
$(1,2)\in\Rel$ but $(2,1)\notin\Rel$.

(iii): $A := \{1,2\}$ and $\Rel := \{(1,1)\}$. 
Then $\Rel$ is symmetric and transitive, yet not reflexive: 
$(2,2)\notin\Rel$. 
\end{solution}

\begin{exercise}
Construct relations $\Rel$ such that:
\begin{enumerate}
\item $\Rel$ is reflexive, but neither symmetric nor transitive;
\item $\Rel$ is symmetric, but neither reflexive nor transitive;
\item $\Rel$ is transitive, but neither reflexive nor symmetric. 
\end{enumerate}
\end{exercise}
\begin{solution}
(i): $A := \{1,2,3\}$ and
$\Rel := \{(1,1),(2,2),(3,3),(1,2),(2,3)\}$.
Clearly, $\Rel$ is reflexive.
Since $(1,2)\in\Rel$ but $(2,1)\notin\Rel$, $\Rel$ is not symmetric. 
Since $(1,2)\in\Rel$ and $(2,3)\in\Rel$, yet $(1,3)\notin\Rel$, 
we see that $\Rel$ is not transitive. 

(ii): $A := \{1,2\}$ and
$\Rel := \{(1,2),(2,1)\}$.
Clearly, $\Rel$ is symmetric.
However, $(1,1)\notin\Rel$, so $\Rel$ is not reflexive.
Furthermore $(1,2)\in\Rel$ and $(2,1)\in\Rel$, yet $(1,1)\notin\Rel$, so
transitivity also fails.

(iii): $A := \{1,2\}$ and 
$\Rel := \{(1,2)\}$.
Transitivity is trivially true (there is nothing to check since $2\neq 1$). 
Reflexivity fails since $(1,1)\notin\Rel$.
Finally, $\Rel$ is clearly not symmetric. 
\end{solution}

\begin{exercise}
Let $A$ be the set of all straight lines drawn in the Euclidean
plane. Define a relation $\Rel$ on $A$ by 
\begin{equation*}
\ell_1 \Rel \ell_2 \quad :\Leftrightarrow\quad
\text{the lines $\ell_1$ and $\ell_2$ are perpendicular.}
\end{equation*}
Decide whether $\Rel$ is
(i) reflexive; 
(ii) symmetric;
(iii) transitive;
(iv) an equivalence relation. 
\end{exercise}
\begin{solution} 
Let $\ell_1,\ell_2,\ell_3$ be lines in $A$. 

(i): $\Rel$ is \emph{not reflexive} because no line is perpendicular
to itself.

(ii): $\Rel$ is \emph{symmetric} because if $\ell_1$ is perpendicular
to $\ell_2$, then $\ell_2$ is perpendicular to $\ell_1$.

(iii): $\Rel$ is not \emph{transitive}: The $x$-axis is perpendicular to
the $y$-axis, and the $y$-axis is perpendicular to the $x$-axis;
however, the $x$-axis is not perpendicular to itself.

(iv): In view of (i) or (iii), it is clear that $\Rel$ is
\emph{not} an equivalence relation.
\end{solution}

\begin{exercise}
Let $A$ be the set of all straight lines drawn in the Euclidean
plane. Define a relation $\Rel$ on $A$ by 
\begin{equation*}
\ell_1 \Rel \ell_2 \quad :\Leftrightarrow\quad
\text{the lines $\ell_1$ and $\ell_2$ are parallel.}
\end{equation*}
Decide whether $\Rel$ is
(i) reflexive; 
(ii) symmetric;
(iii) transitive;
(iv) an equivalence relation. 
\end{exercise}
\begin{solution}
Let $\ell_1,\ell_2,\ell_3$ be lines in $A$. 

(i): $\Rel$ is \emph{reflexive} because every line is parallel to
itself.

(ii): $\Rel$ is \emph{symmetric} because if $\ell_1$ is parallel 
to $\ell_2$, then $\ell_2$ is also parallel to $\ell_1$.

(iii): $\Rel$ is \emph{transitive} because if $\ell_1$ is
parallel to $\ell_2$ and $\ell_2$ is parallel to $\ell_3$, then
$\ell_1$ is parallel to $\ell_3$. 

(iv): In view of (i)--(iii), it is clear that $\Rel$ is
an equivalence relation.
\end{solution}

\begin{exercise}
Let 
$n\in\{1,2,3,\ldots\}$, and let 
$A$ be a set containing $n$ distinct elements.
Consider the following two relations:
\begin{enumerate}
\item $\Rel_1 = \menge{(a,a)}{a\in A}$.
\item $\Rel_2 = A\times A$.
\end{enumerate}
Determine whether or not each of these relations is an 
equivalence relation. 
If it is, determine the number of elements in the quotient set.
\end{exercise}
\begin{solution}
(i): 
It is easy to check that $\Rel_1$ is an equivalence relation.
The quotient set $A/\Rel_1 = \menge{[a]}{a\in A}$ 
can be identified with $A$ itself because $[a]=\{a\}$ for
every $a\in A$; thus, it has $n$ elements. 

(ii): 
It is easy to check that $\Rel_2$ is an equivalence relation.
The quotient set $A/\Rel_2 = \menge{[a]}{a\in A} = \{A\}$ has exactly $1$
element, because for every $a\in A$ and every $b\in A$, $b\in[a]$. 
\end{solution}

\begin{exercise}
Let $A$ be a set, and let $\Rel$ be a relation on $A$. 
Then $\Rel$ is called \emph{antisymmetric} if it satisfies the
following property for all $a$ and $b$ in $A$:
If $a\Rel b$ and $b\Rel a$, then $a=b$.
Construct the following:
\begin{enumerate}
\item An example of an equivalence relation that is
antisymmetric.
\item An example of an equivalence relation that is not
antisymmetric.
\end{enumerate}
\end{exercise}
\begin{solution}
(i): Let $A$ be a nonempty set and define $a\Rel b$
$:\Leftrightarrow$ $a=b$. Then $\Rel$ is an equivalence relation
and it is also antisymmetric.

(ii): Consider the equivalence relation of Example~\ref{ex:mod5}. 
It is not antisymmetric because $5\equiv 10$ and $10\equiv 5$ yet
$5\neq 10$. 
\end{solution}

\begin{exercise}
\label{exo:walaa5}
Let $A$ be the set of functions from $\RR$ to $\RR$.
Define a relation $\Rel$ on $A$ by
\begin{equation*}
f\Rel g \quad :\Leftrightarrow \quad
\big( f(0)=g(0) \;\;\text{or}\;\; f(1)=g(1)\big).
\end{equation*}
Decide whether $\Rel$ is
(i) reflexive;
(ii) symmetric; 
(iii) transitive;
(iv) an equivalence relation.
\end{exercise}
\begin{solution}
Let $f,g,h$ be in $A$, i.e., $f,g,h$ are functions from $\RR$ to $\RR$.

(i): Since $f(0)=f(0)$ and $f(1)=f(1)$, it is clear that $f \Rel f$ and
hence $\Rel$ is \emph{reflexive}.

(ii): Since $f(0)=g(0) \Rightarrow g(0)=f(0)$ and also
$f(1)=g(1) \Rightarrow g(1)=f(1)$, it follows that $\Rel$ is symmetric.

(iii): Suppose that $f(x)=x-1$, $g(x)=2x^2-1$, and $h(x)=(x-2)^2$.
Then $f(0)=-1 = g(0)$ and so $f \Rel g$.
Also, $g(1) = 1 = h(1)$ and so $g \Rel h$.
However, $f(0)=-1\neq 4 = h(0)$ and $f(1)=0\neq 1 = h(1)$ and so 
$f\negthinspace\negthinspace\negthinspace\not\Rel h$. It follows that $\Rel$ is \emph{not transitive.}

(iv): Because transitivity fails, $\Rel$ is \emph{not an equivalence
relation}. 
\end{solution}

\begin{exercise}
\label{exo:walaa6}
Let $A$ be the set of functions from $\RR$ to $\RR$.
Define a relation $\Rel$ on $A$ by
\begin{equation*}
f\Rel g \quad :\Leftrightarrow \quad
\big( f(x)-g(x)=1 \;\;\text{for every $x\in\RR$.}\big).
\end{equation*}
Decide whether $\Rel$ is
(i) reflexive;
(ii) symmetric; 
(iii) transitive;
(iv) an equivalence relation.
\end{exercise}
\begin{solution}
Let $f,g,h$ be in $A$, i.e., $f,g,h$ are functions from $\RR$ to $\RR$.

(i): Since $f(x)-f(x)=0\neq 1$ for every $x\in\RR$,
it is clear that 
$f\notRel f$ 
and thus
$\Rel$ is \emph{not reflexive}.

(ii): If $f\Rel g$, i.e., $f(x)-g(x)=1$ for every $x\in\RR$,
then $g(x)-f(x)=-1\neq 1$ for every $x\in\RR$.
Hence $g\notRel f$ and thus
$\Rel$ is \emph{not symmetric}.

(iii):
If $f\Rel g$ and $g\Rel h$, i.e., $f(x)-g(x)=1$ and $g(x)-h(x)=1$  
for every $x\in\RR$, then
(by adding) $f(x)-h(x) = 2$ for every $x\in\RR$ and hence 
$f\notRel h$.
Thus $\Rel$ is \emph{not transitive}.

(iv): Since each property fails, $\Rel$ is \emph{not an equivalence
relation}.
\end{solution}

\begin{exercise}
\label{exo:walaa7}
Let $A$ be the set of functions from $\RR$ to $\RR$.
Define a relation $\Rel$ on $A$ by
$f\Rel g$ $:\Leftrightarrow$
there exists a constant $c\in\RR$ (possibly depending on $f$ and $g$) 
such that $f(x)-g(x)=c$ for every $x\in\RR$.

Decide whether $\Rel$ is
(i) reflexive;
(ii) symmetric; 
(iii) transitive;
(iv) an equivalence relation.
\end{exercise}
\begin{solution}
Let $f,g,h$ be in $A$, i.e., $f,g,h$ are functions from $\RR$ to $\RR$.

(i): Since $f(x)-f(x)=0$ for every $x\in\RR$,
it is clear that $f\Rel f$ and thus $\Rel$ is \emph{reflexive}.

(ii): If $f\Rel g$, say $f(x)-g(x)=c$ for some $c\in\RR$ and every $x\in\RR$,
then $g(x)-f(x)=-c$ for every $x\in\RR$.
Hence $g\Rel f$ and thus
$\Rel$ is \emph{symmetric}.

(iii):
If $f\Rel g$ and $g\Rel h$, say $f(x)-g(x)=c$ and $g(x)-h(x)=d$  for every
$x\in\RR$ and some constants $c$ and $d$ in $\RR$, then
(by adding) $f(x)-h(x) = c+d$ for every $x\in\RR$ and hence 
$f\Rel h$.
Thus $\Rel$ is \emph{transitive}.

(iv): Since each property holds, $\Rel$ is \emph{an equivalence
relation}.
\end{solution}

\begin{exercise}
Let $A = \RR^2$, and 
define a relation $\Rel$ on $A$ by
$(x,y)\Rel (u,v)$ $:\Leftrightarrow$
$|x|+|y|=|u|+|v|$.

Show that $\Rel$ is an equivalence relation.
Describe the equivalence class $[(1,0)]$ 
geometrically.
Which equivalence class has exactly one element?
\end{exercise}
\begin{solution}
Let $(x,y)$, $(u,v)$, and $(a,b)$ in $\RR^2$.

(i): Clearly, $(x,y)\Rel (x,y)$ because
$|x|+|y|=|x|+|y|$. 
Thus $\Rel$ is \emph{reflexive}.

(ii): If $(x,y)\Rel (u,v)$, 
then $|x|+|y| = |u|+|v|$ and
hence $|u|+|v|=|x|+|y|$ which shows that 
$(u,v)\Rel (x,y)$.
Hence $\Rel$ is \emph{symmetric}.

(iii):
If $(x,y)\Rel (u,v)$ and $(u,v)\Rel (a,b)$, then
$|x|+|y|=|u|+|v|$ and $|u|+|v|=|a|+|b|$
and so $|x|+|y|=|a|+|b|$ which shows that 
$(x,y) \Rel (a,b)$. 
Thus $\Rel$ is \emph{transitive}.

The equivalence class $[(1,0)] =
\menge{(x,y)\in\RR^2}{|x|+|y|=1}$
is the ``diamond'' with vertices
$(\pm 1,0)$ and $(0,\pm 1)$. 

All equivalence classes are scaled version of this diamond
--- and thus contain infinitely many elements --- 
except for $[(0,0)] = \{(0,0)\}$. 
\end{solution}

\begin{exercise}
\label{exo:walaa8}
Let $A$ be the set of functions from $\RR$ to $\RR$.
Define a relation $\Rel$ on $A$ by
\begin{equation*}
f\Rel g \quad :\Leftrightarrow \quad
\big( f(0)=g(1)\;\;\text{and}\;\; f(1)=g(0)\big)
\end{equation*}
Decide whether $\Rel$ is
(i) reflexive;
(ii) symmetric; 
(iii) transitive;
(iv) an equivalence relation.
\end{exercise}
\begin{solution}
Let $f,g,h$ be in $A$, i.e., $f,g,h$ are functions from $\RR$ to $\RR$.

(i): Suppose that $f(x)=x$. Then $f(0)=0\neq 1 = f(1)$, so 
$f\negthinspace\negthinspace\negthinspace\not\Rel f$ 
and thus
$\Rel$ is \emph{not reflexive}.

(ii): 
If $f\Rel g$, then $f(0)=g(1)$ and $f(1)=g(0)$;
hence, $g(0)=f(1)$ and $g(1)=f(0)$, i.e.,
$g\Rel f$.
Thus $\Rel$ is \emph{symmetric}.

(iii):
Suppose that $f(x)=h(x)=|x-1|$ and
$g(x)=x$.
Then $f(0)=1=g(1)$ and $f(1)=0=g(0)$.
Hence $f\Rel g$ and thus (by (ii) or directly)
$g\Rel f$ and so $g\Rel h$.
However, $f(0)=1\neq 0 = f(1)=h(1)$ and thus
$f\negthinspace\negthinspace\negthinspace\not\Rel h$.
Thus $\Rel$ is \emph{not transitive}.

(iv): Since two properties fail, $\Rel$ is \emph{not an equivalence
relation}.
\end{solution}

\begin{exercise}
Suppose that $S = \{1,2,3\}$ and
denote its power set by $A$.
Define a relation on $A$ by
$X\Rel Y$ $:\Leftrightarrow$
$X$ and $Y$ have the same number of elements.
Show that $\Rel$ is an equivalence relation,
and determine the quotient set.
How many elements does the quotient set have?
\end{exercise}
\begin{solution}
Let $X,Y,Z$ be elements in $A$, i.e., subsets of $S$.

(i): Since $X$ and $X$ have the same number of elements, we have 
$X\Rel X$ and so $\Rel$ is reflexive.

(ii): If $X$ and $Y$ have the same number of elements, so do
$Y$ and $X$; consequently, 
$\Rel$ is symmetric.

(iii): 
If $X$ and $Y$ have the same number of elements, say $m$,
and $Y$ and $Z$ have the same number of elements, say $n$,
then $m=n$ and so $\Rel$ is transitive.

Thus, $\Rel$ is an equivalence relation.

We have
\begin{align*} 
A/\Rel &= 
\big\{ [\varnothing],[\{1\}],[\{1,2\}],[\{1,2,3\}]\big\}\\
&= 
\Big\{
\{\varnothing\},\big\{\{1\},\{2\},\{3\}\big\},\big\{\{1,2\},\{1,3\},\{2,3\}\big\},\big\{\{1,2,3\}\big\}\Big\}
\end{align*}
and the quotient set has 4 elements. 
\end{solution}

\begin{exercise}
Suppose that $S = \{1,2,3,4,5\}$ and
denote its power set by $A$.
Define a relation on $A$ by
$X\Rel Y$ $:\Leftrightarrow$
$X\cap\{1,3,5\}= Y\cap\{1,3,5\}$.
Show that $\Rel$ is an equivalence relation and
determine the equivalence class $[\{1,2,3\}]$. 
\end{exercise}
\begin{solution}
Let $X,Y,Z$ be elements in $A$, i.e., subsets of $S$.

(i): Since $X\cap\{1,3,5\}=X\cap\{1,3,5\}$, we have
$X\Rel X$ and $\Rel$ is reflexive.

(ii): If $X\cap\{1,3,5\} = Y\cap\{1,3,5\}$, then
$Y\cap\{1,3,5\} = X\cap\{1,3,5\}$ and hence $\Rel$ is symmetric.

(iii): If $X\cap \{1,3,5\} = Y\cap\{1,3,5\}$ and 
$Y\cap\{1,3,5\}=Z\cap\{1,3,5\}$, then
$X\cap\{1,3,5\}=Z\cap\{1,3,5\}$ and hence $\Rel$ is transitive.

Thus, $\Rel$ is an equivalence relation.

Now $Z\in [\{1,2,3\}]$
$\Leftrightarrow$
$\{1,2,3\}\Rel Z$
$\Leftrightarrow$
$\{1,2,3\}\cap\{1,3,5\} = Z\cap\{1,3,5\}$
$\Leftrightarrow$
$\{1,3\} = Z\cap\{1,3,5\}$
$\Leftrightarrow$
($\{1,3\}\subseteq Z$ and $5\notin Z$)
$\Leftrightarrow$
$Z\in\{ \{1,3\},\{1,2,3\},\{1,3,4\},\{1,2,3,4\}\}$. 
Hence
\begin{equation*}
 [\{1,2,3\}]
=\{ \{1,3\},\{1,2,3\},\{1,3,4\},\{1,2,3,4\}\}.
\end{equation*}
\end{solution}

\begin{exercise}[Pigeonhole Principle]
If you have $n$ buckets (or pigeonholes) and 
more than $n$ objects (or pigeons) that you distribute,
then one of the buckets contains at least $2$ objects.
This is known as the \emph{Pigeonhole Principle}.
\index{Pigeonhole Principle}
Now assume that you are in a class with 15+ students. 
Deduce that there are two students in your class 
whose birthday falls in the same calendar month. 
\end{exercise}
\begin{solution}
We have 12 buckets, namely the months of a year.
Your class has 15+ students, so there are more than 12 students in it.
Hence there exist two students whose birthday falls in the same
month.
\end{solution}

\begin{exercise}[hair!]
Assuming that the maximum number of hairs on a human head
is 150,000, show that there are (at least) five people in Vancouver having
exactly the same number of hairs on their heads. 
\emph{Hint:} Argue by contradiction and find out the population of
Vancouver!
\end{exercise}
\begin{solution}
This is a variant of the Pigeonhole Principle.
We argue by contradiction and suppose that the result is false. 
Then for every number $n\in\{0,1,2,\ldots,150000\}$
there are no more than 4 people with that number $n$ of hair.
Hence the population of Vancouver can be no larger than
$4\cdot150001 = 600004$.
On the other hand, the population of Vancouver is 
larger than 660000 --- contradiction!
(Sources for the numbers: WolframAlpha.)
\end{solution}

\begin{exercise}
Let $A=\QQ$ be the set of rational numbers and define 
a relation on $A$ by 
$$ \frac{a}{b} \Rel \frac{c}{d}
\quad :\Leftrightarrow\quad 
\frac{a}{b}-\frac{c}{d}\in \ZZ.$$
Show that $\Rel$ is an equivalence relation.
Furthermore, 
determine the equivalence class $[0]$. 
\end{exercise}
\begin{solution}
Consider elements $a/b$, $c/d$, $e/f$ in $A=\QQ$
(so $b,d,f$ are all nonzero of course).

(i): We have $(a/b)-(a/b) = 0\in\ZZ$ so $(a/b)\Rel (a/b)$ 
and $\Rel$ is reflexive.

(ii): 
Suppose $(a/b)\Rel (c/d)$, say $(a/b)-(c/d) = k\in\ZZ$.
Then $(c/d)-(a/b) = -k \in \ZZ$ and so $(c/d)\Rel (a/b)$.
Hence $\Rel$ is symmetric.

(iii): 
Suppose $(a/b)\Rel (c/d)$, say $(a/b)-(c/d) = m\in\ZZ$, 
and $(c/d)\Rel (e/f)$, say $(c/d)-(e/f) = n\in\ZZ$. 
Then 
$$
(a/b)-(e/f)
= \big( (a/b)-(c/d)\big) + \big((c/d)-(e/f)\big)
= m+n \in \ZZ,
$$ 
and therefore $(a/b)\Rel (e/f)$.
Hence 
$\Rel$ is transitive.

Thus, $\Rel$ is an equivalence relation.

Finally, 
$(a/b)\in [0]$
$\Leftrightarrow$ 
$0\Rel (a/b)$
$\Leftrightarrow$ 
$(a,b)\Rel 0$
$\Leftrightarrow$ 
$(a/b)-0 \in \ZZ$
$\Leftrightarrow$
$(a/b)\in\ZZ$.
Therefore, $[0]=\ZZ$.
\end{solution}

\begin{exercise}
Set $A = \ZZ\times \ZZ$
and define a relation $\Rel$ on $A$ by
\begin{equation}
(x,y)\Rel (u,v) \quad :\Leftrightarrow\quad
x+v = y+u.
\end{equation}
Show that $\Rel$ is an equivalence relation.
Moreover, determine the equivalence class $[(0,0)]$.
\end{exercise}
\begin{solution} 
Let $(x,y),(u,v),(a,b)$ be in $A$. 

Reflexivity: 
Because $x+y=y+x$, it is clear that $(x,y)\Rel(x,y)$. 

Symmetry: Assume that $(x,y)\Rel(u,v)$, i.e., $x+v=y+u$.
Then $u+y=v+x$ and hence $(u,v)\Rel(x,y)$.

Transitivity: Assume $(x,y)\Rel(u,v)$, i.e., $x+v=y+u$.
Also assume that $(u,v)\Rel(a,b)$, i.e., $u+b=v+a$. 
Hence $x-y=u-v$ and $u-v=a-b$.
Thus $x-y=a-b$, i.e., $x+b=y+a$ and so $(x,y)\Rel (a,b)$.

Finally, let $(u,v)\in A$.
Then 
$(u,v)\in[(0,0)]$
$\Leftrightarrow$
$(0,0)\Rel (u,v)$
$\Leftrightarrow$
$0+v=0+u$
$\Leftrightarrow$
$v=u$.
Hence 
$[(0,0)]=\menge{(n,n)}{n\in\ZZ}$.
\end{solution}

\begin{exercise}[YOU be the marker!] 
Consider the following statement 
\begin{equation}
\text{``The usual $\leq$ is an equivalence relation on $\RR$.''}
\end{equation}
and the following ``proof'':
\begin{quotation}
Reflexivity: Clear, since $a\leq a$.\\
Symmetry: If $a\leq b$ and $b\leq a$, then $a=b$ and so $a\leq b$.\\
Transitivity:
If $a\leq b$ and $b\leq c$, then obviously $a\leq c$. 
\end{quotation}
Why is this proof wrong?
\end{exercise}
\begin{solution}
The proofs of reflexivity and transitivity are correct.
But the proof of symmetry is wrong.
Symmetry means that if $a\leq b$, then 
we must show that $b\leq a$ (we cannot assume this).
And this is wrong: e.g., $1\leq 2$ is true 
but $2\leq 1$ is clearly false.
\end{solution}

\begin{exercise}[YOU be the marker!]  
Consider the following statement 
\begin{equation*}
\text{``If $\Rel$ is a symmetric and transitive relation on $A$,
 then $\Rel$ is also reflexive.''}
\end{equation*}
and the following ``proof'':
\begin{quotation}
Let $x\in A$.\\
Now get $y\in A$ such that $x\Rel y$.\\
By symmetry, $y\Rel x$.\\
Now $x\Rel y$ and $y\Rel x$, 
so transitivity yields $x\Rel x$ and we are done.
\end{quotation}
Why is this proof wrong?
\end{exercise}
\begin{solution}
All is well \textbf{except} that we cannot guarantee that 
for all $x\in A$, there is some $y\in A$ such that 
$x\Rel y$ and that is where the proof breaks down.

A counterexample is $A = \{1,2\}$
and $\Rel = \{(2,2)\}$ which is symmetric and transitive but 
not reflexive because $(1,1)\notin\Rel$. 
\end{solution}

\begin{exercise}[TRUE or FALSE?]
Mark each of the following statements as either true or false. 
Briefly justify your answer.
\begin{enumerate}
\item ``If $x\in A$ and $A\subseteq B$, then $x\in B$.''
\item ``If $x\notin A$ and $A\subseteq B$, then $x\notin B$.''
\item ``Every relation is an equivalence relation.''
\item ``If $A$ is a set of $10$ elements and 
$\Rel$ is an equivalence relation on $A$, then the quotient set 
$A/\Rel$ has at least 2 elements.
\end{enumerate}
\end{exercise}
\begin{solution}

(i): TRUE: By definition. 

(ii): FALSE: $-1\notin \NN\subseteq \ZZ$ but $-1\in\ZZ$. 

(iii): FALSE: There are many counterexamples, e.g., ``$<$''. 

(iv): FALSE: If $\Rel = A\times A$, i.e., every $a\in A$ is 
related to every $b\in A$,   
then the quotient set has exactly one element, $A$ itself.
\end{solution} 
\chapter{Ordered Fields}

\label{ch:of}

Inequalities are a crucial part of mathematics. 
We now revisit the real numbers in this light.

\section{Axioms of Order and Consequences}

\index{Axioms of Order}

\begin{definition}[Axioms of Order]
Certain elements in $\RR$ are termed \textbf{positive}; 
if $x\in\RR$ is positive, we write $x>0$.
The following hold.
\begin{enumerate}
\item[\textbf{O1}]
For every $x\in\RR$, exactly one of the following hold:
$x>0$, $x=0$, or $-x>0$.
\item[\textbf{O2}]
For all $x,y$ in $\RR$ we have:
[$x>0$ and $y>0$] $\Rightarrow$ $x+y>0$.
\item[\textbf{O3}]
For all $x,y$ in $\RR$ we have:
[$x>0$ and $y>0$] $\Rightarrow$ $xy>0$.
\end{enumerate}
\end{definition}

If $-x$ is positive, we say that $x$ is \textbf{negative}.
Loosely speaking, \textbf{O1} says that every real number
is $0$, positive, or negative; and \textbf{O2} and \textbf{O3} say
that the sum and product of positive numbers is again positive.

Given $x,y$ in $\RR$, we write:
$x>y$ if $x-y>0$;
$x\geq y$ if [$x>y$ or $x=y$];
$y<x$ if $x>y$;
and 
$y\leq x$ if $x\geq y$.
If $x\geq 0$, we say that $x$ is \textbf{nonnegative}\footnote{Warning:
Some authors include $0$ when speaking about positive numbers.};
if $x\leq 0$, then $x$ is \textbf{nonpositive}.

\begin{proposition}
\label{p:ofield}
Let $u,v,x,y,z$ be in $\RR$.
Then the following hold.
\begin{enumerate}
\item 
\label{p:ofieldi}
$x<0$ 
$\Leftrightarrow$
$-x>0$.
\item 
\label{p:ofieldii}
(``$<$'' is transitive)~~
\emph{[}$x<y$ and $y<z$\emph{]} $\Rightarrow$ $x<z$.
\item 
\label{p:ofieldiii}
$x<y$ $\Leftrightarrow$ $x+z<y+z$.
\item 
\label{p:ofieldiv}
\emph{[}$x<y$ and $u<v$\emph{]} $\Rightarrow$ $x+u<y+v$.
\item 
\label{p:ofieldiv+}
\emph{[}$x\leq y$ and $u \leq v$\emph{]} $\Rightarrow$ $x+u\leq y+v$.
\item 
\label{p:ofieldv}
\emph{[}$x<y$ and $u>0$\emph{]} $\Rightarrow$ $ux<uy$.
\item 
\label{p:ofieldvi}
\emph{[}$x<y$ and $u<0$\emph{]} $\Rightarrow$ $ux>uy$.
\item 
\label{p:ofieldvii}
\emph{[}$x\leq y$ and $u>0$\emph{]} $\Rightarrow$ $ux\leq uy$.
\item 
\label{p:ofieldviii}
\emph{[}$x\leq y$ and $u<0$\emph{]} $\Rightarrow$ $ux\geq uy$.
\item 
\label{p:ofieldix}
\emph{[}$0\leq x< y$ and $0\leq u<v$\emph{]} $\Rightarrow$ $ux< vy$.
\item 
\label{p:ofieldx}
$x\neq 0$ $\Leftrightarrow$ $x^2>0$. 
\item 
\label{p:ofieldxi}
$1>0$.
\item 
\label{p:ofieldxii}
$x>0$ $\Leftrightarrow$ $x^{-1}>0$.
\item 
\label{p:ofieldxiii}
$x<0$ $\Leftrightarrow$ $x^{-1}<0$.
\item 
\label{p:ofieldxiv}
$0<x<y$ $\Rightarrow$ $x^{-1}>y^{-1}$. 
\end{enumerate}
\end{proposition}
\begin{proof}
\ref{p:ofieldi}:
$x<0$
$\Leftrightarrow$
$0>x$
$\Leftrightarrow$
$0-x>0$
$\Leftrightarrow$
$-x>0$.

\ref{p:ofieldii}:
Since $x<y$ and $y<z$, we have
$y-x>0$ and $z-y>0$. Using \textbf{O2},
it follows that 
$z-x = (y-x)+(z-y) > 0$, i.e., $z>x$.

\ref{p:ofieldiii}:
$x<y$ 
$\Leftrightarrow$
$y-x>0$
$\Leftrightarrow$
$(y+z)-(x+z)>0$
$\Leftrightarrow$
$x+z<y+z$.

\ref{p:ofieldiv}: 
Indeed, using \textbf{O2},
we have
[$y-x>0$ and $v-u>0$]
$\Rightarrow$ 
$(y+v)-(x+u) = (y-x)+(v-u)>0$
$\Leftrightarrow$
$x+u<y+v$.

\ref{p:ofieldv}: 
Indeed, using \textbf{O3}, we have
[$x<y$ and $u>0$] 
$\Leftrightarrow$
[$y-x>0$ and $u>0$]
$\Rightarrow$
$uy-ux = yu-xu= (y-x)u>0$
$\Leftrightarrow$
$uy>ux$ 
$\Leftrightarrow$ $ux<uy$.

\ref{p:ofieldvi}: 
Indeed, using \textbf{O3}, we have
[$x<y$ and $u<0$] 
$\Leftrightarrow$
[$y-x>0$ and $-u>0$]
$\Rightarrow$
$-uy+ux = y(-u)-x(-u)= (y-x)(-u)>0$
$\Leftrightarrow$
$uy<ux$.

\ref{p:ofieldvii}: 
Exercise~\ref{exo:201011a}.

\ref{p:ofieldviii}: 
Exercise~\ref{exo:201011b}.

\ref{p:ofieldix}: 
Exercise~\ref{exo:201011c}.

\ref{p:ofieldx}: 
``$\Rightarrow$'': 
Suppose that $x\neq 0$.
If $x>0$, then $x^2 = x \cdot x > 0$ by \textbf{O3}.
If $x<0$, then $-x>0$ and hence
$x^2 = x\cdot x = (-x)(-x) = (-x)^2 > 0$
by Proposition~\ref{p:may26:4}\ref{p:may26:4v} and \textbf{O3}. 

``$\Leftarrow$'': Suppose that $x=0$. Then $x^2 = x\cdot x = 0$
by Proposition~\ref{p:may26:4}\ref{p:may26:4ii}. Hence
if  $x^2>0$, then $x\neq 0$. 

\ref{p:ofieldxi}: 
We have $1\neq 0$ by \textbf{M3}.
Hence, by \textbf{M3} and \ref{p:ofieldx}, $1 = 1^2 > 0$.

\ref{p:ofieldxii}: 
``$\Rightarrow$'':
If $x>0$, then 
$x\neq 0$ and (by \ref{p:ofieldx}) $(x^{-1})^2 > 0$;
hence, by \textbf{O3}, 
$x^{-1} = x(x^{-1})^2 > 0$.

``$\Leftarrow$'':
Apply ``$\Rightarrow$'' to $x^{-1}$ and use the fact that
$(x^{-1})^{-1} = x$ by Proposition~\ref{p:100908:2}\ref{p:100908:2v}. 

\ref{p:ofieldxiii}: 
This is similar to 
\ref{p:ofieldxii}. 

\ref{p:ofieldxiv}: 
We have $0<x$ and, by \ref{p:ofieldii}, $0<y$.
Hence, by \textbf{O3}, $xy>0$.
Thus, using Proposition~\ref{p:100908:2}\ref{p:100908:2v} 
and \ref{p:ofieldxii},
$u:= x^{-1}y^{-1}=(xy)^{-1} >0$.
By \ref{p:ofieldv}, 
$xu<yu$, hence
$y^{-1}= xx^{-1}y^{-1}=xu<yu=yx^{-1}y^{-1}=x^{-1}$. 
\end{proof}

The following two results are actually used quite often in mathematical proofs.

\begin{corollary}
\label{c:130923a}
Let $n\in\{2,3,\ldots\}$, and let $x_1,\ldots,x_n$ be nonnegative numbers.
Then\footnote{This result is used in Linear Programming and Convex
Optimization courses in the proof of duality theorems.}
\begin{equation}
x_1+x_2+\cdots + x_n = 0 
\quad\Leftrightarrow\quad
x_1=x_2=\cdots = x_n = 0.
\end{equation}
\end{corollary}
\begin{proof}
Consider the case when $n=2$. 

``$\Leftarrow$'': This is clear because $0+0=0$ by \textbf{A3}.

``$\Rightarrow$'': We argue by contradiction and assume that one of the
numbers, say $x_1$, is positive: $x_1>0$. 
Then, by Proposition~\ref{p:ofield}\ref{p:ofieldiii},
$0 \leq x_2 = x_2+0 < x_2 + x_1 = 0$, which is absurd. 

The general case follows by mathematical induction and is left as 
Exercise~\ref{exo:130923b}. 
\end{proof}

\begin{corollary}
\label{c:130923c}
Let $n\in\{2,3,\ldots\}$, and let $x_1,\ldots,x_n$ be in $\RR$.
Then\footnote{This result is useful in Linear Algebra and Matrix Analysis,
where it is used to prove that the only vector of length $0$ is the zero
vector: $\|x\|=0$ $\Leftrightarrow$ $x=(0,\ldots,0)\in\RR^n$.}
\begin{equation}
x_1^2+x_2^2+\cdots + x_n^2 = 0 
\quad\Leftrightarrow\quad
x_1=x_2=\cdots = x_n = 0.
\end{equation}
\end{corollary}
\begin{proof}
Exercise~\ref{exo:130923d}. 
\end{proof}

\begin{remark} 
\label{r:191010}
Some comments on Proposition~\ref{p:ofield} are in order.
\begin{enumerate}
\item 
\label{r:191010i}
An \textbf{ordered field} is a field in which the axioms of order
\textbf{O1},
\textbf{O2},
\textbf{O3} hold.
Thus, $\RR$ is an ordered field, as is $\QQ$, the field of rational
numbers. 
Since the proofs of Proposition~\ref{p:ofield} utilized only the axioms of
order, all the properties of that Proposition hold true in any ordered
field, e.g., $\QQ$. 
\item Consider the {complex numbers} $\mathbb{C}$, in which 
$\mathrm{i}^2=-1$, where $\mathrm{i} = \sqrt{-1}$. 
Since $\mathrm{i}\neq 0$, item~\ref{p:ofieldx} would imply that
$-1=\mathrm{i}^2>0$, hence $0>1$, 
which contradicts item~\ref{p:ofieldxi}.
This shows that it is intrinsically  
\emph{impossible} to make the complex numbers an
ordered field, no matter how we try to define order there.
\item The natural numbers are embedded in the real numbers via $0\mapsto 0$
and
\begin{equation}
n \mapsto \underbrace{1+1+\cdots + 1}_{\text{$n$ terms}}.
\end{equation}
Using \textbf{O2} and Proposition~\ref{p:ofield}\ref{p:ofieldxi}, we can
show by mathematical induction that $n>0$ for $n\in\{1,2,3,\ldots\}$.
\item 
The finite field $\mathbb{F}_2$ encounted in Section~\ref{sec:fields}
cannot be ordered. For if it could be ordered,
then Proposition~\ref{p:ofield}\ref{p:ofieldxi} would yield $1>0$ and
hence,
by Proposition~\ref{p:ofield}\ref{p:ofieldiii}, 
$0=1+1>0+1=1$, which is absurd. 
\end{enumerate}
\end{remark}

\section{Properties of the Absolute Value}

\begin{definition}[absolute value]
\label{d:absval}
For $x\in\RR$, the \emph{absolute value} is\index{Absolute value}
\begin{equation}
|x| := \begin{cases}
x, &\text{if $x\geq 0$;}\\
-x, &\text{if $x<0$.}
\end{cases}
\end{equation}
\end{definition}

\begin{proposition}
\label{p:absval}
Let $x$ and $y$ be in $\RR$.
Then the following hold.
\begin{enumerate}
\item 
\label{p:absvali}
$|x|\geq 0$.
\item 
\label{p:absvalii}
$|x|=0$ $\Leftrightarrow$ $x=0$.
\item 
\label{p:absvalii+}
$x\leq |x|$.
\item 
\label{p:absvaliii}
$|x\cdot y|=|x|\cdot|y|$.
\item 
\label{p:absvaliii+}
$|-x|=|x|$.
\item 
\label{p:absvalvi}
$-x\leq |x|$.
\item 
\label{p:absvaliv}
$y\neq 0$ $\Rightarrow$ 
$\displaystyle \left|\frac{x}{y}\right|=\frac{|x|}{|y|}$.
\end{enumerate}
\end{proposition}
\begin{proof}
\ref{p:absvali}: 
Case~1: $x\geq 0$. Then $|x|=x\geq 0$ by definition.
Case~2: $x<0$. Then, by definition and 
Proposition~\ref{p:ofield}\ref{p:ofieldi},
$|x|=-x>0$.

\ref{p:absvalii}: 
``$\Leftarrow$'': If $x=0$,
then $|x|=x=0$.
``$\Rightarrow$'': We prove this by contradiction.
Assume that $|x|=0$ and $x\neq 0$. 
Then either $x>0$ or $x<0$.
If $x>0$, then $|x|=x>0$, which is absurd;
otherwise, $x<0$, and then $|x|=-x>0$, which is also absurd.

\ref{p:absvalii+}: 
If $x\geq 0$, then $|x|=x$ and so $x\leq |x|$.
And if $x<0$, then $|x|=-x>0$ and so $x=-|x|<0<|x|$.

\ref{p:absvaliii}: 
We consider 4 cases.

\emph{Case~1}: $x\geq 0$ and $y\geq 0$.\\
Then $|x|=x$ and $|y|=y$; hence,
\begin{equation}
\label{e:100920:a}
|x|\cdot|y| = xy.
\end{equation}
Now if $y>0$, then $xy\geq 0$ by
Proposition~\ref{p:ofield}\ref{p:ofieldvii}.
And if $y=0$, then $xy=x\cdot 0 = 0$ by
Proposition~\ref{p:may26:4}\ref{p:may26:4ii}. 
In either case $xy\geq 0$ and thus 
\begin{equation}
\label{e:100920:b}
|xy| = xy\geq 0. 
\end{equation}
Combining \eqref{e:100920:a} and \eqref{e:100920:b},
we see that $|x|\cdot|y| = |xy|$. 

\emph{Case~2}: $x\geq 0$ and $y< 0$.\\
Suppose first that $x=0$.
Then $x\cdot y = 0\cdot y = 0$, 
$|x y| = |0| = 0$, and $|x|\cdot |y| = |0|\cdot (-y) 
= 0\cdot (-y) = 0$. Thus, $|xy|=0=|x|\cdot|y|$.
Now assume that $x>0$. Then
$|x|\cdot|y| = x\cdot (-y) = x\cdot(-1)y = (-1)(xy) =-(xy)$ and
$xy<0$ (see Proposition~\ref{p:may26:4}\ref{p:may26:4iv} and 
Proposition~\ref{p:ofield}\ref{p:ofieldv}) 
so that $|xy| = -(xy)$. 

\emph{Case~3}: $x<0$ and $y\geq 0$.\\
\hhbcom{This is as Case~2}

\emph{Case~4}: $x<0$ and $y<0$.\\

\ref{p:absvaliii+}: 
Set $y=-1$ in \ref{p:absvaliii}.
Now $1>0$ (by Proposition~\ref{p:ofield}\ref{p:ofieldxi}), hence
$-1<0$ and thus $|-1| = -(-1) = 1$.
Thus $|xy| = |x|\cdot|y|$ turns into
$|-x| = |x(-1)| = |x|\cdot|-1| = |x|\cdot 1 = |x|$.

\ref{p:absvalvi}:
By \ref{p:absvalii+} (applied to $-x$) and \ref{p:absvaliii+}, 
we have
$-x \leq |-x| = |x|$.

\ref{p:absvaliv}: 
We have $x = (x/y)\cdot y$.
Hence, by \ref{p:absvaliii}, 
$|x| = |x/y|\cdot|y|$ and thus $|x|/|y| = |x/y|$. 
\end{proof}

The following result is useful when checking inequalities.

\begin{lemma}
\label{l:absval}
Let $x$ and $y$ be in $\RR$.
Then 
\begin{equation}
|x|\leq y
\quad
\Leftrightarrow
\quad
x\leq y \text{ and } -x\leq y.
\end{equation}
\end{lemma}
\begin{proof}
``$\Rightarrow$'':
Suppose that $|x|\leq y$.
Then, by
Proposition~\ref{p:absval}\ref{p:absvalii+}\&\ref{p:absvalvi},
$x\leq|x|$ and $-x\leq|x|$. Hence,
$\pm x \leq|x|\leq y$.

``$\Leftarrow$'':
Suppose that $\pm x\leq y$.
By definition, $|x|=x$ or $|x|=-x$.
Hence, $|x|\leq y$.
\end{proof}

\begin{theorem}[triangle inequality]
\label{t:triangle}
Let $x$ and $y$ be in $\RR$. Then\index{Triangle inequality}
\begin{equation}
|x+y|\leq |x|+|y|.
\end{equation}
\end{theorem}
\begin{proof}
By Proposition~\ref{p:absval}\ref{p:absvalii+},
$x\leq |x|$ and $y\leq |y|$.
Hence, by Proposition~\ref{p:ofield}\ref{p:ofieldiv+}, 
\begin{equation} 
x+y \leq |x|+|y|. 
\end{equation}
Similarly,
by Proposition~\ref{p:absval}\ref{p:absvalvi},
$-x\leq |x|$ and $-y\leq |y|$; thus,
\begin{equation} 
-(x+y)= -x+(-y) \leq |x|+|y|. 
\end{equation}
Altogether, 
using Lemma~\ref{l:absval}, we see that 
$|x+y| \leq |x|+|y|$. 
\end{proof}

\begin{corollary}
\label{c:triangle}
Let $x$ and $y$ be in $\RR$.
Then
\begin{equation}
\big| |x|-|y|\big| \leq |x-y|.
\end{equation}
\end{corollary}
\begin{proof}
On the one hand, using Theorem~\ref{t:triangle}, 
$|x| = |(x-y)+y| \leq |x-y| + |y|$.
On the other hand,
$|y| \leq |y-x| + |x| = |x-y|+|x|$.
Altogether,
\begin{equation}
|x|-|y| \leq |x-y|
\quad\text{and}\quad
-\big(|x|-|y|\big) = |y|-|x| \leq |x-y|,
\end{equation}
which implies that $\big||x|-|y|\big| \leq |x-y|$
by Lemma~\ref{l:absval}. 
\end{proof}

\begin{lemma}[general triangle inequality]
\label{l:generaltriangle}
Let $n\geq 2$ be an integer, and
let $x_1,x_2,\ldots,x_n$ be real numbers.
Then\index{Triangle inequality}
\begin{equation}
\label{e:gentriangle}
\left|\sum_{i=1}^{n} x_i\right|
\leq \sum_{i=1}^{n} |x_i|.
\end{equation}
\end{lemma}
\begin{proof}
Exercise~\ref{exo:triangle}.
\end{proof}

\begin{lemma}
\label{l:maxmin}
Let $x$ and $y$ be in $\RR$.
Then\index{Maximum}\index{Minimum}
\begin{equation}
\max\{x,y\} = \frac{x+y+|x-y|}{2}
\quad\text{and}\quad
\min\{x,y\} = \frac{x+y-|x-y|}{2}.
\end{equation}
\end{lemma}

\section{Archimedean Fields and Consequences}

\begin{definition}[archimedean property]
For all $x,y$ in $\RR$ such that $x>0$ and $y>0$,
there exists $n\in\NN$ such that $ny>x$.
\index{Archimedean property}\index{Archimedean field}
\end{definition}

\begin{lemma}
\label{l:entier}
For every $x>0$, there exists $n\in\NN$ such that $n>x$.
\end{lemma}
\begin{proof}
Clear from the archimedean property.
\end{proof}

For every $x\in\RR$, there exists a unique integer $k\in\ZZ$ such that 
$k\leq x < k+1$. One writes $\lfloor x \rfloor = k$ and says
``$k$ is equal to the \emph{floor} of $x$''.\index{Floor}
We thus have $\lfloor 4/3\rfloor = 1$, $\lfloor 0\rfloor = 0$,
and $\lfloor -4/3\rfloor = -2$. 
In Lemma~\ref{l:entier}, we may choose $n=\lfloor x+1\rfloor$. 

\begin{corollary}
\label{c:entier}
For every $\varepsilon>0$, there exists
$n\in\NN$ such that $n>0$ and $\tfrac{1}{n}<\varepsilon$.
\end{corollary}
\begin{proof}
Set $x=1/\varepsilon$. Then $x>0$ by
Proposition~\ref{p:ofield}\ref{p:ofieldxii}.
Hence, by Lemma~\ref{l:entier}, there exists $n\in\NN$
such that $n>x>0$.
In turn, by Proposition~\ref{p:ofield}\ref{p:ofieldxiv}, 
$\tfrac{1}{n} < \tfrac{1}{x} = \varepsilon$. 
\end{proof}

\begin{theorem}[Bernoulli's inequality]
Let $x\geq -1$ and let $n\in\NN$.
Then\index{Bernoulli's inequality} 
\begin{equation}
\label{e:bernoulli}
(1+x)^n \geq 1+nx.
\end{equation}
\end{theorem}
\begin{proof}
By Mathematical Induction on $n\geq 0$.

\textbf{Base Case}: When $n=0$, then clearly
$(1+x)^0 = 1 = 1+0 = 1+nx$. 

\textbf{Inductive Step}: Now assume $n\geq 0$ is such that 
\eqref{e:bernoulli} holds. We need to show that
\eqref{e:bernoulli} is true when $n$ is replaced by $n+1$.
Using the assumption that $1+x\geq 0$ and the induction hypothesis
in \eqref{e:bernoulli:b}, 
we see that
\begin{subequations}
\begin{align}
(1+x)^{n+1} &= (1+x)(1+x)^n\\
&\geq (1+x)(1+nx)\label{e:bernoulli:b}\\
&= 1+ (n+1)x + nx^2\\
&\geq 1+(n+1)x,
\end{align}
\end{subequations}
which shows that \eqref{e:bernoulli} holds for $n+1$.
The proof is complete by the principle of mathematical induction. 
\end{proof}

\begin{theorem}
\label{t:100920}
Let $y>1$.
Then for every $M>0$, there exists $n\in\NN$ such that $y^n>M$.
\end{theorem}
\begin{proof}
Set $x:= y-1 > 0$. 
Using the archimedean property, we obtain $n\in\NN$ such that
$nx > M$.
Because $x>0>-1$, we deduce from 
Bernoulli's inequality that 
\begin{equation}
y^n = (1+x)^n \geq 1+nx > 1+M > M
\end{equation}
as announced. 
\end{proof}

\begin{corollary}
\label{c:100920}
Let $z\in\RR$ such that $0<z<1$.
Then for every $\varepsilon>0$, there exists $n\in\NN$ such that
$z^n<\varepsilon$.
\end{corollary}
\begin{proof}
By Proposition~\ref{p:ofield}\ref{p:ofieldxiv}, 
$y := \tfrac{1}{z}>1$. 
Let $\varepsilon>0$. 
Now Theorem~\ref{t:100920} implies the existence of 
$n\in\NN$ such that $\tfrac{1}{z^n} = \big(\tfrac{1}{z}\big)^n = y^n >
\tfrac{1}{\varepsilon}$. 
Using Proposition~\ref{p:ofield}\ref{p:ofieldxiv} again,
we see that 
$z^n<\varepsilon$. 
\end{proof}

\begin{remark}
An ordered field with the archimedean property is called
an \textbf{archimedean field}. So, $\RR$ and $\QQ$ are archimedean. 
However, there do exist ordered fields that are not archimedean; e.g.,
the real rational functions $\RR(x)$ (see Exercises
\ref{exo:rerafu} and \ref{exo:rerafu2}). 
\end{remark}

\section*{Exercises}\markright{Exercises}
\addcontentsline{toc}{section}{Exercises}
\setcounter{theorem}{0}

\begin{exercise}
\label{exo:201011a}
Prove Proposition~\ref{p:ofield}\ref{p:ofieldvii}.
\end{exercise}
\begin{solution}
Suppose that $x\leq y$ and that $u>0$. 
We must show that $ux\leq uy$.
First, $x\leq y$ means $y\geq x$, which in turn
is [$y>x$ or $y=x$]. 
\emph{Case~1:} $y=x$. 
Then $uy=ux$. 
It follows that $uy\geq ux$, as claimed.
\emph{Case~2:} $y>x$. Then $y-x>0$.
Since $u>0$, \textbf{O3} yields 
$(y-x)u>0$. On the other hand,
\begin{align*}
(y-x)u &= u(y-x) \tag{by \textbf{M2}}\\
&=u(y+(-x)) \tag{def of subtraction}\\
&=uy+u(-x) \tag{by \textbf{D}}\\
&=uy+u((-1)x) \tag{by Proposition~\ref{p:may26:4}\ref{p:may26:4iv}}\\
&=uy+(u(-1))x \tag{by \textbf{M1}}\\
&=uy+((-1)u)x \tag{by \textbf{M2}}\\
&=uy+(-1)(ux) \tag{by \textbf{M1}}\\
&=uy+(-(ux)) \tag{by Proposition~\ref{p:may26:4}\ref{p:may26:4iv}}\\
&=uy-ux \tag{def of subtraction}
\end{align*}
Altogether, $uy-ux>0$. Hence $uy>ux$, thus
$uy \geq ux$ and so $ux\leq uy$.
\end{solution}

\begin{exercise}
\label{exo:201011b}
Prove Proposition~\ref{p:ofield}\ref{p:ofieldviii}.
\end{exercise}
\begin{solution} 
Suppose that $x\leq y$ and that $u<0$. 
We must show that $ux\geq uy$.

First, $x\leq y$ means $y\geq x$, which in turn
is [$y>x$ or $y=x$]. 

\emph{Case~1:} $y=x$. 
Then $uy=ux$. 
It follows that $ux\geq uy$, as claimed.

\emph{Case~2:} $y>x$. Then $y-x>0$.
Since $u<0$, we have $-u>0$ and \textbf{O3} yields 
$(y-x)(-u)>0$. On the other hand,
\begin{align*}
(y-x)(-u) &= (-u)(y-x) \tag{by \textbf{M2}}\\
&=(-u)(y+(-x)) \tag{def of subtraction}\\
&=(-u)y+(-u)(-x) \tag{by \textbf{D}}\\
&=((-1)u)y+ux \tag{by Proposition~\ref{p:may26:4}\ref{p:may26:4iv}\&\ref{p:may26:4v}}\\
&=(-1)(uy)+ux \tag{by \textbf{M1}}\\
&=-(uy)+ux \tag{by Proposition~\ref{p:may26:4}\ref{p:may26:4iv}}\\
&=ux+(-(uy)) \tag{by \textbf{A2}}\\
&=ux-uy \tag{def of subtraction}
\end{align*}
Altogether, $ux-uy>0$. Hence $ux>uy$, thus
$ux \geq uy$.
\end{solution}

\begin{exercise}
\label{exo:201011c}
Prove Proposition~\ref{p:ofield}\ref{p:ofieldix}.
\end{exercise}
\begin{solution} 
Suppose $0\leq x < y$ and $0\leq u < v$.
We must show that $ux<vy$.
Note that $y>0$ and $v>0$, 
so $vy > 0$ by \textbf{O3}.

\emph{Case~1:} $x=0$ or $u=0$.
Then $ux=0$ by Proposition~\ref{p:may26:4}\ref{p:may26:4ii} and we are done
in this case.

\emph{Case~2:} $x>0$ and $u>0$.
By Proposition~\ref{p:ofield}\ref{p:ofieldv},
$ux<uy$. 
Now $u<v$ and $y>0$, so again by 
Proposition~\ref{p:ofield}\ref{p:ofieldv},
$yu<yv$.

Altogether, 
$ux<uy = yu <yv=vy$ and so $ux<vy$. 
\end{solution}

\begin{exercise}
\label{exo:130923b}
Complete the proof of Corollary~\ref{c:130923a}
by showing the missing induction step. 
(You do not need to verify the base case.)
\end{exercise}
\begin{solution}
The base case was already dealt with in the proof of
Corollary~\ref{c:130923a}. 
Denote the statement in question by $S(n)$. 

\textbf{Inductive Step:}
Assume that $n\geq 2$ is an integer such that the statement is true.
Now let $x_1,\ldots,x_n,x_{n+1}$ be nonnegative real numbers.

If all of them are equal to $0$, then 
$(x_1+\cdots + x_n)+x_{n+1} = 0 + 0 = 0$,
where we used the inductive hypothesis for the first equality.
Let us consider the converse, i.e., we assume that 
$x_1+\cdots + x_{n+1} = 0$, which we rewrite as 
\begin{equation*}
0 = (x_1+x_2+\cdots + x_n) + x_{n+1}.
\end{equation*}
Set $y_1 := x_1+\cdots + x_n$ and $y_2 := x_{n+1}$.
Both $y_1$ and $y_2$ are nonnegative, and they add up to $0$.
Since we showed that $S(2)$ is true, it follows that
$y_1 = x_1+\cdots +x_n=0$ and $y_2 = x_{n+1} = 0$.
Using the inductive hypothesis again (on $y_1$!),
it follows that $x_1=x_2=\cdots = x_n=x_{n+1}=0$.
We thus showed that $S(n+1)$ is true as well.

Therefore, by the Principle of Mathematical Induction, the entire result
follows. 
\end{solution}

\begin{exercise}
\label{exo:130923d}
Prove Corollary~\ref{c:130923c}.
\emph{Hint}: You may use Corollary~\ref{c:130923a}.
\end{exercise}
\begin{solution}
The implication ``$\Leftarrow$'' is clear (or very easily proved by
induction, using \textbf{A3}). 

It thus remains to prove the implication ``$\Rightarrow$''.
So suppose that $x_1^2+x_2^2+\cdots + x_n^2=0$.
Note that each $x_i^2$ is nonnegative:
indeed, if $x_i = 0$, then $x_i^2 = 0^2 = 0$ by 
Proposition~\ref{p:may26:4}\ref{p:may26:4ii};
otherwise, $x_i\neq 0$ in which case $x_i^2 > 0$ by 
Proposition~\ref{p:ofield}\ref{p:ofieldx}. 
It therefore follows from Corollary~\ref{c:130923a} (applied
to the nonnegative numbers $x_1^2,\ldots,x_n^2$) that 
$x_1^2 = \cdots = x_n^2 = 0$. Furthermore, 
Proposition~\ref{p:may26:4}\ref{p:may26:4iii} now yields that
each $x_i = 0$, as announced. 
\end{solution}

\begin{exercise}
\label{exo:triangle}
Prove Lemma~\ref{l:generaltriangle}. 
\end{exercise}
\begin{solution}
By Mathematical Induction on $n$. 

\textbf{Base Case}: $n=2$.
Then we need to prove that $|x_1+x_2| \leq |x_1|+|x_2|$, which indeed is
true by the classical triangle inequality (Theorem~\ref{t:triangle}). 

\textbf{Inductive Step}: Assume that $\nnn$ is such that $n\geq 2$ 
and \eqref{e:gentriangle} holds. We must show that \eqref{e:gentriangle} also holds with
$n$ replaced by $n+1$, i.e.,
\begin{equation}
\label{e:hesucks}
\left|\sum_{k=1}^{n+1} x_k\right|
\stackrel{?}{\leq} \sum_{k=1}^{n+1} |x_k|.
\end{equation}
Indeed, using the triangle inequality (Theorem~\ref{t:triangle}) 
in \eqref{e:hesucksb}
and the \textbf{induction hypothesis} in \eqref{e:hesucksc}, 
\begin{subequations}
\begin{align}
\left|\sum_{k=1}^{n+1} x_k\right|
&= \left|\left(\sum_{k=1}^{n} x_k\right) + x_{n+1}\right| \\
&\leq \left|\sum_{k=1}^{n} x_k\right| + |x_{n+1}| \label{e:hesucksb}\\
&\leq \left(\sum_{k=1}^{n}|x_k|\right) + |x_{n+1}|\label{e:hesucksc}\\
&= \sum_{k=1}^{n+1}|x_k|,
\end{align}
\end{subequations}
we verified \eqref{e:hesucks}, as required.
Therefore, by the Principle of Mathematical Induction, 
Lemma~\ref{l:generaltriangle} is verified. 
\end{solution}

\begin{exercise}
\label{exo:maxmin}
Prove Lemma~\ref{l:maxmin}.
\end{exercise}
\begin{solution}
\emph{Case~1}: $x\geq y$.\\
On the one hand,
$$\max\{x,y\} = x \quad\text{and}\quad\min\{x,y\} = y.$$
On the other hand, 
$x-y\geq 0$ and hence $|x-y|=x-y$.
It follows that 
$$\frac{x+y+|x-y|}{2} = \frac{x+y+(x-y)}{2} = \frac{2x}{2} = x$$
and
$$\frac{x+y-|x-y|}{2} = \frac{x+y-(x-y)}{2} = \frac{2y}{2} = y.$$
Altogether, the result follows in this case. 

\emph{Case~2}: $x< y$.\\
On the one hand,
$$\max\{x,y\} = y \quad\text{and}\quad\min\{x,y\} = x.$$
On the other hand, 
$x-y<0$ and hence $|x-y|=-(x-y)=y-x$.
It follows that 
$$\frac{x+y+|x-y|}{2} = \frac{x+y+(y-x)}{2} = \frac{2y}{2} = y$$
and
$$\frac{x+y-|x-y|}{2} = \frac{x+y-(y-x)}{2} = \frac{2x}{2} = x.$$
Altogether, the result follows in this case. 
\end{solution}

\begin{exercise}
\label{exo:uncleshawn}
Let $n\in\NN\smallsetminus\{3\}$. 
Show that
\begin{equation*}
n^2 \leq 2^n
\end{equation*}
\emph{Hint}: Use induction for $n\geq 4$.
\end{exercise}
\begin{solution}
The statement is clearly true for 
$n\in\{0,1,2,4\}$ since
$0^2 = 0 \leq 1 = 2^0$,
$1^2 = 1 \leq 2 = 2^1$,
$2^2 = 4 \leq 4 = 2^2$,
and 
$4^2 = 16 \leq 16 = 2^4$. 

Now assume that the inequality is true for some integer $n\geq 4$.
Then $4n \leq n^2$ and hence
\begin{align*}
(n+1)^2 &= n^2 + 2n+1  
\leq n^2 + \frac{n^2}{2} + 1  \quad\text{[since $4n\leq n^2$]}\\
&= n^2\big(1+\tfrac{1}{2}\big) + 1
\leq \frac{3}{2}2^n + 1 \quad\text{[induction hypothesis]}\\
&< \frac{3}{2}2^n + \frac{1}{2}2^n
= 2\cdot 2^n 
= 2^{n+1}.
\end{align*}
Therefore, the result is proven using the Principle
of Mathematical Induction.

\emph{Alternatively},
we could argue
that 
$(n+1)^2 = n^2+2n+1
= n^2(1+2\cdot\tfrac{1}{n}+\tfrac{1}{n^2})
\leq 2^n(1+2\cdot\tfrac{1}{n}+\tfrac{1}{n^2})
$.
We would be done if the last term is 
less than or equal to $2^{n+1} = 2^n\cdot 2$,
so it suffices to show that 
$1+2\cdot\tfrac{1}{n}+\tfrac{1}{n^2} \leq 2$, 
i.e., that 
$2\cdot\tfrac{1}{n}+\tfrac{1}{n^2} \leq 1$.
But this is OK because $n\geq 4$ implies 
$2\cdot\tfrac{1}{n} \leq 2\cdot\tfrac{1}{4} = \tfrac{1}{2}$
and 
$\tfrac{1}{n^2} \leq \tfrac{1}{4^2} = \tfrac{1}{16} \leq \tfrac{1}{2}$. 
\end{solution}

\begin{exercise}
\label{exo:101011:2}
Show that
\begin{equation*}
2^n < n!
\end{equation*}
for every integer $n\geq 4$. 
\end{exercise}
\begin{solution}
We use induction.

When $n=4$, we have
$2^4 = 16 < 24 = 4!$.

Now assume that for some integer $n\geq 4$,
we have $2^n < n!$. 
Then
\begin{equation*}
2^{n+1} = 2\cdot 2^n
< 2\cdot n!
<(n+1)\cdot n! 
=(n+1)!
\end{equation*}
and the result thus follows from the Principle
of Mathematical Induction.
\end{solution}

\begin{exercise}
Let $x>0$. 
Show that
\begin{equation*}
x + \frac{1}{x} \geq 2,
\end{equation*}
and that equality holds if and only if 
$x=1$.
\end{exercise}
\begin{solution}
Using Proposition~\ref{p:ofield}\ref{p:ofieldx},
Proposition~\ref{p:may26:4}\ref{p:may26:4ii}, 
and Proposition~\ref{p:ofield}\ref{p:ofieldvii},
we have
\begin{subequations}
\begin{align*}
(x-1)^2 \geq 0 &\Leftrightarrow
x^2 - 2x + 1 \geq 0
\Leftrightarrow
x^2 +1 \geq 2x
\Leftrightarrow
x + \frac{1}{x} \geq 2.
\end{align*}
\end{subequations}
The last inequality is thus an equality if and only if
$(x-1)^2=0$
$\Leftrightarrow$
$x-1=0$ (by Proposition~\ref{p:may26:4}\ref{p:may26:1iii})
$\Leftrightarrow$
$x=1$.
\end{solution}

\begin{exercise}
\label{exo:101011:1}
Let $n\in\NN\smallsetminus\{0\}$ and let $k\in\NN$.
Show that 
\begin{equation*}
{n \choose k}\frac{1}{n^k} \leq \frac{1}{k!}.
\end{equation*}
\end{exercise}
\begin{solution}
If $k\geq n+1$, then the inequality is true because
${n \choose k} = 0<\tfrac{1}{k!}$. 
Thus, we assume that $k\in\{0,1,\ldots,n\}$. 
Then
$\frac{n!}{(n-k)!} = (n)(n-1)\cdots (n-k+1)$
and so
\begin{equation*}
\frac{1}{n^k}\frac{n!}{(n-k)!} = 
\frac{(n)(n-1)\cdots (n-k+1)}{\underbrace{(n)(n)\cdots(n)}_{k
\text{ factors}}}
= \frac{n}{n}\frac{n-1}{n}\cdots \frac{n-k+1}{n}
\leq 1.
\end{equation*}
Hence, by Lemma~\ref{l:choose}\ref{l:choose:ii},
\begin{equation*}
{n \choose k}\frac{1}{n^k}
= \frac{1}{n^k} \frac{n!}{(n-k)!}\frac{1}{k!}
\leq \frac{1}{k!},
\end{equation*}
as claimed.
\end{solution}

\begin{exercise}
\label{exo:101011:3}
Let $n\geq 1$ be an integer. Show that 
\begin{equation*}
\Big( 1+ \frac{1}{n}\Big)^n \leq \sum_{k=0}^{n} \frac{1}{k!}.
\end{equation*}
\emph{Hint:} Exercise~\ref{exo:101011:1}. 
\end{exercise}
\begin{solution}
Indeed, using the Binomial Theorem (Theorem~\ref{t:binomial})
and Exercise~\ref{exo:101011:1}, we obtain
\begin{equation*}
\Big( 1+ \frac{1}{n}\Big)^n 
= \sum_{k=0}^n {n \choose k}\frac{1}{n^k}
\leq 
\sum_{k=0}^n \frac{1}{k!}
\end{equation*}
as required.
\end{solution}

\begin{exercise}
\label{exo:101011:4}
Let $n\geq 4$ be an integer. 
Use Exercise~\ref{exo:101011:2} and Theorem~\ref{t:geosum}
to show that
\begin{equation*}
\sum_{k=0}^n \frac{1}{k!} < 3.
\end{equation*}
\end{exercise}
\begin{solution}
By Exercise~\ref{exo:101011:2}, we have 
\begin{equation*}
\frac{1}{k!} < \left(\frac{1}{2}\right)^k = 2^{-k}
\qquad
\text{for all integers $k\geq 4$.}
\end{equation*}
Therefore, using also the formula for the geometric sum
(Theorem~\ref{t:geosum}), we obtain
\begin{align*}
\sum_{k=0}^n \frac{1}{k!}
&= \sum_{k=0}^3 \frac{1}{k!} +  \sum_{k=4}^n \frac{1}{k!}
= \big( 1 + 1 + \tfrac{1}{2}+\tfrac{1}{6}\big) 
+  \sum_{k=4}^n \frac{1}{k!}
= \big(2+\tfrac{2}{3}\big) + \sum_{k=4}^n \frac{1}{k!}\\
&<  \big(2+\tfrac{2}{3}\big) + \sum_{k=4}^n 2^{-k}
=  \big(2+\tfrac{2}{3}\big) + 2^{-4}\sum_{k=4}^n 2^{-k+4}
=  \big(2+\tfrac{2}{3}\big) + 2^{-4}\sum_{k=0}^{n-4} 2^{-(k+4)+4}\\
&=  \big(2+\tfrac{2}{3}\big) + 2^{-4}\sum_{k=0}^{n-4} (1/2)^k
=  \big(2+\tfrac{2}{3}\big) + 2^{-4}\frac{1-(1/2)^{n-3}}{1-(1/2)}\\
&<  \big(2+\tfrac{2}{3}\big) + 2^{-4}\frac{1-0}{1/2}
=  2+\tfrac{2}{3} + \tfrac{1}{8}
=  2+\tfrac{19}{24} < 3. 
\end{align*}
\end{solution}

\begin{exercise}
Let $n \geq 4$ be an integer. 
Use Exercises \ref{exo:101011:3} and 
\ref{exo:101011:4} to show that
\begin{equation*}
\left( \frac{n}{3}\right)^n \leq \frac{n!}{3}\,.
\end{equation*}
\end{exercise}
\begin{solution}
By induction on $n$.

When $n=1$, we indeed have
$(1/3)^1 = 1/3 \leq 1/3 = 1!/3$.

Now assume the inequality is true for some integer $n\geq 1$.
Then, using the two exercises and the induction hypothesis, 
we have
\begin{align*}
\left( \frac{n+1}{3}\right)^{n+1}
&= \frac{n+1}{3}  \left(\frac{n+1}{3}\right)^n 
= \frac{n+1}{3}  \left(\frac{n\big(1+1/n\big)}{3}\right)^n 
= \frac{n+1}{3}  \left(\frac{n}{3}\right)^n
\left(1 + \frac{1}{n}\right)^n\\
&\leq \frac{n+1}{3}  \left(\frac{n}{3}\right)^n
\sum_{k=0}^{n} \frac{1}{k!}
< \frac{n+1}{3}  \left(\frac{n}{3}\right)^n \cdot 3
= (n+1) \left(\frac{n}{3}\right)^n 
\leq (n+1) \frac{n!}{3}
= \frac{(n+1)!}{3}.
\end{align*}
The result thus follows from the Principle of Mathematical Induction.
\end{solution}

\begin{exercise}
Let $n \geq 3$ be an integer. 
Use the Binomial Theorem (and mathematical induction) to show that 
\begin{equation*}
n^n > (n+1)!\,.
\end{equation*}
\end{exercise}
\begin{solution}
By induction on $n$.

When $n=3$, we indeed have
$3^3 = 27 > 24 = 4!$. 

Now assume the inequality is true for some integer $n\geq 3$.
Then, using the Binomial Theorem and the induction hypothesis, 
we have
\begin{align*}
(n+1)^{n+1}
&= \sum_{k=0}^{n+1} {n+1 \choose k} n^{n+1-k}
= n^{n+1} + (n+1)n^{n} + \sum_{k=2}^{n+1} {n+1 \choose k} n^{n+1-k}\\
&= n^{n} \bigg(n + (n+1)+ \sum_{k=2}^{n+1} {n+1 \choose k}
n^{1-k}\bigg)\\
&> n^n(2n+1)
> n^n(n+2)
> (n+1)!(n+2)
= (n+2)!. 
\end{align*}
The result thus follows from the Principle of Mathematical Induction.
\end{solution}

\begin{exercise}[a non-archimedean ordered field]
\label{exo:rerafu2}
Consider the field of real rational functions $\RR(x)$
(see Exercise~\ref{exo:rerafu}) and \emph{define} 
$$ f(x) = \frac{ a_m x^m+a_{m-1} x^{m-1} + \cdots + a_1x + a_0}{b_n x^n +
b_{n-1}x^{n-1} + \cdots + b_1x+b_0}$$
to be positive when the $a_mb_n>0$. 
(This does not mean that the function $f$ is everywhere positive!)
Show that with this definition, $\RR(x)$ becomes an ordered field,
i.e., check \textbf{O1}--\textbf{O3}.
Furthermore, show that for every $\nnn$, $x>n$ and hence $\RR(x)$ is
\emph{not} archimedean. 
\end{exercise}
\begin{solution} 
Suppose $f$ has the representation as given, where we 
also abbreviate the numerator and denominator polynomials: 
$$ f(x) = \frac{ a_m x^m+a_{m-1} x^{m-1} + \cdots + a_1x + a_0}{b_n x^n +
b_{n-1}x^{n-1} + \cdots + b_1x+b_0} =: \frac{a(x)}{b(x)}.$$
Then $b_n\neq 0$. 
Hence $a_mb_n$ is defined and we have a positive element 
if $a_mb_n>0$, a negative element if $a_mb_n<0$,
and $0$ is $a_m=0$, i.e., $f$ is the zero rational function.
This proves \textbf{O1}.

Now assume that $f$ is positive, i.e., $a_mb_n>0$  and that 
$$ g(x) = \frac{ c_p x^p+c_{p-1} x^{p-1} + \cdots + c_1x + c_0}{d_q x^q +
d_{q-1}x^{q-1} + \cdots + d_1x+d_0} =: \frac{c(x)}{d(x)}$$
is positive as well, i.e., $c_pd_q>0$. 
Then 
\begin{align*}
(f+g)(x)
&= \frac{a(x)}{b(x)} + \frac{c(x)}{d(x)}
= \frac{a(x)d(x)}{b(x)d(x)} + \frac{c(x)b(x)}{d(x)b(x)}
= \frac{a(x)d(x)+b(x)c(x)}{b(x)d(x)}\\
&= \frac{(a_mx^m+ \cdots)(d_qx^q+\cdots) + (b_nx^n+\cdots)(c_px^p+\cdots)}%
{(b_nx^n+\cdots)(d_qx^q+\cdots)}\\
&= \frac{(a_md_qx^{m+q}+\cdots) + (b_nc_px^{n+p}+\cdots)}{b_nd_qx^{n+q}+\cdots}.
\end{align*}

\emph{Case~1:} $b_n>0$ and $d_q>0$.
Then $a_m>0$ and $c_p>0$. Hence 
$a_md_q>0$, $b_nc_p>0$, and $b_nd_q>0$, which shows $f+g$ is positive.

\emph{Case~2:} $b_n<0$ and $d_q<0$.
Then $a_m<0$ and $c_p<0$. Hence 
$a_md_q>0$, $b_nc_p>0$, and $b_nd_q>0$, which shows $f+g$ is positive.

\emph{Case~3:} $b_n>0$ and $d_q<0$.
Then $a_m>0$ and $c_p<0$. Hence 
$a_md_q<0$, $b_nc_p<0$, and $b_nd_q<0$, which shows $f+g$ is positive.

\emph{Case~4:} $b_n<0$ and $d_q>0$.
This is similar to \emph{Case~3}.

Altogether, \textbf{O2} holds. 

Now consider 
\begin{align*}
(f\cdot g)(x)
&= \frac{a(x)}{b(x)} \cdot \frac{c(x)}{d(x)}
= \frac{a(x)c(x)}{b(x)d(x)}
= \frac{(a_mx^m+ \cdots)(c_px^p+\cdots)}%
{(b_nx^n+\cdots)(d_qx^q+\cdots)}
= \frac{a_mc_px^{m+p}+\cdots}{b_nd_qx^{n+q}+\cdots}.
\end{align*}

\emph{Case~1:} $b_n>0$ and $d_q>0$.
Then $a_m>0$ and $c_p>0$. Hence 
$a_mc_p>0$, and $b_nd_q>0$, which shows $f\cdot g$ is positive.

\emph{Case~2:} $b_n<0$ and $d_q<0$.
Then $a_m<0$ and $c_p<0$. Hence 
$a_mc_p>0$, and $b_nd_q>0$, which shows $f\cdot g$ is positive.

\emph{Case~3:} $b_n>0$ and $d_q<0$.
Then $a_m>0$ and $c_p<0$. Hence 
$a_mc_p<0$, and $b_nd_q<0$, which shows $f\cdot g$ is positive.

\emph{Case~4:} $b_n<0$ and $d_q>0$.
This is similar to \emph{Case~3}.

Altogether, we verified \textbf{O3}. 

Finally, consider 
$$
f(x) = x = \frac{1\cdot x^1 + 0}{1}
\quad\text{and}\quad 
g(x) = n = \frac{n}{1}.
$$
Then 
$$
(f-g)(x) = \frac{1\cdot x^1 + 0}{1} - \frac{n}{1}
= \frac{1\cdot x^1 -n}{1}
$$
and since $1\cdot 1= 1>0$, it follows that $f-g>0$, 
i.e., $f>g$. 

In view of Lemma~\ref{l:entier},
we see that $\RR(x)$ cannot have the archimedean property.
\end{solution}

\begin{exercise}[Cauchy--Schwarz inequality]
\label{exo:CS}
\index{Cauchy--Schwarz inequality}
Let $x_1,x_2,y_1,y_2$ be numbers in an ordered field $\FF$.
Show that 
$$(x_1y_1+x_2y_2)^2 \leq (x_1^2+x_2^2)(y_1^2+y_2^2).$$
\emph{Hint:} Exercise~\ref{exo:151007a}.
\end{exercise}
\begin{solution}
Using Exercise~\ref{exo:151007a}, we have for 
all $x,y,u,v$: 
$$ (xu-yv)^2+(xv+yu)^2 = (x^2+y^2)(u^2+v^2).$$
Setting $x=x_1,v=y_1,y=x_2,u=y_2$ we thus learn that 
$$ (x_1y_2-x_2y_1)^2+(x_1y_1+x_2y_2)^2 = (x_1^2+x_2^2)(y_2^2+y_1^2).$$
On the other hand, squares in an ordered field are nonnegative 
because of Proposition~\ref{p:ofield}\ref{p:ofieldx}, 
$0^2=0\geq 0$ and Remark~\ref{r:191010}\ref{r:191010i}. 
Altogether, we deduce that 
$$(x_1y_1+x_2y_2)^2 \leq  (x_1y_2-x_2y_1)^2+(x_1y_1+x_2y_2)^2
= (x_1^2+x_2^2)(y_1^2+y_2^2),$$
as claimed. 
\end{solution}

\begin{exercise}
Let $x$ and $y$ be numbers in $\RR$ (or, in fact, in any ordered
field) such that $x<y$.
Show that there exists a number $m$ such that
$x<m<y$. 
\end{exercise}
\begin{solution}
Recall that $1>0$ 
(Proposition~\ref{p:ofield}\ref{p:ofieldxi}).
Hence $2=1+1>0$ by \textbf{O2}.
It follows that $1/2>0$ 
(Proposition~\ref{p:ofield}\ref{p:ofieldxii}).
Now set
$$ m := (x+y)/2.$$
We have the following using Proposition~\ref{p:ofield}:
$x<y$
$\Rightarrow$
$2x=x+x < x+y$
$\Rightarrow$
$x = (2x)/2 < (x+y)/2 = m$.
The proof that $m<y$ is similar
(start by adding $y$ instead of $x$ to the inequality $x<y$). 
\end{solution}

\begin{exercise}
Prove or disprove: 
\begin{quotation}
``If $x$ and $y$ are in $\RR$, then there exists $M>0$ such that 
$|x|+|y|\leq M\cdot |x+y|$.''
\end{quotation}
\end{exercise}
\begin{solution}
The statement is false.
Suppose to the contrary that the statement is true. 
Let $x=1$ and $y=-1$. 
Because we assume the statement is true, 
there exists $M>0$ as in the statement. 
Then 
$$2 = 1+ 1 = |x|+y| \leq M\cdot|x+y| = M\cdot |1+(-1)| = M\cdot 0 = 0,$$
which is a contradiction because $2>0$. 
\end{solution}

\begin{exercise}
Prove or disprove: 
\begin{quotation}
``If $x$ and $y$ are in $\RR$, then 
$|x|^2+|y|^2\geq |x+y|^2$.''
\end{quotation}
\end{exercise}
\begin{solution}
The statement is false:
Consider $x=y=1$.
Then LHS is equal to $1^2+1^2=2$
while RHS is equal to $(1+1)^2 = 4$. 
\end{solution}

\begin{exercise}
Prove or disprove: 
\begin{quotation}
``There exists $M>0$ such that if $x$ and $y$ are in $\RR$, then 
$|x|^2+|y|^2\geq M\cdot |x+y|^2$.''
\end{quotation}
\end{exercise}
\begin{solution}
The statement is true, e.g., with $M=1/2$ (or smaller). 
Let $x$ and $y$ be in $\RR$.
We must show that 
$$|x|^2 + |y|^2 \stackrel{?}{\geq} \thalb |x+y|^2.$$
Note that $|z|\in\{z,-z\}$ and so $|z|^2 \in \{z^2,(-z)^2\}$
and so $|z|^2 = z^2$.
Hence, we must show that 
$x^2 + y^2 \stackrel{?}{\geq} \thalb (x+y)^2$
or equivalently, 
$$2x^2+2y^2 \stackrel{?}{\geq} (x+y)^2.$$
Expanding the RHS gives 
$$2x^2+2y^2 \stackrel{?}{\geq} x^2 +y^2+2xy$$
or $x^2+y^2-2xy \stackrel{?}{\geq} 0$.
But this is true because $x^2+y^2-2xy = (x-y)^2$ is a square,
hence nonnegative!
\end{solution}

\begin{exercise}
\label{exo:200724a}
Show that for every real number $x>-1$, we have 
\begin{equation*}
1-x \leq \frac{1}{1+x}.
\end{equation*}
\end{exercise}
\begin{solution}
We know that squares are nonnegative 
(use Proposition~\ref{p:ofield}\ref{p:ofieldx}),
so $x^2 \geq 0$ and hence 
$1-x^2\leq 1$.
On the other hand, $1-x^2=(1+x)(1-x)$.
So altogether
\begin{equation}
\label{exo:200724ae1}
(1+x)(1-x)\leq 1. 
\end{equation}
Because $x>-1$, we have $1+x>0$. 
Dividing \eqref{exo:200724ae1} by $1+x$
and recalling Proposition~\ref{p:ofield}\ref{p:ofieldvii}, we obtain the result. 
\end{solution}

\begin{exercise}
\label{exo:200724b}
Show that for all real numbers $x$ and $y$ such that $-1<x<y$ we have 
\begin{equation*}
\frac{x}{1+x} < \frac{y}{1+y};
\end{equation*}
in other words, the function $z\mapsto z/(1+z)$ 
is strictly increasing on $\left]-1,+\infty\right[$. 
\end{exercise}
\begin{solution}
Assume that $-1<x<y$.
Then $0<1+x<1+y$.
Taking the inverse (Proposition~\ref{p:ofield}\ref{p:ofieldxiv}), we obtain 
$(1+x)^{-1}>(1+y)^{-1}$. 
Hence 
$1-(1+x)^{-1}<1-(1+y)^{-1}$ and so 
\begin{align*}
\frac{x}{1+x} &=  \frac{(1+x)-1}{1+x} =1-(1+x)^{-1}
<1-(1+y)^{-1}
= \frac{(1+y)-1}{1+y}
= \frac{y}{1+y}
\end{align*}
as claimed. 

Alternatively, we can (but don't have to, see above!) use Calculus:
\begin{equation*}
    \Big( \frac{z}{1+z}\Big)' = \frac{1}{(1+z)^2}>0
\end{equation*}
for $z>-1$ and thus the function is strictly increasing. 
\end{solution}

\begin{exercise}
\label{exo:200724c}
Suppose that $0\leq x \leq 1$,
that $0<\alpha \leq 1$,
that $0\leq x_0\leq 1$,
that $n\in\NN$, and that 
\begin{equation}
\label{e:exo:200724c1}
x \leq \frac{x_0}{nx_0+\alpha^{-n}}.
\end{equation}
Let 
\begin{equation}
\label{e:201020a}
    0 \leq y \leq \alpha(1-x)x.
\end{equation}
Show that 
\begin{equation}
y \leq \frac{x_0}{(n+1)x_0+\alpha^{-(n+1)}}.
\end{equation}
\emph{Hint:}
Exercises~\ref{exo:200724a} and \ref{exo:200724b}.
\end{exercise}
\begin{solution}
We start by pointing out that $y$ can actually be chosen
because $\alpha> 0$, $x\geq 0$, $1-x\geq 0$
imply that $\alpha(1-x)x\geq 0$.
Therefore, 
\begin{align}
y
&\leq 
\alpha(1-x)x = (\alpha x)(1-x)\tag{by \eqref{e:201020a}}\\
&\leq \alpha \frac{x}{1+x} \tag{by Exercise~\ref{exo:200724a}}\\
&\leq \alpha 
\frac{\frac{x_0}{nx_0+\alpha^{-n}}}{1+\frac{x_0}{nx_0+\alpha^{-n}}} \tag{by \eqref{e:exo:200724c1} and Exercise~\ref{exo:200724b}}\\
&= \frac{1}{\frac{1}{\alpha}} \frac{x_0}{nx_0+\alpha^{-n}+x_0}\notag\\
&= \frac{x_0}{\tfrac{1}{\alpha}(n+1)x_0+\alpha^{-(n+1)}}\notag\\
&\leq \frac{x_0}{(n+1)x_0+\alpha^{-(n+1)}} \tag{because $1/\alpha\geq 1$}
\end{align}
as claimed. 
\end{solution}

\begin{exercise}
\label{exo:200724d}
Show that if $x\in\RR$ satisfies $|x|\geq 4$,
then $-x^2+x \leq -3|x|$.
\end{exercise}
\begin{solution}
We have 
$|x|\geq 4$
$\Leftrightarrow$
$-|x|\leq -4$
$\Leftrightarrow$
$-|x|\cdot|x|\leq -4|x|$
$\Leftrightarrow$
$-x^2\leq -4|x|$
$\Rightarrow$
$-x^2 \leq -x-3|x|$
$\Leftrightarrow$
$-x^2+x\leq -3|x|$.
\end{solution}

\begin{exercise}
Show that if $x$ and $y$ are two nonzero real numbers, then 
\begin{equation*}
    \frac{1}{2}\Big|\frac{x}{|x|}+\frac{y}{|y|} \Big| 
    \leq \frac{|x+y|}{|x|+|y|}.
\end{equation*}
\end{exercise}
\begin{solution}
\emph{Case~1:} $x>0$ and $y<0$.\\ 
Then $x/|x|=1$ while $y/|y|=-1$.
Hence the LHS $=0$ and we are done.

\emph{Case~2:} $x<0$ and $y>0$. \\
This is similar to \emph{Case~1}.

\emph{Case~3:} $x>0$ and $y>0$.\\
Then $x/|x|=y/|y|=1$ and the LHS 
$=\thalb|1+1|=1$. Luckily,
the RHS $=(x+y)/(x+y) = 1$
\end{solution}

\begin{exercise}[YOU be the marker!] 
Consider the following statement 
\begin{equation*}
\text{``If $x$ and $y$ are in $\RR$, then $|x+y|=|x|+|y|$.''}
\end{equation*}
and the following ``proof'':
\begin{quotation}
Let $x$ and $y$ be in $\RR$.\\
\emph{Case~1}: $x\geq 0$ and $y\geq 0$.\\
Then $x+y\geq 0$ and hence $|x+y|=x+y=|x|+|y|$.\\
\emph{Case~2}: $x< 0$ and $y< 0$.\\
Then $-x>0$ and $-y>0$, hence $-(x+y)=(-x)+(-y) > 0$.\\
Thus $|x+y| = -(x+y) = (-x)+(-y) = |x| + |y|$.\\
Altogether, we verified the statement. 
\end{quotation}
Why is this proof wrong?
\end{exercise}
\begin{proof}
Everything written down is mathematically sound except 
for the last sentence: We are not done because
there are missing cases that were not considered, 
e.g., $x>0$ and $y<0$. 
(And the result is wrong: consider $x=1$ and $y=-1$.)
\end{proof}

\begin{exercise}[TRUE or FALSE?]
Mark each of the following statements as either true or false. 
Briefly justify your answer.
\begin{enumerate}
\item ``There exists a field that cannot be ordered.''
\item ``In an ordered field, the sum of two negative numbers is negative.''
\item ``In an ordered field, the product of two negative numbers is negative.''
\item ``If $x\in\RR$, then $(1+x)^3 \geq 1+3x$.''
\end{enumerate}
\end{exercise}
\begin{solution}
(i): TRUE: The field with two elements, $\FF_2$, cannot be ordered 
because $1+1=0$. 

(ii): TRUE: [$x<0$ and $y<0$]
$\Leftrightarrow$ [$-x>0$ and $-y>0$]
$\Rightarrow$ $-(x+y)= (-x)+(-y)>0$
$\Rightarrow$ $x+y<0$.

(iii): FALSE: 
$-1<0$ but $(-1)(-1) = (1)(1) = 1>0$.

(iv): FALSE: Consider $x=-11$.
Then $(1+x)^3 = (-10)^3 = -1000$
but $1+3(-11)=-32$. 
\end{solution}

\begin{exercise}[TRUE or FALSE?]
Determine whether or not each of the following statements
is either TRUE or FALSE.
If it is TRUE, explain why. 
If it is FALSE, provide a counterexample.
\begin{enumerate}
\item
$|x+y| = |x| + |y|$.
\item 
$|x-y| \leq |x|-|y|$.
\item
$|x-y| \leq |x|+|y|$.
\end{enumerate}
\end{exercise}
\begin{solution}
(i): This is FALSE.
E.g., let $x=1$ and $y=-1$.
Then $|x+y| = |1+(-1)| = |0| = 0$
while $|x|+|y| = |1| + |-1| = 1 + 1 =2$.

(ii): This is FALSE. For instance, consider $x=0$ and $y=1$.
Then $|x-y|=|0-1| = 1$ but $|x|-|y| = |0|-|1| = 0 -1 = -1 $. 

(iii): This is TRUE. The triangle inequality (applied to $x$ and $-y$)
and Proposition~\ref{p:absval}\ref{p:absvaliii+} yield 
$|x-y| = |x+(-y)| \leq |x| + |-y| = |x| + |y|$. 
\end{solution} 
\chapter{Sequences and Limits}

\section{Definitions and Examples}

\begin{definition}[sequence]
A \textbf{sequence} of real numbers is formally a function from $\NN$ to $\RR$,
written as
$(a_n)_\nnn$ or $(a_0,a_1,a_2,\ldots,a_n,\ldots)$. 
If $n_0\in\ZZ$, it is convenient to refer to $(a_n)_{n\geq n_0} =
(a_{n_0},a_{n_0+1},\ldots)$ also as a sequence.\index{Sequence}
\end{definition}

Note that a sequence is different from a set (which is unordered).
Thus, e.g.,
\begin{equation}
(0,1,0,1,0,1,\ldots) \neq (1,0,1,0,1,0,\ldots)
\end{equation}
whereas
\begin{equation}
\{0,1,0,1,0,1,\ldots\} = \{0,1\} = \{1,0\} = \{1,0,1,0,1,0,\ldots\}.
\end{equation}

\begin{example}
\label{ex:seq}
Here are some sequences.
\begin{enumerate}
\item 
\label{ex:seq:i}
If $a_n\equiv a$, for some constant $a\in\RR$, then we obtain 
the constant sequence $(a,a,a,\ldots)$.\\[-2mm]
\item If $a_n=\tfrac{1}{n}$ for every $n\geq 1$, then we obtain the
sequence
$(1,\tfrac{1}{2},\tfrac{1}{3},\ldots,\tfrac{1}{n},\ldots)$.\\[-2mm]
\item If $a_n=(-1)^n$ for every $n\in\NN$, one obtains
$(1,-1,1,-1,\ldots)$. \\[-2mm]
\item If $a_n = 1/3^n$ for every $n\in\NN$, we get
$(1,\tfrac{1}{3},\tfrac{1}{9},\tfrac{1}{27},\ldots)$.\\[-2mm]
\item More generally, let $\alpha\in\RR$ and set $a_n = \alpha^n$ for every $n\in\NN$.
One obtains $(a_n)_\nnn = (1,\alpha,\alpha^2,\ldots,\alpha^n,\ldots)$.
\\[-2mm]
\item If $a_n=\tfrac{n}{n+1}$ for every $\nnn$, we obtain
$(0,\tfrac{1}{2},\tfrac{2}{3},\tfrac{3}{4},\ldots)$.\\[-2mm]
\item 
\label{ex:seq:fibo}
Set $f_0=0$, $f_1=1$, and $f_n=f_{n-1}+f_{n-2}$ for $n\geq 2$.
Then one obtains the \emph{recursively defined} sequence of 
\textbf{Fibonacci numbers}:
$(0,1,1,2,3,5,8,13,21,\ldots)$.\index{Fibonacci numbers} \\[-2mm]
\item 
\label{ex:seq:logseq}
Let $r\in\RR$ and $x_0\in\RR$, 
and $x_{n+1}=r(1-x_n)x_n$ for $n\geq 0$.
This is the (again recursively defined) \textbf{logistic sequence}\index{logistic sequence}.
\end{enumerate}
\end{example}

\begin{definition}[convergence, limit, divergence]
\label{d:limit}
Let $(a_n)_\nnn$ be a sequence of real numbers, and let $\ell\in\RR$.
Then $(a_n)_\nnn$ \textbf{converges} to $\ell$,
written as $\lim_\nnn a_n=\ell$, 
$\lim_{n\to\pinf}a_n=\ell$, $\lim a_n=\ell$, or $a_n\to\ell$,
if\index{Convergent sequence}\index{Convergence}\index{Limit}
\begin{quotation}
\noindent
for every $\varepsilon>0$, there exists $N(\varepsilon)\in\NN$ such that\\
$n\geq N(\varepsilon)$ $\Rightarrow$ $|a_n-\ell|<\varepsilon$ 
(i.e., $|a_n-\ell|<\varepsilon$ for \emph{every} $n\geq N(\varepsilon)$). 
\end{quotation}
If this happens, $\ell$ is called the \textbf{limit} of the sequence
$(a_n)_\nnn$. If $(a_n)_\nnn$ has no limit, then it \textbf{diverges}.
\index{Divergent sequence}\index{Divergence}
\end{definition}

Equivalently, we can say that $(a_n)_\nnn$ converges to $\ell$
if eventually all terms of the sequence lie in
the open interval\footnote{We follow here the European (especially
French) notation of open
intervals $\left]a,b\right[$, which is convenient to distinguish
from a pair $(a,b)$.}
\begin{equation}
\left]\ell-\varepsilon,\ell+\varepsilon\right[
:= \menge{x\in\RR}{|x-\ell|<\varepsilon}
= \menge{x\in\RR}{\ell-\varepsilon < x < \ell+\varepsilon},
\end{equation}
for every $\varepsilon>0$.

Let us now revisit Example~\ref{ex:seq}.

\begin{example}
\label{ex:cseq}
If $a_n\equiv a$ for some constant $a\in\RR$,
then $\lim a_n=a$. Indeed, let $\varepsilon > 0$
and set (e.g.) $N(\varepsilon) = 0$. Then
for every $n\geq N(\varepsilon) = 0$, we have
\begin{equation}
|a_n-a| = |a-a| = |0| = 0 < \varepsilon,
\end{equation}
as required.
\end{example}

In Definition~\ref{d:limit}, we talk about \emph{the} limit of sequence.
The next result shows that this language is justified.

\begin{theorem}[the limit is unique]
A sequence has at most one limit.
\end{theorem}
\begin{proof}
Let $(a_n)_\nnn$ be a sequence.
There is nothing to prove if the sequence does not converge.
Thus, assume that $(a_n)_\nnn$ converges.
We need to show that the limit is unique and
assume to the contrary 
that there are \emph{two} limits $\ell_1$ and $\ell_2$, where 
$\ell_1\neq\ell_2$.

Now take any $0<\varepsilon \leq |\ell_1-\ell_2|$. 
Since $a_n\to\ell_1$, there exists $N_1\in \NN$ such that 
\begin{equation}
n\geq N_1 
\quad\Rightarrow\quad
|a_n-\ell_1|<\thalb \varepsilon.
\end{equation}
And since $a_n\to\ell_2$, there exists $N_2\in \NN$ such that 
\begin{equation}
n\geq N_2 
\quad\Rightarrow\quad
|a_n-\ell_2|<\thalb \varepsilon.
\end{equation}
Thus, for every $n\geq \max\{N_1,N_2\}$, we have
\begin{subequations}
\begin{align}
|\ell_1-\ell_2|
&=|(\ell_1-a_n)-(\ell_2-a_n)|
\leq |\ell_1-a_n| + |\ell_2-a_n|\\
&<\thalb\varepsilon + \thalb\varepsilon
=\varepsilon\\
&\leq |\ell_1-\ell_2|,
\end{align}
\end{subequations}
which is absurd.
\end{proof}

\begin{example}
Let $a_n = \frac{1}{n}$, for every $n\geq 1$,
and let $\varepsilon > 0$.
Now pick\footnote{This is possible because of Corollary~\ref{c:entier}
and here we use the archimedean property of the real numbers.} 
$N(\varepsilon)\in\NN$ such that 
$N(\varepsilon) > \frac{1}{\varepsilon}$.
Then
\begin{equation}
|a_n-0|= \left| \tfrac{1}{n} - 0\right| = \tfrac{1}{n} < \varepsilon,
\quad
\text{for all $n\geq N(\varepsilon)$.}
\end{equation}
Therefore,
\begin{equation}
\label{e:1/n}
\tfrac{1}{n}\to 0.
\end{equation}
\end{example}

\begin{example}
\label{ex:altseq}
The sequence defined by
$a_n = (-1)^n$ diverges.
\end{example}
\begin{proof}
We argue by contradiction.
Suppose that there exists $\ell\in\RR$ such that $a_n\to\ell$.
Then, for $\varepsilon=1$, there exists $N=N(1)\in\NN$ such that
\begin{equation}
|a_n-\ell| < 1,
\quad
\text{for all $n\geq N$.}
\end{equation}
Hence, for all $n\geq N$, the triangle inequality 
(Theorem~\ref{t:triangle}) thus yields
\begin{subequations}
\begin{align}
2 &= |a_{n+1}-a_n| = |(a_{n+1}-\ell) + (\ell-a_n)|\\
&\leq |a_{n+1}-\ell| + |\ell-a_n| 
=|a_{n+1}-\ell| + |a_n-\ell| < 1+ 1= 2,
\end{align}
\end{subequations}
which is absurd.
This is the desired contradiction and therefore
$(a_n)_\nnn$ diverges. 
\end{proof}

\begin{example}
\label{ex:n/(n+1)}
$\displaystyle \frac{n}{n+1} \to 1$.
\end{example}
\begin{proof}
Let $\varepsilon>0$.
Then
\begin{equation}
\left|\frac{n}{n+1}-1\right|
=\frac{1}{n+1} < \frac{1}{n}\leq \varepsilon,
\quad\text{for every } n\geq \frac{1}{\varepsilon}.
\end{equation}
(Hence any integer $N\geq \frac{1}{\varepsilon}$ does the job.)
\end{proof}

\begin{definition}[boundedness]
\label{d:boundedseq}
Let $(a_n)_\nnn$ be a sequence.
Then $(a_n)_\nnn$ is
\textbf{bounded above} if 
there exists $\beta\in\RR$ such that 
$a_n\leq\beta$, for every $\nnn$.
Analogously, $(a_n)_\nnn$ is
\textbf{bounded below} if 
there exists $\alpha\in\RR$ such that 
$\alpha\leq a_n$, for every $\nnn$.
Finally,
$(a_n)_\nnn$ is \textbf{bounded}
if $(a_n)_\nnn$ is both bounded above and below.\index{Bounded
sequence}\index{bounded above (sequence)}\index{bounded below
(sequence)} 
\end{definition}

It follows easily (Exercise~\ref{exo:bdseq}) that
$(a_n)_\nnn$ is bounded if and only if
\begin{equation}
\label{e:bdseq}
\text{there exists $\gamma\geq 0$ such that
$|a_n|\leq\gamma$,
for every $\nnn$.}
\end{equation}

The next result is fundamental, yielding
a necessary condition for convergence.

\begin{theorem}
\label{t:convbound}
Every convergent sequence is bounded.
\end{theorem}
\begin{proof}
Let $(a_n)_\nnn$ be convergent, with limit $\ell$.
Then (for $\varepsilon=1$) there exists $N\in\NN$
such that 
\begin{equation}
|a_n-\ell| < 1,
\quad
\text{for all $n\geq N$.}
\end{equation}
Hence, using Theorem~\ref{t:triangle}, 
\begin{equation}
|a_n| = |(a_n-\ell)+\ell|
\leq |a_n-\ell| + |\ell|
<1 + |\ell|,
\quad
\text{for all $n\geq N$.}
\end{equation}
Therefore,
\begin{equation}
|a_n| \leq \gamma := \max\big\{
|a_0|,|a_1|,\ldots,|a_{N-1}|,1+|\ell|\big\},
\quad
\text{for all $n\in\NN$,}
\end{equation}
i.e., the sequence $(a_n)_\nnn$ is bounded. 
\end{proof}

\begin{remark}
The sequence defined by $a_n = (-1)^n$ is clearly
bounded (since $|a_n| = |(-1)^n| = 1$ for all $n\in\NN$)
but $(a_n)_\nnn$ is \emph{not} convergent (Example~\ref{ex:altseq}).
Therefore, the converse of Theorem~\ref{t:convbound} fails.
\end{remark}

\begin{example}[geometric sequence]
\label{ex:geoseq}\index{Geometric sequence} 
Let $\alpha\in\RR$. 
Then exactly one of the following holds.
\begin{enumerate}
\item 
\label{ex:geoseq:i}
$|\alpha|<1$ and $\alpha^n\to 0$.
\item 
\label{ex:geoseq:ii}
$\alpha=1$ and $\alpha^n= 1\to 1$.
\item 
\label{ex:geoseq:iii}
$\alpha=-1$ and $(\alpha^n)_\nnn=((-1)^n)_\nnn$ diverges.
\item 
\label{ex:geoseq:iv}
$|\alpha|>1$ and $(\alpha^n)_\nnn$ diverges.
\end{enumerate}
\end{example}
\begin{proof}
\ref{ex:geoseq:i}:
If $\alpha=0$, then $\alpha^n = 0^n = 0$ for every $n\geq 1$
and thus $\alpha^n\to 0$ by Example~\ref{ex:cseq}. 
Now assume that $\alpha\neq 0$, and
let $\varepsilon>0$. 
Apply Corollary~\ref{c:100920} (with $z = |\alpha|\in \zeroun$)
to obtain $N\in\NN$ such that $|\alpha|^N<\varepsilon$.
Hence,
\begin{equation}
|\alpha^n - 0| = |\alpha^n| = |\alpha|^n \leq |\alpha|^N <
\varepsilon,
\quad
\text{for every $n\geq N$.}
\end{equation}
Thus, $\alpha^n\to 0$. 

\ref{ex:geoseq:ii}:
Clear from Example~\ref{ex:cseq}.

\ref{ex:geoseq:iii}:
Clear from Example~\ref{ex:altseq}.

\ref{ex:geoseq:iv}:
Suppose to the contrary that $(\alpha^n)_\nnn$ converges.
By Theorem~\ref{t:convbound}, 
$(\alpha^n)_\nnn$ is bounded,
say $|\alpha^n|\leq\gamma$ for all $n\in\NN$ and for
some $\gamma> 0$. 
On the other hand, Theorem~\ref{t:100920} implies that there is some
integer $n\in\NN$ such that $|\alpha^n| = |\alpha|^n>\gamma$,
which is absurd. 
Therefore, $(\alpha^n)_\nnn$ diverges.
\end{proof}

\begin{proposition}[a bridge to calculus]
\label{p:calc1}
Let $a$ and $L$ be two real numbers, and 
let $f\colon \RR\smallsetminus\{a\}\to\RR$. 
Recall that, by definition, 
$$\lim_{x\to a} f(x)= L$$
if and only if 
\begin{quotation}
\noindent
for every $\varepsilon>0$, there exists $\delta>0$ such that\\
$0<|x-a|<\delta$ $\Rightarrow$ $|f(x)-L|<\varepsilon$. 
\end{quotation}
In fact, we have $\lim_{x\to a} f(x)=L$ if and only if 
whenever $(a_n)_\nnn$ is a sequence in $\RR\smallsetminus\{a\}$ with
$a_n\to a$, then $f(a_n)\to L$. 
\end{proposition}
\begin{proof}
Exercise~\ref{exo:calc1}.
\end{proof}

\begin{remark}[quantifier statements and their negation]
Sometimes, it is convenient to write complicated statements
using quantifiers. 
For instance, to express that a sequence $(a_n)_\nnn$ converges
to $\ell$, we can rewrite the definition as
\begin{equation}
\label{e:forall}
(\forall\varepsilon>0)
(\exi N\in\NN)
(\forall n\in\{N,N+1,\ldots\})
\quad
|a_n-\ell|<\varepsilon.
\end{equation}
Note that this is read left to right.
``$\forall$'' is the \emph{for all quantifier}
and ``$\exi$'' is the \emph{there exists quantifier}.
While this looks technical, the advantage is that it is easy
to logically negate such expressions.
E.g., the logical negation of \eqref{e:forall} is
\begin{equation}
(\exi\varepsilon>0)
(\forall N\in\NN)
(\exi n\in\{N,N+1,\ldots\})
\quad
|a_n-\ell|\geq \varepsilon.
\end{equation}
Note that the negation is obtained by swapping 
$\forall$ with $\exi$ throughout, and by negating the very
right-most statement. 
More generally, let $Q_1,\ldots,Q_n$ be quantifier statements
and let $p$ be a statement.
Then the logical negation of
\begin{equation}
(Q_n)(Q_{n-1})\cdots (Q_1)\quad p
\end{equation}
is 
\begin{equation}
(Q_n^*)(Q_{n-1}^*)\cdots (Q_1^*)\quad (\lnot p),
\end{equation}
where $Q_k^*$ is $Q_k$ but with $\forall$ and $\exi$
interchanged.
The fact that this useful rule works can be proved
by our beloved principle of mathematical induction on the
number of quantifiers!
Let us conclude by negating $\lim_{x\to a} f(x)=L$ from
Proposition~\ref{p:calc1}.
First, $\lim_{x\to a} f(x)=L$ means
\begin{equation}
\label{e:forall2}
(\forall \varepsilon>0)
(\exi \delta>0)
(\forall x\in\RR\smallsetminus\{a\})
\quad
|x-a| < \delta
\;\;\Rightarrow\;\;
|f(x)-L|<\varepsilon. 
\end{equation}
Recall that $p\Rightarrow q$ means
$(\lnot p)\lor q$.
Hence $\lnot(p\Rightarrow q) = (\lnot(\lnot p))\land (\lnot q)
= p \land (\lnot q)$. 
We deduce that the negation of \eqref{e:forall2} is
\begin{equation}
(\exi \varepsilon>0)
(\forall \delta>0)
(\exi x\in\RR\smallsetminus\{a\})
\quad
|x-a| < \delta
\;\;\land\;\;
|f(x)-L|\geq \varepsilon.
\end{equation}
(This is very useful in the proof of Proposition~\ref{p:calc1}.)
\end{remark}

\section{Limit Laws}

\begin{theorem}[Sum Law]
\label{t:sumlaw}
Let $(a_n)_\nnn$ and $(b_n)_\nnn$ be convergent sequences
with limits $\alpha$ and $\beta$, respectively.
Define the sequence $(c_n)_\nnn$ by
$c_n := a_n+b_n$, for every $n\in\NN$.
Then $(c_n)_\nnn$ converges to $\alpha+\beta$;
in short:\index{Sum Law (for limits)} 
\begin{equation}
\lim_{\nnn}\;(a_n+b_n) = \lim_\nnn a_n + \lim_\nnn b_n.
\end{equation}
\end{theorem}
\begin{proof}
Let $\varepsilon>0$.
Since $\thalb\varepsilon > 0$ and
$a_n\to\alpha$ and $b_n\to\beta$,
there exist $N_1\in\NN$ and $N_2\in\NN$ such that
\begin{equation}
n\geq N_1
\quad\Rightarrow\quad
|a_n-\alpha|<\thalb\varepsilon
\end{equation}
and 
\begin{equation}
n\geq N_2
\quad\Rightarrow\quad
|b_n-\beta|<\thalb\varepsilon. 
\end{equation}
Set $N:=\max\{N_1,N_2\}$.
Then for every $n\geq N$, we have
\begin{subequations}
\begin{align}
|c_n-(\alpha+\beta)| &= |(a_n+b_n)-(\alpha+\beta)| =
|(a_n-\alpha)+(b_n-\beta)|\\
&\leq |a_n-\alpha| + |b_n-\beta| < \thalb\varepsilon + \thalb\varepsilon =
\varepsilon.
\end{align}
\end{subequations}
Therefore, $c_n\to\alpha+\beta$. 
\end{proof}

\begin{example}
$\displaystyle \frac{n}{n+1} + \frac{1}{3^n}\to 1$.
\end{example}
\begin{proof}
On the one hand, $\frac{n}{n+1}\to 1$ by Example~\ref{ex:n/(n+1)}. 
On the other hand, $\frac{1}{3^n} = (1/3)^n \to 0$ 
by Example~\ref{ex:geoseq}\ref{ex:geoseq:i}.
Altogether, using Theorem~\ref{t:sumlaw},
$ \frac{n}{n+1} + \tfrac{1}{3^n}\to 1+0 = 1$.
\end{proof}

\begin{theorem}[Product Law]
\label{t:prodlaw}
Let $(a_n)_\nnn$ and $(b_n)_\nnn$ be convergent sequences
with limits $\alpha$ and $\beta$, respectively.
Define the sequence $(c_n)_\nnn$ by
$c_n := a_nb_n$, for every $n\in\NN$.
Then $(c_n)_\nnn$ converges to $\alpha\beta$;
in short:\index{Product Law (for limits)} 
\begin{equation}
\lim_{\nnn}\;(a_nb_n) = \Big(\lim_\nnn a_n\Big)\Big(\lim_\nnn b_n\Big).
\end{equation}
\end{theorem}
\begin{proof}
The sequence $(a_n)_\nnn$ is convergent, hence
bounded by Theorem~\ref{t:convbound}. Thus, 
\begin{equation}
\text{there exists $\gamma_1\geq 0$ such that
for every $\nnn$, we have $|a_n|\leq\gamma_1$.}
\end{equation}
Now take any 
\begin{equation}
\gamma > \max\big\{\gamma_1,|\beta|\big\}\geq 0,
\end{equation}
and let $\varepsilon>0$. 
Since $a_n\to\alpha$ and $b_n\to\beta$,
for $\frac{\varepsilon}{2\gamma} > 0$,
there exist $N_1$ and $N_2$ in $\NN$ such that
\begin{equation}
|a_n-\alpha|<\frac{\varepsilon}{2\gamma}
\;\;\text{for every $n\geq N_1$}
\quad\text{and}\quad
|b_n-\beta|<\frac{\varepsilon}{2\gamma}
\;\;\text{for every $n\geq N_2$.}
\end{equation}
Thus, for every $n\geq N := \max\{N_1,N_2\}$, we have
\begin{subequations}
\begin{align}
|a_nb_n-\alpha\beta| &=|a_n(b_n-\beta) + (a_n-\alpha)\beta|\\
&\leq |a_n(b_n-\beta)| + |(a_n-\alpha)\beta| 
= |a_n|\cdot|b_n-\beta| + |a_n-\alpha|\cdot|\beta| \\
&<\gamma\frac{\varepsilon}{2\gamma} + \gamma\frac{\varepsilon}{2\gamma} =
\varepsilon.
\end{align}
\end{subequations}
Therefore, $a_nb_n\to\alpha\beta$.
\end{proof}

\begin{corollary}[Constant-Multiple Law]
\label{c:c-mlaw}
Let $(a_n)_\nnn$ be convergent to $\alpha$ and let $c\in\RR$.
Then $(ca_n)_\nnn$ converges to $c\alpha$; in short:
\index{Constant-Multiple Law (for limits)}
\begin{equation}
\lim_{\nnn}\;(ca_n) = c\;\lim_\nnn a_n. 
\end{equation}
\end{corollary}
\begin{proof}
The constant sequence $(c)_\nnn$ converges to $c$ by Example~\ref{ex:cseq}.
The result thus follows from Theorem~\ref{t:prodlaw}.
\end{proof}

\begin{corollary}[Difference Law]
\label{c:difflaw}
Let $(a_n)_\nnn$ and $(b_n)_\nnn$ be convergent sequences
with limits $\alpha$ and $\beta$, respectively.
Define the sequence $(c_n)_\nnn$ by
$c_n := a_n-b_n$, for every $n\in\NN$.
Then $(c_n)_\nnn$ converges to $\alpha-\beta$;
in short:\index{Difference Law (for limits)}
\begin{equation}
\lim_{\nnn}\;(a_n-b_n) = \lim_\nnn a_n - \lim_\nnn b_n.
\end{equation}
\end{corollary}
\begin{proof}
This follows from 
Theorem~\ref{t:sumlaw} and 
Corollary~\ref{c:c-mlaw} because
$c_n = a_n + (-1)b_n$.
\end{proof}

\begin{proposition}[Reciprocal Law]
\label{p:replaw}
Let $(b_n)_\nnn$ be convergent to $\beta\neq 0$.
Then there exists $n_0\in\NN$ such that 
$b_n\neq 0$ for every $n\geq n_0$.\index{Reciprocal Law (for limits)}
Moreover, the sequence $(1/b_n)_{n\geq n_0}$ converges to $1/\beta$.
\end{proposition}
\begin{proof}
Since $\beta\neq 0$, we have $|\beta|>0$ by
Proposition~\ref{p:absval}\ref{p:absvali}\&\ref{p:absvalii}.
Thus, since $\tfrac{1}{2}|\beta|>0$, there exists $n_0\in\NN$ such that
\begin{equation}
n\geq n_0
\quad\Rightarrow\quad
|b_n-\beta|<\frac{|\beta|}{2}.
\end{equation}
On the other hand, $|b_n-\beta| \geq \big| |b_n|-|\beta|\big| \geq
|\beta|-|b_n|$ by Corollary~\ref{c:triangle}. 
Altogether,
$|\beta|-|b_n| \leq |b_n-\beta| < \thalb|\beta|$ and thus
\begin{equation}
n\geq n_0
\quad\Rightarrow\quad
|b_n| > \frac{|\beta|}{2} > 0;
\end{equation}
in particular, $b_n\neq 0$, and also
\begin{equation}
n\geq n_0
\quad\Rightarrow\quad
\frac{1}{|b_n|} < \frac{2}{|\beta|}.
\end{equation}
Now let $\varepsilon >0$.
Then $|\beta|^2\varepsilon/2>0$ and hence 
there exists $n_1\in\NN$ such that 
\begin{equation}
n\geq n_1
\quad\Rightarrow\quad
|b_n-\beta| < \frac{\varepsilon|\beta|^2}{2}.
\end{equation}
Altogether, for $n\geq N := \max\{n_0,n_1\}$, we have
\begin{equation}
\left|\frac{1}{b_n}-\frac{1}{\beta}\right|
= \left| \frac{\beta-b_n}{b_n\beta}\right|
= \frac{1}{|\beta|}\frac{1}{|b_n|}|b_n-\beta|
<\frac{1}{|\beta|}\frac{2}{|\beta|} \frac{\varepsilon|\beta|^2}{2} =
\varepsilon.
\end{equation}
Therefore, $1/b_n\to 1/\beta$.
\end{proof}

\begin{corollary}[Quotient Law]
\label{c:quotlaw}
Let $(a_n)_\nnn$ and $(b_n)_\nnn$ be convergent sequences
with limits $\alpha$ and $\beta$, respectively.
Assume that $\beta\neq 0$.
Then there exists $n_0\in\NN$ such that
$b_n\neq 0$ for every $n\geq n_0$ and 
the sequence $(c_n)_{n\geq n_0}$,defined by
$c_n := a_n/b_n$ for every $n\geq n_0$, 
converges to $\alpha/\beta$;
in short:\index{Quotient Law (for limits)} 
\begin{equation}
\lim_{\nnn}\;\left(\frac{a_n}{b_n}\right) = \frac{\displaystyle \lim_\nnn
a_n}{\displaystyle \lim_\nnn b_n}. 
\end{equation}
\end{corollary}
\begin{proof}
Since $c_n = a_n/b_n = a_n\cdot \frac{1}{b_n}$,
the result follows by combining Theorem~\ref{t:prodlaw} and
Proposition~\ref{p:replaw}.
\end{proof}

\begin{example}
\label{ex:limitlaws}
$\displaystyle \frac{2n^2 + 4n +5}{n^2+1} \to 2$.
\end{example}
\begin{proof}
We write, for every $n\geq 1$, 
\begin{equation}
\frac{2n^2 + 4n +5}{n^2+1}
= \frac{2+\frac{4}{n} + \frac{5}{n^2}}{1+\frac{1}{n^2}}. 
\end{equation}
Now 
\begin{itemize}
\item $2\to 2$ (by Example~\ref{ex:cseq});\\[-2mm]
\item $\frac{4}{n}\to 0$ 
(by \eqref{e:1/n} and  Corollary~\ref{c:c-mlaw});\\[-2mm]
\item $\frac{5}{n^2} = \frac{1}{n}\frac{5}{n}\to 0$
(by \eqref{e:1/n}, Corollary~\ref{c:c-mlaw}, and
Theorem~\ref{t:prodlaw});\\[-2mm]
\item $1\to 1$ (by Example~\ref{ex:cseq});\\[-2mm]
\item $\frac{1}{n^2} = \frac{1}{n}\frac{1}{n}\to 0$
(by \eqref{e:1/n} and Theorem~\ref{t:prodlaw}). 
\end{itemize}
Hence, using Theorem~\ref{t:sumlaw} repeatedly,
we have $2+\frac{4}{n} + \frac{5}{n^2} \to 2+0+0 = 2$
and $1+\frac{1}{n^2}\to 1+0=1\neq 0$.
Finally, using Corollary~\ref{c:quotlaw}, we have
\begin{equation}
\frac{2n^2 + 4n +5}{n^2+1}
= \frac{2+\frac{4}{n} + \frac{5}{n^2}}{1+\frac{1}{n^2}}
\to \frac{2}{1} = 2.
\end{equation}
\end{proof}

\begin{theorem}
\label{t:limleq}
Let $(a_n)_\nnn$ and $(b_n)_\nnn$ be two convergent sequences
such that $a_n\leq b_n$ for every $\nnn$.
Then
\begin{equation}
\lim_\nnn a_n \leq \lim_\nnn b_n.
\end{equation}
\end{theorem}
\begin{proof}
Denote the limits of $(a_n)_\nnn$ and $(b_n)_\nnn$ by
$\alpha$ and $\beta$, respectively.
We argue by contradiction and assume that $\alpha > \beta$. 
Then $\varepsilon := (\alpha-\beta)/2 > 0$ so there exist
$N_1$ and $N_2$ in $\NN$ such that
\begin{equation}
n \geq N_1
\;\Rightarrow\;
|a_n-\alpha| < \varepsilon
\qquad\text{and}\qquad
n \geq N_2
\;\Rightarrow\;
|b_n-\beta| < \varepsilon. 
\end{equation}
Set $N := \max\{ N_1,N_2\}$.
Then, 
for every $n\geq N$, 
\begin{equation}
\alpha-a_n\leq |a_n-\alpha| < \varepsilon
\qquad\text{and}\qquad
b_n-\beta \leq |b_n-\beta| < \varepsilon.
\end{equation}
Therefore,
for every $n\geq N$, 
\begin{equation}
b_n < \varepsilon + \beta = \frac{\alpha-\beta}{2} + \beta
= \frac{\alpha+\beta}{2} = \alpha - \frac{\alpha-\beta}{2} 
= \alpha-\varepsilon < a_n,
\end{equation}
which contradicts the hypothesis $a_n\leq b_n$. 
\end{proof}

\begin{corollary}
\label{c:limleq}
Let $(b_n)_\nnn$ be a convergent sequence such that
$\alpha \leq b_n\leq \gamma$, for every $\nnn$.
Then
\begin{equation}
\alpha \leq \lim_\nnn b_n \leq \gamma.
\end{equation}
\end{corollary}
\begin{proof}
Combine Example~\ref{ex:cseq} with Theorem~\ref{t:limleq}.
\end{proof}

\begin{theorem}[Squeeze Theorem]
\label{t:squeeze}
Let $(a_n)_\nnn$, $(b_n)_\nnn$, and $(c_n)_\nnn$ be sequences
such that for every $\nnn$, $a_n\leq b_n\leq c_n$.
Suppose that $(a_n)_\nnn$ and $(c_n)_\nnn$ converge to a common
limit $\ell$. Then $(b_n)_\nnn$ converges to $\ell$ as well.
\index{Squeeze Theorem}
\end{theorem}
\begin{proof}
Let $\varepsilon>0$.
Since $a_n\to\ell$, there exists $N_1\in\NN$ such that
\begin{equation}
n\geq N_1
\quad\Rightarrow\quad
|a_n-\ell|< \varepsilon
\;\Leftrightarrow\;
\ell-\varepsilon < a_n < \ell+\varepsilon.
\end{equation}
Since $c_n\to\ell$, there exists $N_2\in\NN$ such that
\begin{equation}
n\geq N_2
\quad\Rightarrow\quad
|c_n-\ell|< \varepsilon
\;\Leftrightarrow\;
\ell-\varepsilon < c_n < \ell+\varepsilon.
\end{equation}
Now set $N := \max\{N_1,N_2\}$. 
Combining the above, we see that for every $n\geq N$,
\begin{equation}
\ell-\varepsilon < a_n \leq b_n \leq c_n < \ell+\varepsilon;
\end{equation}
consequently, $|b_n-\ell|<\varepsilon$.
\end{proof}

\begin{proposition}
\label{p:absantoabsa}
Let $(a_n)_\nnn$ be convergent to $\alpha$.
Then $(|a_n|)_\nnn$ converges to $|\alpha|$. 
\end{proposition}
\begin{proof}
Exercise~\ref{exo:absantoabsa}.
\end{proof}

\section{Convergence to $\pm\infty$}

\index{Convergence to $\pm\infty$}
\index{Divergence to $\pm\infty$} 
The symbols $+\infty$ and $-\infty$ are convenient
to capture the behaviour of sequences that grow arbitrarily
large in absolute value. These symbols are \emph{not}
real numbers; e.g., $\infty + 1 = +\infty = +\infty + 0$ but $1\neq 0$ etc.

\begin{definition}[divergence to infinity]
Let $(a_n)_\nnn$ be a sequence of real numbers.
Then $(a_n)_\nnn$ \textbf{diverges to $\pinf$}, written 
$a_n\to\pinf$ or $\displaystyle\lim_\nnn a_n=\pinf$ if 
\begin{quotation}
\noindent
for every $\gamma\in\RR$, there exists $N\in\NN$ such that\\
$n\geq N$ $\Rightarrow$ $a_n>\gamma$.
\end{quotation}
Similarly, $(a_n)_\nnn$ \textbf{diverges to $\minf$}, written 
$a_n\to\minf$ or $\displaystyle\lim_\nnn a_n=\minf$ if 
\begin{quotation}
\noindent
for every $\gamma\in\RR$, there exists $N\in\NN$ such that\\
$n\geq N$ $\Rightarrow$ $a_n<\gamma$.
\end{quotation}
\end{definition}

Note that if $a_n\to\pinf$, then $(a_n)$ is not bounded above;
the converse is false, however. 

\begin{example}\
\begin{enumerate}
\item 
$(3n)_\nnn$ diverges to $\pinf$.
\item
$(-3^n)_\nnn$ diverges to $\minf$.
\item
$((-3)^n)_\nnn$ diverges, but it does not
diverge to $\pinf$ or to $\minf$.
\end{enumerate}
\end{example}

The study of such sequences can be reduced to
the study of null sequences (i.e., sequences that converge to $0$).
We leave the proofs as Exercise~\ref{exo:dinfty1} 
and Exercise~\ref{exo:dinfty2}.

\begin{theorem}
\label{t:dinfty1}
Let $(a_n)_\nnn$ be a sequence that diverges to
either $\pinf$ or $\minf$. 
Then there exists $n_0\in\NN$ such that
$a_n\neq 0$ for every $n\geq n_0$.
Furthermore, the sequence $(1/a_n)_{n\geq n_0}$
converges and
\begin{equation}
\lim_\nnn\; \frac{1}{a_n} = 0.
\end{equation}
\end{theorem}

\begin{theorem}
\label{t:dinfty2}
Let $(a_n)_\nnn$ be a sequence of positive (resp.~negative) real numbers
such that $a_n\to 0$. Then $(1/a_n)_\nnn$ diverges to $\pinf$
(resp.~$\minf$). 
\end{theorem}

\section{Series}
\label{sec:series}
\index{Series}
\index{Infinite Series}
Let $(a_n)_\nnn$ be a sequence of real numbers.
The sequence of \textbf{partial sums} $(s_n)_\nnn$, defined by
\begin{equation}
s_n := \sum_{k=0}^{n} a_k,
\qquad\text{for every $n\in\NN$},
\end{equation}
is called an \textbf{infinite series} and written as 
\begin{equation}
\text{
$\displaystyle \sum_{k=0}^\infty a_k$ ~~~or~~~
$\displaystyle \sum_{k\in\NN} a_k$.
}
\end{equation}
If the sequence $(s_n)_\nnn$ converges, then we 
say that the series $ \sum_{k=0}^\infty a_k$ converges, and we denote
its limit also by 
$\textstyle \sum_{k=0}^\infty a_k$.

Thus, the notation 
\begin{equation}
\displaystyle \sum_{k=0}^\infty a_k
\end{equation}
has \emph{two meanings}: 
it denotes the sequence
\begin{equation}
\bigg(\sum_{k=0}^{n}a_k\bigg)_\nnn
\end{equation}
of partial sums; if this sequence of partial sums converges,
then it denotes also its limit.
As for sequences, it is sometimes convenient to start the index
of a series at an integer different from $0$.

\begin{remark}
It appears that series are much more special than sequences.
In fact, the notions are equivalent:
given an arbitrary sequence $(c_n)_\nnn$, we set
\begin{equation}
a_0 := c_0
\qquad\text{and}\qquad
a_n := c_n-c_{n-1}, \;\text{for $n\geq 1$.}
\end{equation}
Then, for every $\nnn$, 
\begin{equation}
\sum_{k=0}^n a_k = a_0 + \sum_{k=1}^{n} a_k
= c_0 + (c_1-c_0) + \cdots + (c_n-c_{n-1}) = c_n.
\end{equation}
Therefore, $(c_n)_\nnn = \sum_{k=0}^\infty a_k$. 
\end{remark}

\begin{example}
\label{ex:seite24}
Consider the series
\begin{equation}
\sum_{k=1}^\infty \frac{1}{k(k+1)},
\end{equation}
with partial sums
\begin{equation}
s_n := \sum_{k=1}^{n} \frac{1}{k(k+1)}, 
\quad\text{for all integers $n\geq 1$.}
\end{equation}
It can be shown by induction that (see Exercise~\ref{exo:seite24})
\begin{equation}
\label{e:seite24}
s_n = \sum_{k=1}^{n} \frac{1}{k(k+1)} = \frac{n}{n+1},
\quad\text{for all integers $n\geq 1$.}
\end{equation}
In view of Example~\ref{ex:n/(n+1)},
we deduce that
\begin{equation}
\sum_{k=1}^\infty \frac{1}{k(k+1)} = 1.
\end{equation}
\end{example}

\begin{theorem}[Geometric Series]
\label{t:geoseries}
Let $x\in\RR$ be such that $|x|<1$.
Then\index{Geometric series}
\begin{equation}
\sum_{k=0}^\infty x^k = \frac{1}{1-x}.
\end{equation}
\end{theorem}
\begin{proof}
Indeed, for the partial sums we have, using Theorem~\ref{t:geosum}, 
\begin{equation}
s_n := \sum_{k=0}^{n} x^k = \frac{1-x^{n+1}}{1-x},
\quad\text{for every $\nnn$.}
\end{equation}
On the other hand, $x^{n+1} \to 0$ by
Example~\ref{ex:geoseq}\ref{ex:geoseq:i}. 
Altogether,
\begin{equation}
s_n = \frac{1-x^{n+1}}{1-x} = \frac{1}{1-x} - \frac{1}{1-x}\cdot x^{n+1}
\to \frac{1}{1-x} - \frac{1}{1-x}\cdot 0 = \frac{1}{1-x},
\end{equation}
by Corollary~\ref{c:c-mlaw} and Corollary~\ref{c:difflaw}.
\end{proof}

Since a series is nothing but a sequence of partial sums, 
we obtain from Theorem~\ref{t:sumlaw}, 
Corollary~\ref{c:c-mlaw}, and
Corollary~\ref{c:difflaw} the following laws\footnote{Products are not so
easy and they will be studied later.} for series.

\begin{theorem}
\label{t:serieslaws}
Let $\sum_{k=0}^\infty a_k$ and
$\sum_{k=0}^\infty b_k$ be convergent series,
and let $c\in\RR$.
Then the series
$\sum_{k=0}^\infty (a_k+b_k)$,
$\sum_{k=0}^\infty (a_k-b_k)$,
and 
$\sum_{k=0}^\infty ca_k$
are convergent as well; furthermore, we have
\index{Limit Laws (for series)} 
\begin{equation}
\sum_{k=0}^\infty (a_k\pm b_k) = 
\sum_{k=0}^\infty a_k \pm
\sum_{k=0}^\infty b_k 
\qquad\text{and}\qquad
\sum_{k=0}^\infty ca_k = 
c \sum_{k=0}^\infty a_k.
\end{equation}
\end{theorem}

The last two results are very powerful.
Let us provide an illustration.

\begin{example}
Suppose that 
\begin{equation}
x= 0.13424242\ldots = 0.13\overline{42}.
\end{equation}
This means\index{Decimal expansion}
\begin{equation}
x = \frac{13}{100} + \frac{42}{10^4} + \frac{42}{10^6} + \cdots
= \frac{13}{100} + \sum_{k=0}^\infty \frac{42}{10^{4+2k}}.
\end{equation}
The infinite series is, using Theorem~\ref{t:geoseries} and
Theorem~\ref{t:serieslaws}, equal to 
\begin{align}
\sum_{k=0}^\infty \frac{42}{10^{4+2k}} &= 
\frac{42}{10^4}\sum_{k=0}^\infty \left(\frac{1}{10^2}\right)^k 
= \frac{42}{10^4}\cdot \frac{1}{1-10^{-2}}
= \frac{42}{10^4-10^2} = \frac{42}{9900} = \frac{7}{1650}.
\end{align}
Therefore,
\begin{equation}
x = \frac{13}{100} + \frac{7}{1650} = \frac{443}{3300}. 
\end{equation}
\end{example}

\section*{Exercises}\markright{Exercises}
\addcontentsline{toc}{section}{Exercises}
\setcounter{theorem}{0}

\begin{exercise}
Suppose $(a_n)_\nnn$ is a sequence and let $\ell\in\RR$.
Let us say that $(a_n)_\nnn$ is ``super convergent'' to
$\ell$ if there exists $N\in \NN$ such that 
for every
$\varepsilon>0$
we have $n\geq N$ $\Rightarrow$ $|a_n-\ell|<\varepsilon$.
Show that if $(a_n)_\nnn$ super converges to $\ell$,
then $(a_n)_\nnn$ converges to $\ell$ (in the usual sense).
What about the converse?
\end{exercise}
\begin{solution}
Clearly, every super convergent sequence converges,
since there is a ``super'' $N$ that does not depend on
$\varepsilon$. 

Now assume that $(a_n)_\nnn$ super converges to $\ell$.
So pick up the $N\in\NN$ such that 
for every
$\varepsilon>0$
we have $n\geq N$ $\Rightarrow$ $|a_n-\ell|<\varepsilon$.
Let $n\geq N$. Then $|a_n-\ell|<\varepsilon$ for 
\emph{every} $\varepsilon > 0$. 
Then $|a_n-\ell|=0$ (because otherwise we could
choose $\varepsilon = |a_n-\ell|/2$ and obtain a contradiction).
Hence $a_N=a_{N+1}=\cdots = \ell$.
This means that the super convergent sequences are precisely
those that are \emph{eventually constant}.
This exercise also illustrates that the \emph{order of the
quantifiers} ``for every'' and ``there exist'' is crucial.
Moreover, $(1/n)_{n\geq 1}$ is convergent $0$, but not super
convergent to $0$. 
\end{solution}

\begin{exercise}
\label{exo:bdseq}
Let $(a_n)_\nnn$ be a sequence.
Show that $(a_n)_\nnn$ is bounded if and only if
\eqref{e:bdseq} holds.
\end{exercise}
\begin{solution}
``$\Rightarrow$'': Suppose first that $(a_n)_\nnn$ is bounded, say bounded
above by $\beta$ and bounded below by $\alpha$:
\begin{equation*}
\alpha \leq a_n \leq \beta, \quad\text{for every $\nnn$.}
\end{equation*}
Note that $-|\alpha| \leq \alpha$ and $\beta\leq |\beta|$. 
Set 
\begin{equation*}
\gamma := \max\{|\alpha|,|\beta|\}.
\end{equation*}
Then
\begin{equation*}
-\gamma\leq -|\alpha|\leq \alpha\leq a_n \leq \beta\leq |\beta|\leq\gamma, \quad\text{for every $\nnn$.}
\end{equation*}
In particular,
\begin{equation*}
a_n \leq \gamma \;\text{and}\; -a_n\leq \gamma, \quad\text{for every $\nnn$.}
\end{equation*}
By Lemma~\ref{l:absval}, $|a_n|\leq \gamma$ for every $\nnn$. 

``$\Leftarrow$'': Now assume that for some $\gamma\geq 0$, $|a_n|\leq \gamma$, for 
every $\nnn$. 
Then, 
for every $\nnn$, 
by Lemma~\ref{l:absval},
$\pm a_n\leq \gamma$ and so $-\gamma\leq a_n\leq \gamma$.
Thus, the sequence $(a_n)_\nnn$ is bounded above by $\gamma$ and bounded
below by $-\gamma$. Therefore, $(a_n)_\nnn$ is bounded. 
\end{solution}

\begin{exercise}
Prove that the sequence of Fibonacci numbers
$(f_n)_\nnn$ considered in
Example~\ref{ex:seq}\ref{ex:seq:fibo}
satisfies $f_n \geq n-1$ for all $\nnn$.
\emph{Hint:} Use strong induction (see Exercise~\ref{exo:strongind}) 
with a suitable $n_0$.
Then conclude that $(f_n)_\nnn$ is not bounded,
and hence not convergent.
\end{exercise}
\begin{solution}
We shall use strong induction with $n_0=3$.
First, we deal with $n\in\{0,1,2\}$ directly:\\
(i) If $n=0$, then $f_0 = 0 \geq -1 = 0-1$, as required.\\
(ii) If $n=1$, then $f_1 = 1 \geq 0 = 1-1$, as required.\\
(iii) If $n=2$, then $f_2 = 1 \geq 1 = 2-1$, as required.\\

Now we turn to the actual induction. 

\textbf{Base Case:}
If $n=3$, then $f_3 = 2 \geq 2 = 3-1$, as required.

\textbf{Inductive Step:}
Suppose that for some $n\geq 3$, we have $f_k\geq k-1$
for every $k\in\{3,\ldots,n\}$. Note that by
our preliminary work, we know that this actually holds for 
every $k\in\{0,1,\ldots,n\}$!

We thus conclude that 
\begin{subequations}
\begin{align}
f_{n+1} &= f_n + f_{n-1} 
\geq (n-1) + \big((n-1)-1\big)\\
&= 2n-3 \\
&\geq n\label{e:130928a}\\
&=(n+1)-1,
\end{align}
\end{subequations}
where \eqref{e:130928a} is true exactly because $n\geq 3$. 
This completes the proof of the inductive step.

Therefore, by the Principle of Strong Mathematical Induction, the
statement is proven. 
\end{solution}

\begin{exercise}
Prove that 
$\displaystyle \frac{n}{2^n}\to 0$.
\emph{Hint}: Exercise~\ref{exo:uncleshawn}. 
\end{exercise}
\begin{solution}
By Exercise~\ref{exo:uncleshawn},
for every $n\geq 4$,
\begin{equation*}
\frac{n^2}{2^n} \leq 1
\quad\Rightarrow\quad
\frac{n}{2^n} \leq \frac{1}{n}.
\end{equation*}
Now let $\varepsilon>0$,
and take $N(\varepsilon) >  \max\big\{4,\tfrac{1}{\varepsilon}\big\}$
in $\NN$.
Then
\begin{equation*}
\left|\frac{n}{2^n}-0\right| = \frac{n}{2^n} \leq 
\frac{1}{n} \leq \frac{1}{N(\varepsilon)} < \varepsilon,
\quad\text{for every $n\geq N(\varepsilon)$.}
\end{equation*}
Or, alternatively, use the Squeeze Theorem!
\end{solution}

\begin{exercise}
\label{exo:calc1}
Prove Proposition~\ref{p:calc1}.
\end{exercise}
\begin{solution}
``$\Rightarrow$'': Assume that $\lim_{x\to a} f(x)=L$ and let
$(a_n)_\nnn$ be a sequence such that $a_n\to a$ and $a_n\neq a$ for every
$\nnn$. Now let $\varepsilon >0$. 
By definition of the limit, there exists $\delta>0$ such that
$0<|x-a|<\delta$ implies $|f(x)-L|<\varepsilon$. 
Now for this $\delta$, we pick up $N$ such that 
$n\geq N$ implies $|a_n-a|<\delta$.
If $n\geq N$, then $0<|a_n-a|<\delta$ and hence
$|f(a_n)-L|<\varepsilon$. We have shown that $f(a_n)\to L$. 

``$\Leftarrow$'': 
We argue by contradiction. Hence we assume there exists 
some $\varepsilon>0$ such that for every $\delta>0$,
there exists $x$ such that $0<|x-a|<\delta$ and
$|f(x)-L|\geq \varepsilon$. 
In particular, for $\delta = 1/(n+1)$, where $\nnn$, 
we find $x_n$ such that $0<|x_n-a|<1/(n+1)$ yet $|f(x_n)-L| \geq \varepsilon$. 
Note that for the sequence $(x_n)_\nnn$ so generated, we have
$x_n\to a$ and $x_n\neq a$ for every $\nnn$.
By hypothesis we should have $|f(x_n)-L|<\varepsilon$ for $n$ sufficiently
large; however, we have $|f(x_n)-L| \geq \varepsilon$ for every $\nnn$
--- this is absurd!
\end{solution}

\begin{exercise}
Prove Corollary~\ref{c:c-mlaw} directly, using the definition of
convergence.
\end{exercise}
\begin{solution}
If $c=0$, then $(ca_n)_\nnn = (0)_\nnn$ is a constant sequence converging
to $0=c\alpha$. 

We thus assume that $c\neq 0$.
Let $\varepsilon>0$. 
Since $\varepsilon/|c|>0$ and $a_n\to\alpha$, 
there exists $N\in\NN$ such that
\begin{equation*}
n\geq N \;\Rightarrow\;
|a_n-\alpha| < \frac{\varepsilon}{|c|}. 
\end{equation*}
It follows that for every $n\geq N$, we have
\begin{equation*}
|c a_n - c\alpha| = |c(a_n-\alpha)| = |c|\cdot |a_n-\alpha|
< |c|\frac{\varepsilon}{|c|} = \varepsilon,
\end{equation*}
as required. 
\end{solution}

\begin{exercise}
Prove Corollary~\ref{c:difflaw} directly, using the definition of
convergence.
\end{exercise}
\begin{solution}
Let $\varepsilon>0$.
Since $\thalb\varepsilon > 0$ and
$a_n\to\alpha$ and $b_n\to\beta$,
there exist $N_1\in\NN$ and $N_2\in\NN$ such that
\begin{equation*}
n\geq N_1
\quad\Rightarrow\quad
|a_n-\alpha|<\thalb\varepsilon
\end{equation*}
and 
\begin{equation*}
n\geq N_2
\quad\Rightarrow\quad
|b_n-\beta|<\thalb\varepsilon. 
\end{equation*}
Set $N:=\max\{N_1,N_2\}$.
Then for every $n\geq N$, we have
\begin{align*}
|c_n-(\alpha-\beta)| &= |(a_n-b_n)-(\alpha-\beta)| =
|(a_n-\alpha)+(\beta-b_n)|\\
&\leq |a_n-\alpha| + |b_n-\beta| < \thalb\varepsilon + \thalb\varepsilon =
\varepsilon.
\end{align*}
Therefore, $c_n\to\alpha-\beta$. 
\end{solution}

\begin{exercise}
Using the limit laws, determine the limit of the sequence $(a_n)_\nnn$,
where 
\begin{equation*}
a_n := \frac{7n^3 + 3n - 7}{2n^3 + 2n+3},
\quad\text{where $n\in\NN$.}
\end{equation*}
\end{exercise}
\begin{solution}
Dividing by $n^3$ and 
arguing as in Example~\ref{ex:limitlaws},
we deduce that for every $n\geq 1$, 
\begin{equation*}
a_n = \frac{7 + 3\frac{1}{n}\frac{1}{n} -
7\frac{1}{n}\frac{1}{n}\frac{1}{n}}{2+
2\frac{1}{n}\frac{1}{n}+3\frac{1}{n}\frac{1}{n}\frac{1}{n}} \to 
\frac{7+3\cdot 0 \cdot 0 - 7\cdot 0\cdot 0 \cdot 0}{2+2\cdot 0 \cdot 0
+3\cdot 0\cdot 0 \cdot 0} = \frac{7}{2}.
\end{equation*}
Hence $\lim_\nnn a_n = \frac{7}{2}$. 
\end{solution}

\begin{exercise}[sum law for limits from calculus]
\label{exo:calc2}
Let $f$ and $g$ be functions defined on $\RR\smallsetminus\{a\}$
with limits $L$ and $M$ as $x\to a$, respectively.
Prove that the sum law for limits in calculus holds, i.e.., 
$$\lim_{x\to a} \big(f(x)+g(x)\big) = \lim_{x\to a}f(x) + \lim_{x\to
a}g(x).$$
\emph{Hint:} Proposition~\ref{p:calc1} and Theorem~\ref{t:sumlaw}. 
\end{exercise}
\begin{solution}
Let $(x_n)_\nnn$ be a sequence in $\RR\smallsetminus\{a\}$ that converges
to $a$. By Proposition~\ref{p:calc1}, it suffices to show
that $(f+g)(x_n)=f(x_n)+g(x_n)\to L+M$.
Now, again by Proposition~\ref{p:calc1}, $f(x_n)\to L$ and $g(x_n)\to M$.
Now by Theorem~\ref{t:sumlaw}, $f(x_n)+g(x_n)\to L+M$ and we are done.
\end{solution}

\begin{exercise}[product law for limits from calculus]
\label{exo:calc3}
Let $f$ and $g$ be functions defined on $\RR\smallsetminus\{a\}$
with limits $L$ and $M$ as $x\to a$, respectively.
Prove that the product law for limits in calculus holds, i.e., 
$$\lim_{x\to a} \big(f(x)\cdot g(x)\big) = \lim_{x\to a}f(x) \cdot \lim_{x\to
a}g(x).$$
\emph{Hint:} Proposition~\ref{p:calc1} and Theorem~\ref{t:prodlaw}. 
\end{exercise}
\begin{solution}
Let $(x_n)_\nnn$ be a sequence in $\RR\smallsetminus\{a\}$ that converges
to $a$. By Proposition~\ref{p:calc1}, it suffices to show
that $(f\cdot g)(x_n)=f(x_n)\cdot g(x_n)\to L\cdot M$.
Now, again by Proposition~\ref{p:calc1}, $f(x_n)\to L$ and $g(x_n)\to M$.
Now by Theorem~\ref{t:prodlaw}, $f(x_n)\cdot g(x_n)\to L\cdot M$ and we are done.
\end{solution}

\begin{exercise}[squeeze theorem from calculus]
\label{exo:calc4}
Let $f,g,h$ be functions defined on $\RR\smallsetminus\{a\}$
such that $f(x)\leq g(x)\leq h(x)$ for every $x\in\RR\smallsetminus\{a\}$, and 
suppose that $L = \lim_{x\to a} f(x)=\lim_{x\to a} h(x)$. 
Prove that $\lim_{x\to a} g(x)=L$. 
\emph{Hint:} Proposition~\ref{p:calc1} and Theorem~\ref{t:squeeze}. 
\end{exercise}
\begin{solution}
Let $(x_n)_\nnn$ be a sequence in $\RR\smallsetminus\{a\}$ that converges
to $a$. By Proposition~\ref{p:calc1}, it suffices to show
that $b_n := g(x_n)\to L$. 
Now, by assumption, $a_n := f(x_n) \leq b_n = g(x_n) \leq c_n := h(x_n)$,
and, again by Proposition~\ref{p:calc1}, $a_n\to L$ and $c_n\to L$. 
Now by Theorem~\ref{t:squeeze}, $b_n\to L$, as required.
\end{solution}

\begin{exercise}
Provide examples of sequences $(a_n)$ and $(b_n)$
such that $a_n\to \pinf$ and $b_n\to 0$ for each of the following.
\begin{enumerate}
\item $(a_nb_n)$ diverges to $\pinf$.
\item $(a_nb_n)$ diverges to $\minf$.
\item $(a_nb_n)$ is unbounded but does not diverge to $\pinf$ or to
$\minf$.
\item $(a_nb_n)$ converges to some real number $\gamma$ that can be chosen arbitrarily in advance. 
\item $(a_nb_n)$ is bounded but not convergent.
\end{enumerate}
This illustrates why it's hard to define the product $0\cdot\infty$.
\end{exercise}
\begin{solution}
Many examples are possible. Here is a set of examples that does the
job.
\begin{enumerate}
\item $a_n= n^2$ and $b_n=\tfrac{1}{n}$ for $n\geq 1$. 
Then $a_nb_n=n\to\pinf$. 
\item $a_n = n^2$ and $b_n=-\tfrac{1}{n}$ for $n\geq 1$. 
Then $a_nb_n=-n\to\minf$. 
\item $a_n = n^2$ and $b_n = (-1)^n\tfrac{1}{n}$ for $n\geq 1$.
Then $a_nb_n = (-1)^nn$. 
\item $a_n = n$ and $b_n = \gamma\tfrac{1}{n}$ for $n\geq 1$.
Then $a_nb_n = \gamma$.
\item $a_n = n$ and $b_n = (-1)^n\tfrac{1}{n}$ for $n\geq 1$.
Then $a_nb_n = (-1)^n$. 
\end{enumerate}
\end{solution}

\begin{exercise}
\label{exo:dinfty1}
Prove Theorem~\ref{t:dinfty1}.
\end{exercise}
\begin{solution}
Suppose that $a_n\to\pinf$. (The proof when $a_n\to\minf$ is similar.)
For $\gamma=0$, there exists $n_0\in\NN$ such that
whenever $n\geq n_0$, we have $a_n>0$.
It follows that the sequence $(1/a_n)_{n\geq n_0}$ is a 
well-defined sequence of positive real numbers.
Now let $\varepsilon>0$. Since $a_n\to\pinf$,
there exists $N\in\NN$ such that $a_n>\frac{1}{\varepsilon}$,
for every $n\geq N$.
It follows that for every $n\geq N$, we have
\begin{equation*}
\left|\frac{1}{a_n}-0\right|
=\frac{|1|}{|a_n|} = \frac{1}{a_n} < \frac{1}{\frac{1}{\varepsilon}} =
\varepsilon.
\end{equation*}
\end{solution}

\begin{exercise}
\label{exo:dinfty2}
Prove Theorem~\ref{t:dinfty2} when the sequence has 
exclusively positive terms. 
\end{exercise}
\begin{solution} 
Suppose that $(a_n)_\nnn$ is a sequence of positive numbers. 
Also assume that $a_n\to 0$. 
We must show that $(1/a_n)_\nnn$ diverges to $\pinf$. 
To this end, let $\gamma>0$. 
Then $\varepsilon := 1/\gamma > 0$.
Because $a_n\to 0$, there exists $N\in\NN$ such that 
if $n\geq N$, then $a_n = |a_n| = |a_n-0|<\varepsilon$;
hence $1/a_n > 1/\varepsilon = \gamma$
and we are done.
\end{solution}

\begin{exercise}
\label{exo:seite24}
Prove \eqref{e:seite24} by mathematical induction.
\end{exercise}
\begin{solution}
The statement to prove is
\begin{equation*}
\sum_{k=1}^n \frac{1}{k(k+1)} = \frac{n}{n+1},
\quad\text{for every integer $n\geq 1$.}
\end{equation*}
Denote the corresponding statement by $S(n)$.

\textbf{Base Case}:
When $n=1$, $S(n)$ reads
$\sum_{k=1}^{1} \frac{1}{k(k+1)} = \frac{1}{1(1+1)} = 
\frac{1}{2} = \frac{1}{1+1}. \checkmark$

\textbf{Inductive Step}:
Suppose that $S(n)$ is true for some integer $n\geq 1$.
Then, using the inductive hypothesis in \eqref{e:201021a}, 
we obtain 
\begin{subequations}
\begin{align}
\sum_{k=1}^{n+1} \frac{1}{k(k+1)} 
&= \left(\sum_{k=1}^n \frac{1}{k(k+1)}\right) + \frac{1}{(n+1)(n+2)}\\
&= \frac{n}{n+1} +  \frac{1}{(n+1)(n+2)}\label{e:201021a}\\
&= \frac{n(n+2)}{(n+1)(n+2)} +  \frac{1}{(n+1)(n+2)}\\
&= \frac{n^2+2n+1}{(n+1)(n+2)}\\
&= \frac{(n+1)^2}{(n+1)(n+2)}\\
&= \frac{n+1}{n+2},
\end{align}
\end{subequations}
which verifies the statement $S(n+1)$.
The proof is thus complete by the principle of mathematical induction.
\end{solution}

\begin{exercise}
Determine the limit of each of the following infinite series.
\begin{enumerate}
\item $1+\frac{1}{2} + \frac{1}{4}+\frac{1}{8} + \frac{1}{16}+\cdots$.\\[-2mm]
\item $1-\frac{1}{2} + \frac{1}{4}-\frac{1}{8} + \frac{1}{16}\mp \cdots$. 
\end{enumerate}
\end{exercise}
\begin{solution}
Both series are geometric series $\sum_{k=0}^\infty x^k$, 
with $x=\pm\tfrac{1}{2}$. It follows from 
Theorem~\ref{t:geoseries} that the limits
are $1/(1\mp\tfrac{1}{2})$.
Thus, 
$1+\frac{1}{2} + \frac{1}{4}+\frac{1}{8} +\cdots = 2$ and
$1-\frac{1}{2} + \frac{1}{4}-\frac{1}{8} \pm \cdots = \frac{2}{3}$. 
\end{solution}

\begin{exercise}
Convert $1.23455555\ldots = 1.234\overline{5}$ into a fraction.
\end{exercise}
\begin{solution}
Set $x :=  1.234\overline{5}$.
Then, using the geometric series, 
\begin{align*}
x&= 1 + \frac{2}{10} + \frac{3}{100} + \frac{4}{1000}
+ \frac{5}{10000}\Big(1 + \frac{1}{10} + \frac{1}{100} + \cdots\Big)\\
&=  1 + \frac{2}{10} + \frac{3}{100} + \frac{4}{1000}
+ \frac{5}{10000}\frac{1}{1-\frac{1}{10}}
=  1 + \frac{2}{10} + \frac{3}{100} + \frac{4}{1000}
+ \frac{5}{1000}\frac{1}{10- 1}\\
&=  1 + \frac{2}{10} + \frac{3}{100} + \frac{4}{1000}
+ \frac{5}{1000}\frac{1}{9}
=  \frac{9000}{9000} + \frac{2\cdot 900}{9000} + 
\frac{3\cdot 90}{9000} + \frac{4\cdot 9}{9000}
+ \frac{5}{9000}\\
&=\frac{9000+1800+270+36+5}{9000}
= \frac{11111}{9000}.
\end{align*}
\end{solution}

\begin{exercise}
Let $\alpha$ and $\beta$ be real numbers.
Define a sequence $(a_n)_\nnn$ by
\begin{equation*}
a_0 := \alpha,\quad
a_1 := \beta,\quad
a_{n} := \frac{a_{n-2}+a_{n-1}}{2}, 
\;\;\text{for every $n\geq 2$.}
\end{equation*}
Show that $(a_n)_\nnn$ converges and determine its limit.\\
\emph{Hint:} 
For every integer $k\geq 0$, we have
$a_{k+1}-a_k = (-1/2)^k(\beta-\alpha)$.
\end{exercise}
\begin{solution}
The identity in the hint is easily proved by induction on $k\geq 0$.
Indeed, the identity is clear for $k=0$.

Now assume the identity holds for some integer $k\geq 0$.
Then
\begin{align*}
a_{(k+1)+1} - a_{k+1} & = a_{k+2} - a_{k+1} 
= \frac{a_{k}+a_{k+1}}{2} - a_{k+1}
= \frac{a_k-a_{k+1}}{2}\\
&= (-1/2)(a_{k+1}-a_k)
= (-1/2)(-1/2)^k(\beta-\alpha)
= (-1/2)^{k+1}(\beta-\alpha),
\end{align*}
and, by the Principle of Mathematical Induction, the identity thus
holds for all $k\in\NN$.

Hence, for every $n\geq 1$,
\begin{align*}
a_n &= a_0 + (a_1-a_0)+\cdots + (a_n-a_{n-1})
=\alpha + \sum_{k=0}^{n-1}(a_{k+1}-a_k)
=\alpha + \sum_{k=0}^{n-1}\left(\frac{-1}{2}\right)^k(\beta-\alpha)\\
&\to \alpha + \frac{2}{3}(\beta-\alpha)
= \alpha + \frac{1}{1-(-1/2)}(\beta-\alpha) 
= \alpha + \frac{2}{3}(\beta-\alpha) 
= \frac{\alpha+2\beta}{3}.
\end{align*}
\end{solution}

\begin{exercise}
\label{exo:strangeharm}
Set $h_n = 1 + \tfrac{1}{2} + \cdots + \tfrac{1}{n}$ for every $n\geq
1$, which yields the sequence of \emph{harmonic
numbers}\index{Harmonic numbers}
$(h_n)_\nnn$. 
Show that
\begin{equation*}
\sum_{k=2}^{n} \frac{1}{k(k-1)}h_k = 2 - \frac{1}{n+1} -
\frac{h_{n+1}}{n}, \quad\text{for every $n\geq 1$.}
\end{equation*}
\end{exercise}
\begin{solution}
The statement to prove is 
\begin{equation*}
\sum_{k=2}^{n} \frac{1}{k(k-1)}h_k = 2 - \frac{1}{n+1} -
\frac{h_{n+1}}{n}, \quad\text{for every $n\geq 1$.}
\end{equation*}
Denote this statement by $S(n)$.

\textbf{Base Case}:
When $n=1$, then $S(n)$ turns into 
$0 = 2 - \tfrac{1}{1+1} - \tfrac{3}{2}. \checkmark$

\textbf{Inductive Step}:
Suppose that $S(n)$ is true for some integer $n\geq 1$.
Then
\begin{align*}
\sum_{k=2}^{n+1} \frac{1}{(k-1)k}h_k &= 
\bigg(\sum_{k=2}^{n} \frac{1}{(k-1)k} {h_k}\bigg) +
\frac{1}{(n+1)n}h_{n+1}
= 2 - \frac{1}{n+1} - \frac{h_{n+1}}{n} + \frac{1}{(n+1)n}h_{n+1}\\
&= 2 - \frac{1}{n+1} - \bigg(\frac{1}{n}-\frac{1}{(n+1)n}\bigg){h_{n+1}}
= 2 - \frac{1}{n+1} - \frac{1}{n+1}{h_{n+1}}\\
&= 2 - \frac{1}{n+1} - \frac{1}{n+1}\bigg(h_{n+2}-\frac{1}{n+2}\bigg)
= 2 - \frac{1}{n+1}\bigg(1-\frac{1}{n+2}\bigg) - \frac{1}{n+1}h_{n+2}\\
&= 2 - \frac{1}{n+2} - \frac{1}{n+1}h_{n+2}
= 2 - \frac{1}{(n+1)+1} - \frac{h_{(n+1)+1}}{n+1},
\end{align*}
and, by the Principle of Mathematical Induction, the identity thus
holds for all $n\geq 1$.
\end{solution}

\begin{exercise}
\label{exo:Binet}
Use strong induction to prove that the sequence of Fibonacci numbers
$(f_n)_\nnn$ considered in
Example~\ref{ex:seq}\ref{ex:seq:fibo}
satisfies 
\begin{equation*}
\ww f_n = \left(\frac{1+\ww}{2}\right)^n -
\left(\frac{1-\ww}{2}\right)^n,
\end{equation*}
for every $\nnn$. 
\end{exercise}
\begin{solution}
\textbf{Base Case:}
When $n=0$, the identity turns into $0=0$, which is clearly true.
When $n=1$, the identity turns into
$\ww\cdot 1 = \ww = (1+\ww)/2 - (1-\ww)/2$, which is again true. 

\textbf{Inductive Step:}
Suppose the identity holds for some $n-1$ and $n$, where $n\geq
1$. 
Then, using that
\begin{equation*}
\left(\frac{1+\ww}{2}\right)^2 = 
\frac{1+5+2\ww}{4} = \frac{3+\ww}{2}
\end{equation*}
and that
\begin{equation*}
\left(\frac{1-\ww}{2}\right)^2 = 
\frac{1+5-2\ww}{4} = \frac{3-\ww}{2}, 
\end{equation*}
we obtain 
\begin{align*}
\ww f_{n+1} &= \ww (f_{n-1}+f_{n}) 
=\ww f_{n-1} + \ww f_n\\
&= 
\left(\frac{1+\ww}{2}\right)^{n-1} -
\left(\frac{1-\ww}{2}\right)^{n-1}
+ \left(\frac{1+\ww}{2}\right)^n -
\left(\frac{1-\ww}{2}\right)^n\\
&=\left(\frac{1+\ww}{2}\right)^{n-1}
\left(1 + \frac{1+\ww}{2}\right)-
\left(\frac{1-\ww}{2}\right)^{n-1}
\left(1 + \frac{1-\ww}{2}\right)\\
&=\left(\frac{1+\ww}{2}\right)^{n-1}
\left(\frac{3+\ww}{2}\right)-
\left(\frac{1-\ww}{2}\right)^{n-1}
\left(\frac{3-\ww}{2}\right)\\
&=\left(\frac{1+\ww}{2}\right)^{n-1}
\left(\frac{1+\ww}{2}\right)^2-
\left(\frac{1-\ww}{2}\right)^{n-1}
\left(\frac{1-\ww}{2}\right)^2\\
&=\left(\frac{1+\ww}{2}\right)^{n+1}
-\left(\frac{1-\ww}{2}\right)^{n+1},
\end{align*}
as required. This completes the proof of the inductive step.

Therefore, by the Principle of Strong Mathematical Induction,
the formula is proven. 
\end{solution}

\begin{exercise}
\label{exo:Binet2}
Explain why for all $n$ sufficiently large, 
we can compute the Fibonacci numbers 
$(f_n)_\nnn$ 
(see Example~\ref{ex:seq}\ref{ex:seq:fibo})
by 
\begin{equation*}
 f_n = \operatorname{round}\left(\frac{1}{\ww}\Big(\frac{1+\ww}{2}\Big)^n\right)
\end{equation*}
where $\operatorname{round}(x)$ returns the nearest integer to
$x$. 
\emph{Hint:} Exercise~\ref{exo:Binet}
\end{exercise}
\begin{solution}
By Exercise~\ref{exo:Binet}, we have, for every $\nnn$, 
\begin{equation*}
f_n = \frac{1}{\ww}\left(\frac{1+\ww}{2}\right)^n -
\frac{1}{\ww}\left(\frac{1-\ww}{2}\right)^n. 
\end{equation*}
Now note that $(1-\ww)^2 = 1 -2\ww+5 = 6 - 2\ww < 4$ because
$\sqrt{5}>1$. 
It follows that
\begin{equation*}
\left| \frac{1-\ww}{2}\right| < 1.
\end{equation*}
Hence it follows from 
Example~\ref{ex:geoseq} (geometric sequence) and
Corollary~\ref{c:c-mlaw} (constant-multiple law) that 
\begin{equation*}
-\frac{1}{\ww}\left(\frac{1-\ww}{2}\right)^n \to 0.
\end{equation*}
Thus,
\begin{equation*}
f_n = \frac{1}{\ww}\left(\frac{1+\ww}{2}\right)^n 
+ \text{(a sequence converging to $0$)}.
\end{equation*}
Eventually, the sequence converging to $0$ will have absolute
value $<0.5$; from then onwards, since it is clear that all
Fibonacci numbers are \emph{integers}, the result will hold. 

\emph{Note:} This approximation is actually working for 
\emph{every} $\nnn$ --- try it with Maple or your Calculator!
\end{solution}

\begin{exercise}
\label{exo:IsaoPat}
Let $(a_n)_\nnn$ be a sequence in $\left]0,1\right[$ and define
a sequence $(b_n)_\nnn$ by
$b_0 = a_0$ and 
\begin{equation}
\label{e:140811b}
b_{n+1} = \frac{a_{n+1}b_n}{a_{n+1}+b_n - a_{n+1}b_n}, 
\quad
\text{for every $n\geq 1$.}
\end{equation}
Show that
\begin{equation}
\label{e:140811a}
b_n = \frac{1}{\big(\sum_{i=0}^n 1/a_i\big)-n},
\quad\text{for every $n\in\NN$.}
\end{equation}
\end{exercise}
\begin{solution}
We show this by induction. 

\textbf{Base Case}:
When $n=0$, then \eqref{e:140811a} reads
$b_0 = 1/(1/a_0 - 0) = a_0.~ \checkmark$

\textbf{Inductive Step}:
Suppose that \eqref{e:140811a} is true for some integer $n\in\NN$. 
Then
\begin{equation*}
\sum_{i=0}^n \frac{1}{a_i} = n + \frac{1}{b_n}.
\end{equation*}
and hence
\begin{subequations}
\label{e:140811c}
\begin{align}
\sum_{i=0}^{n+1} \frac{1}{a_i}
&= \left(\sum_{i=0}^n \frac{1}{a_i}\right) + \frac{1}{a_{n+1}}\\
&=  n + \frac{1}{b_n} +  \frac{1}{a_{n+1}}.
\end{align}
\end{subequations}
On the other hand,
by \eqref{e:140811b}, 
\begin{equation}
\label{e:140811d}
\frac{1}{b_{n+1}} = \frac{a_{n+1}+b_n - a_{n+1}b_n}{a_{n+1}b_n} = 
\frac{1}{b_n}+\frac{1}{a_{n+1}} - 1.
\end{equation}
Combining 
\eqref{e:140811c} and \eqref{e:140811d}, we obtain 
\begin{equation*}
\sum_{i=0}^{n+1} \frac{1}{a_i}
= n + 1 + \frac{1}{b_{n+1}}.
\end{equation*}
Solving this for $b_{n+1}$ is precisely 
\eqref{e:140811a} with $n$ replaced by $n+1$.

Therefore, by the Principle of Mathematical Induction,
the identity holds for every $n\in\NN$.
\end{solution}

\begin{exercise}
\label{exo:IsaoPat2}
Consider Exercise~\ref{exo:IsaoPat} and
assume that $a_n = (1/2)^n$. 
Find (i) a simple formula for $b_n$
and (ii) $\lim_\nnn b_n$. 
\end{exercise}
\begin{solution}
Let $\alpha \in\left]0,1\right[$ 
and assume that $a_n = \alpha^n$ for every $\nnn$.
We have 
\begin{equation*}
\sum_{i=0}^n 1/a_i = \sum_{i=0}^n (1/\alpha)^i = 
\frac{1-(1/\alpha)^{n+1}}{1-(1/\alpha)}
\end{equation*}
by Theorem~\ref{t:geosum}. 
Hence
\begin{equation*}
b_n = \frac{1}{\frac{1-(1/\alpha)^{n+1}}{1-(1/\alpha)}
 - n}
\end{equation*}
In our present question, $\alpha = 1/2$ and so $1/\alpha=2$.
We thus obtain
\begin{equation*}
b_n = \frac{1}{\frac{1-2^{n+1}}{1-2} - n} = \frac{1}{2^{n+1}-(n+1)}, 
\end{equation*}
which gives (i). 
Next, the Binomial Theorem yields 
\begin{equation*}
2^n = (1+1)^n = \sum_{k=0}^n {n \choose k} \geq {n \choose 1} + {n \choose 2} 
= n + n(n-1)/2. 
\end{equation*}
Using this for $n+1$ instead of $n$, we obtain
$ 2^{n+1} - (n+1) \geq n(n+1)/2$ 
and hence that 
\begin{equation*}
0 \leq b_n = \frac{1}{2^{n+1}-(n+1)} \leq \frac{1}{ n(n+1)/2} =
\frac{2}{n(n+1)}
\to 0, 
\end{equation*}
which gives (ii). 
\end{solution}

\begin{exercise}
\label{exo:absantoabsa}
Prove Proposition~\ref{p:absantoabsa}.
\end{exercise}
\begin{solution}
Let $\varepsilon>0$. 
Since $a_n\to\alpha$, by definition, there exists $N\in\NN$ such that
$(\forall n\geq N)$ $|a_n-\alpha|<\varepsilon$. 
On the other hand, 
$(\forall n\geq N)$ $||a_n|-|\alpha||\leq |a_n-\alpha|$
by Corollary~\ref{c:triangle}. 
Altogether,
$(\forall n\geq N)$ $||a_n|-|\alpha|| <\varepsilon$
which shows that $|a_n|\to|\alpha|$. 
\end{solution}

\begin{exercise}
Let $(a_n)_\nnn$ be a sequence in $\RR$ such that $a_n\to\alpha$. 
Show that $\max\{a_n,0\}\to\max\{\alpha,0\}$. 
\emph{Hint:} Lemma~\ref{l:maxmin} and Proposition~\ref{p:absantoabsa}.
\end{exercise}
\begin{solution}
Lemma~\ref{l:maxmin} implies that
$$(\forall x\in\RR)\quad 
\max\{x,0\} = \frac{x+|x|}{2}.$$
Using the limit laws (Theorem~\ref{t:sumlaw} and
Corollary~\ref{c:c-mlaw}) along with
Proposition~\ref{p:absantoabsa},  we see that
\begin{equation*}
\max\{a_n,0\} = \frac{a_n+|a_n|}{2} \to \frac{\alpha+|\alpha|}{2}
= \max\{\alpha,0\},
\end{equation*}
as claimed. 
\end{solution}

\begin{exercise}
Let $(a_n)_\nnn$ and $(b_n)_\nnn$ be sequences in $\RR$ such that
$a_n\to\alpha$ and $b_n\to\beta$.
Show that $\max\{a_n,b_n\}\to\max\{\alpha,\beta\}$. 
\emph{Hint:} Lemma~\ref{l:maxmin} and Proposition~\ref{p:absantoabsa}.
\end{exercise}
\begin{solution}
Lemma~\ref{l:maxmin} implies that
$$(\forall x\in\RR)(\forall y\in\RR) \quad 
\max\{x,y\} = \frac{x+y+|x-y|}{2}.$$
Using the limit laws (Theorem~\ref{t:sumlaw} and
Corollary~\ref{c:c-mlaw}) along with
Proposition~\ref{p:absantoabsa},  we see that
\begin{equation*}
\max\{a_n,b_n\} = \frac{a_n+b_n+|a_n-b_n|}{2} \to \frac{\alpha+\beta+|\alpha-\beta|}{2}
= \max\{\alpha,\beta\},
\end{equation*}
as claimed. 
\end{solution}

\begin{exercise}[more on the logistic sequence]
\index{logistic sequence}
Let $r\in[0,1]$,
let $x_0\in [0,1]$, and define the logistic sequence via 
\begin{equation*}
x_{n+1} = r(1-x_n)x_n
\end{equation*}
for every $\nnn$. 
Show that $x_n\to 0$. 
\emph{Hint:}
Discuss cases. The hard case turns out to be easy if 
we use mathematical induction and 
Exercise~\ref{exo:200724c}.
\end{exercise}
\begin{solution}
An easy mathematical induction gives $x_n\in[0,1]$ for every $\nnn$.

If $r=0$, then clearly $(x_n)_\nnn = (x_0,0,0,0,\ldots)$ converges to $0$.

So assume that $0<r\leq 1$.
If $x_0=0$, then $(x_n)_\nnn = (0,0,0,\ldots)$ clearly converges to $0$.
So we also assume that $x_0>0$.
We show by mathematical induction that 
\begin{equation}
\label{e:logseq1}
x_n \leq \frac{x_0}{nx_0+r^{-n}}
\end{equation}
for every $\nnn$. 
The base case is clear since 
\begin{equation*}
x_0 = \frac{x_0}{0+1} = \frac{x_0}{0\cdot x_0 + r^{-0}}.
\end{equation*}
Now assume that for some integer $n\geq 0$, we have 
\begin{equation*}
x_n \leq \frac{x_0}{nx_0+r^{-n}}.
\end{equation*}
By Exercise~\ref{exo:200724c}, 
\begin{equation*}
x_{n+1} \leq \frac{x_0}{(n+1)x_0+r^{-(n+1)}}.
\end{equation*}
This completes the inductive step and therefore 
\eqref{e:logseq1} is proven. 
Note that \eqref{e:logseq1} implies 
\begin{equation*}
0\leq x_n \leq \frac{x_0}{nx_0+r^{-n}} \leq \frac{x_0}{nx_0} = \frac{1}{n}\to 0.
\end{equation*}
Hence $(x_n)_\nnn$ converges to $0$ by the Squeeze Theorem
(Theorem~\ref{t:squeeze}).
\end{solution}

\begin{exercise}[even more on the logistic sequence]
\index{logistic sequence}
Let $x_0\in \RR$, such that $|x_0|\geq 4$ and define the logistic sequence 
(with $r=1$) via 
\begin{equation*}
x_{n+1} = (1-x_n)x_n.
\end{equation*}
Show that $x_n\to -\infty$. 
\emph{Hint:}
Exercise~\ref{exo:200724d}.
\end{exercise}
\begin{solution}
First, let us prove by induction that 
\begin{equation}
\label{e:200724rain}
x_n \leq -4\cdot 3^n
\end{equation}
for every integer $n\geq 1$.

By Exercise~\ref{exo:200724d}, we have 
\begin{equation*}
x_1 = (1-x_0)x_0 = -x_0^2 + x_0 \leq -3|x_0|\leq -3\cdot 4=-4\cdot 3^1
\end{equation*}
and so the base case holds.

Now assume that $n\geq 1$ is an integer such that 
$x_n \leq -4\cdot 3^n$. 
Then $x_n<0$ and hence 
$|x_n| = -x_n \geq 4\cdot 3^n \geq 4$ and 
Exercise~\ref{exo:200724d} now yields
\begin{align*}
x_{n+1} &= (1-x_n)x_n = -x_n^2 + x_n \leq -3|x_n|
= -3(-x_n)\\
&= 3x_n \leq 3\big(-4\cdot 3^n) = -4\cdot 3^{n+1}
\end{align*}
which completes the proof of the inductive step. 

We have verified \eqref{e:200724rain} and the result now follows.
\end{solution}

\begin{exercise}[YOU be the marker!] 
Consider the following statement 
\begin{equation*}
\text{``If $(a_n)_\nnn$ is a sequence and $\alpha\in\RR$ such that $|a_n|\to|\alpha|$, 
then $a_n\to\alpha$.''}
\end{equation*}
and the following ``proof'':
\begin{quotation}
Let $\varepsilon>0$.\\
Because $|a_n|\to|\alpha|$,
there exists $N\in\NN$ such that 
$n\geq N$ 
$\Rightarrow$ 
$\big||a_n|-|\alpha|\big|<\varepsilon$.\\
It follows that 
$|a_n-\alpha|\leq \big||a_n|-|\alpha|\big|<\varepsilon$
and hence $a_n\to\alpha$. 
\end{quotation}
Why is this proof wrong?
\end{exercise}
\begin{solution}
The last line is wrong: We don't generally have 
$|a_n-\alpha|\leq \big||a_n|-|\alpha|\big|$,
the inequality goes the wrong way!

A counterexample to the statement is the sequence $(a_n)_\nnn= ((-1)^n)_\nnn$ 
and $\alpha =1$.
\end{solution}

\begin{exercise}[TRUE or FALSE?]
Mark each of the following statements as either true or false. 
Briefly justify your answer.
\begin{enumerate}
\item ``If $a_n\to \alpha$, $b_n\to\beta$, and $a_n<b_n$, then 
$\alpha<\beta$.''
\item ``If $(a_n)_\nnn$ is a convergent sequence, 
then $\sum_{\nnn}a_n$ is a convergent series.''
\item ``If $p(x)$ is a polynomial and $a_n\to\alpha$, 
then $p(a_n)\to p(\alpha)$.''
\item ``If $a_n\to\alpha\neq 0$, then $(a_n/|a_n|)_\nnn$ converges.''
\end{enumerate}
\end{exercise}
\begin{solution}
(i): FALSE: Consider $a_n=0\to 0=\alpha$ and $b_n=1/2^n\to\beta =0$. 

(ii): FALSE: $a_n\equiv 1\to 1$ but $\sum_\nnn a_n$ diverges to $+\infty$.

(iii): TRUE: This follows from the limit laws. 

(iv): TRUE: Suppose first that $\alpha>0$.
Then eventually $a_n$ is positive and so $a_n/|a_n|=1$.
Hence $a_n/|a_n|\to 1$ in this case. 
And if $\alpha<0$, then $a_n/|a_n|\to -1$. 
\end{solution} 
\chapter{The Completeness Axiom}

\section{The Trouble with $\sqrt{2}$}

Even the Greeks knew already that $\sqrt{2}\notin\mathbb{Q}$
--- the proof is a gorgeous example of the 
proof-by-contradiction
technique.

\begin{theorem}[$\sqrt{2}$ is irrational]
\label{t:squareroot2}\index{irrational}
There is no rational number $x$ such that $x^2 = 2$.
\end{theorem}
\begin{proof}\index{Proof by contradiction (Example)}
Suppose to the contrary that there exists $x\in\mathbb{Q}$ such that
$x^2=2$, say
\begin{equation}
x = \frac{m}{n}, 
\quad
\text{where $m\in\ZZ$ and $n\in\ZZ\smallsetminus\{0\}$}
\quad
\text{and}
\quad
x^2 =2,
\end{equation}
and where we assume that the fraction is in lowest terms 
(all common factors except for $\pm 1$ are canceled out). 
Then $2 = x^2 = m^2/n^2$ and so 
\begin{equation}
2n^2 = m^2.
\end{equation}
This means that $m^2$ is \emph{even}. 
Therefore, $m$ is even (see Exercise~\ref{exo:evensquare}), say
\begin{equation}
m= 2k, \quad \text{for some integer $k$.}
\end{equation}
But then the equality $2n^2 = m^2$ turns into
$2n^2 = (2k)^2 = 4k^2$.
Therefore, $n^2= 2k^2$ and so $n^2$ is even.
Again, we deduce that $n$ is even, say 
\begin{equation}
n= 2l, \quad \text{for some integer $l$.}
\end{equation}
But now $m$ and $n$ do have a common factor, namely $2$ --- contradiction!
\end{proof}

Hence, the existence of $\sqrt{2}$ does not follow
from any of the previous axioms for a field.
One last axiom is needed, the so-called completeness axiom.
Its precise formulation requires a new notion for sequences.

\section{Cauchy Sequences and the Completeness Axiom}

\begin{definition}[Cauchy sequence]
Let $(a_n)_\nnn$ be a sequence of real numbers.
Then $(a_n)_\nnn$ is a \textbf{Cauchy sequence} if
\index{Cauchy sequence}
\begin{quotation}
\noindent
for every $\varepsilon>0$, there exists $N\in\NN$ such that\\
for all $m,n$ in $\NN$\quad
[$m\geq N$ and $n\geq N$]\;\;
$\Rightarrow$ \;\; $|a_n-a_m| < \varepsilon$.
\end{quotation}
\end{definition}

\begin{theorem}
\label{t:convCauchy}
Every convergent sequence is also a Cauchy sequence.
\end{theorem}
\begin{proof}
Let $(a_n)_\nnn$ be a convergent sequence, with limit $\ell$.
Let $\varepsilon > 0$.
Then, since $\varepsilon/2>0$,
there exists $N\in\NN$ such that 
\begin{equation}
|a_n-\ell| < \frac{\varepsilon}{2},
\quad\text{for all $n\geq N$.}
\end{equation}
It follows that for all $n\geq N$ and $m\geq N$, we have
\begin{align*}
|a_n-a_m| &= |(a_n-\ell) + (\ell-a_m)|\\
&\leq |a_n-\ell| + |\ell-a_m| \\
&=|a_n-\ell|+|a_m-\ell|\\
&< \frac{\varepsilon}{2} + \frac{\varepsilon}{2},
\end{align*}
as required. 
\end{proof}

\begin{definition}[Completeness Axiom]
\label{d:compaxiom}\index{Completeness Axiom}
The real numbers are \textbf{complete}, i.e.,
every Cauchy sequence in $\RR$ is actually convergent.
\end{definition}

This is an incredibly important property.
It encodes convergence without mentioning the limit,
which is crucial in applications where one typically
does not know beforehand what the solution to a problem is.
We mention in passing that it can be shown
that every complete ordered field with the archimedean property is
essentially indistinguishable from the real numbers.
Furthermore, one can construct the real numbers as
equivalence classes of Cauchy sequences of rational numbers
but this is a fairly lengthy process.

\section{Series Expansions of Real Numbers}

As a first application of the completeness axiom, 
we see that the decimal expansions lead to real numbers.
In fact, this works for any base, not just the usual base $10$.

\begin{theorem}
\label{t:seriesnumber}
Let $b$ be an integer such that $b\geq 2$,
and let $k\in\NN$.
Suppose that for all $n\geq -k$, we have 
$a_n\in\{0,1,\ldots,b-1\}$, i.e., the $a_n$ are the digits with respect to
the base $b$.
Then \index{Decimal expansion of real numbers}
\begin{equation}
\pm \sum_{n=-k}^\infty a_n b^{-n}
\end{equation}
is a Cauchy sequence and hence converges to a real number.
\end{theorem}
\begin{proof}
Define the partial sums
\begin{equation}
s_n := \sum_{l=-k}^n a_l b^{-l},
\quad\text{for every integer $n\geq -k$.}
\end{equation}
We shall show that $(s_n)_{n\geq -k}$ is a Cauchy sequence.
Let $\varepsilon > 0$.
Since $0<1/b<1$, we have $1/b^n\to 0$ by
Example~\ref{ex:geoseq}\ref{ex:geoseq:i}. 
Hence, there exists $N\in \NN$ such that
\begin{equation}
\label{e:101028:a}
\text{
for all $n\geq N$, we have $b^{-n}<\varepsilon$.
}
\end{equation}
Now let $n\geq m\geq N\geq -k$. 
Then, using Theorem~\ref{t:geosum} and \eqref{e:101028:a},
\begin{subequations}
\begin{align}
|s_n-s_m| &= \left| \sum_{l=m+1}^{n}a_lb^{-l} \right|
=\sum_{l=m+1}^{n} a_lb^{-l}
\leq \sum_{l=m+1}^{n} (b-1)b^{-l}
= \sum_{l=0}^{n-m-1} (b-1)b^{-(l+m+1)}\\
&= (b-1)b^{-m-1}\sum_{l=0}^{n-m-1} (1/b)^l
= (b-1)b^{-m-1}\frac{1-(1/b)^{n-m}}{1-(1/b)}\\
&<(b-1)b^{-m-1}\frac{1}{1-(1/b)}
=(b-1)b^{-m}\frac{1}{b\big(1-(1/b)\big)}
= b^{-m}\\
&<\varepsilon,
\end{align}
\end{subequations}
as required.
\end{proof}

\begin{remark}
If $b=10$, we get the familiar decimal expansions.
Popular in Computer Science are binary expansions, i.e., $b=2$.
The Babylonians liked to work with $b=60$, because
60 is divisible by $2,3,4,5,6,10,12,15,20,30$. 
\end{remark}

The converse of the previous result is also true.

\begin{theorem}
\label{t:numberseries}
Let $b$ be an integer such that $b\geq 2$.
Then every real number has a series expansion with respect to base $b$.
\end{theorem}
\begin{proof}
The proof is algorithmic, we show how to do it.
Let $x\geq 0$ (otherwise, apply what we do next to $-x$ and then take the
negative). 
Let $k\geq 0$ be the smallest integer such that
\begin{equation}
\label{e:101018a}
0 \leq x < b^{k+1};
\end{equation}
note that by Theorem~\ref{t:100920}, $k$ does exist.
We now construct a sequence of digits $(a_l)_{l\geq -k}$ 
in $\{0,1,\ldots,b-1\}$ such that for every $n\geq -k$ and
\begin{equation}
\label{e:191024a}
x_n := \sum_{l=-k}^n a_l b^{-l},
\end{equation}
we have
\begin{equation}
\label{e:101018c}
x_n \leq x < x_n + b^{-n}.
\end{equation}

\textbf{Base Case}:
Consider the inequalities
\begin{equation}
0 = 0\cdot b^{k} < 1\cdot b^k < \cdots < (b-1)\cdot b^k < b\cdot b^k=
b^{k+1},
\end{equation}
which give rise to a partition of the interval $\left[0,b^{k+1}\right[$
into $b$ subintervals
$\left[0,b^k\right[$,
$\left[b^k,2b^k\right[$,
\ldots,
$\left[(b-1)b^k,b^{k+1}\right[$.
In view of \eqref{e:101018a},
there exists a unique integer $a_{-k}\in\{0,1,\ldots,b-1\}$ such that
\begin{equation}
x_{-k} = a_{-k}b^k \leq x < (a_{-k}+1)b^k = x_{-k} + b^k.
\end{equation}

\textbf{Inductive Step}:
Suppose that for some $n\geq -k$ the numbers 
$a_{-k},\ldots,a_n$ are already constructed such that 
\eqref{e:191024a} holds and 
\begin{equation}
\label{e:101018bb}
x_n \leq x < x_n + b^{-n}.
\end{equation}
We consider the inequalities
\begin{equation}
\label{e:101018b}
x_n < x_n + b^{-n-1} < x_n + 2b^{-n-1} < \cdots <
x_n+ b b^{-n-1} = x_n + b^{-n},
\end{equation}
which 
give rise to a partition of the interval 
$\left[x_n,x_n+b^{-n}\right[$
into $b$ subintervals
\begin{subequations}
\begin{align}
&\left[x_n,x_n+b^{-n-1}\right[,\\
&\left[x_n+b^{-n-1},x_n+2b^{-n-1}\right[,\\
&\qquad \vdots\\
&\left[x_n + (b-1)b^{-n-1},x_n+b^{-n}\right[.
\end{align}
\end{subequations}
As before, in view of \eqref{e:101018bb},
there exists a unique integer $a_{n+1}\in\{0,1,\ldots,b-1\}$
such that
\begin{equation}
x_n+a_{n+1}b^{-n-1}\leq x < x_n+ a_{n+1}b^{-n-1}+ b^{-n-1}.
\end{equation}
We now set $x_{n+1} = x_n + a_{n+1}b^{-n-1}$, which yields the
required. 

Therefore, by the principlde of mathematical induction, 
\eqref{e:101018c} is true for all $n\geq -k$, and
we clearly have
\begin{equation}
|x_n-x| = |x-x_n| = x-x_n < b^{-n}.
\end{equation}
Hence, $\lim_n x_n = x$ and thus
\begin{equation}
x = \sum_{l=-k}^\infty a_l b^{-l},
\end{equation}
which is the announced series expansion with respect to $b$.
\end{proof}

\begin{remark}
The last two results allow us to identify real numbers
with the usual (infinite) series expansions.
Note that a consequence of Theorem~\ref{t:numberseries}
is that every real number can be arbitrarily well approximated
by rational numbers (since the partial sums of the expansions
are all rational numbers).
\end{remark}

\section{Subsequences and the Bolzano-Weierstrass Theorem}

\begin{definition}[subsequence]
Let $(a_n)_\nnn$ be a sequence of real numbers, and let
\begin{equation}
n_0 < n_1 < n_2 < \cdots < n_k < n_{k+1} < \cdots
\end{equation}
be a strictly increasing sequence of nonnegative integers.
Then the sequence \index{Subsequence}
\begin{equation}
(a_{n_k})_{k\in\NN} = \big( a_{n_0}, a_{n_1}, a_{n_2},\ldots\big)
\end{equation}
is a \textbf{subsequence} of the sequence $(a_n)_\nnn$.
\end{definition}

\begin{proposition}
\label{p:allgo}
Let $(a_n)_\nnn$ be a sequence in $\RR$ such that $a_n\to\ell$,
and let $(a_{n_k})_{k\in\NN}$ be a subsequence of $(a_n)_\nnn$.
Then $(a_{n_k})_{k\in\NN}$ also converges to $\ell$.
\end{proposition}
\begin{proof}
Let $\varepsilon > 0$.
Since $a_n\to \ell$, there exists $N\in\NN$ such that
\begin{equation}
k\geq N \quad \Rightarrow 
\quad |a_k-\ell| < \varepsilon.
\end{equation}
It's clear that for every $k\in\NN$, $k\leq n_k$.
Hence if $k\geq N$, then we have $n_k \geq k\geq N$ and therefore 
$|a_{n_k}-\ell|<\varepsilon$.
\end{proof}

\begin{theorem}[Bolzano-Weierstrass]
\label{t:BW}\index{Bolzano-Weierstrass Theorem}
Let $(x_n)_\nnn$ be a bounded sequence of real numbers. 
Then $(x_n)_\nnn$ possesses a convergent subsequence.
\end{theorem}
\begin{proof}
Since $(x_n)_\nnn$ is bounded, 
there exist $\alpha$ and $\beta$ in $\RR$ such that
\begin{equation}
\alpha \leq x_n \leq \beta,
\quad\text{for all $n\in\NN$.}
\end{equation}
Now consider the interval
\begin{equation}
[\alpha_0,\beta_0] := 
[\alpha,\beta] := \menge{x\in\RR}{\alpha\leq x\leq \beta}.
\end{equation}
Using induction on $k$, we shall construct a sequence of (nested)
intervals $\big([\alpha_k,\beta_k]\big)_{k\in\NN}$ such that
the following hold, for every $k\in\NN$:
\begin{enumerate}
\item
\label{proof:bw:i}
$[\alpha_k,\beta_k]$ contains infinitely many terms of
the sequence $(x_n)_\nnn$.
\item
\label{proof:bw:ii}
$[\alpha_k,\beta_k] \subseteq [\alpha_{k-1},\beta_{k-1}]$
for $k\geq 1$.
\item
\label{proof:bw:iii}
$\beta_k-\alpha_k = 2^{-k}(\beta-\alpha)$.
\end{enumerate}

\textbf{Base Case:}
When $k=0$, we set $\alpha_0=\alpha$ and $\beta_0=\beta$.
Clearly, \ref{proof:bw:i} and \ref{proof:bw:iii} hold
(and we don't have to worry about \ref{proof:bw:ii}, since this item
is claimed only when $k\geq 1$).

\textbf{Inductive Step:}
Suppose that $[\alpha_k,\beta_k]$ with properties
\ref{proof:bw:i}--\ref{proof:bw:iii} is constructed, for some
$k\in\NN$.
Set $\mu := (\alpha_k+\beta_k)/2$, which is the midpoint of the
interval  $[\alpha_k,\beta_k]$.
Since  $[\alpha_k,\beta_k]$ contains infinitely many terms of
$(x_n)_\nnn$, at least one of the intervals
$[\alpha_k,\mu]$ and $[\mu,\beta_k]$ must contain infinitely
many terms of the sequence as well. We now set
\begin{equation}
[\alpha_{k+1},\beta_{k+1}] 
:=
\begin{cases}
[\alpha_k,\mu], &\text{if $[\alpha_k,\mu]$ contains infinitely many
terms of $(x_n)_\nnn$;}\\
[\mu,\beta_k], &\text{otherwise.}
\end{cases}
\end{equation}
It is clear that $[\alpha_{k+1},\beta_{k+1}]$
satisfies \ref{proof:bw:i}--\ref{proof:bw:iii}
with $k$ replaced by $k+1$. 

Having constructed the intervals, we now construct
--- again by induction ---
a subsequence $(x_{n_k})_{k\in\NN}$ of $(x_n)_\nnn$
such that $x_{n_k}\in[\alpha_k,\beta_k]$, for every $k\in\NN$.

\textbf{Base Case:}
Set $x_{n_0} = x_0$.
Since $\alpha\leq x_0\leq \beta$, we have $x_{n_0}\in
[\alpha_0,\beta_0]$.

\textbf{Inductive Step:}
Suppose $x_{n_0},\ldots,x_{n_k}$ are already constructed
with the desired properties for some $k\in\NN$. 
Since $[\alpha_{k+1},\beta_{k+1}]$ contains infinitely many terms
of the sequence $(x_n)_\nnn$, there exists
$n>n_k$ such that $x_n\in [\alpha_{k+1},\beta_{k+1}]$.
Set $n_{k+1} := n$. Then $x_{n_{k+1}}\in
[\alpha_{k+1},\beta_{k+1}]$ as required.

We shall now show that $(x_{n_k})_{k\in\NN}$ is convergent.
It suffices to show that  $(x_{n_k})_{k\in\NN}$ 
is a Cauchy sequence.
Let $\varepsilon > 0$. Since $1/2^n\to 0$, 
there exists $N\in\NN$ such that 
\begin{equation}
\frac{\beta-\alpha}{2^N} < \varepsilon.
\end{equation}
Now let $k$ and $l$ be integers $\geq N$.
Then
\begin{equation}
x_{n_k} \in [\alpha_k,\beta_k] \subseteq [\alpha_N,\beta_N]
\quad\text{and}\quad
x_{n_l} \in [\alpha_l,\beta_l] \subseteq [\alpha_N,\beta_N].
\end{equation}
Therefore,
$|x_{n_k}-x_{n_l}| \leq $ length of the interval $[\alpha_N,\beta_N]$
$= 2^{-N}(\beta-\alpha) < \varepsilon$. 
\end{proof}

%
%

\begin{definition}[cluster point]
Let $(a_n)_\nnn$ be a sequence, and let $c\in\RR$.
Then $c$ is a \textbf{cluster point}\footnote{Also known as
\textbf{accumulation point} or \textbf{limit point}of $(a_n)_\nnn$.} of $(a_n)_\nnn$ if some
subsequence of $(a_n)_\nnn$ converges to $c$.
\index{Cluster point}\index{Accumulation point}
\end{definition}

The Bolzano-Weierstrass theorem thus states that
every bounded sequence has a cluster point.

\begin{example}
\label{ex:101102:1}
The sequence $(a_n)_\nnn = ((-1)^n)_\nnn$ possesses two cluster points:
$-1$ and $+1$ since clearly
$a_{2k} = (-1)^{2k} = 1\to 1$
and 
$a_{2k+1} = (-1)^{2k+1} = -1\to -1$.
This sequence has no other cluster point.
\end{example}

\begin{example}
The sequence $(a_n)_\nnn = (n)_\nnn$ has no cluster point.
Indeed, $n\to\pinf$, and so does every subsequence of $(n)_\nnn$.
\end{example}

\begin{example}
The sequence $(a_n)_\nnn$ defined by
\begin{equation}
a_n := \begin{cases}
\frac{1}{n}, &\text{if $n$ is odd;}\\
n, &\text{if $n$ is even}
\end{cases}
\end{equation}
is unbounded (since $(a_{2k})_{k\in\NN} = (2k)_{k\in\NN}$ diverges
to $\pinf$); however $(a_n)_\nnn$ possesses exactly one cluster point,
namely $0$ (as $a_{2k+1} = 1/(2k+1)\to 0$). 
\end{example}

\section{Monotone Sequences}

\begin{definition}
Let $(a_n)_\nnn$ be a sequence. We say that:
\begin{enumerate}
\item $(a_n)_\nnn$ is \textbf{increasing},
if $a_n\leq a_{n+1}$ for every $\nnn$;\index{increasing sequence}
\item $(a_n)_\nnn$ is \textbf{strictly increasing},
if $a_n< a_{n+1}$ for every $\nnn$;\index{strictly increasing sequence} 
\item $(a_n)_\nnn$ is \textbf{decreasing},\index{decreasing sequence}
if $a_n\geq a_{n+1}$ for every $\nnn$;
\item $(a_n)_\nnn$ is \textbf{strictly decreasing},\index{strictly decreasing sequence} 
if $a_n> a_{n+1}$ for every $\nnn$.
\end{enumerate}
The sequence is \textbf{monotone}\index{Monotone sequence}
if it satisfies one of the above properties.
\end{definition}

The next result is fundamental since it provides a powerful
technique for obtaining convergent sequences.

\begin{theorem}[Bounded Monotone Convergence Theorem] 
\label{t:monconv}
Every bounded monotone sequence of real numbers converges.
\end{theorem}
\begin{proof}
Let $(a_n)_\nnn$ be a bounded monotone sequence,
say (without loss of generality) increasing.
By the Bolzano-Weierstrass theorem (Theorem~\ref{t:BW}),
$(a_n)_\nnn$ possesses a convergent subsequence,
say $a_{n_k}\to \ell$. We now show that the entire
sequence converges to $\ell$.

To this end, let $\varepsilon>0$.
Then there exists $K\in\NN$ such that
\begin{equation}
k\geq K
\quad\Rightarrow\quad
|a_{n_k}-\ell| < \varepsilon.
\end{equation}
Set $N:= n_{K}$.
For every integer $n\geq N$,
there exists an integer $k\geq K$ such that
$n_k \leq n< n_{k+1}$.
Since $(a_n)_\nnn$ is increasing, we have
\begin{equation}
\ell-\varepsilon < a_{n_k} \leq a_n \leq a_{n_{k+1}} <
\ell+\varepsilon;
\end{equation}
consequently,
$\ell-\varepsilon < a_n <\ell+\varepsilon$
$\Leftrightarrow$
$|a_n-\ell| < \varepsilon$.
\end{proof}

\section*{Exercises}\markright{Exercises}
\addcontentsline{toc}{section}{Exercises}
\setcounter{theorem}{0}


Recall that an integer $a$ is \emph{even} if $a=2b$, where $b\in\ZZ$;
and that $a$ is \emph{odd} if $a=2b+1$, where $b\in\ZZ$.

\begin{exercise}
Prove that if $m$ is an even integer, then $m^2$ is even as well.
\end{exercise}
\begin{solution}
Suppose that $m$ is an even integer, say $m=2k$, where $k\in\ZZ$.
Then $m^2 = (2k)^2 = 4k^2 = 2(2k^2)$.
Since $2k^2\in\ZZ$, it follows that $m^2$ is even.
\end{solution}

\begin{exercise}
Prove that if $m$ is an odd integer, then $m^2$ is odd as well.
\end{exercise}
\begin{solution}
Suppose that $m$ is an odd integer, say $m=2k+1$, where $k\in\ZZ$.
Then $m^2 = (2k+1)^2 = 4k^2 +4k+1= 2(2k^2+2k)+1$.
Since $2k^2+2k\in\ZZ$, it follows that $m^2$ is odd.
\end{solution}

\begin{exercise}
\label{exo:evensquare}
Prove that if $m$ is an integer such that $m^2$ is even,
then $m$ is even as well.
\emph{Hint}:  Prove the contrapositive. 
\end{exercise}
\begin{solution}
Suppose that $m$ is not even. Then $m$ is odd, say 
$m=2k+1$, for some integer $k\in\ZZ$.
Hence 
$m^2 = (2k+1)^2 = 4k^2 + 4k + 1 = 2(2k^2+2k) + 1$
is odd because $2k^2+2k\in\ZZ$, and therefore $m^2$ is not even.
\end{solution}

\begin{exercise}
\label{exo:oddsquare}
Prove that if $m$ is an integer such that $m^2$ is odd,
then $m$ is odd as well.
\emph{Hint}:  Prove the contrapositive. 
\end{exercise}
\begin{solution}
Suppose that $m$ is not odd. Then $m$ is even, say 
$m=2k$, for some integer $k\in\ZZ$.
Thus 
$m^2 = (2k)^2 = 4k^2 = 2(2k^2)$
is even because $2k^2\in\ZZ$, and this 
shows that $m^2$ is not odd.
\end{solution}

\begin{exercise}
\label{exo:130928c}
Prove that if $m$ is an integer such that
$m^2$ is divisible by $3$, then $m$ must be divisible by $3$.
\emph{Hint:} Prove the contrapositive.
\end{exercise}
\begin{solution}
Suppose that $m$ is not divisible by $3$.
Then we write $m$ as $m=3k\pm 1$, where $k\in\ZZ$.
It follows that
\begin{equation*}
m^2 = (3k\pm 1)^2 = (3k)^2 \pm 2\cdot 3k + 1
= 3\big( 3k^2 \pm 2k\big) + 1
\end{equation*}
is not divisible by $3$ (it has remainder $1$).
\end{solution}

\begin{exercise}
Show that there exists no rational number $x$ such that $x^2=3$.
\emph{Hint:} You may use Exercise~\ref{exo:130928c}. 
\end{exercise}
\begin{solution}
Suppose to the contrary that there exists $x\in\mathbb{Q}$ such that
$x^2=3$, say
\begin{equation*}
x = \frac{m}{n}, 
\quad
\text{where $m\in\ZZ$ and $n\in\ZZ\smallsetminus\{0\}$}
\quad
\text{and}
\quad
x^2 =3,
\end{equation*}
and where we assume that the fraction is in lowest terms 
(all common factors except for $\pm 1$ are canceled out). 
Then $3 = x^2 = m^2/n^2$ and so 
\begin{equation*}
3n^2 = m^2.
\end{equation*}
It follows from Exercise~\ref{exo:130928c} that
$m$ is divisible by $3$, say 
\begin{equation*}
m= 3k, \quad \text{for some integer $k$.}
\end{equation*}
But then the equality $3n^2 = m^2$ turns into
$3n^2 = (3k)^2 = 9k^2$.
Therefore, $n^2= 3k^2$.
Again, by Exercise~\ref{exo:130928c}, it follows that 
\begin{equation*}
n= 3l, \quad \text{for some integer $l$.}
\end{equation*}
But now $m$ and $n$ do have a nontrivial common factor, namely $3$ --- contradiction!
\end{solution}

\begin{exercise}
\label{exo:200811a} 
Prove that if $m$ is an integer such that
$m^2$ is divisible by $5$, then $m$ must be divisible by $5$.
\emph{Hint:} Prove the contrapositive.
\end{exercise}
\begin{solution}
Suppose that $m$ is not divisible by $5$.

\emph{Case~1:}
$m=5k\pm 1$, where $k\in\ZZ$.\\
It follows that
\begin{equation*}
m^2 = (5k\pm 1)^2 = (5k)^2 \pm 2\cdot 5k + 1
= 5\big( 5k^2 \pm 2k\big) + 1
\end{equation*}
is not divisible by $5$ (it has remainder $1$).

\emph{Case~2:}
$m=5k\pm 2$, where $k\in\ZZ$.\\
It follows that
\begin{equation*}
m^2 = (5k\pm 2)^2 = (5k)^2 \pm 4\cdot 5k + 4
= 5\big( 5k^2 \pm 4k\big) + 4
\end{equation*}
is not divisible by $5$ (it has remainder $4$).
\end{solution}

\begin{exercise}
\label{exo:200811b} 
Show that there exists no rational number $x$ such that $x^2=5$.
\emph{Hint:} You may use Exercise~\ref{exo:200811a}. 
\end{exercise}
\begin{solution}
Suppose to the contrary that there exists $x\in\mathbb{Q}$ such that
$x^2=5$, say
\begin{equation*}
x = \frac{m}{n}, 
\quad
\text{where $m\in\ZZ$ and $n\in\ZZ\smallsetminus\{0\}$}
\quad
\text{and}
\quad
x^2 =5,
\end{equation*}
and where we assume that the fraction is in lowest terms 
(all common factors except for $\pm 1$ are canceled out). 
Then $5 = x^2 = m^2/n^2$ and so 
\begin{equation*}
5n^2 = m^2.
\end{equation*}
It follows from Exercise~\ref{exo:200811a} that
$m$ is divisible by $5$, say 
\begin{equation*}
m= 5k, \quad \text{for some integer $k$.}
\end{equation*}
But then the equality $5n^2 = m^2$ turns into
$5n^2 = (5k)^2 = 25k^2$.
Therefore, $n^2= 5k^2$.
Again, by Exercise~\ref{exo:200811a}, it follows that 
\begin{equation*}
n= 5l, \quad \text{for some integer $l$.}
\end{equation*}
But now $m$ and $n$ do have a nontrivial 
common factor, namely $5$ --- contradiction!
\end{solution}

\begin{exercise}
Prove that $0.49999\ldots = 0.4\overline{9}$ is equal to $0.5$.
\end{exercise}
\begin{solution}
Using the geometric series, we compute
\begin{align*}
0.4\overline{9}&= \frac{4}{10} 
+ \frac{9}{100}\Big(1 + \frac{1}{10} + \frac{1}{100} + \cdots\Big)
=  \frac{4}{10} + \frac{9}{100}\frac{1}{1-\frac{1}{10}}
=  \frac{4}{10} + \frac{9}{100- 10}\\
&=  \frac{4}{10} + \frac{9}{90}
=  \frac{4}{10} + \frac{1}{10}
=\frac{5}{10}
=0.5,
\end{align*}
as claimed.
\end{solution}

\begin{exercise}
Prove that $0.9999\ldots = 0.\overline{9}$ is equal to $1$.
\end{exercise}
\begin{solution}
Using the Geometric Series, we compute
\begin{align*}
0.\overline{9}&= 
\frac{9}{10}\Big(1 + \frac{1}{10} + \frac{1}{100} + \cdots\Big)
=  \frac{9}{10}\frac{1}{1-\frac{1}{10}}
=  \frac{9}{10}\frac{1}{\frac{9}{10}}
=  1,
\end{align*}
as claimed.
\end{solution}

\begin{exercise} 
Prove that $0.3333\ldots = 0.\overline{3}$ is equal to $1/3$.
\end{exercise}
\begin{solution}
Using the Geometric Series, we compute
\begin{align*}
0.\overline{3}&= 
\frac{3}{10}\Big(1 + \frac{1}{10} + \frac{1}{100} + \cdots\Big)
=  \frac{3}{10}\frac{1}{1-\frac{1}{10}}
=  \frac{3}{10}\frac{1}{\frac{9}{10}}
=  \frac{1}{3},
\end{align*}
as claimed.
\end{solution}

\begin{exercise}
Prove that $\tfrac{1}{3} = \tfrac{1}{4} + \tfrac{1}{16} +
\tfrac{1}{64} + \tfrac{1}{256}+\cdots $, 
i.e., the binary expansion of $1/3$ is $0.\overline{01}$.
\index{Binary expansion}
\end{exercise}
\begin{solution}
Indeed, using the Geometric Series, we get
\begin{align*}
\tfrac{1}{4} + \tfrac{1}{16} +
\tfrac{1}{64} + \tfrac{1}{256}+\cdots &= \tfrac{1}{4}\big(1+ \tfrac{1}{4} + \tfrac{1}{16} +
\tfrac{1}{64} + \tfrac{1}{256}+\cdots \big)
= \tfrac{1}{4}\frac{1}{1-\tfrac{1}{4}}=\frac{1}{4-1} = \frac{1}{3},
\end{align*}
as claimed.
\end{solution}

\begin{exercise} 
Prove that $\tfrac{1}{3} = \tfrac{2}{7} + \tfrac{2}{49} +
\tfrac{2}{343}+\cdots $, 
i.e., the expansion of $1/3$ in base $7$ is $0.\overline{2}$.
\index{Binary expansion}
\end{exercise}
\begin{solution}
Indeed, using the Geometric Series, we get
\begin{align*}
\tfrac{2}{7} + \tfrac{2}{49} +
\tfrac{2}{343} + \cdots &= \tfrac{2}{7}\big(1+ \tfrac{1}{7} + \tfrac{1}{7^2} +
\cdots \big)
= \tfrac{2}{7}\frac{1}{1-\tfrac{1}{7}}=
\tfrac{2}{7}\frac{1}{\tfrac{6}{7}} = \frac{2}{6} = \frac{1}{3},
\end{align*}
as claimed.
\end{solution}

\begin{exercise}
Prove that $\tfrac{1}{3} = \tfrac{2}{8} + \tfrac{5}{64} +
\tfrac{2}{512} + \tfrac{5}{4096}+\cdots $, 
i.e., the \emph{octal} (base $8$) expansion of $1/3$ is $0.\overline{25}$.
\index{Binary expansion}
\end{exercise}
\begin{solution}
Indeed, using the Geometric Series, we get
\begin{align*}
\tfrac{2}{8} + \tfrac{5}{64} +
\tfrac{2}{512} + \tfrac{5}{4096}+\cdots &= 
\tfrac{2}{8} + \tfrac{5}{8^2}
+\tfrac{2}{8^3} + \tfrac{5}{8^4} +\cdots \\
&=\big(\tfrac{2}{8} + \tfrac{5}{8^2} \big)\big(1 + \tfrac{1}{8^2}+\tfrac{1}{8^4}+\tfrac{1}{8^6}+\cdots\big)\\
&=\big(\tfrac{2\cdot 8}{8^2} + \tfrac{5}{8^2} \big)\Big(1 + \tfrac{1}{8^2}+\big(\tfrac{1}{8^2}\big)^2+\big(\tfrac{1}{8^2}\big)^3+\cdots\Big)\\
&= \tfrac{21}{64}\frac{1}{1-\tfrac{1}{64}}
= \tfrac{21}{64}\frac{1}{\tfrac{63}{64}}
=\frac{21}{63} = \frac{1}{3},
\end{align*}
as claimed. 
\end{solution}

\begin{exercise}
Prove that $\tfrac{2}{3} = \tfrac{1}{2} + \tfrac{1}{8} +
\tfrac{1}{32} + \tfrac{1}{128}+\cdots $, 
i.e., the binary expansion of $2/3$ is $0.\overline{10}$.
\index{Binary expansion}
\end{exercise}
\begin{solution}
Indeed, using the Geometric Series, we get
\begin{align*}
\tfrac{1}{2} + \tfrac{1}{8} +
\tfrac{1}{32} + \tfrac{1}{128}+\cdots &= \tfrac{1}{2}\big(1+ \tfrac{1}{4} + \tfrac{1}{16} +
\tfrac{1}{64} + \tfrac{1}{256}+\cdots \big)
= \tfrac{1}{2}\frac{1}{1-\tfrac{1}{4}}=
\tfrac{1}{2}\frac{4}{4-1}= \frac{2}{3},
\end{align*}
as claimed.
\end{solution}

\begin{exercise} 
Prove that $\tfrac{2}{3} = \tfrac{2}{4} + \tfrac{2}{16} +
\tfrac{2}{64} + \cdots $, 
i.e., the expansion of $2/3$ with respect to base 4 is $0.\overline{2}$.
\end{exercise}
\begin{solution}
Indeed, using the Geometric Series, we get
\begin{align*}
\tfrac{2}{4} + \tfrac{2}{16} +
\tfrac{2}{64} + \tfrac{2}{4^4}+\cdots 
&= \tfrac{2}{4}\big(1+ \tfrac{1}{4} + \tfrac{1}{4^2} +
\tfrac{1}{4^3} + \cdots \big)
= \tfrac{2}{4}\frac{1}{1-\tfrac{1}{4}}=
\tfrac{2}{4}\frac{1}{\frac{3}{4}}= \frac{2}{3},
\end{align*}
as claimed.
\end{solution}

\begin{exercise} 
\label{exo:base16a}
When working in base 16, one needs 15 digits,
which are typically written as 
$0,1,2,\ldots,9,A,B,C,D,E,F$,
 where $A,B,C,D,E,F$ stand for $10,11,12,13,14,15$
respectively. 
Prove that $\tfrac{2}{3} = \tfrac{10}{16} + \tfrac{10}{16^2} +
\tfrac{10}{16^3} + \cdots $, 
i.e., the expansion of $2/3$ with respect to base 16 is $0.\overline{A}$.
\end{exercise}
\begin{solution}
Indeed, using the Geometric Series, we get
\begin{align*}
\tfrac{10}{16} + \tfrac{10}{16^2} +
\tfrac{10}{16^3} + \tfrac{10}{16^4}+\cdots 
&= \tfrac{10}{16}\big(1+ \tfrac{1}{16} + \tfrac{1}{16^2} +
\tfrac{1}{16^3} + \cdots \big)\\
&= \tfrac{10}{16}\frac{1}{1-\tfrac{1}{16}}=
\tfrac{10}{16}\frac{1}{\frac{15}{16}}= \frac{2}{3},
\end{align*}
as claimed.
\end{solution}

\begin{exercise}
A rational number $x$ has the \emph{octal} (base 8) representation
$111.111111\cdots = 111.\overline{1}$.
Determine $x$, i.e., write $x$ as a quotient (in lowest terms) of integers. 
\end{exercise}
\begin{solution}
On the one hand, since we are working in base $8$, we have 
\begin{align*}
0.\overline{1} &= \tfrac{1}{8} + \tfrac{1}{64}+ \cdots
= \tfrac{1}{8} \sum_{k=0}^\infty \big(\tfrac{1}{8}\big)^k
= \tfrac{1}{8} \frac{1}{1-\tfrac{1}{8}} = \tfrac{1}{7}.
\end{align*}
On the other hand, the digits $111$ in base 8 mean
$1\cdot 8^2 + 1\cdot 8 + 1 = 64+8+1 = 73$. 
Altogether, we conclude that
$x= 73\tfrac{1}{7} = \tfrac{512}{7}$. 
\end{solution}

\begin{exercise} 
A rational number $x$ has the \emph{hexadecimal} (base 16, 
see also Exercise~\ref{exo:base16a}) representation
$0.ABBAABBA\cdots = 0.\overline{ABBA}$.
Determine $x$, i.e., write $x$ as a quotient (in lowest terms) of integers. 
\emph{Hint:} Let \texttt{Maple} or \texttt{WolframAlpha} help you.
\end{exercise}
\begin{solution}
Because we are working with base 16, we have 
\begin{align*}
x = 0.\overline{ABBA} 
&= \tfrac{10}{16} + \tfrac{11}{16^2}+\tfrac{11}{16^3}+\tfrac{10}{16^4}+ \cdots
= \Big(\tfrac{10}{16}+\tfrac{11}{16^2}+\tfrac{11}{16^3}+\tfrac{10}{16^4}\Big) 
\sum_{k=0}^\infty \big(\tfrac{1}{16^4}\big)^k\\
&= \tfrac{21981}{32768} \frac{1}{1-\tfrac{1}{16^4}} = \frac{862}{1285}.
\end{align*}
\end{solution}

\begin{exercise} 
Let $(a_n)_\nnn$ converge to $\ell$.
Show that $\ell$ is a cluster point of $(a_n)_\nnn$
and that $(a_n)_\nnn$ has no other cluster point.
\end{exercise}
\begin{solution}
It is clear that $\ell$ is a cluster point.

Let $\ell'$ be a possibly different cluster point.
Then there exists a subsequence $(a_{n_k})_\kkk$ such that 
$a_{n_k}\to \ell'$. 
By Proposition~\ref{p:allgo}, we have $\ell'=\ell$.
\end{solution}

\begin{exercise}
Show that the sequence $((-1)^n)_\nnn$ has no cluster points different
from $\pm 1$.
\end{exercise}
\begin{solution}
Write $(a_n)_\nnn = ((-1)^n)_\nnn$ and suppose to the contrary
that $c$ is a cluster point of $(a_n)_\nnn$ that is different from
$\pm 1$. Then there exists a subsequence $(a_{n_k})_\kkk$ converging to
$c$. Now any term of this subsequence is either $1$ or $-1$, so
$|a_{n_k}-c| \in \{ |1-c|,|-1-c|\}$. But since $\min \{ |1-c|,|-1-c|\}>0$,
it is clear that $|a_{n_k}-c|$ cannot become less than an arbitrary
$\varepsilon$ (e.g., it cannot become less than $(1/2)\min \{
|1-c|,|-1-c|\} > 0$) --- absurd!
\end{solution}

\begin{exercise}
Let $(a_n)_\nnn$ be a sequence of real numbers.
Show that $(a_n)_\nnn$ is convergent
if and only if
$(a_n)_\nnn$ is bounded and $(a_n)_\nnn$ possesses
exactly one cluster point.
\end{exercise}
\begin{solution}
``$\Rightarrow$'': 
Assume that $(a_n)_\nnn$ is convergent, say to $\ell$. 
Then $(a_n)_\nnn$ is bounded by Theorem~\ref{t:convbound}.
It is clear that $\ell$ is a cluster point of $(a_n)_\nnn$.
Hence $(a_n)_\nnn$ possesses at least one cluster point.
On the other hand, if $\ell'$ is another cluster point,
then there exists a subsequence $(a_{n_k})_{k\in\NN}$ 
converging to $\ell'$. By Proposition~\ref{p:allgo},
$\ell=\ell'$. Altogether, $(a_n)_\nnn$ has exactly one cluster point.

``$\Leftarrow$'': Let $\ell$ be the unique cluster point of 
$(a_n)_\nnn$. We argue by contradiction and assume that
$(a_n)_\nnn$ does not converge to $\ell$.
Then there exits $\varepsilon>0$ such that
for every $k\in\NN$ there exists $n=n_k\geq k$
such that $|a_{n_k}-\ell|\geq\varepsilon$. 
We use this to inductively construct a subsequence
$(a_{n_k})_{k\in\NN}$ of $(a_n)_\nnn$ such that
$|a_{n_k}-\ell|\geq\varepsilon$. 
By Bolzano-Weierstrass, this subsequence has a convergent
subsequence whose limit would be $\varepsilon$-away from $\ell$,
so there are at least two cluster points which is absurd.
\end{solution}

\begin{exercise}
Is the converse of Theorem~\ref{t:monconv} true?
That is, is it true that every convergent sequence is bounded and
monotone?
\end{exercise}
\begin{solution}
NO. Every convergent sequence is bounded by Theorem~\ref{t:convbound}.
However, not every convergent sequence is monotone. 
Consider, e.g., $(a_n)_\nnn = (\tfrac{1}{n+1}(-1)^n)_\nnn$.
Then, by the Squeeze Theorem, $a_n\to 0$. However, 
$(a_n)_\nnn$ is neither increasing nor decreasing; consequently,
it is not monotone. 
\end{solution}

\begin{exercise}
\label{exo:171107a}
Prove that if $m$ is an integer such that $m^3$ is even, then $m$
is even as well.
\emph{Hint:} Show the contrapositive. 
\emph{Remark:} This result is useful in showing that
$\sqrt[3]{2}\not\in\QQ$, see Exercise~\ref{exo:cuberoot2}. 
\end{exercise}
\begin{solution}
The contrapositive is
\begin{equation}
\label{e:171107a}
\tag{C}
\text{If $m$ is an odd integer, then $m^3$ is odd.}
\end{equation}
To prove \eqref{e:171107a}, let $m$ be an odd integer,
say $m=2n+1$, where $n\in\ZZ$. 
Then $m^3 = (2n+1)^3 = (2n)^3 + 3(2n)^2 + 3(2n) + 1
= 8n^3 + 12n^2 + 6n+1 = 2(4n^3+6n^2+3n) + 1$ is odd
because $4n^3+6n^2+3n\in\ZZ$. 
\end{solution}

\begin{exercise}
\label{exo:cuberoot2}
Prove that there is no rational number $x$ such that $x^3 = 2$. 
\emph{Hint:} Mimic the proof of Theorem~\ref{t:squareroot2}
using Exercise~\ref{exo:171107a}. 
\end{exercise}
\begin{solution}
Suppose to the contrary that there exists $x\in\mathbb{Q}$ such that
$x^3=2$, say
\begin{equation*}
x = \frac{m}{n}, 
\quad
\text{where $m\in\ZZ$ and $n\in\ZZ\smallsetminus\{0\}$}
\quad
\text{and}
\quad
x^3 =2,
\end{equation*}
and where we assume that the fraction is in lowest terms 
(all common factors except for $\pm 1$ are canceled out). 
Then $2 = x^3 = m^3/n^3$ and so 
\begin{equation*}
2n^3 = m^3.
\end{equation*}
This means that $m^3$ is \emph{even}. 
Therefore, $m$ is even (see Exercise~\ref{exo:171107a}), say
\begin{equation*}
m= 2k, \quad \text{for some integer $k$.}
\end{equation*}
But then the equality $2n^3 = m^3$ turns into
$2n^3 = (2k)^3 = 8k^3$.
Therefore, $n^3= 4k^3=2(2k^3)$ and so $n^3$ is even.
Again, we deduce that $n$ is even, say 
\begin{equation*}
n= 2l, \quad \text{for some integer $l$.}
\end{equation*}
But now $m$ and $n$ do have a nontrivial common factor, 
namely $2$ --- contradiction!
\end{solution}

\begin{exercise} 
\label{exo:171107a+}
Prove that if $m$ is an integer such that $m^3$ is divisible by $3$, then $m$
is divisible by $3$ as well.
\emph{Hint:} Show the contrapositive. 
\emph{Remark:} This result is useful in showing that
$\sqrt[3]{3}\not\in\QQ$, see Exercise~\ref{exo:cuberoot2+}. 
\end{exercise}
\begin{solution}
The contrapositive is
\begin{equation}
\label{e:171107a+}
\tag{C}
\text{If $m$ is an integer not divisible by $3$, then $m^3$ is not divisible by $3$.}
\end{equation}
To prove \eqref{e:171107a+}, write $m = 3k\pm 1$, 
where $k\in\ZZ$. 
Then $m^3 = (3k\pm 1)^3 = (3k)^3 + 3(3k)^2(\pm 1) + 3(3k)(\pm 1)^2 + (\pm 1)^3
= 27k^3 \pm 27k^2 + 9k \pm 1 
= 3(9k^3\pm 9n^2+3k) \pm 1$ is not divisible by $3$
because the remainder after dividing by $3$ is either $1$ or $2$. 
\end{solution}

\begin{exercise} 
\label{exo:cuberoot2+}
Prove that there is no rational number $x$ such that $x^3 = 3$. 
\emph{Hint:} Mimic the proof of Theorem~\ref{t:squareroot2}
using Exercise~\ref{exo:171107a+}. 
\end{exercise}
\begin{solution}
Suppose to the contrary that there exists $x\in\mathbb{Q}$ such that
$x^3=3$, say
\begin{equation*}
x = \frac{m}{n}, 
\quad
\text{where $m\in\ZZ$ and $n\in\ZZ\smallsetminus\{0\}$}
\quad
\text{and}
\quad
x^3 =3,
\end{equation*}
and where we assume that the fraction is in lowest terms 
(all common factors except for $\pm 1$ are canceled out). 
Then $3 = x^3 = m^3/n^3$ and so 
\begin{equation*}
3n^3 = m^3.
\end{equation*}
This means that $m^3$ is divisible by $3$. 
Therefore, $m$ is also divisible by $3$ (see Exercise~\ref{exo:171107a+}), say
\begin{equation*}
m= 3k, \quad \text{for some integer $k$.}
\end{equation*}
But then the equality $3n^3 = m^3$ turns into
$3n^3 = (3k)^3 = 27k^3$.
Therefore, $n^3= 9k^3=3(3k^3)$ and so $n^3$ is 
divisible by $3$. 
Again, we deduce that $n$ is divisible by $3$, say 
\begin{equation*}
n= 3l, \quad \text{for some integer $l$.}
\end{equation*}
But now $m$ and $n$ do have a nontrivial common factor, 
namely $3$ --- contradiction!
\end{solution}

\begin{exercise}[YOU be the marker!] 
Consider the following statement 
\begin{equation*}
\text{``The binary (i.e., base $2$) representation of $-1$ is 
$1.\overline{1}$.''}
\end{equation*}
and the following ``proof'':
\begin{quotation}
First, because we work with base $2$, we have 
$1.\overline{1} = 1 + 1\cdot 2 + 1\cdot 2^2 + \cdots 
= \sum_{n=0}^\infty 2^n$.\\
On the other hand, this is the geometric series 
$\sum_{n=0}^\infty 2^n = 1/(1-2) = -1$.\\
Altogether, the result follows.
\end{quotation}
Why is this proof wrong?
\end{exercise}
\begin{solution}
The first line is wrong because 
$1.\overline{1} = 1 + \sum_{n=1}^{\infty} \tfrac{1}{2^n}$. 

The second line is also wrong because the geometric series formula 
$\sum_{n=0}^\infty\alpha^n$ requires $|\alpha|<1$.
But here $|2| = 2>1$.
\end{solution}

\begin{exercise}[TRUE or FALSE?]
Mark each of the following statements as either true or false. 
Briefly justify your answer.
\begin{enumerate}
\item ``If $a$ is an integer and $a \geq 2$, then $\sqrt{a}$ is irrational.''
\item ``If $(a_n)_\nnn$ has a bounded subsequence, then $(a_n)_\nnn$ 
converges.''
\item ``If $(a_n)_\nnn$ converges, then $(a_n)_\nnn$ is monotone.''
\item ``If $(a_n)_\nnn$ is not a Cauchy sequence of real numbers, 
 then $(a_n)_\nnn$ cannot have a cluster point.''
\end{enumerate}
\end{exercise}
\begin{solution}
(i): FALSE: Consider $a=4 = 2^2$. 

(ii): FALSE: $a_n\equiv (-1)^n$. 

(iii): FALSE: $((-1)^n/(n+1))_\nnn$. 

(iv): FALSE: $((-1)^n)_\nnn$. 
\end{solution} 
\chapter{Square Roots}

We now show that square roots do exist.
This is an application of the completeness axiom. 
In fact, we provide an algorithm for computing square roots,
which not only goes back to the Babylonians, but which is also
used nowadays in numerical computations. 

\section{Existence of Square Roots}

Let $a\in\RR$ such that $a>0$, a number of which we wish
to compute the square root of. Now if 
$x\in\RR$ is a square root of $a$, i.e., $x^2=a$,
then $x\neq 0$ and $x=a/x$; otherwise $x\neq a/x$.
In the latter case, this motivates the replacement of
$x$ by
\begin{equation}
\frac{1}{2}\Big( x + \frac{a}{x}\Big),
\end{equation}
i.e., by the arithmetic mean of $x$ and $a/x$. 
It is remarkable that this idea always works,
i.e., it generates a sequence of real numbers that converges
to the square root of $a$.

\begin{theorem}
\label{t:squareroot}
Let $a>0$ and let $x_0>0$ be real numbers.
Define the sequence $(x_n)_\nnn$ by 
\begin{equation}
x_{n+1} := \frac{1}{2}\Big( x_n + \frac{a}{x_n}\Big),
\quad\text{for every $\nnn$.}
\end{equation}
Then $(x_n)_\nnn$ converges to the 
unique positive solution of the equation $x^2=a$.
\end{theorem}
\begin{proof}
The proof proceeds along several steps.

\textsc{Step 1:} 
$x_n>0$, for every $\nnn$.\\
Indeed, we prove this by induction on $n$.
By assumption, $x_0>0$.
Now assume that $x_n>0$ for some $\nnn$.
Then $1/x_n>0$ and hence $a/x_n>0$ since $a>0$.
It follows that $x_n + a/x_n>0$ and further that
$x_{n+1} = (x_n + a/x_n)/2>0$ since $2=1+1>0$.

\textsc{Step 2:} 
$x_n^2\geq a$, for every integer $n\geq 1$. \\
Indeed, if $n\geq 1$, then 
\begin{subequations}
\begin{align}
x_n^2 - a &= \frac{1}{4}\Big( x_{n-1}+\frac{a}{x_{n-1}}\Big)^2 -a 
= \frac{1}{4}\Big( x_{n-1}^2 + 2a + \frac{a^2}{x_{n-1}^2}\Big)  -
a\\
&= \frac{1}{4}\Big( x_{n-1}^2 - 2a + \frac{a^2}{x_{n-1}^2}\Big)
=  \frac{1}{4}\Big( x_{n-1}-\frac{a}{x_{n-1}}\Big)^2 \geq 0.
\end{align}
\end{subequations}

\textsc{Step 3:} 
$x_{n+1}\leq x_n$, for every integer $n\geq 1$. \\
Indeed, if $n\geq 1$, then using Steps 1 and 2, 
\begin{equation}
x_n-x_{n+1} = x_n -  \frac{1}{2}\Big( x_n + \frac{a}{x_n}\Big)
= \frac{1}{2x_n}\big(x_n^2 -a \big) \geq 0.
\end{equation}

Now define
\begin{equation}
y_n := \frac{a}{x_n},
\quad\text{for every $\nnn$.}
\end{equation}

\textsc{Step 4:} 
$y_n^2\leq a$, for every integer $n\geq 1$.\\
Indeed, let $n\geq 1$. Then 
$1/x_n^2\leq 1/a$ by Step~2.
Multiplying by $a^2>0$ gives
$y_n^2 = (a/x_n)^2 = a^2/x_n^2 \leq a^2/a = a$.

\textsc{Step 5:} 
$y_n\leq y_{n+1}$, for every integer $n\geq 1$.\\
This follows immediately from Step~3. 

\textsc{Step 6:} 
$y_n\leq x_{n}$, for every integer $n\geq 1$.\\
For otherwise, there would exist $n\geq 1$
such that $y_n>x_n>0$ and so $y_n^2>x_n^2$,
which contradicts $y_n^2\leq a \leq x_n^2$ (see Steps 2 and 4).

\textsc{Step 7:} 
$(x_n)_{\nnn}$ converges to some real number $x>0$.\\
Combining Step~3 and Step~6, we see that
$x_1\geq x_n\geq x_{n+1}\geq y_{n+1}\geq y_1 > 0$, for
every $n\geq 1$. Hence $(x_n)_\nnn$ is bounded and monotone
(decreasing). 
By Theorem~\ref{t:monconv}, $(x_n)_\nnn$ converges. 
Since $x_n \geq y_1>0$ for every $n\geq 1$, 
it follows from Corollary~\ref{c:limleq} that
$x\geq y_1>0$. 

\textsc{Step 8:} $x^2 =a $.\\
Indeed, using the Limit Laws,
we obtain
\begin{equation}
x = \lim x_{n+1} = \lim \frac{1}{2}\Big( x_n + \frac{a}{x_n}\Big)
= \frac{1}{2} \Big( \lim x_n + \frac{a}{\lim x_n}\Big) =
\frac{1}{2}\Big( x + \frac{a}{x}\Big). 
\end{equation}
Thus $2x = x+a/x$
$\Leftrightarrow$ 
$x=a/x$ 
$\Leftrightarrow$ 
$x^2 = a$.

\textsc{Step 9:} There is only one positive solution to $x^2=a$.\\
Suppose that $z>0$ satisfies $z^2 =a$.
Then $0 = a - a = x^2-z^2 = (x+z)(x-z)$.
Now $x+z>0$, so division by $x+z$ yields $0=x-z$.
\end{proof}

Theorem~\ref{t:squareroot} thus makes the following
notion well defined.

\begin{definition}[square root]
Let $a\in\RR$ such that $a\geq 0$.
Then the unique nonnegative real number $x$ such that
$x^2=a$ is called the \textbf{square root} of $a$,
and written $x=\sqrt{a}$.\index{Square root}
\end{definition}

\begin{remark}
\label{r:squareroot}
Let 
$a\in\RR$ and consider the equation
\begin{equation}
x^2=a.
\end{equation}
\begin{enumerate}
\item
When $a=0$, the equation becomes $x^2=0$.
By Proposition~\ref{p:may26:4}\ref{p:may26:4iii},
this equation has only one solution, namely $x=0$.
So $\sqrt{0} = 0$. 
\item Now assume that $a>0$.
Then $\sqrt{a}$ is the unique positive solution of the equation
$x^2=a$ by Theorem~\ref{t:squareroot}. 
Writing $x^2 = a$ as $x^2 = \sqrt{a}^2$, we see that
\begin{equation}
x^2=a
\;\Leftrightarrow\;
x^2 - \sqrt{a}^2 = 0
\;\Leftrightarrow\;
\big(x+\sqrt{a}\big)\big(x-\sqrt{a}\big)=0.
\end{equation}
In turn, Proposition~\ref{p:may26:4}\ref{p:may26:4iii}
yields $x+\sqrt{a}=0$ or $x-\sqrt{a}=0$,
which gives rise to the two distinct solutions
$\sqrt{a}$ and $-\sqrt{a}<0$.
\item When $a<0$, the equation $x^2=a$ has no solution in $\RR$ by 
Proposition~\ref{p:ofield}\ref{p:ofieldx}. 
\end{enumerate}
\end{remark}

Let $\alpha$ and $\beta$ be two nonnegative real numbers.
Then the uniqueness of square roots quickly leads to
\begin{equation}
\label{e:squareprod}
\sqrt{\alpha\beta} = \sqrt{\alpha}\sqrt{\beta}.
\end{equation}

\begin{lemma}[arithmetic mean -- geometric mean inequality]
\label{l:amgm}
\index{AM -- GM Inequality}
\index{Arithmetic mean}
\index{Geometric mean}
Let $\alpha\geq 0$ and let $\beta\geq 0$.
Then
\begin{equation}
\label{e:amgm}
\sqrt{\alpha\beta} \leq \frac{\alpha+\beta}{2},
\end{equation}
and equality holds if and only if $\alpha=\beta$.
\end{lemma}

\section{Speed of Convergence}

\begin{theorem}
Let $a>0$, let $x_0>0$, and let the sequence
$(x_n)_\nnn$ be defined as in Theorem~\ref{t:squareroot}, i.e.,
\begin{equation}
\label{e:halloween:a}
x_{n+1} := \frac{1}{2}\Big( x_n + \frac{a}{x_n}\Big),
\quad\text{for every $\nnn$.}
\end{equation}
Define the corresponding sequence of relative errors 
$(\varepsilon_n)_\nnn$ by
$\varepsilon_n = x_n/\sqrt{a} - 1$, i.e., by
\begin{equation}
\label{e:halloween:b}
x_n = \sqrt{a}\big(1+\varepsilon_n\big),
\quad\text{for every $n\in\NN$.}
\end{equation}
Then for every $n\geq 1$, we have
$\varepsilon_n\geq 0$ and also
\begin{equation}
\label{e:halloween:c}
\varepsilon_{n+1} = \frac{1}{2}
\frac{\varepsilon_n^2}{1+\varepsilon_n}
\leq \frac{1}{2}\min\big\{ \varepsilon_n,\varepsilon_n^2\big\}. 
\end{equation}
\end{theorem}
\begin{proof}
Combining Steps 1 and 2 in the proof of Theorem~\ref{t:squareroot},
we see that 
$x_n \geq \sqrt{a}$ and hence $\varepsilon_n \geq 0$,
for $n\geq 1$. Next, substituting \eqref{e:halloween:b}
into \eqref{e:halloween:a}, we obtain
\begin{equation}
\sqrt{a}\big(1+\varepsilon_{n+1}\big) = 
\frac{1}{2}\bigg( \sqrt{a}\big(1+\varepsilon_n\big) +
\frac{a}{\sqrt{a}\big(1+\varepsilon_n\big)}\bigg),
\quad\text{for every $\nnn$.}
\end{equation}
After dividing by $\sqrt{a}$, this turns into
\begin{equation}
1+\varepsilon_{n+1} = 
\frac{1}{2}\bigg( 1+\varepsilon_n + 
\frac{1}{1+\varepsilon_n}\bigg),
\quad\text{for every $\nnn$.}
\end{equation}
Hence, 
\begin{equation}
\varepsilon_{n+1} = 
\frac{1}{2}\bigg( -1 +\varepsilon_n + 
\frac{1}{1+\varepsilon_n}\bigg)
= \frac{1}{2} 
\frac{\varepsilon_n^2}{1+\varepsilon_n}, 
\quad\text{for every $\nnn$.}
\end{equation}
Now suppose that $n\geq 1$. 
Then $\varepsilon_n\geq 0$.
On the one hand, since 
$1+\varepsilon_n\geq 1$, we get 
$\varepsilon_n^2/(1+\varepsilon_n)\leq \varepsilon_n^2$.
On the other hand, we also have 
$\varepsilon_n^2 \leq \varepsilon_n^2 + \varepsilon_n
= \varepsilon_n(\varepsilon_n+1)$; thus, 
$\varepsilon_n^2/(1+\varepsilon_n)\leq\varepsilon_n$.
Altogether, we obtain
\eqref{e:halloween:c}. 
\end{proof}

\begin{remark}
The relative error discussed in the last result becomes
small very quickly.
Suppose it is, at iteration $n$, less than or equal to $10^{-3}$.
Then at iteration $n+1$, the error is no worse
than $\tfrac{1}{2} 10^{-6}$. Roughly speaking, the number of correct
digits is expected to \emph{double} at every iteration, this is 
also termed \textbf{quadratic convergence}.

Furthermore, in Calculus~I we considered \textbf{Newton's method}
for solving the equation $f(x)=0$, via the 
iteration\index{Newton's method}
\begin{equation}
x_{n+1} = x_n  - \frac{f(x_n)}{f'(x_n)}.
\end{equation}
Now if $f(x) = x^2-a$, i.e., we aim to find a square root,
then $f'(x) = 2x$ and Newton's method takes the form
\begin{equation}
\label{e:Newton}
x_{n+1} = x_n - \frac{x_n^2-a}{2x_n} = \frac{x_n}{2} + \frac{a}{2x_n},
\end{equation}
which is \eqref{e:halloween:a}!
\end{remark}

We conclude with two numerical examples.
Note that $a/x_n\leq \sqrt{a}\leq x_n$
by Steps 2 and 4 in the proof of Theorem~\ref{t:squareroot}. 
We compute numerically the square roots of $4$ and of $2$. 
Observe the rapid convergence and doubling of accuracy. 
The code on the next page is done 
in \texttt{Julia}:

\newpage

\label{p:sqrt2}


\includepdf[
  pages=1-1,
  scale=0.9,
  offset= 0 -20mm, 
  pagecommand={\thispagestyle{headings}}
]{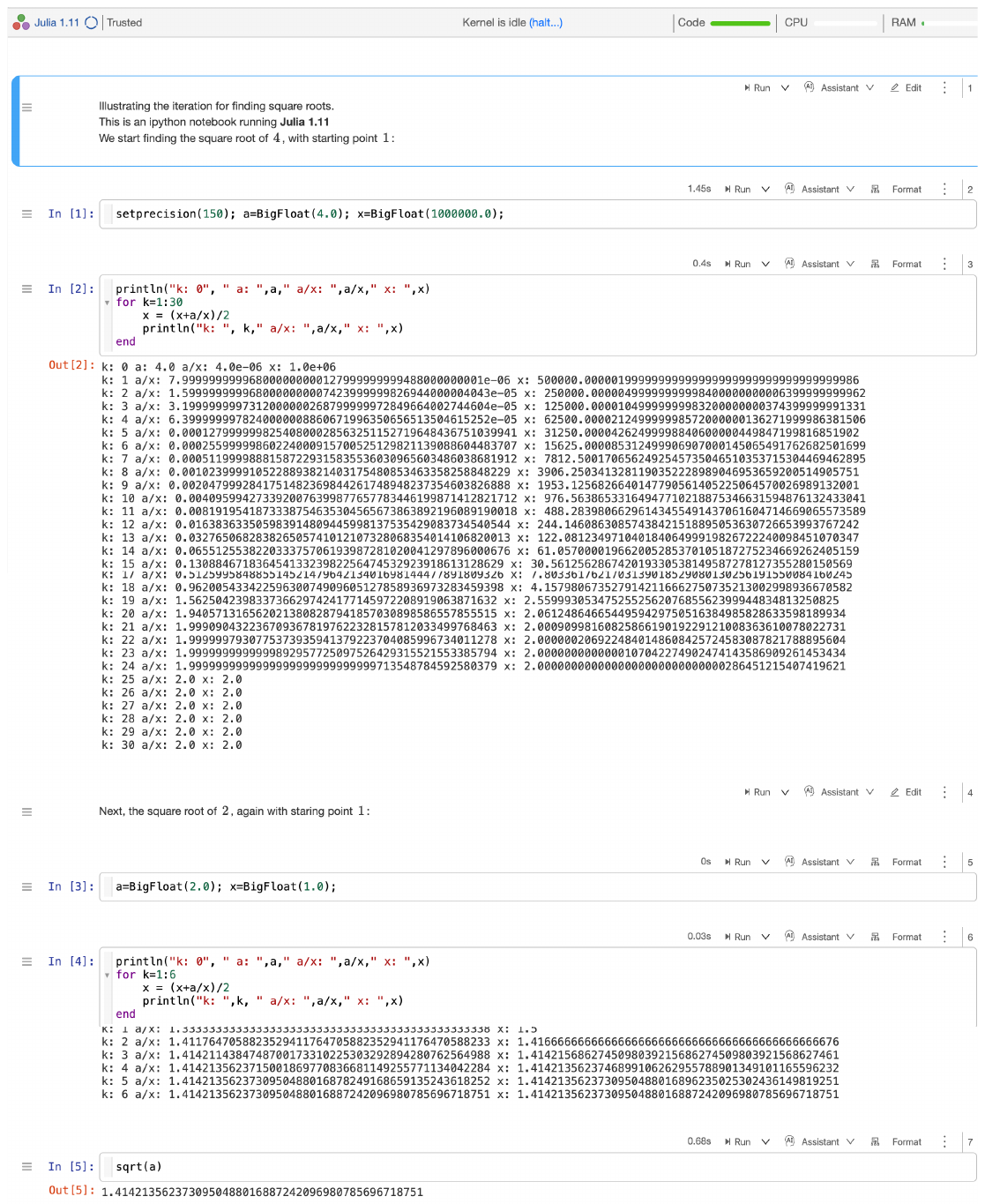}
















\section*{Exercises}\markright{Exercises}
\addcontentsline{toc}{section}{Exercises}
\setcounter{theorem}{0}

\begin{exercise} 
Let $k$ be an integer such that $k\geq 2$, let $a>0$, and let $x_0>0$.
Define a sequence $(x_n)_\nnn$ by
\begin{equation}
\label{e:kthroot}
x_{n+1} := \frac{1}{k}\Big( (k-1)x_n + \frac{a}{x_n^{k-1}}\Big),
\quad\text{for every $\nnn$.}
\end{equation}
Show that $(x_n)_\nnn$ converges to the unique positive solution of
$x^k=a$. 
\end{exercise}
\begin{solution}
First, an easy induction gives $x_n>0$ for every $\nnn$.
Next, 
\begin{align*}
x_n > 0
&\Rightarrow 
\frac{a}{x_n^k} + k-1 >0\\
&\Rightarrow 
\frac{a}{x_n^k} -1 >-k\\
&\Rightarrow 
\frac{1}{k}
\bigg(\frac{a}{x_n^k} -1 \bigg) > -1. 
\end{align*}
Note that 
\begin{equation}
\label{e:kroota}
x_{n+1} = \frac{1}{k}x_n\Big( (k-1) + \frac{a}{x_n^{k}}\Big)
= x_n\bigg(1+\frac{1}{k}\Big(\frac{a}{x_n^k}-1\Big)\bigg).
\end{equation}
Recall that Bernoulli yields 
$y\geq -1$ $\Rightarrow$ 
$(1+y)^k \geq 1+ky$. 
Applying this with 
\begin{equation*}
    y = \frac{1}{k}\Big(\frac{a}{x_n^k}-1\Big) > -1
\end{equation*}
gives us 
\begin{equation*}
    \bigg(1+\frac{1}{k}\Big(\frac{a}{x_n^k}-1\Big)\bigg)^k 
    \geq 1 + k\frac{1}{k}\Big(\frac{a}{x_n^k}-1\Big)
    = \frac{a}{x_n^k}. 
\end{equation*}
Hence 
\begin{equation*}
x_{n+1}^k = x_n^k\bigg(1+\frac{1}{k}\Big(\frac{a}{x_n^k}-1\Big)\bigg)^k  
\geq x_n^k \frac{a}{x_n^k} = a. 
\end{equation*}
Thus 
\begin{equation*}
(\forall n\geq 1)\quad 
\frac{a}{x_n^k} \leq 1; 
\;\;\text{hence}\;\;
\frac{1}{k}\bigg(\frac{a}{x_n^k}-1\bigg)\leq 0.
\end{equation*}
In view of \eqref{e:kroota}, 
$x_{n+1}\leq x_n$ for all $n\geq 1$.
Since the sequence $(x_n)_\nnn$ consists of positive terms,
it follows that there exists $x \in\RR$ such that 
\begin{equation*}
    x_n\to x \geq 0.
\end{equation*}
Now observe that the definition of $x_{n+1}$ yields 
\begin{equation*}
kx_{n+1}x_n^{k-1} = (k-1)x_n^k + a.
\end{equation*}
Taking the limit as $n\to\infty$ gives
\begin{equation*}
kxx^{k-1} = (k-1)x^k + a,
\end{equation*}
and therefore $x^k = a$. 
Finally, $0\leq z_1 < z_2$
$\Rightarrow$ 
$z_1^k < z_2^k$ which shows that there is only one 
nonnegative $k$th root!
\end{solution}

\begin{exercise}
Consider the function $f(x) = x^k -a$. 
Derive the update formula of the sequence generated by Newton's method
and compare to \eqref{e:kthroot}.
\end{exercise}
\begin{solution}
Recall that Newton's method updates via
\begin{equation*}
x_{n+1} = x_n  - \frac{f(x_n)}{f'(x_n)}.
\end{equation*}
Now $f'(x) = kx^{k-1}$ and hence 
\begin{equation*}
x_{n+1} = x_n  - \frac{x_n^k-a}{kx_n^{k-1}} = 
x_n\bigg(1-\frac{1}{k}\bigg) + \frac{a}{kx_n^{k-1}}
\end{equation*}
which yields the formula \eqref{e:kthroot}. 
\end{solution}

\begin{exercise}
\label{exo:130928f}
Let $\alpha$ and $\beta$ be real numbers such that
$0\leq \alpha < \beta$. 
Show that $\sqrt{\alpha}<\sqrt{\beta}$.
\end{exercise}
\begin{solution}
We argue by contradiction.
If the result was false, we would have 
$\sqrt{\alpha}\geq\sqrt{\beta}>0$ and
hence $\alpha = \sqrt{\alpha}^2\geq \sqrt{\beta}^2 = \beta$ by
Proposition~\ref{p:ofield}\ref{p:ofieldix}, which is absurd.
\end{solution}

\begin{exercise}
\label{exo:goldenmean}
Consider
\begin{equation}
\sqrt{1+\sqrt{1+\sqrt{1+\sqrt{1+\sqrt{1+\cdots}}}}},
\end{equation}
i.e., the sequence $(a_n)_\nnn$ defined by
$a_0:=1$ and $a_{n+1} := \sqrt{1+a_n}$, for $n\geq 0$. 
Show that this sequence is convergent, and determine its limit.
\emph{Hint:} Use induction twice to show that $(a_n)_{\nnn}$ is strictly
increasing, and that $a_n\leq 2$ for every $\nnn$.
\end{exercise}
\begin{solution}

\textsc{Step 0:} Let $0\leq \alpha < \beta$.
Then $\sqrt{\alpha}<\sqrt{\beta}$.
(See also Exercise~\ref{exo:130928f}.)
For otherwise, we would have
$\sqrt{\alpha}\geq\sqrt{\beta}>0$ and
hence $\alpha = \sqrt{\alpha}^2\geq \sqrt{\beta}^2 = \beta$ by
Proposition~\ref{p:ofield}\ref{p:ofieldix}, which is absurd.

\textsc{Step 1:} $(a_n)_\nnn$ is strictly increasing.\\
We prove by induction that 
\begin{equation*}
a_n < a_{n+1},
\quad
\text{for every $\nnn$.}
\end{equation*}

\textbf{Base Case}: When $n=0$, we do have
$a_0 = 1 < \sqrt{2} = \sqrt{1+1} = \sqrt{1+a_0} = a_1$.

\textbf{Inductive Step}: Assume that for some integer $n\geq 0$
we have $a_n<a_{n+1}$.
Then $1+a_n<1+a_{n+1}$ and hence
$a_{n+1} = \sqrt{1+a_n} < \sqrt{1+a_{n+1}}=a_{n+2}$, 
which completes the proof. 

\textsc{Step 2:} $a_n \leq 2$, for every $\nnn$.\\
We prove this also by induction.
Indeed, the \textbf{Base Case} $a_0 = 1<2$ is clear.
Suppose now that for some integer $n\geq 0$, we have $a_n\leq 2$
(this is the \textbf{inductive hypothesis}). 
Then $1+a_n\leq 1+2= 3<4$ and hence $a_{n+1}=\sqrt{1+a_n}\leq
\sqrt{3}<\sqrt{4}=2$ which completes the proof of the
\textbf{inductive step}.

\textsc{Step 3:} 
$a_n\to\alpha\in[1,2]$.\\
Note that $a_n\geq 1$ for every $\nnn$ by Step~1. 
Now combine Theorem~\ref{t:monconv} with Steps 1 and 2.

It remains to determine the limit $\alpha$.
Now $a_{n+1}^2 = 1+a_n$, for every $\nnn$;
thus, taking limits yields
$\alpha^2 = 1+\alpha$.
Hence $(\alpha-1/2)^2 = 5/4$. 
Since $\alpha\geq 1$, it follows that $\alpha-1/2=\sqrt{5}/2$
and therefore that $a_n \to (1+\sqrt{5})/2$. 
\end{solution}

\begin{exercise}
Consider the sequence $(a_n)_\nnn$ defined in
Exercise~\ref{exo:goldenmean}. \index{Inverse Symbolic Calculator}
Use a programming language of your choice to numerically compute the
first 40 terms of this sequence. 
Hand in (i) a printout of your computer code and
(ii) a printout of the output of your computer code. 
\end{exercise}
\begin{solution}
See next page!
\end{solution}

\includepdf[
  pages=1-1,
  scale=0.8,
  offset= 0 -20mm, 
  pagecommand={\thispagestyle{headings}}
]{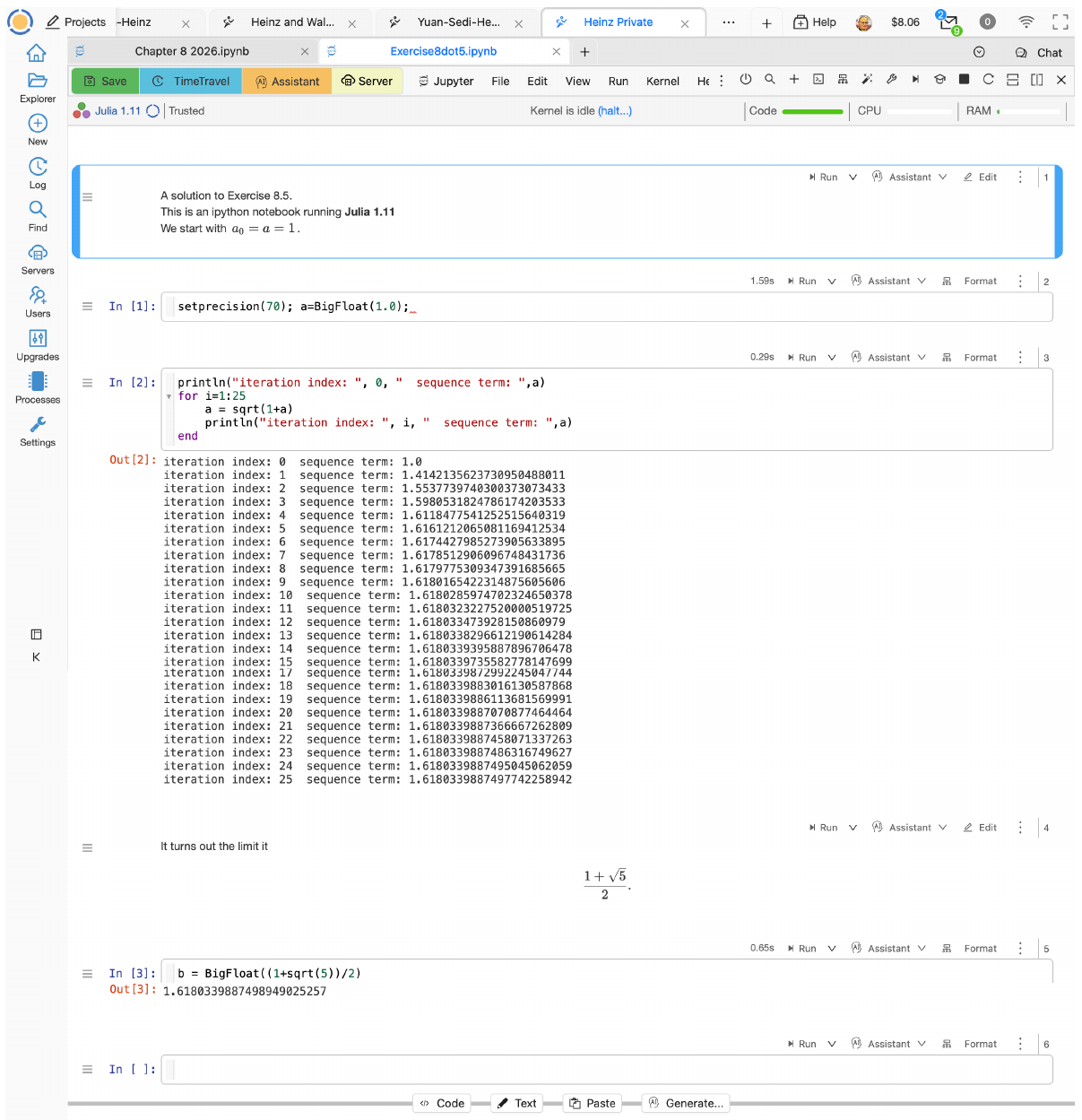}

\begin{exercise}
\label{exo:260304a}
Let $\alpha\geq 0$ and let $\beta\geq 0$.
Prove that $\sqrt{\alpha\beta}=\sqrt{\alpha}\sqrt{\beta}$.
\end{exercise}
\begin{solution}
In view of Theorem~\ref{t:squareroot}
and Remark~\ref{r:squareroot},
we note that given $\gamma\geq 0$,
$\sqrt{\gamma}$ is the unique nonnegative number such that
$\sqrt{\gamma}^2=\gamma$.

Thus, to verify that $\sqrt{\alpha\beta}=\sqrt{\alpha}\sqrt{\beta}$,
it suffices to show that
$\big(\sqrt{\alpha}\sqrt{\beta}\big)^2=\alpha\beta$,
since $\sqrt{\alpha}\sqrt{\beta}$ is clearly nonnegative.
Indeed,
$\big(\sqrt{\alpha}\sqrt{\beta}\big)^2
= \sqrt{\alpha}^2\sqrt{\beta}^2 = \alpha\beta$.
\end{solution}

\begin{exercise}
Prove Lemma~\ref{l:amgm}. 
\end{exercise}
\begin{solution}
By \eqref{e:squareprod},
\begin{equation*}
\sqrt{\alpha\beta} = \sqrt{\alpha}\sqrt{\beta}.
\end{equation*}
Hence,
\begin{subequations}
\begin{align}
\sqrt{\alpha\beta}\leq\frac{\alpha+\beta}{2}
&\Leftrightarrow
2\sqrt{\alpha}\sqrt{\beta}\leq \alpha+\beta\\
&\Leftrightarrow
0\leq \sqrt{\alpha}^2 - 2\sqrt{\alpha}\sqrt{\beta}+\sqrt{\beta}^2\\
&\Leftrightarrow
0\leq \big(\sqrt{\alpha}-\sqrt{\beta}\big)^2,\label{e:orban}
\end{align}
\end{subequations}
which is true by Proposition~\ref{p:ofield}\ref{p:ofieldx}. 
Moreover, we have equality in \eqref{e:orban} if and only if
$\sqrt{\alpha}=\sqrt{\beta}$, which is equivalent to $\alpha=\beta$.
\end{solution}

\begin{exercise}
Compute the square roots of $9$ and $3$ numerically, similarly to the
example provided in this section, using a programming language of your
choice (e.g., \texttt{Maple}, \texttt{Python}, \texttt{Octave},
\texttt{Java}, etc.). 
Hand in (i) a printout of your computer code and
(ii) a printout of the output of your computer code. 
\end{exercise}
\begin{solution}
    Modify the code provided on page~\pageref{p:sqrt2} to compute $\sqrt{3}$ and $\sqrt{9}$.
\end{solution}

\begin{exercise}
Consider the sequence $(f_n)_\nnn$ of Fibonacci numbers defined
in Example~\ref{ex:seq}\ref{ex:seq:fibo}. 
Using Exercise~\ref{exo:Binet} and the limit laws, show that
\begin{equation*}
\frac{f_{n+1}}{f_n}\to \frac{1+\ww}{2}.
\end{equation*}
\end{exercise}
\begin{solution}
Set $\alpha := (1+\ww)/2$  and observe that $\alpha>1$.
It follows that $1/\alpha^n\to 0$ by Example~\ref{ex:geoseq}. 
By Exercise~\ref{exo:Binet}, we have that
\begin{equation*}
f_n = \frac{1}{\ww}\left(\frac{1+\ww}{2}\right)^n
+ \varepsilon_n = \frac{1}{\ww}\alpha^n + \varepsilon_n,
\end{equation*}
where $\varepsilon_n\to 0$ (because $|1-\ww|<2$). 
Using the limit laws, we deduce that 
\begin{align*}
\frac{f_{n+1}}{f_n} &= 
\frac{ \frac{1}{\ww}\alpha^{n+1}
+ \varepsilon_{n+1}}{\frac{1}{\ww}\alpha^n
+ \varepsilon_n}
= \frac{\alpha +
\ww\varepsilon_{n+1}/\alpha^n}{1+\ww\varepsilon_n/\alpha^n}\\
&\to
\frac{\alpha+\ww\cdot 0\cdot 0}{1+\ww\cdot 0 \cdot 0} = \alpha =
\frac{1+\ww}{2}, 
\end{align*}
as announced. 
\end{solution}

\begin{exercise}
Find all solutions to the equation
$\sqrt{x+1}-\sqrt{x-1}=2$ where $x\in\RR$.
\end{exercise}
\begin{solution}
We claim there is no solution and argue by contradiction. 
Suppose to the contrary that $x$ is a solution of the equation.
Then $\sqrt{x+1} = \sqrt{x-1}+2$.
Squaring yields
$x+1 = x-1 + 4\sqrt{x-1} + 4$.
Hence $4\sqrt{x-1}=-2$ or 
$\sqrt{x-1}=-1/2$ which is absurd.
\end{solution}

\begin{exercise}[Harmonic Mean]
Let $x>0$ and $y>0$. The \emph{Harmonic Mean} of $x$ and $y$ is
defined by $h(x,y) := 2xy/(x+y)$. Show that 
$h(x,y) \leq \sqrt{xy}$.
\index{Harmonic mean}
\end{exercise}
\begin{solution}
We have $0 \leq (x-y)^2 = x^2-2xy+y^2$,
hence $2xy \leq x^2+y^2$. 
Thus $4xy \leq x^2 + y^2 + 2xy = (x+y)^2$.
It follows that $4(xy)^2 \leq (xy)(x+y)^2$
and thus (taking square roots)
$2xy \leq \sqrt{xy}(x+y)$.
Re-arranging yields
$h(x,y) = 2xy/(x+y) \leq \sqrt{xy}$. 
\end{solution}

\begin{exercise}[a trip to Vancouver]
Suppose you are driving to Vancouver.
\begin{enumerate}
\item
Suppose you go 50km/h in the first hour of your drive,
and 100km/h in the second hour of your drive.
Determine your average speed for the first couple of hours.
\item
Suppose your average speed driving to Vancouver was 50km/h (snow
storm!) while your average speed returning from Vancouver to
Kelowna was 100km/h. What was your average speed for the entire
round trip?
\end{enumerate}
\end{exercise}
\begin{solution}
(i): Denote your average speed for the first hour by $x$,
and for the second hour by $y$. 
You travel $x$ km in the first hour, and $y$ km in the second
hour. Altogether, you drove $x+y$ km in the first couple of
hours. Hence your average speed for that period is 
\begin{equation*}
\frac{x+y}{2},
\end{equation*}
which is the arithmetic mean of $x$ and $y$. 
For our concrete question, we obtain an average speed of 75 km/h.

(ii): Now let $x$ be the average speed of your trip from Kelowna
to Vancouver, and let $y$ be the average speed of your trip from
Vancouver back to Kelowna. 
Denote by $d$ the distance from Kelowna to Vancouver.
It took you $d/x$ hours from Kelowna to Vancouver,
and $d/y$ hours from Vancouver to Kelowna. 
Altogether, it took you $(d/x) + (d/y)$ hours for the entire
trip.
The average speed for your trip is therefore
\begin{equation*}
\frac{2d}{\frac{d}{x} + \frac{d}{y}} =
\frac{2}{\frac{1}{x}+\frac{1}{y}} = \frac{2xy}{x+y},
\end{equation*}
which is the harmonic mean of $x$ and $y$.
For our concrete question, we compute 
$(2\cdot50\cdot 100)/(50+100) = 200/3 = 66\tfrac{2}{3} \approx
66.67$ km/h. Kind of surprising, isn't it?
\end{solution}

\begin{exercise}
\label{exo:walaa1}
Let $n\geq 1$ be an integer. Show that
\begin{equation}
\label{e:140902e}
\sqrt{n} \leq \sum_{k=1}^n \frac{1}{\sqrt{k}} \leq 2\sqrt{n}-1.
\end{equation}
\emph{Hint:}
Show first that for $k\geq 1$, we have
\begin{equation}
\label{e:140902c}
\frac{1}{2\sqrt{k+1}} < \sqrt{k+1}-\sqrt{k} < \frac{1}{2\sqrt{k}}.
\end{equation}
\end{exercise}
\begin{solution}
Note that
$(\sqrt{k+1}-\sqrt{k})(\sqrt{k+1}+\sqrt{k}) = \sqrt{k+1}^2 - \sqrt{k}^2 =
k+1 - k = 1$ and hence 
\begin{equation}
\label{e:140902a}
\frac{1}{\sqrt{k+1}+\sqrt{k}} = \sqrt{k+1}-\sqrt{k}.
\end{equation}
Furthermore, since $\sqrt{k}<\sqrt{k+1}$, we estimate 
\begin{equation}
\label{e:140902b}
\frac{1}{2\sqrt{k+1}} = \frac{1}{\sqrt{k+1}+\sqrt{k+1}}
< \frac{1}{\sqrt{k}+\sqrt{k+1}}
< \frac{1}{\sqrt{k}+\sqrt{k}} =
\frac{1}{2\sqrt{k}}.
\end{equation}
Combining \eqref{e:140902a} and \eqref{e:140902b}
yields \eqref{e:140902c}.
Multiplying \eqref{e:140902c} by $2$ followed by summing 
over $k\in\{1,\ldots,n-1\}$ yields
\begin{equation}
\label{e:140902d}
\sum_{l=2}^n \frac{1}{\sqrt{l}}=\sum_{k=1}^{n-1}\frac{1}{\sqrt{k+1}} \leq 2\sqrt{n}-2 \leq 
\sum_{k=1}^{n-1} \frac{1}{\sqrt{k}}.
\end{equation}
This implies 
\begin{equation*} 
\sum_{k=1}^n \frac{1}{\sqrt{k}} 
= 1 + \sum_{l=2}^n \frac{1}{\sqrt{l}}
\leq 1 + 2\sqrt{n}- 2 = 2\sqrt{n}-1,
\end{equation*}
which establishes the right inequality in \eqref{e:140902e}. 
We also learn from \eqref{e:140902d} that
\begin{equation*} 
\sum_{k=1}^n \frac{1}{\sqrt{k}} 
= \frac{1}{\sqrt{n}} + \sum_{k=1}^{n-1} \frac{1}{\sqrt{k}}
\geq  \frac{1}{\sqrt{n}} + 2\sqrt{n}-2.
\end{equation*}
To complete the proof, it suffices to show that
\begin{equation*}
\frac{1}{\sqrt{n}} + 2\sqrt{n}-2 \geq \sqrt{n}.
\end{equation*}
But this inequality is equivalent to 
$0 \leq \sqrt{n}-2+1/\sqrt{n} = (\sqrt{\sqrt{n}} - 1/\sqrt{\sqrt{n}})^2$,
which is true. 
\end{solution}

\begin{exercise}[absolute value revisited]
\index{Absolute value}
Let $x\in\RR$.
Show that $$|x|= \sqrt{x^2}.$$
\end{exercise}
\begin{solution}
If $x\neq 0$, then $x^2>0$;
otherwise $x=0$ and $x^2\geq 0$.
In any case, $x^2\geq 0$ and so
$\sqrt{x^2}$ is well defined. 
Recall that $\sqrt{a}$ is the unique nonnegative solution $r$ 
to the equation $r^2=a$.

\emph{Case~1:} $x\geq 0$.
On the one hand, $|x|=x\geq 0$.
On the other hand, $x^2=x^2$, and so $x=\sqrt{x^2}$.
Altogether, $|x|=\sqrt{x^2}$.

\emph{Case~2:} $x< 0$.
On the one hand, $|x|=-x\geq 0$.
On the other hand, $(-x)^2=x^2$, and so $-x=\sqrt{x^2}$.
Altogether, $|x|=\sqrt{x^2}$.
\end{solution}

\begin{exercise}
Let $a_0=1$ and define recursively\footnote{This sequence plays a
role in the so-called FISTA algorithm, one of the most
successful optimization methods of the new millenium.
See Amir Beck and Marc Teboulle,
A Fast Iterative Shrinkage-Thresholding Algorithm for Linear
Inverse Problems,
\emph{SIAM Journal on Imaging Sciences}~2, pp.~183--202 (2009).}
$$a_{n+1} = \frac{1+\sqrt{1+4a_n^2}}{2}.$$
Show the following for every $\nnn$:
\begin{enumerate}
\item 
$a_n \geq (n+2)/2$
\item
$a_n \leq n+1$.
\end{enumerate}
\emph{Hint:} Use induction. 
\end{exercise}
\begin{solution}
(i): 
If $n=0$, then $a_0=1=(0+2)/2$ and the base case is fine.

Now assume the inequality holds for some $\nnn$.
Then
\begin{align*}
a_{n+1} &= \frac{1+\sqrt{1+4a_n^2}}{2} \geq \frac{1+\sqrt{4a_n^2}}{2}
= \frac{1+2a_n}{2}\\
&\geq \frac{1+2\frac{n+2}{2}}{2}
= \frac{n+3}{2} = \frac{(n+1)+2}{2},
\end{align*}
which completes the proof of the induction step.

(ii): When $n=0$, then $a_0 = 1 = 0+1$, so the base case is
clear.
Now assume the inequality holds for some $\nnn$.
We start by observing that
\begin{equation}
\label{e:171026}
2n+4 \geq 1 + \sqrt{1+4(n+1)^2}.
\end{equation}
Indeed,
\begin{align*}
2n+4 \geq 1 + \sqrt{1+4(n+1)^2}
&\Leftrightarrow
2n+3 \geq \sqrt{1+4(n+1)^2}\\
&\Leftrightarrow
(2n+3)^2 \geq {1+4(n+1)^2}\\
&\Leftrightarrow
4n^2+12n+9 \geq 1+4n^2+8n+4\\
&\Leftrightarrow
4n+4 \geq 0,
\end{align*}
which is obviously true!
We thus obtain
\begin{align*}
a_{n+1} &= \frac{1+\sqrt{1+4a_n^2}}{2} 
\tag*{[definition]}\\
&\leq \frac{1+\sqrt{1+4(n+1)^2}}{2} 
\tag*{[inductive hypothesis]}\\
&\leq \frac{2n+4}{2} 
\tag*{[use \eqref{e:171026}]}\\
&= n+2\\
&= (n+1)+1,
\end{align*}
as required. This completes the proof of the inductive step and
we are done by induction. 
\end{solution}

\begin{exercise}[continuity of taking the square root]
Let $(a_n)_\nnn$ be a sequence of nonnegative real numbers
converging to $\ell$.
Show that $\sqrt{a_n}\to\sqrt{\ell}$.
\emph{Remark:} This shows that taking the square root is a
continuous function.
\end{exercise}
\begin{solution}
Since $a_n\to\ell$ and $a_n\geq 0$ for all $\nnn$,
it follows that $\ell\geq 0$ by Corollary~\ref{c:limleq}. 
Let $\varepsilon>0$. 

\emph{Case~1:} $\ell=0$.\\
Since $\varepsilon^2>0$, there exists $N\in\NN$ such that
for every $n\geq N$, we have $a_n=|a_n-0|<\varepsilon^2$;
thus,  $|\sqrt{a_n}-0|=\sqrt{a_n}<\varepsilon$ as required.

\emph{Case~2:} $\ell>0$.\\
Since $\sqrt{\ell}\varepsilon>0$, there exists
$N\in\NN$ such that
for every $n\geq N$, we have
$|a_n-\ell|<\sqrt{\ell}\varepsilon$. 
It follows that
\begin{align*}\left|\sqrt{a_n}-\sqrt{\ell}\right|
&=
\frac{\left|\sqrt{a_n}-\sqrt{\ell}\right|\left|\sqrt{a_n}+\sqrt{\ell}\right|}{\left|\sqrt{a_n}+\sqrt{\ell}\right|}
=\frac{\left|\left(\sqrt{a_n}-\sqrt{\ell}\right)\left(\sqrt{a_n}+\sqrt{\ell}\right)
\right|}{\sqrt{a_n}+\sqrt{\ell}}\\
&=\frac{\left|a_n-\ell\right|}{\sqrt{a_n}+\sqrt{\ell}}
\leq \frac{\left|a_n-\ell\right|}{\sqrt{\ell}}
<\frac{\sqrt{\ell}\varepsilon}{\sqrt{\ell}} = \varepsilon,
\end{align*}
as required.
\end{solution}

\begin{exercise} 
\label{exo:1/root}
Show that 
\begin{equation*}
    \frac{1}{\sqrt{n}} \to 0.
\end{equation*}
\end{exercise}
\begin{solution}
Let $n\geq 1$.
Then $0 < \sqrt{n}<\sqrt{n+1}$ and thus
\begin{equation*}
0<    \frac{1}{\sqrt{n+1}} < \frac{1}{\sqrt{n}}.
\end{equation*}
It follows that $(1/\sqrt{n})_\nnn$ is
decreasing and bounded below (by $0$), hence 
there exists $\ell \in\RR$ such that 
\begin{equation*}
\frac{1}{\sqrt{n}} \to \ell. 
\end{equation*}
By the limit laws, 
\begin{equation*}
0 \leftarrow \frac{1}{n} = \frac{1}{\sqrt{n}} \cdot \frac{1}{\sqrt{n}}\to \ell\cdot \ell. 
\end{equation*}
Therefore, $\ell^2=0$ and thus $\ell=0$.
\end{solution}

\begin{exercise} Does the sequence 
$(\sqrt{n+9}-\sqrt{n})_\nnn$ converge?
If so, find the limit.
If not, explain why not.
\end{exercise}
\begin{solution}
Let $n\geq 1$.
Then 
\begin{align*}
0 < \sqrt{n+9}-\sqrt{n}
&= \frac{\big(\sqrt{n+9}+\sqrt{n}\big)\big(\sqrt{n+9}-\sqrt{n}\big)}%
{\sqrt{n+9}+\sqrt{n}}
= \frac{\sqrt{n+9}^2-\sqrt{n}^2}%
{\sqrt{n+9}+\sqrt{n}}
= \frac{9}{\sqrt{n+9}+\sqrt{n}}\\
&< \frac{9}{\sqrt{n}}
\to 0
\end{align*}
using Exercise~\ref{exo:1/root}.
By the Squeeze Theorem, we conclude that 
$\sqrt{n+9}-\sqrt{n} \to 0$.
\end{solution}

\begin{exercise}[YOU be the marker!] 
Consider the following statement 
\begin{equation*}
    \text{``$\sqrt{n+9}-\sqrt{n}\to 3$''}
\end{equation*}
and the following ``proof'':
\begin{quotation}
We have $\sqrt{n+9}-\sqrt{n}
= \sqrt{n} + \sqrt{9}-\sqrt{n} = \sqrt{9}$.\\
Now $\sqrt{9} = 3$ because the square root is 
nonnegative and $3^2 = 9$.\\
Altogether, the result follows.
\end{quotation}
Why is this proof wrong?
\end{exercise}
\begin{solution}
The error occurs right at the beginning:
It is \emph{not} true that 
$\sqrt{n+9} = \sqrt{n} + \sqrt{9}$.
For instance, when $n=16$,
we have $\sqrt{n+9} = \sqrt{16+9} = \sqrt{25} = 5$
while $\sqrt{n}+\sqrt{9} = \sqrt{16} + \sqrt{9} = 4 + 3 = 7$. 
\end{solution}

\begin{exercise}[YOU be the marker!] 
Consider the following statement 
\begin{equation*}
\text{``$2=1$''}
\end{equation*}
and the following ``proof'':
\begin{quotation}
First, $4-6 = 1-3$ $\Rightarrow$ 
$4-6+\tfrac{9}{4} = 1-3+\tfrac{9}{4}$
$\Rightarrow$ $2^2-2\cdot 2\cdot \tfrac{3}{2} 
+\big(\tfrac{3}{2}\big)^2= 1^2 - 2\cdot 1\cdot \tfrac{3}{2}
+\big(\tfrac{3}{2}\big)^2$.\\
Hence $\big(2-\tfrac{3}{2}\big)^2 = \big(1-\tfrac{3}{2}\big)^2$. \\
Taking square roots, we get
$2-\tfrac{3}{2} = 1 - \tfrac{3}{2}$.\\
Therefore, $2=1$.
\end{quotation}
Why is this proof wrong?
\end{exercise}
\begin{solution}
The error occurs when taking the square root.
Indeed, the square root is always nonnegative, 
so $\sqrt{\big(1-\tfrac{3}{2}\big)^2} = |1-\tfrac{3}{2}| = 
\tfrac{3}{2}-1$ which is different from $1-\tfrac{3}{2}$. 
\end{solution}

\begin{exercise}[TRUE or FALSE]
Each of the following statements is either true or false. 
Briefly justify your answer.
\begin{enumerate}
\item ``If $\alpha\geq 0$ and $\beta\geq 0$ are real numbers, then 
$\sqrt{\alpha\beta} = \sqrt{\alpha}\sqrt{\beta}$.''
\item ``If $\alpha\geq 0$ and $\beta\geq 0$ are real numbers, then 
$\sqrt{\alpha+\beta} = \sqrt{\alpha}+\sqrt{\beta}$.''
\item ``The square root of any nonnegative real number exists and 
is again nonnegative and real.''
\item ``The square root of any nonnegative rational number exists and 
is again nonnegative and rational.''
\end{enumerate}
\end{exercise}
\begin{solution}
(i): TRUE: See Exercise~\ref{exo:260304a}.

(ii): FALSE: Consider $\alpha=9$ and $\beta=16$. 

(iii): TRUE: We even have an algorithm that quickly finds square roots!

(iv): FALSE: $\sqrt{2}\in\RR\smallsetminus\QQ$. 
\end{solution} 
\chapter{Series}

\hhbcom{Stromberg's book is an excellent source, see page 58 or so}

\section{Basic Convergence Tests}

\index{Series}

\begin{proposition}[Cauchy Test]
\label{p:cauchycrit}\index{Cauchy test (for series)}
Let $(a_n)_\nnn$ be a sequence of real numbers.
Then the series $\sum_{k\in\NN} a_k$ converges if and only if
\begin{quotation}
\noindent
for every $\varepsilon>0$, there exists $N\in\NN$ such that
$\displaystyle\left|\sum_{k=m}^n a_k\right| < \varepsilon$ 
for all $n\geq m\geq N$.
\end{quotation}
\end{proposition}
\begin{proof}
Recall that  $\sum_{k\in\NN} a_k$ converges if and only if
the sequence $(s_n)_\nnn$ of partial sums converges,
where $s_n = \sum_{k=0}^n a_k$.
In turn, $(s_n)_\nnn$ converges if and only if it is a Cauchy 
sequence (Theorem~\ref{t:convCauchy} and Definition~\ref{d:compaxiom}), 
which yields the result.
\end{proof}

\begin{lemma}
\label{l:convnull}
Let $(a_n)_\nnn$ be a sequence of real numbers such that
$\sum_{k\in\NN} a_k$ converges. Then $a_n\to 0$.
\end{lemma}
\begin{proof}
Indeed, let $\varepsilon >0$.
By Proposition~\ref{p:cauchycrit},
there exists $N\in\NN$ such that for every $n=m\geq N$, we have
$|\sum_{k=n}^n a_k| = |a_n|<\varepsilon$.
Therefore, $a_n\to 0$.
\end{proof}

The contrapositive of the last result gives immediately
a test for divergence.

\begin{corollary}[Divergence Test]
\index{Divergence Test}
\label{c:divtest}
Let $(a_n)_\nnn$ be a sequence of real numbers such that $a_n\not\to 0$.
Then $\sum_{k\in\NN}a_k$ diverges.
\index{Divergence test (for series)}
\end{corollary}

\begin{example}
$\sum_{k\in\NN} (-1)^k$ diverges because $(-1)^n\not\to 0$.
\end{example}

\begin{theorem}[Bounded Partial Sums Test]
\index{Bounded Partial Sums Test}
\label{t:seriesbounded}
Let $(a_n)_\nnn$ be a sequence of nonnegative real numbers.
Then $\sum_{k\in\NN} a_k$ converges if and only if the sequence of
partial sums $(\sum_{k=0}^n a_k)_\nnn$ is bounded.
\end{theorem}
\begin{proof}
Denote the sequence of partial sums of $(a_n)_\nnn$ by $(s_n)_\nnn$.
Then, since each $a_n\geq 0$, it is clear that $(s_n)_\nnn$ is
increasing. 
If $(s_n)_\nnn$ is bounded, it is convergent by
Theorem~\ref{t:monconv}. Conversely, if $(s_n)_\nnn$ is convergent,
then it is bounded by Theorem~\ref{t:convbound}.
\end{proof}

\begin{example}[Harmonic Series]
\label{ex:harser}\index{Harmonic Series}\index{Harmonic Numbers}
The \emph{Harmonic Series}
$\sum_{k\geq 1} \frac{1}{k}$ diverges to $\pinf$
because the increasing sequence of partial sums 
$(s_n)_\nnn = (\sum_{k=1}^n\frac{1}{k})_{k\in\NN}$ is unbounded:
indeed,
\begin{subequations}
\begin{align}
s_1 &= 1,\\
s_2 &= 1 + \tfrac{1}{2},\\
s_4 &= 1 + \tfrac{1}{2} + \underbrace{\tfrac{1}{3}+\tfrac{1}{4}}_{>
2\cdot \tfrac{1}{4} = \tfrac{1}{2}} >
1+2\cdot\tfrac{1}{2},\\
s_8 &= s_4 + \underbrace{\tfrac{1}{5} + \cdots + \tfrac{1}{8}}_{>
4\cdot\tfrac{1}{8} = \tfrac{1}{2}} > 1+ 3\cdot\tfrac{1}{2},\\
&\quad \vdots\\
s_{2^n} &\geq 1 + n\cdot\frac{1}{2},
&\quad 
\end{align}
\end{subequations}
and the result now follows from Theorem~\ref{t:seriesbounded}. 
Note that the harmonic series also illustrates that the converse
of Lemma~\ref{l:convnull} fails: indeed, $\frac{1}{n}\to 0$ yet
the harmonic series is not convergent.
\end{example}

\begin{example}[$p$ Series]
\label{ex:pseries}\index{p Series}
Let $p$ be an integer such that $p\geq 2$. 
Then $\sum_{k\geq 1} \frac{1}{k^p}$ is convergent\footnote{In fact, 
formulas are known when $p$ is even such as $\sum_{k\geq
1}\frac{1}{k^2} = \pi^2/6$. The situation is much less understood when
$p$ is odd.}.
Since each term in this series is positive
and in view of Theorem~\ref{t:seriesbounded}, it suffices to 
show that the (increasing) sequence of partial sums $(s_n)_\nnn$ 
is bounded.  Indeed,
let $n\geq 1$. Then there exists $m\in\NN$ such that
$n+1 \leq 2^{m+1}$ (see Theorem~\ref{t:100920}). 
Thus, using Theorem~\ref{t:geoseries}, 
\begin{subequations}
\begin{align}
s_n &\leq s_{2^{m+1}-1}
= \sum_{k=1}^{2^{m+1}-1} \frac{1}{k^p} 
= 1 + \left(\frac{1}{2^p} + \frac{1}{3^p}\right) + \cdots
+ \left(\sum_{k=2^m}^{2^{m+1}-1} \frac{1}{k^p}\right)\\
&\leq \sum_{l=0}^{m}2^l\frac{1}{(2^l)^p}
=\sum_{l=0}^{m} \left(\frac{1}{2^{p-1}}\right)^l
\leq \sum_{l=0}^{\pinf} \big(2^{-p+1}\big)^l
= \frac{1}{1-2^{-p+1}}
<+\infty.
\end{align}
\end{subequations}
\end{example}

\section{Alternating Series}

\begin{theorem}[Leibniz Alternating Series Test]
\label{t:Leibniz}\index{Leibniz Test}
\index{Alternating Series Test}
\index{Leibniz Alternating Series Test}
Let $(a_n)_\nnn$ be a decreasing sequence of nonnegative real numbers
such that $a_n\to 0$. 
Then $a_0-a_1+a_2-a_3\pm\cdots = \sum_{k\in\NN} (-1)^k a_k$ is convergent. 
\end{theorem}
\begin{proof}
Denote the sequence of partial sums by $(s_n)_\nnn$.
Then $s_{2n+2}-s_{2n} = a_{2n+2} - a_{2n+1} \leq 0$ 
and so $(s_{2n})_{n\in\NN}$ is decreasing.
Similarly, $(s_{2n+1})_{n\in\NN}$ is increasing.
Furthermore, $s_{2n+1}-s_{2n} = -a_{2n+1}\leq 0$ and
so $s_{2n+1}\leq s_{2n}$, for every $n\in\NN$. 
Altogether,
\begin{equation}
\label{e:premt2}
s_0\geq s_2\geq s_4\geq \cdots s_{2n}\geq s_{2n+1}\geq \cdots \geq
s_{3}\geq s_1.
\end{equation}
Therefore, $(s_{2n})_{n\in\NN}$ is decreasing and bounded,
hence convergent, say to $\alpha$,  by Theorem~\ref{t:monconv}.
Similarly,  $(s_{2n+1})_{n\in\NN}$ is increasing and bounded,
hence convergent, say to $\beta$.
Since $s_{2n+1}-s_{2n} = -a_{2n+1}\to 0$, it follows that 
$\alpha=\beta$. 
We now show that the entire sequence $(s_n)_\nnn$ is in fact
converging to $\alpha$. 
To this end, let $\varepsilon>0$.
Then there exist $N_1$ and $N_2$ in $\NN$ such that
\begin{equation}
n\geq N_1
\quad\Rightarrow\quad
|s_{2n}-\alpha|<\varepsilon
\end{equation}
and
\begin{equation}
n\geq N_2
\quad\Rightarrow\quad
|s_{2n+1}-\alpha|<\varepsilon. 
\end{equation}
Now set $N := \max\{2N_1,2N_2+1\}$.
Then if $n\geq N$, we have $|s_n-\alpha|<\varepsilon$.
\end{proof}

\begin{example}[Alternating Harmonic Series]
\label{ex:altharser}\index{Alternating Harmonic Series}
The \emph{Alternating Harmonic Series}
\begin{equation}
\sum_{k\geq 1} \frac{(-1)^{k-1}}{k} = 1 - \frac{1}{2} + \frac{1}{3} -
\frac{1}{4} \pm \cdots
\end{equation}
converges\footnote{In fact, the limit of this series is $\ln(2)$ as
can be shown with tools from Analysis.} 
by Theorem~\ref{t:Leibniz}.
\end{example}

\begin{example}
The \emph{alternating odd harmonic series} 
\begin{equation}
\sum_{k\in\NN} \frac{(-1)^k}{2k+1} = 1 -\frac{1}{3} + \frac{1}{5} -
\frac{1}{7}\pm \cdots
\end{equation}
converges\footnote{In fact, the value of this series is
$\frac{\pi}{4}$.} 
by Theorem~\ref{t:Leibniz}.
\end{example}

\section{Absolutely Convergent Series}

\begin{definition}[Absolute Convergence]
\index{absolute convergence (for series)}
\index{absolutely convergent series}
Let $(a_n)_\nnn$ be a sequence of real numbers.
Then $\sum_{k\in\NN} a_k$ \textbf{converges absolutely}
if $\sum_{k\in\NN} |a_k|$ is convergent.
\end{definition}

\begin{theorem}
\label{t:absconvimpliesregconv}
Every absolutely convergent series is convergent.
\end{theorem}
\begin{proof}
Let $(a_n)_\nnn$ be sequence such that
$\sum_{k\in\NN} a_k$ is absolutely convergent, i.e.,
 $\sum_{k\in\NN} |a_k|$ is convergent.
By the Cauchy test (Proposition~\ref{p:cauchycrit}),
this means that for every $\varepsilon>0$, there exists $N\in\NN$
such that for all $n\geq m\geq N$, we have 
$\big|\sum_{k=m}^n |a_k|\big|<\varepsilon$;
thus, 
$|\sum_{k=m}^{n}a_k|\leq \sum_{k=m}^n |a_k|=\big|\sum_{k=m}^n
|a_k|\big|<\varepsilon$.
Again by  Proposition~\ref{p:cauchycrit},
it follows that $\sum_{k\in\NN}a_k$ is convergent.
\end{proof}

\begin{remark}
The converse of Theorem~\ref{t:absconvimpliesregconv}
fails in general. Indeed, the alternating harmonic series
(Example~\ref{ex:altharser}) converges, but the harmonic series
diverges (Example~\ref{ex:harser}); 
thus the alternating harmonic series converges but
not absolutely\footnote{A series that converges but not absolutely is
also called \textbf{conditionally convergent}.}.
\end{remark}

\begin{theorem}[Comparison Test]
\label{t:compartest}\index{Comparison Test}
Let $(b_n)_\nnn$ be a sequence of nonnegative numbers such that 
$\sum_{k\in\NN}b_k$ converges, and let $(a_n)_\nnn$ be a sequence of
real numbers such that for every $n\in\NN$, $|a_n|\leq b_n$.
Then the series $\sum_{k\in\NN}a_k$ is absolutely convergent.
\end{theorem}
\begin{proof}
Let $\varepsilon>0$. Since $\sum_{k\in\NN}b_k$ is convergent and since
all terms of this series are nonnegative, by
Proposition~\ref{p:cauchycrit} 
there exists $N\in\NN$ such that for all $n\geq m\geq N$, we have 
\begin{equation}
\sum_{k=m}^{n}|a_k| 
\leq \sum_{k=m}^{n} b_k
= \left|\sum_{k=m}^{n} b_k\right|
<\varepsilon.
\end{equation}
Again from Proposition~\ref{p:cauchycrit} 
it follows that the series 
$\sum_{k\in\NN} a_k$ is absolutely convergent.
\end{proof}

The contrapositive of the last theorem yields a useful test for
divergence.

\begin{corollary}
\label{c:compartest}
Let $(a_n)_\nnn$ be a sequence of nonnegative numbers such that 
$\sum_{k\in\NN}a_k$ diverges, and let $(b_n)_\nnn$ be a sequence of
real numbers such that for every $n\in\NN$, $a_n\leq b_n$.
Then the series $\sum_{k\in\NN}b_k$ is divergent.
\end{corollary}
\begin{proof}
Indeed, if $\sum_{k\in\NN}b_k$ were not divergent, then 
by Theorem~\ref{t:compartest} it would follow that 
$\sum_{k\in\NN} a_k$ is convergent.
\end{proof}

Let us revisit the $p$-series (Example~\ref{ex:pseries})
with a much simpler proof.\index{p-series}
\begin{example}
If $p$ is an integer such that $p\geq 2$, 
then (Exercise~\ref{exo:previsited})
\begin{equation}
\label{e:previsited}
\frac{1}{n^p} \leq \frac{1}{n^2} \leq \frac{2}{n(n+1)},
\quad\text{for every $n\geq 1$.}
\end{equation}
In view of Example~\ref{ex:seite24} and Theorem~\ref{t:serieslaws},
it follows from the comparison test that 
$\sum_{k\in\NN}1/k^p$ is absolutely convergent.
\end{example}

\begin{theorem}[Ratio Test]
\label{t:ratiotest}\index{Ratio Test}
Let $(a_n)_\nnn$ be a sequence of nonzero real numbers and
let $\beta\in\left]0,1\right[$.
Suppose that for some $n_0\in\NN$ we have
\begin{equation}
\left| \frac{a_{n+1}}{a_n}\right|
\leq \beta,
\quad
\text{for every $n\geq n_0$.}
\end{equation}
Then the series $\sum_{k\in\NN}a_k$ converges absolutely.
\end{theorem}
\begin{proof}
Since (absolute) convergence does not change when finitely many terms
of a series are changed, we assume for convenience that $n_0=0$, i.e.,
$|a_{n+1}/a_n|\leq \beta$, for all $\nnn$.
An easy mathematical induction yields
\begin{equation}
|a_n| \leq |a_0|\beta^n,
\quad
\text{for every $n\in\NN$.}
\end{equation}
Now $\sum_{k\in\NN} |a_0|\beta^k$ is absolutely convergent 
by Theorem~\ref{t:geoseries} and Theorem~\ref{t:serieslaws};
thus, the comparison test (Theorem~\ref{t:compartest}) yields
the absolute convergence of $\sum_{k\in\NN}a_k$.
\end{proof}

Exercise~\ref{exo:noconclusion} illustrates
that no conclusion can be drawn if $|a_{n+1}/a_n|\leq 1$.

Often, the ratio test is used in the following form.

\begin{corollary}[Ratio Test -- Utility Version]
\label{c:utirat}
\index{Ratio Test}
Let $(a_n)_\nnn$ be a sequence of nonzero real numbers and assume that
\begin{equation}
\label{e:newtest}
q := \lim_\nnn \left|\frac{a_{n+1}}{a_n}\right|
\end{equation}
exists. 
Then the following hold:
\begin{enumerate}
\item 
\label{c:utirat1}
If $q<1$, then 
the series $\sum_{k\in\NN}a_k$ converges absolutely.
\item 
\label{c:utirat2}
If $q>1$, then 
the series $\sum_{k\in\NN}a_k$ diverges.
\end{enumerate}
\end{corollary}
\begin{proof}
\ref{c:utirat1}:
Define $q_n := |a_{n+1}/a_n|$, for every $\nnn$. 
Set $\beta := (q+1)/2\in\zeroun$ and
$\varepsilon := (1-q)/2$. Then 
$q+\varepsilon = \beta$. 
Since $q_n\to q$, there exists $n_0\in\NN$ such that
$\varepsilon > |q_n-q| \geq q_n-q$ for every $n\geq n_0$.
In particular, $q_n<q+\varepsilon=\beta<1$ for 
every $n\geq n_0$. 
The conclusion now follows from Theorem~\ref{t:ratiotest}. 

\ref{c:utirat2}:
Exercise~\ref{exo:utirat2}. 
\end{proof}

\begin{example}
Consider the series $\sum_{\nnn} n^23^{-n}$. 
Set $a_n=n^23^{-n}$, for all $\nnn$. 
Using the limit laws, we deduce that 
\begin{align}
\left|\frac{a_{n+1}}{a_n}\right|
&=\frac{(n+1)^23^n}{n^23^{n+1}}=
\frac{1}{3}\bigg(1+\frac{1}{n}\bigg)^2
\to \frac{1}{3} =: q < 1.
\end{align}
Therefore, the series $\sum_{\nnn} n^23^{-n}$
converges absolutely by Corollary~\ref{c:utirat}\ref{c:utirat1}. 
In passing, we note that \texttt{Maple} or \texttt{WolframAlpha}
will give you even a formula for the partial sums
(which you could verify by induction); using this, one then
sees that the limit of this series is $3/2$. 
\end{example}

\hhbcom{One could develop
the theory of re-arrangements as in 
Forster page 43f.}

\hhbcom{Lay's book page 310 has lots of examples.}

\hhbcom{Forster Seite 40}

\section*{Exercises}\markright{Exercises}
\addcontentsline{toc}{section}{Exercises}
\setcounter{theorem}{0}

\begin{exercise}
\label{exo:previsited}
Prove \eqref{e:previsited}.
\end{exercise}
\begin{solution}
Let $n$ and $p$ be integers such that $n\geq 1$ and $p\geq 2$.
Then $0<1\leq n\leq n^2\leq \cdots \leq n^p$, and thus
$1/n^p \leq 1/n^2$. 
Also, $n(n+1)=n^2+n \leq n^2 + n\cdot n = 2n^2$. 
Therefore, 
$1/n^2 \leq 2/\big(n(n+1)\big)$. 
\end{solution}

\begin{exercise}
\label{exo:noconclusion}
Consider two instances of the series $\sum_{k\geq 1}a_k$, 
namely when
(i) $(\forall\nnn)$ $a_n = 1/n$, and when 
(ii) $(\forall\nnn)$ $a_n = 1/n^2$. 
Show that in either case $a_n\to 0$, 
$|a_{n+1}/a_{n}|<1$, and 
$|a_{n+1}/a_{n}| \to 1$.
Explain why no conclusion regarding convergence or divergence of
the series can be drawn.
\end{exercise}
\begin{solution}
We know that $1/n\to 0$, and hence $1/n^2 = (1/n)(1/n)\to 0\cdot 0 =
0$ by the product law for limits.
Let $n\geq 1$. Then $n<n+1$ and thus $n/(n+1)<1$;
thus, 
\begin{equation*}
1=\frac{1}{1+0}\leftarrow\frac{1}{1+\frac{1}{n}}=\frac{n}{n+1}=\frac{\frac{1}{n+1}}{\frac{1}{n}} < 1.
\end{equation*}
We also have $n^2 < (n+1)^2$; thus, 
\begin{equation*}
1=\Big(\frac{1}{1+0}\Big)^2\leftarrow\Big(\frac{1}{1+\frac{1}{n}}\Big)=\Big(\frac{n}{n+1}\Big)^2 = 
\frac{\frac{1}{(n+1)^2}}{\frac{1}{n^2}} < 1.
\end{equation*}
These examples dash any hopes for a Ratio Test with $q<1$
replaced by $1$:
indeed, in case~(i), we obtain the Harmonic Series
which is divergent (Example~\ref{ex:harser})
while in case~(ii), we obtain a convergent $p$-series 
(Example~\ref{ex:pseries}). 
\end{solution}

\begin{exercise}
Determine whether or not 
the series $\sum_{n\in\NN} n!/n^n$ absolutely converges.
\end{exercise}
\begin{solution}
We set $a_n := n!/n^n$, for $n\geq 1$. 
Then
\begin{align*}
\left|\frac{a_{n+1}}{a_n}\right|
&= \left|\frac{(n+1)!/(n+1)^{n+1}}{n!/n^n} \right|
= \frac{(n+1)!}{n!}\frac{n^{n}}{(n+1)^{n+1}}\\
&= (n+1) \frac{n^n}{(n+1)(n+1)^n}
= \frac{n^n}{(n+1)^n} \\
&= \left(\frac{n}{n+1}\right)^n
= \left(\frac{1}{1+1/n}\right)^n
=\frac{1}{\left(1+\frac{1}{n}\right)^n}.
\end{align*}
Now, using Bernoulli's inequality, we have
$\big(1+\frac{1}{n}\big)^n \geq 1 + n\frac{1}{n}=2$;
equivalently $\big(1+\frac{1}{n}\big)^{-n} \leq 2^{-1}$.
Altogether, 
\begin{equation*}
\left|\frac{a_{n+1}}{a_n}\right|
= \frac{1}{\left(1+\frac{1}{n}\right)^n} 
\leq \frac{1}{2}.
\end{equation*}
Therefore, by the Ratio Test (Theorem~\ref{t:ratiotest}),
the series is absolutely convergent.
\end{solution}

\begin{exercise}
Determine whether or not 
the series $\sum_{n\in\NN} n^5/4^n$ absolutely converges.
\end{exercise}
\begin{solution}
Set $a_n:=n^54^{-n}$, for all $\nnn$. 
Using the limit laws, we deduce that 
\begin{align*}
\left|\frac{a_{n+1}}{a_n}\right|
&=\frac{(n+1)^54^n}{n^54^{n+1}}=
\frac{1}{4}\bigg(1+\frac{1}{n}\bigg)^5
\to \frac{1}{4} =: q < 1.
\end{align*}
Therefore, the series $\sum_{\nnn} n^54^{-n}$
converges absolutely by Corollary~\ref{c:utirat}\ref{c:utirat1}. 
\end{solution}

\begin{exercise}
Determine whether or not 
the series $\sum_{n\geq 1} (n+10)/(n^2-3n+1)$ absolutely converges.
\end{exercise}
\begin{solution}
Let $n$ be an integer such that $n\geq 1$.
Then $n+10> n$. Also, 
$3n>1$
$\Leftrightarrow$
$n^2> n^2 -3n+1$
$\Leftrightarrow$
$1/(n^2 -3n+1) > 1/n^2$.
Thus, 
\begin{equation*}
\frac{n+10}{n^2-3n+1}
> \frac{n}{n^2} = \frac{1}{n}.
\end{equation*}
Since the Harmonic Series diverges (Example~\ref{ex:harser}),
it follows from Corollary~\ref{c:compartest} that the series
under investigation diverges as well.
\end{solution}

\begin{exercise}
Recall that the sequence of \emph{harmonic numbers} $(h_n)_\nnn$ 
\index{Harmonic numbers} is exactly the sequence of partial sums
of the harmonic series. In view of Example~\ref{ex:harser},
the sequence $(h_n)_\nnn$ diverges to $+\infty$.
Show that $\sum_{k\geq 2}h_k/(k(k-1))$ converges
and deduce that $h_n/(n(n-1))\to 0$. 
\emph{Hint:} Exercise~\ref{exo:strangeharm}.
\end{exercise} 
\begin{solution}
It is clear that the terms of the series are all nonnegative.
By Exercise~\ref{exo:strangeharm}, the series of partial sums
is bounded above by $2$. 
Hence, by Theorem~\ref{t:seriesbounded}, the series converges. 
Finally, by Lemma~\ref{l:convnull}, the terms converge to $0$.
\end{solution}

\begin{exercise}
\label{exo:utirat2}
Prove Corollary~\ref{c:utirat}\ref{c:utirat2}.
\end{exercise} 
\begin{solution}
Set $b_n := |a_{n+1}/a_n|$ for $\nnn$.
Then $b_n\to q>1$. 
Let $\varepsilon := (q-1)/2>0$.
Then there exists $N\in\NN$ such that 
$(\forall n\geq N)$ $|b_n-q|<\varepsilon$. 
Hence for $n\geq N$, we have 
$\varepsilon > |b_n-q| \geq q-b_n$ and thus
$b_n > q-\varepsilon = (1+q)/2 > 1$. In particular,
$(\forall n\geq N)$ $|a_{n+1}/a_n|=b_n\geq 1$.
An easy induction gives
$$(\forall n\geq N)\quad |a_{n+1}|\geq|a_n|\geq \cdots \geq |a_N|>0. $$
It follows that $|a_n|\not\to 0$. 
By Proposition~\ref{p:absantoabsa}, $a_n\not\to 0$. 
Finally, by Corollary~\ref{c:divtest}, $\sum_{k\in\NN} a_k$ diverges.
\end{solution}

\begin{exercise}
Determine whether each of the following series is either
convergent or divergent. Briefly justify your answer.
In case of convergence, you do not need to find the limit.
\begin{enumerate}
\item $\displaystyle \sum_{n=0}^{\infty} \frac{(-3)^n}{n!}$
\item $\displaystyle \sum_{n=0}^{\infty} (-3)^n\cdot n!$
\item $\displaystyle \sum_{n=0}^{\infty} \frac{(-1)^n}{3n+2}$
\item $\displaystyle \sum_{n=0}^{\infty} \frac{1}{\sqrt{3n+1}}$
\item $\displaystyle \sum_{n=1}^{\infty} \frac{1}{n^3}$
\end{enumerate}
\end{exercise} 
\begin{solution}
(i): We use the \textbf{Ratio Test} 
(Corollary~\ref{c:utirat}) to demonstrate \textbf{convergence}. 
Set $a_n := (-3)^n/n!$ for every $\nnn$. 
Then
\begin{equation*}
\left| \frac{a_{n+1}}{a_n}\right| =
\frac{\left|(-3)^{n+1}/(n+1)!\right|}{\left|(-3)^n/n!\right|} =
\frac{3^{n+1}\cdot n!}{3^n \cdot (n+1)!} = \frac{3}{n+1} \to 0 <
1,
\end{equation*}
as needed!

(ii): We use the \textbf{Divergence Test} 
(Corollary~\ref{c:divtest}). 
Set $a_n := (-3)^n\cdot n!$ for every $\nnn$. 
Then $|a_n| \geq 1$ for every $\nnn$ and so $a_n\not\to 0$ as
needed. 

(iii): We use the \textbf{Leibniz (Alternating Series) Test} 
(Theorem~\ref{t:Leibniz}) for \textbf{convergence}. 
Set $a_n := 1/(3n+2)$  for every $\nnn$. 
Then $a_n>a_{n+1} > 0$ and $a_n\to 0$ as needed. 

(iv): We use the comparison test for \textbf{divergence}
(Corollary~\ref{c:compartest}).
Set $b_n := 1/\sqrt{3n+1}$ and $a_n := 1/(3(n+1))$, for every
$\nnn$.
Then 
\begin{equation*}
b_n = \frac{1}{\sqrt{3n+1}} \geq \frac{1}{3n+1}>\frac{1}{3n+3} =
a_n,
\end{equation*}
for every $\nnn$.
Since $\sum_{n=1}^\infty \tfrac{1}{n} = \sum_{n=0}^\infty
\tfrac{1}{n+1}$ diverges (harmonic series, see
Example~\ref{ex:harser}) so must $\sum_{\nnn}a_n$ and we are done. 

(v): This is a $p$ series with $p=3\geq 2$ and we are thus
\textbf{convergent} by Example~\ref{ex:pseries}. 
\end{solution}

\begin{exercise} 
\label{exo:Cauchy0}
Suppose that $(a_n)_{n\geq 1}$ is a sequence of real 
numbers such that $a_n\to 0$.
Show that 
\begin{equation*}
\frac{1}{n}\sum_{i=1}^na_i \to 0.
\end{equation*}
\emph{Hint:} You could proceed along the following steps:
\begin{quotation}
1.~Get $N_1\geq 1$ such that $|a_{N_1}+\cdots + a_n|/n < \thalb\varepsilon$ 
for $n\geq N_1$. \\
2.~Then get $N_2 \geq 1$ such that 
$|a_{1}+\cdots + a_{N_1-1}|/n < \thalb\varepsilon$ for $n\geq N_2$.\\
3.~Now combine Steps~1 and 2.
\end{quotation}
\end{exercise}
\begin{solution}
Let $\varepsilon > 0$.

Because $a_n\to 0$ and $\varepsilon/2 > 0$, 
there exists $N_1\geq 1$ such that 
\begin{equation*}
    n\geq N_1 \;\;\Rightarrow\;\;
    |a_n| = |a_n-0| < \thalb\varepsilon.
\end{equation*}
It follows that for all $n\geq N_1$, 
\begin{align}
\label{e:200813a}
\frac{|a_{N_1}+a_{N_1+1}+\cdots + a_n|}{n}
&\leq \frac{|a_{N_1}|+|a_{N_1+1}|+\cdots + |a_n|}{n} 
< \frac{(n-N_1+1)\thalb\varepsilon}{n} 
\leq \thalb\varepsilon.
\end{align}
Next, because $1/n\to 0$ and $N_1$ is \emph{fixed}, it follows
that 
\begin{equation*}
\frac{a_1+a_2+\cdots+a_{N_1-1}}{n} \to 0.
\end{equation*}
Hence there exists $N_2\geq 1$ such that 
for all $n\geq N_2$, we have 
\begin{equation}
\label{e:200813b}
\frac{|a_{1}+a_{2}+\cdots+a_{N_1-1}|}{n} < \thalb\varepsilon.
\end{equation}
Now take an arbitrary $n\geq N := \max\{N_1,N_2\}$. 
Combining \eqref{e:200813a} and \eqref{e:200813b}, 
we estimate 
\begin{align*}
\left|\frac{a_1+a_2+\cdots + a_n}{n}\right|
&\leq 
\frac{|a_1+\cdots+a_{N_1-1}|}{n}
+ \frac{|a_{N_1}+\cdots+a_{n}|}{n}
< \thalb\varepsilon + \thalb\varepsilon =\varepsilon,
\end{align*}
and we're done!
\end{solution}

\begin{exercise} 
\label{exo:Cauchy1}
Suppose that $(b_n)_{n\geq 1}$ is a sequence of real 
numbers such that $b_n\to \beta$.
Show that 
\begin{equation*}
\frac{1}{n}\sum_{i=1}^nb_i \to \beta.
\end{equation*}
\emph{Hint:} You may use Exercise~\ref{exo:Cauchy0}.
\end{exercise}
\begin{solution}
Set $a_n \equiv b_n-\beta$. Then $a_n\to 0$.
Using Exercise~\ref{exo:Cauchy0}, 
we obtain
\begin{align*}
\frac{1}{n}\sum_{i=1}^nb_i
&=
\frac{1}{n}\sum_{i=1}^n (a_i+\beta)
=
\bigg(\frac{1}{n}\sum_{i=1}^n a_i\bigg) +
\bigg(\frac{1}{n}\sum_{i=1}^n \beta\bigg)
=
\bigg(\frac{1}{n}\sum_{i=1}^n a_i\bigg) +
\beta \\
&\to 0 + \beta = \beta
\end{align*}
as desired.
\end{solution}

\begin{exercise} 
\label{exo:Cauchy4}
Let $(a_n)_\nnn$ be a sequence of real numbers such that 
\begin{equation*}
a_{n+1}-a_n\to \ell\in\RR.
\end{equation*}
Show that 
\begin{equation*}
\frac{a_n}{n}\to \ell.
\end{equation*}
\emph{Hint:} You may use Exercise~\ref{exo:Cauchy1}.
\end{exercise}
\begin{solution}
Set $b_n := a_{n}-a_{n-1}$ for $n\geq 1$.
Then $b_n\to \ell$ by assumption.
Now Exercise~\ref{exo:Cauchy1} yields
\begin{equation*}
\ell \leftarrow \frac{1}{n}\sum_{i=1}^n b_i 
=
\frac{1}{n}\sum_{i=1}^n (a_i-a_{i-1})
= \frac{1}{n}\big( a_n - a_0),
\end{equation*}
which implies 
\begin{equation*}
\frac{a_n}{n} = 
\bigg(\frac{1}{n}\sum_{i=1}^n b_i\bigg) + \frac{a_0}{n}
\to \ell+0 = \ell
\end{equation*}
as announced.
\end{solution}

\begin{exercise} 
\label{exo:Cauchy2}
Prove or disprove: 
Suppose that $(b_n)_{n\geq 1}$ is a sequence of real 
numbers such that 
\begin{equation*}
\frac{1}{n}\sum_{i=1}^nb_i \to \beta \in\RR.
\end{equation*}
Then $b_n\to \beta$.
\end{exercise}
\begin{solution}
The statement is FALSE!
Here is a counterexample.
Set $b_n = (-1)^n$.
Then 
\begin{equation*}
\sum_{i=1}^n b_i
= \sum_{i=1}^n (-1)^i
= \begin{cases}
0, &\text{if $n$ is even;}\\
-1, &\text{if $n$ is odd.}
\end{cases}
\end{equation*}
It follows that 
\begin{equation*}
0\leftarrow \frac{-1}{n}
\leq \frac{1}{n}\sum_{i=1}^n b_n \leq 0,
\end{equation*}
and the Squeeze Theorem now gives 
\begin{equation*}
\frac{1}{n}\sum_{i=1}^n b_n \to 0.
\end{equation*}
On the other hand, $b_n=(-1)^n\not\to 0$.
\end{solution}

\begin{exercise} 
\label{exo:Cauchy3}
Prove or disprove: 
Suppose that $(b_n)_{n\geq 1}$ is a sequence of positive real 
numbers such that $b_n\to+\infty$. 
Then 
\begin{equation*}
\frac{1}{n}\sum_{i=1}^nb_i \to +\infty.
\end{equation*}
\end{exercise}
\begin{solution}
The statement is TRUE!
Let $C>0$. 
Because $b_n\to\pinf$, 
there exists $N\geq 1$ such that 
\begin{equation*}
    n\geq N
    \;\;\Rightarrow\;\;
    b_n>C.
\end{equation*}
Set $S := \sum_{i=1}^{N-1} b_i$, which is just a positive real constant.
It follows that 
\begin{align*}
\frac{1}{n}\sum_{i=1}^n b_i
&= \bigg(\frac{1}{n}\sum_{i=1}^{N-1} b_i\bigg)
+ \bigg(\frac{1}{n}\sum_{i=N}^n b_i\bigg)
\geq \frac{S}{n} + \frac{n-N+1}{n} C
\to 0 + C,
\end{align*}
which yields the result because $C$ was an arbitrary positive constant and 
so Theorem~\ref{t:seriesbounded} applies.
\end{solution}

\begin{exercise}[YOU be the marker!] 
Consider the following statement 
\begin{equation*}
\text{``$\sum_{n\geq 1}\frac{1}{n^2}$ diverges''}
\end{equation*}
and the following ``proof'':
\begin{quotation}
We know that $\sum_{n\geq 1} \tfrac{1}{n}$ diverges (Harmonic Series)!\\
Next, $n\geq 1$ $\Rightarrow$ $n^2\geq n$ $\Rightarrow$ 
$\tfrac{1}{n^2}\leq \tfrac{1}{n}$.\\
Therefore, by the Comparison Test, $\sum_{n\geq 1}\tfrac{1}{n^2}$ diverges.
\end{quotation}
Why is this proof wrong?
\end{exercise}
\begin{solution}
The last line makes an invalid statement. 
If you want to show that a series $\sum_{n\geq 1}a_n$ diverges
by comparing to $\sum_{n\geq 1}\tfrac{1}{n}$, then you have to 
show that $a_n\geq \tfrac{1}{n}$, not $a_n\leq \tfrac{1}{n}$.
(In fact, the series $\sum_{n\geq 1}\tfrac{1}{n^2}$ converges, 
as we know from Example~\ref{ex:pseries}.) 
\end{solution}

\begin{exercise}[YOU be the marker!] 
Consider the following statement 
\begin{equation*}
\text{``$\sum_{n\geq 1}\frac{1}{n}$ converges''}
\end{equation*}
and the following ``proof'':
\begin{quotation}
Set $a_n := \tfrac{1}{n}$ for all $n\geq 1$.\\
Then $|a_{n+1}/a_n| = n/(n+1) < 1$.\\
Hence the series converges by the Ratio Test.
\end{quotation}
Why is this proof wrong?
\end{exercise}
\begin{solution}
The last line makes an invalid statement. 
If you want to apply the Ratio Test, you need that 
$|a_{n+1}/a_{n}| \leq q <1$ for some $q<1$.
This is not true here because $|a_{n+1}/a_n| \to 1$.
Hence we cannot apply the Ratio Test.
(And indeed, the series, which is the Harmonic Series, 
does not converge.)
\end{solution}

\begin{exercise}[TRUE or FALSE?]
Mark each of the following statements as either true or false. 
Briefly justify your answer.
\begin{enumerate}
\item ``The alternating even harmonic series $\tfrac{1}{2}-\tfrac{1}{4}
+\tfrac{1}{6}-\tfrac{1}{8}\pm \cdots$ converges.''
\item ``If $\sum_{n\geq 1}a_n$ converges, then 
$\sum_{n\geq 1}|a_n|$ converges.''
\item ``Both $(\tfrac{1}{n})_{n\geq 1}$ and $\sum_{n\geq 1}\tfrac{1}{n}$ converge.''
\item ``If $\sum_{\nnn}a_n=\alpha$, then $a_n\to\alpha$.''
\end{enumerate}
\end{exercise}
\begin{solution}
(i): TRUE: This follows from the Leibniz Alternating Series Test.

(ii): FALSE: Consider $a_n\equiv (-1)^n\tfrac{1}{n}$. 

(iii): FALSE: While $1/n\to 0$, the Harmonic Series does diverge.

(iv): FALSE: $a_n\equiv (1/2)^n$ has the series being equal to $\alpha = 2$,
but $a_n\to 0 \neq 2= \alpha$. 
\end{solution} 
\chapter{Functions and Cardinality}
\label{cha:otherconv}

\section{Basic Definitions}

\begin{definition}
Let $X$ and $Y$ be nonempty sets.
We write 
\begin{equation}f\colon X\to Y\end{equation} and say $f$ is a \textbf{function}
with \textbf{domain} $X$ if for every $x\in X$, \index{Function}
$f(x)$ is an element in $Y$.
The set $Y$ is the \textbf{codomain} of $f$. 
The \textbf{graph} of the function $f$ is \index{Graph}
\begin{equation}
\gr f := \menge{(x,y)\in X\times Y}{y=f(x)}.
\end{equation}
The \textbf{range} of $f$ is the set \index{Range}
\begin{equation}
\ran f := f(X) := \bigcup_{x\in X} \{f(x)\}
= \menge{y\in Y}{\text{there exists $x\in X$ such that $f(x)=y$}}.
\end{equation}
The range of $f$ may be equal to --- or a proper subset of --- 
the codomain of $f$. 
More generally, if $S$ is a subset of $X$, we write
\begin{equation}
f(S) := \bigcup_{x\in S} \{f(x)\}
= \menge{y\in Y}{\text{there exists $x\in S$ such that $f(x)=y$}}
\end{equation}
and say that $f(S)$ is the \textbf{image} of $S$ under $f$ in $Y$.
Furthermore, if $T$ is a subset of $Y$, we write
\index{Image}\index{Pre-image}
\begin{equation}
f^{-1}(T) := \menge{x\in X}{f(x)\in T}
\end{equation}
and say that $f^{-1}(T)$ is the \textbf{pre-image} of $T$ under $f$ in $X$.
It will be convenient to define the \textbf{identity mapping} on $X$ by
\index{Identity mapping}
\begin{equation} \label{e:Id}
\Id_X \colon X\to X\colon x\mapsto x,
\end{equation}
and similarly for $\Id_Y$.
\end{definition}

Thus the graph of the function $f\colon X\to Y$ 
is a binary relation from $X$ to $Y$ (in the sense of
Definition~\ref{d:relation}) with the following two properties:
\begin{enumerate}
\item
for every $x\in X$, there exists $y\in Y$ such that $(x,y)\in\gr f$;
\item 
if $(x,y_1)$ and $(x,y_2)$ belong to $\gr f$, then $y_1=y_2$.
\end{enumerate}
Property~(ii) is also known as the \emph{Vertical Line Test} in
Calculus~I.

\begin{proposition}
\label{p:101108:1}
Let $f\colon X\to Y$,
let $A,A_1,A_2$ be subsets of $X$,
and let $B,B_1,B_2$ be subsets of $Y$.
Then the following hold.
\begin{enumerate}
\item 
\label{p:101108:1i}
$A \subseteq f^{-1}\big(f(A)\big)$.
\item 
\label{p:101108:1ii}
$f\big( f^{-1}(B)\big) \subseteq B$.
\item 
\label{p:101108:1iii}
$f(A_1\cap A_2)\subseteq f(A_1)\cap f(A_2)$.
\item 
\label{p:101108:1iv}
$f(A_1\cup A_2)=f(A_1)\cup f(A_2)$.
\item 
\label{p:101108:1v}
$f^{-1}(B_1\cap B_2) = f^{-1}(B_1)\cap f^{-1}(B_2)$.
\item 
\label{p:101108:1vi}
$f^{-1}(B_1\cup B_2) = f^{-1}(B_1)\cup f^{-1}(B_2)$.
\item 
\label{p:101108:1vii}
$f^{-1}(Y\smallsetminus B) = X\smallsetminus f^{-1}(B)$. 
\end{enumerate}
\end{proposition}
\begin{proof}
\ref{p:101108:1i}:
Let $a\in A$. Then $f(a)\in f(A)$ and so $a\in f^{-1}\big(f(A)\big)$.
Thus, $A\subseteq f^{-1}\big(f(A)\big)$. 

\ref{p:101108:1iii}:
Let $y\in f(A_1\cap A_2)$.
Then there exists $x\in A_1\cap A_2$ such that $y= f(x)$.
Since $x$ belongs to both $A_1$ and $A_2$,
we deduce that $y= f(x)\in f(A_1)$ and that $y= f(x)\in f(A_2)$.
Thus, $y\in f(A_1)\cap f(A_2)$ and therefore
$f(A_1\cap A_2)\subseteq f(A_1)\cap f(A_2)$.

\ref{p:101108:1v}:
Let $x\in X$.
The announced equality follows because of the equivalences
$x\in f^{-1}(B_1\cap B_2)$
$\Leftrightarrow$
$f(x)\in B_1\cap B_2$
$\Leftrightarrow$
[$f(x)\in B_1$ and $f(x)\in B_2$]
$\Leftrightarrow$
[$x\in f^{-1}(B_1)$ and $x\in f^{-1}(B_2)$]
$\Leftrightarrow$
$x\in f^{-1}(B_1)\cap f^{-1}(B_2)$.

\ref{p:101108:1vii}:
Let $x\in X$.
The announced equality follows because of the equivalences
$x\in f^{-1}(Y\smallsetminus B)$
$\Leftrightarrow$
$f(x)\in Y\smallsetminus B$
$\Leftrightarrow$
[$f(x)\in Y$ and $f(x)\notin B$]
$\Leftrightarrow$
$f(x)\notin B$
$\Leftrightarrow$
$x\notin f^{-1}(B)$
$\Leftrightarrow$
$x\in X\smallsetminus f^{-1}(B)$.
\end{proof}

\begin{definition}
Let $f\colon X\to Y$ be a function. 
Then $f$ is:
\begin{enumerate}
\item \textbf{injective} (or \textbf{one-to-one}) 
if $f(x_1)=f(x_2)$ always implies $x_1=x_2$;
\item \textbf{surjective} (or \textbf{onto})
if $\ran f = Y$;
\item \textbf{bijective} if $f$ is both injective and surjective.
\end{enumerate}
If $f$ is injective (surjective, bijective, resp.), 
then $f$ is called an 
\textbf{injection} (\textbf{surjection}, \textbf{bijection}, resp.).
\index{injective}\index{one-to-one}
\index{surjective}\index{onto}
\index{bijective}
\end{definition}

If $f\colon X\to Y$ is injective and $x_1$ and $x_2$ are two different
elements of $X$, then $f(x_1)\neq f(x_2)$, which explains the name.

\begin{example}
\label{ex:quinny:1}
$f\colon \RR\to \RR\colon x\mapsto x^2$ is
neither injective nor surjective.
\end{example}

\begin{example}
\label{ex:quinny:2}
$f\colon \left[0,+\infty\right[\to \RR \colon x\mapsto x^2$ is 
injective but not surjective.
\end{example}

\begin{example}
\label{ex:quinny:3}
$f\colon \RR\to \left[0,+\infty\right[\colon x\mapsto x^2$ is
surjective but not injective.
\end{example}

\begin{example}
\label{ex:quinny:4}
$f\colon \left[0,+\infty\right[ \to \left[0,+\infty\right[\colon
x\mapsto x^2$ is both injective and surjective, i.e., bijective.
\end{example}

Additional assumptions on $f$ yield
the following strengthening of Proposition~\ref{p:101108:1}.

\begin{proposition}
\label{p:101108:2}
Let $f\colon X\to Y$,
let $A,A_1,A_2$ be subsets of $X$,
and let $B$ be a subset of $Y$.
Then the following hold.
\begin{enumerate}
\item 
\label{p:101108:2i}
If $f$ is injective, then $A = f^{-1}\big(f(A)\big)$.
\item 
\label{p:101108:2ii}
If $f$ is surjective, then  $f\big( f^{-1}(B)\big) = B$.
\item 
\label{p:101108:2iii}
If $f$ is injective, then $f(A_1\cap A_2) =  f(A_1)\cap f(A_2)$.
\end{enumerate}
\end{proposition}
\begin{proof}
\ref{p:101108:2i}:
By Proposition~\ref{p:101108:1}\ref{p:101108:1i},
$A\subseteq f^{-1}\big(f(A)\big)$.
Now take $x\in f^{-1}\big(f(A)\big)$. 
Then $f(x)\in f(A)$. Hence there exists $a\in A$ such that
$f(x)=f(a)$. Since $f$ is injective, it follows that $x=a\in A$.
Therefore, $f^{-1}\big(f(A)\big)\subseteq A$. 
\end{proof}

\section{Composition and the Inverse}

\begin{definition}
Let $X,Y,Z$ be nonempty sets,
let $f\colon X\to Y$,
and let $g\colon Y\to Z$.
The \textbf{composition} $g\circ f$
is the function\index{Composition}
\begin{equation}
g\circ f \colon X\to Z\colon x\mapsto g\big(f(x)\big).
\end{equation}
\end{definition}

The order in a composition is critical.

\begin{example}
Let $f\colon \RR\to\RR \colon x\mapsto x+1$,
let $g\colon \RR\to\RR\colon x\mapsto 2x$,
and let $x\in \RR$.
Then $(g\circ f)(x) = g\big(f(x)\big)
= g(x+1) = 2(x+1) = 2x+2$
while
$(f\circ g)(x) = f\big(g(x)\big) = f(2x) = 2x+1$.
\end{example}

\begin{proposition}
\label{p:quinlan}
Let $X,Y,Z$ be nonempty sets,
let $f\colon X\to Y$,
and let $g\colon Y\to Z$.
Then the following hold.
\begin{enumerate}
\item 
\label{p:quinlan:i}
If $f$ and $g$ are injective, then so is $g\circ f$.
\item 
\label{p:quinlan:ii}
If $f$ and $g$ are surjective, then so is $g\circ f$.
\item 
\label{p:quinlan:iii}
If $f$ and $g$ are bijective, then so is $g\circ f$.
\end{enumerate}
\end{proposition}
\begin{proof}
\ref{p:quinlan:i}:
Suppose that $f$ and $g$ are injective and 
that $g\big(f(x_1)\big) = 
(g\circ f)(x_1)=(g\circ f)(x_2) = g\big(f(x_2)\big)$.
Since $g$ is injective, it follows
that $f(x_1)=f(x_2)$.
Now, since $f$ is also injective, we obtain $x_1=x_2$.
Therefore, $g\circ f$ is injective.
\end{proof}

\begin{definition}\label{d:invfunc}
Let $f\colon X\to Y$ be bijective.
Then the \textbf{inverse function}
is the function from $Y$ to $X$ with graph\index{Inverse function}
\begin{equation}
\gr f^{-1} = \menge{(y,x)\in Y\times X}{(x,y)\in \gr f},
\end{equation}
i.e., for all $x\in X$ and $y\in Y$,
\begin{equation}
y= f(x)
\quad\Leftrightarrow\quad
x= f^{-1}(y).
\end{equation}
\end{definition}

\begin{theorem}
\label{t:invinv}
Let $f\colon X\to Y$ be bijective.
Then $f^{-1}\colon Y\to X$ is also bijective,
$\big(f^{-1}\big)^{-1} = f$, 
$f^{-1}\circ f = \Id_X$,
and $f\circ f^{-1} = \Id_Y$.
\end{theorem}
\begin{proof}
It is clear that $f^{-1}$ is surjective.
To show that $f^{-1}$ is injective, assume
that $f^{-1}(y_1)=f^{-1}(y_2)$.
Applying $f$ yields $y_1=y_2$, and so $f^{-1}$ is injective.
Now let $x\in X$. Then $f(x)\in Y$ and $f^{-1}(f(x))= x$ so
that $f^{-1}\circ f = \Id_X$. 
One sees similarly that $f\circ f^{-1} = \Id_Y$.
Finally, 
\begin{equation}
\gr \big(f^{-1}\big)^{-1} = \menge{(x,y)\in X\times Y}{(y,x)\in \gr
f^{-1}} = \gr f,
\end{equation}
as announced. 
\end{proof}

\begin{example}
The function 
$f\colon\left[0,\pinf\right[\to\left[0,\pinf\right[\colon x\mapsto x^2$
is bijective, with inverse function
$f^{-1}\colon\left[0,\pinf\right[\to\left[0,\pinf\right[\colon y\mapsto\sqrt{y}$.
\end{example}

\begin{theorem}
\label{t:compinv}
Let $X,Y,Z$ be nonempty sets,
let $f\colon X\to Y$ and 
$g\colon Y\to Z$ both be bijective.
Then $g\circ f\colon X\to Z$ is bijective with inverse
\begin{equation}
(g\circ f)^{-1} = f^{-1}\circ g^{-1}.
\end{equation}
\end{theorem}
\begin{proof}
By Proposition~\ref{p:quinlan}\ref{p:quinlan:iii},
$g\circ f$ is bijective. 
Set $h := (g\circ f)^{-1}$,
let $z\in Z$, and let $x\in X$.
Then $x= h(z)$
$\Leftrightarrow$
$z= (g\circ f)(x)$
$\Leftrightarrow$
$z= g\big(f(x)\big)$
$\Leftrightarrow$
$f(x)= g^{-1}(z)$
$\Leftrightarrow$
$x = f^{-1}\big(g^{-1}(z)\big)$
$\Leftrightarrow$
$x= (f^{-1}\circ g^{-1})(z)$.
Therefore, $(g\circ f)^{-1} = h = f^{-1}\circ g^{-1}$.
\end{proof}

\section{Cardinality}

\begin{definition}\label{d:cardinality}
Let $X$ and $Y$ be sets. 
We say that $X$ and $Y$ 
have the \textbf{same cardinality} 
(also known as {equipotent} or {equinumerous})
and write $X\approx Y$ if 
there exists $f\colon X\to Y$ such that $f$ is bijective.
By convention, one has $\emp\approx \emp$.
\index{Cardinality}
\index{equipotent}
\index{equinumerous}
\end{definition}

\begin{example}
\
\begin{enumerate}
\item $\{\text{red},\text{green},\text{blue}\} \approx \{1,2,3\}$.
\item Denote by $2\ZZ$ the set of even integers.
Then $\ZZ\approx 2\ZZ$ via $m\mapsto 2m$.
\item Denote by $2\NN$ the set of even nonnegative integers.
Then $\NN\approx 2\NN$ via $m\mapsto 2m$.
\end{enumerate}
\end{example}

We first observe that having the same cardinality 
acts just like an equivalence relation.

\begin{theorem}
\label{t:cardeqrel}
Let $X$, $Y$, and $Z$ be sets.
Then the following hold.
\begin{enumerate}
\item 
\label{t:cardeqrel:i}
$X\approx X$.
\item 
\label{t:cardeqrel:ii}
If $X\approx Y$, then $Y\approx X$.
\item 
\label{t:cardeqrel:iii}
If $X\approx Y$ and $Y\approx Z$, then $X\approx Z$.
\end{enumerate}
\end{theorem}
\begin{proof}
\ref{t:cardeqrel:i}:
$\Id_X\colon X\to X\colon x\mapsto x$ 
is a bijection from $X$ to $X$, so
$X\approx X$.

\ref{t:cardeqrel:ii}:
Suppose that $X\approx Y$, i.e.,
there is a bijection $f\colon X\to Y$.
By Theorem~\ref{t:invinv},
$f^{-1}$ is a bijection from $Y$ to $X$ so that $Y\approx X$.

\ref{t:cardeqrel:iii}:
Suppose that $X\approx Y$ via $f\colon X\to Y$,
and that $Y\approx Z$ via $g\colon Y\to Z$.
By Theorem~\ref{t:compinv} or
Proposition~\ref{p:quinlan}\ref{p:quinlan:iii}, $g\circ f\colon X\to Z$ is a bijection;
hence, $X\approx Z$. 
\end{proof}

\begin{definition}[cardinal number]\label{d:cardinal}
Let $X$ be a set.\index{Cardinal number}
We say that the \textbf{cardinal number} of $X$ is \textbf{finite}
and equal to $n$, written $\# X = n$, 
if $\{1,2,\ldots,n\}\approx X$; otherwise, the cardinal number is
\textbf{transfinite}.\index{transfinite}
The cardinal number of $\NN$ is denoted by 
\begin{equation}
\aleph_0
\end{equation}
and pronounced
``aleph\footnote{$\aleph$ is the first letter of the Hebrew
alphabet!} naught''; so, $\# \NN = \aleph_0$. 
If $Y$ is another set and $X\approx Y$, we write $\# X=\# Y$.
\end{definition}

Observe that $\# \emp=0$,
that the cardinal number of 
$\{\text{red},\text{green},\text{blue}\}$ is $3$,
and that the cardinal number of $2\NN$ is $\aleph_0$.
In view of Theorem~\ref{t:cardeqrel}, we may think about cardinal
numbers as equivalence classes.

\section{Countability}

\begin{definition}
Let $X$ be a set.
Then $X$ is \textbf{countably infinite} 
(also known as denumerable) 
if $\NN\approx X$.
If $X$ is empty, finite, or countably infinite, then $X$ is called
\textbf{countable}. 
If $X$ is not countable, it is called \textbf{uncountable}.
\index{denumerable}\index{countably infinite}
\index{countable}\index{uncountable}
\end{definition}

\begin{example}
$\ZZ$ is countably infinite.
Indeed, we list a bijection from $\NN$ to $\ZZ$, namely:
$0\mapsto 0$, $1\mapsto -1$, $2\mapsto 1$,
$3\mapsto -2$, $4\mapsto 2$, \ldots, i.e.,
\begin{equation}
\NN\to \ZZ\colon
n\mapsto
\begin{cases}
\displaystyle \frac{n}{2}, &\text{if $n$ is even;}\\[+5 mm]
\displaystyle \frac{-(n+1)}{2}, &\text{if $n$ is odd.}
\end{cases}
\end{equation}
\end{example}

The next result states that $\aleph_0$ is the smallest
transfinite cardinal number.

\begin{proposition}
\label{p:countinfinite}
Let $Y$ be countably infinite,
and let $X$ be an infinite subset of $Y$.
Then $X$ is countably infinite.
\end{proposition}
\begin{proof}
Let $f\colon\NN\to Y$ be bijective.
Consider the sequence $f(n)_\nnn$, which lies in $Y$.
Since $f$ is onto and since $X\subseteq Y$, we have
$X\subseteq f(\NN)$. 
Hence $X = \{f(n_k)\}_{k\in\NN}$, where $(n_k)_{k\in\NN}$ is a subsequence
of $(n)_\nnn$. Thus, 
the map $\NN\to X\colon k\mapsto f(n_k)$ is surjective and also
injective (since $f$ is injective). 
Therefore, $\NN\approx X$, i.e., $X$ is denumerable. 
\end{proof}

It is sometimes cumbersome to construct the bijection when verifying
that a set is countably infinite. The following result gives an
easier-to-use criterion.

\begin{proposition}
\label{p:nicecount}
Let $X$ be a nonempty set.
Then the following are equivalent.
\begin{enumerate}
\item
\label{p:nicecounti}
$X$ is countable.
\item
\label{p:nicecountii}
There exists an injection $f\colon X\to\NN$.
\item
\label{p:nicecountiii}
There exists a surjection $g\colon \NN\to X$.
\end{enumerate}
\end{proposition}
\begin{proof}
``\ref{p:nicecounti}$\Rightarrow$\ref{p:nicecountii}'':
There exists a bijection $h\colon I\to X$, where $I\subseteq \NN$.
Then $f := h^{-1}$ is a bijection from $X$
onto $I$ and thus an injection from $X$ to $\NN$. 

``\ref{p:nicecountii}$\Rightarrow$\ref{p:nicecountiii}'':
Since $f$ is an injection from $X$ to $\NN$, it is a bijection
from $X$ to $f(X)$; thus,
$f^{-1}$ is a bijection from $f(X)$ to $X$. 
Now let $x\in X$ and define
\begin{equation}
g\colon \NN\to X\colon
n\mapsto
\begin{cases}
f^{-1}(n), &\text{if $n\in f(X)$;}\\
x, &\text{if $n\notin f(X)$.}
\end{cases}
\end{equation}
Then $g\big(f(X)\big) = f^{-1}\big(f(X)\big) = X$
and $g\big(\NN\smallsetminus f(X)\big) = \{x\}$.
Thus, $g$ is a surjection from $\NN$ to $X$.

``\ref{p:nicecountiii}$\Rightarrow$\ref{p:nicecounti}'':
Given $g$, define
\begin{equation}
h\colon X\to\NN\colon x\mapsto \min\menge{\nnn}{g(n)=x}.
\end{equation}
Then $h$ is an injection from $X$ to $\NN$ and hence
it is a bijection from $X$ to $h(X)$. If $X$ is finite, it is clearly
countable. If $X$ is infinite, then $h(X)$ is an infinite
subset of $\NN$ and thus countably infinite by
Proposition~\ref{p:countinfinite}.
\end{proof}

\begin{corollary}
\label{c:131126a}
Let $X$ be a nonempty set. 
Then $X$ is countable if and only if
there exists a sequence $(x_n)_\nnn$ such that
$\{x_n\}_\nnn = X$. 
\end{corollary}

\begin{theorem}
\label{t:countunion}
For each $\nnn$, suppose that $S_n$ is a countable set.
Then 
\begin{equation}
\bigcup_{\nnn} S_n
\end{equation}
is countable as well; in other words, the countable union of countable sets
is countable.
\end{theorem}
\begin{proof}
For each $\nnn$, there exists a sequence $(x_{n,k})_{k\in\NN}$ such
that $S_n = \{x_{n,k}\}_{k\in\NN}$. 
Now write
\begin{equation*}
\begin{matrix}
S_0: & x_{0,0} &\to & x_{0,1} & & x_{0,2} & \to & x_{0,3} & \cdots \\
 & &\swarrow  &  &\nearrow &  & \swarrow &  & \nearrow\\
S_1: & x_{1,0} &    & x_{1,1} & & x_{1,2} &     & x_{1,3} & \cdots\\
 & \downarrow & \nearrow   &  & \swarrow &  & \nearrow  &  & \\
S_2: & x_{2,0} &    & x_{2,1} & & x_{2,2} &     & x_{2,3} & \\
 & &\swarrow  &  &\nearrow &  &  &  &  \\
S_3: & x_{3,0} &    & x_{3,1} & & x_{3,2} &     & x_{3,3} & \\
 & \downarrow & \nearrow   &  &  &  &     &  & \\
S_4: & x_{4,0} &    & x_{4,1} & & x_{4,2} &     & x_{4,3} & \cdots\\
\vdots &  &   &  &  &  &     &  & \\
\end{matrix}
\end{equation*}
The arrows generate a new sequence $$(y_n)_\nnn 
= (x_{0,0},x_{0,1},x_{1,0},x_{2,0},x_{1,1},x_{0,2},x_{0,3},\ldots)$$
such that $\{y_n\}_\nnn = \bigcup_{\nnn} S_n$. 
Now apply Corollary~\ref{c:131126a}. 
\end{proof}

\section{Cardinality of Subsets of Real Numbers}

\begin{example}
$\QQ$ is countable.
\end{example}
\begin{proof}
Indeed, for every integer $n\geq 1$,
$A_n := \menge{k/n}{k\in\NN}$ and 
$B_n := \menge{-k/n}{k\in\NN}$ are countable,
hence so is $A_n\cup B_n = \menge{k/n}{k\in\ZZ} = \frac{1}{n}\ZZ$.
It follows that
\begin{equation}
\QQ = \bigcup_{n\in\{1,2,3,\ldots\}} \tfrac{1}{n}\ZZ
\end{equation}
is countable by Theorem~\ref{t:countunion}.
\end{proof}

\begin{example}
The open interval $\left]0,1\right[$ is uncountable.
\end{example}
\begin{proof}
Suppose to the contrary that  $\left]0,1\right[$ is countable, say
by a sequence of real numbers $(x_n)_{n\geq 1}$ such that
$\{x_n\}_{n\geq 1} =  \left]0,1\right[$. 
We use the decimal expansion to write
\begin{subequations}
\begin{align}
x_1 &= 0.a_{1,1}a_{1,2}a_{1,3}\cdots\\
x_2 &= 0.a_{2,1}a_{2,2}a_{2,3}\cdots\\
x_3 &= 0.a_{3,1}a_{3,2}a_{3,3}\cdots\\
\vdots &
\end{align}
\end{subequations}
Now define a sequence of digits by
\begin{equation}
b_k := \begin{cases}
4, &\text{if $a_{k,k}\neq 4$;}\\
6, &\text{if $a_{k,k}=4$,}
\end{cases}
\end{equation}
and then set
\begin{equation}
y := 0.b_1b_2b_3\cdots.
\end{equation}
Note that for every $k\geq 1$, we have
$b_k\neq a_{k,k}$; thus, $y$ is not in the set
$\{x_n\}_{n\geq 1}$ although clearly $y\in\left]0,1\right[$ 
--- contradiction!
\end{proof}

\begin{corollary}
Any nonempty open interval $\left]\alpha,\beta\right[$ is uncountable,
as is $\RR$.
\end{corollary}
\begin{proof}
Indeed, $\left]\alpha,\beta\right[\approx \left]0,1\right[$ via
$x\mapsto (x-\alpha)/(\beta-\alpha)$.
And if $\RR$ were countable, then so would be any infinite subset (see
Proposition~\ref{p:countinfinite}) and in particular
nonempty open intervals, which is absurd.
\end{proof}

\begin{corollary}
The set of irrational numbers $\RR\smallsetminus\QQ$ is uncountable.
\end{corollary}
\begin{proof}
Recall that $\QQ$ is countable. 
Now if $\RR\smallsetminus\QQ$ were countable, so would
be $\RR = \QQ \cup (\RR\smallsetminus\QQ)$, which is absurd.
\end{proof}

\section{Cantor's Theorem}

\begin{theorem}[Cantor]
\label{t:cantor}
Let $X$ be a nonempty set and denote its power set by $\mathcal{P}(X)$.
Then $|X|<|\mathcal{P}(X)|$ in the sense that there exists an
injection from $X$ to $\mathcal{P}(X)$ but there is no bijection
from $X$ to $\mathcal{P}(X)$.\index{Cantor's Theorem}
\end{theorem}
\begin{proof}
It is clear that there is an injection from $X$ to $\mathcal{P}(X)$,
namely
\begin{equation}
X\to \mathcal{P}(X)\colon x\mapsto \{x\}.
\end{equation}
To show that there is no bijection from $X$ to $\mathcal{P}(X)$, 
we assume to the contrary that 
\begin{equation}
f\colon X\to\mathcal{P}(X) \text{~ is bijective.}
\end{equation}
Because $f(x)\subseteq X$, it may or may not be that $x\in f(x)$ for $x\in X$.
Now set 
\begin{equation}
Y := \menge{x\in X}{x\notin f(x)}.
\end{equation}
Since $f$ is surjective, there exists $y\in X$ such that
\begin{equation}
f(y) = Y.
\end{equation}
Now if $y\in Y$, we have, by definition of $Y$, $y\notin f(y)=Y$,
which is absurd.
But, if $y\notin Y$, then $y\notin f(y)$ and thus $y$ should belong to
$Y$, which is equally absurd.
\end{proof}

\section{Zorn's Lemma}

\begin{definition}[partially and totally ordered sets]
Let $A$ be a nonempty set and let $\preccurlyeq$ be a 
relation on $A$ (review Section~\ref{sec:psar} if needed). 
Consider the following conditions:
\begin{dingautolist}{192}
\item 
\label{order:eins}
$(\forall a \in A)\quad a\preccurlyeq a$.
\item 
\label{order:zwei}
$(\forall a\in A)(\forall b \in A)(\forall c \in A)\quad
\left[\,a\preccurlyeq b\;\;\text{and}\;\;b\preccurlyeq c\,\right]
\;\;\Rightarrow\;\;a \preccurlyeq c$.
\item 
\label{order:vier}
$(\forall a\in A)(\forall b \in A)\quad 
\left[\,a\preccurlyeq b\;\;\text{and}\;\;b\preccurlyeq a\,\right]
\;\;\Rightarrow\;\;a=b$.
\item 
\label{order:funf}
$(\forall a\in A)(\forall b \in A)\quad 
a\preccurlyeq b\;\;\text{or}\;\;b\preccurlyeq a$.
\end{dingautolist}

We say that  $(A,\preccurlyeq)$ is:
\begin{enumerate}
\item 
a \textbf{partially ordered set}\index{partially ordered set} if
 \ref{order:eins}, \ref{order:zwei}, and \ref{order:vier} hold;
\item 
\textbf{totally ordered set}\index{totally ordered set} if
it is partially ordered and 
\ref{order:funf} holds, i.e., \ref{order:eins}--\ref{order:funf}
hold. 
\end{enumerate}
\end{definition}

For instance, nonempty subsets of $\RR$ are totally ordered via
$\leq$.
Also, for any set $X$, the power set $\mathcal{P}(X)$ is
partially ordered via $\subseteq$. 

\begin{definition}
Let $(A,\preccurlyeq)$ be a partially ordered set, let
$B\subseteq A$, and let $a\in A$. 
Then $B$ is a \textbf{chain}\index{chain} if $(B,\preccurlyeq)$ is
totally ordered. 
The element $a\in A$ is an \textbf{upper bound}\index{upper bound} of $B$ if
$(\forall b\in B)$ $b \preccurlyeq a$. 
Finally, $a$ is a \textbf{maximal element}\index{maximal element} of $A$
if $(\forall c \in A)$ $a\preccurlyeq c$ $\Rightarrow$ $c=a$. 
\end{definition}

The following fact can be viewed as a fundamental axiom in set
theory. 

\begin{fact}[Zorn's Lemma] 
\label{f:Zorn}
\index{Zorn's lemma}
Let $A$ be a partially ordered set such that every chain in $A$ 
has an upper bound. Then $A$ contains a maximal element.
\end{fact}

\begin{theorem}[Axiom of Choice]
Let $I$ be a nonempty set,
and let, for every $i\in I$, $X_i$ be a nonempty set.
Then the Cartesian product $\prod_{i\in I} X_i$ is nonempty.
\end{theorem}
\begin{proof}
Set $X := \bigcup_{i\in I} X_i$, 
define a subset $A$ of the power set of $I\times X$ by
requiring that if $f\in A$ and $\{(i,x),(i,y)\}\subseteq f$,
then $x=y\in X_i$. We identify such $f\in A$ as the graph a function 
also called $f$ such that $f(i)\in X_i$ for every $i$ in the domain of $f$. 
Now order $A$ partially with ``$\subseteq$''.
Since $\varnothing\in A$, we have $A\neq\varnothing$. 
Let $B$ be a chain in $A$. 
Then $\bigcup_{f\in B} f$ is an upper bound of $B$ in $A$. 
By Zorn's Lemma (Fact~\ref{f:Zorn}), 
$A$ contains a maximal element, say $g$. 
We claim that the domain of $g$, call it $J$, is $I$. 
Suppose not. Then there exists $i \in I\smallsetminus J$.
Consider $h := g \cup \{(i,x_i)\}$, where $x_i\in X_i$ (recall
that $X_i$ is nonempty). 
Then $h\in A$ and $g \subsetneqq h$. This contradicts the
maximality of $g$.
It follows that the domain of $g$ is $I$, i.e., $g$ can be viewed
as a function such that $(g(i))_{i\in I}\in \prod_{i\in I} X_i$. 
\end{proof}

The next result requires some knowledge of basic linear algebra.

\begin{theorem}[Existence of a basis]
\label{t:Hamel}
Let $V$ be a vector space.
Then $V$ admits a basis. 
\end{theorem}
\begin{proof}
If $V=\{0\}$, then $\varnothing$ is a basis.
So assume that $V\neq\{0\}$. 
Now create a subset $A$ of the power set of $V$
by requiring $S\in A$ if $S$ is a linearly independent subset of
$V$. Taking $v\in V\smallsetminus \{0\}$, we see that
$\{v\}\in A$ so that $A\neq\varnothing$.
Partially order $A$ by inclusion. 
Let $B = (S_i)_{i\in I}$ be a chain in $A$. 
Set $S := \bigcup_{i\in I} S_i$. 
We claim that $S$ is linearly independent, i.e., $S\in A$.
Suppose not.
Then there exist vectors $v_1,\ldots,v_n$ in $S$
and scalars $\alpha_1,\ldots,\alpha_n$, not all zero, such that
\begin{equation}
\label{e:260305a}
\alpha_1 v_1+\alpha_2 v_2 + \cdots + \alpha_nv_n = 0.
\end{equation}
Get $i(1),\ldots,i(n)$ such that each $v_j \in S_{i(j)}$. 
Since $B$ is a chain, there exists $m\in\{1,\ldots,n\}$ such that
$\bigcup_{1\leq j\leq n} S_{i(j)} = S_{i(m)}\in A$. 
But then \eqref{e:260305a} is absurd since $S_{i(m)}$ is linearly
independent. 

We have thus shown that $S$ is indeed linearly independent. 
Hence $S\in A$ and $S$ is an upper bound of $B$.
By Zorn's Lemma, there exists a maximal element $M\in A$.
Hence $M$ is linearly independent.

We claim that $M$ is a spanning set.
Suppose not.
Then there exists $v\in V$ such that
$v$ is not in the span of $M$.
But then $M\cup\{v\}$ is not only in $A$ but 
also a proper superset of $M$.
This contradicts the maximality of $M$.
Hence $M$ is indeed a spanning set.

Altogether, $M$ is both linearly independent and a spanning set,
i.e., $M$ is a basis of $V$.
\end{proof}

\begin{remark}
Some comments are in order.
\begin{enumerate}
\item 
Let $V$ denote the vector space of all sequences in $\RR$, with
the usual addition and scalar multiplication.
Then $V$ admits a basis by Theorem~\ref{t:Hamel}.
However, nobody knows an actual concrete basis of $V$.
This type of existence result is typical when
working with Zorn's Lemma.
\item 
Zorn's Lemma can also be used to prove the existence of various
interesting quantities in mathematics including:
nonzero linear functionals in Functional Analysis,
maximal ideals in Algebra, and 
nonmeasurable sets in Measure Theory.
\item Zorn's Lemma and the Axiom of Choice are, in fact,
equivalent. 
\end{enumerate}
\end{remark}

\section*{Exercises}\markright{Exercises}
\addcontentsline{toc}{section}{Exercises}
\setcounter{theorem}{0}

\begin{exercise}
Prove items 
\ref{p:101108:1ii},
\ref{p:101108:1iv},
and \ref{p:101108:1vi}
of Proposition~\ref{p:101108:1}.
\end{exercise}
\begin{solution}
\ref{p:101108:1ii}:
Take $y\in f\big(f^{-1}(B)\big)$.
Then $y= f(x)$, where $x\in f^{-1}(B)$.
Thus, $f(x)\in B$ and therefore $y=f(x)\in B$.
It follows that $f\big(f^{-1}(B)\big)\subseteq B$.

\ref{p:101108:1iv}:
Let $y\in f(A_1\cup A_2)$.
Then $y= f(x)$, where $x\in A_1 \cup A_2$.
If $x\in A_1$, then $y = f(x)\in f(A_1)$.
And if $x\in A_2$, then $y=f(x)\in f(A_2)$.
Thus, $y = f(x) \in f(A_1)\cup f(A_2)$ and therefore
$f(A_1\cup A_2)\subseteq f(A_1)\cup f(A_2)$.

Conversely, since $A_1$ and $A_2$ are both subsets of $A_1\cup A_2$, 
it is clear that $f(A_1)\subseteq f(A_1\cup A_2)$
and that $f(A_2)\subseteq f(A_1\cup A_2)$.
Therefore, $f(A_1)\cup f(A_2)\subseteq f(A_1\cup A_2)$.

\ref{p:101108:1vi}:
Let $x\in X$.
Then 
$x\in f^{-1}(B_1\cup B_2)$
$\Leftrightarrow$
$f(x)\in B_1\cup B_2$
$\Leftrightarrow$
[$f(x)\in B_1$ or $f(x)\in B_2$]
$\Leftrightarrow$
[$x\in f^{-1}(B_1)$ or $x\in f^{-1}(B_2)$]
$\Leftrightarrow$
$x\in f^{-1}(B_1)\cup f^{-1}(B_2)$.
\end{solution}

\begin{exercise} 
Verify Example~\ref{ex:quinny:1}.
\end{exercise}
\begin{solution}
Not injective: 
$f(-1) = (-1)^2 = 1 = 1^2 = f(1)$ but $-1\neq 1$. 

Not surjective:
There is not $x\in\RR$ such that $f(x)=x^2=-1$.
\end{solution}

\begin{exercise} 
Verify Example~\ref{ex:quinny:2}.
\end{exercise}
\begin{solution}
Not surjective:
There is not $x\in\RR$ such that $f(x)=x^2=-1$.

Injective:
Assume without loss of generality that $0\leq x < y$.
Then $0 \leq x^2 < y^2$ and so $f(x)<f(y)$.
So $f$ is strictly increasing, hence injective.
\end{solution}

\begin{exercise} 
Verify Example~\ref{ex:quinny:3}.
\end{exercise}
\begin{solution}
Not injective: 
$f(-1) = (-1)^2 = 1 = 1^2 = f(1)$ but $-1\neq 1$. 

Surjective:
Let $y\geq 0$. 
We have proved that there exists a unique square root
$\sqrt{y}\geq 0$ such that $y=\sqrt{y}^2 = f(\sqrt{y})$. 
Hence $f$ is surjective. 
\end{solution}

\begin{exercise} 
Verify Example~\ref{ex:quinny:4}.
\end{exercise}
\begin{solution}
Injective: 
Assume without loss of generality that $0\leq x < y$.
Then $0 \leq x^2 < y^2$ and so $f(x)<f(y)$.
So $f$ is strictly increasing, hence injective.

Surjective:
Let $y\geq 0$. 
We have proved that there exists a unique square root
$\sqrt{y}\geq 0$ such that $y=\sqrt{y}^2 = f(\sqrt{y})$. 
Hence $f$ is surjective. 
\end{solution}

\begin{exercise}
Prove Proposition~\ref{p:101108:2}\ref{p:101108:2ii}.
\end{exercise}
\begin{solution}
It follows from Proposition~\ref{p:101108:1}\ref{p:101108:1ii} that
$f(f^{-1}(B))\subseteq B$. 

Conversely, assume that $f$ is surjective, and let $y\in B$. 
Then there exists $x\in X$ such that $f(x)=y$. 
Hence $x\in f^{-1}(y)\subseteq f^{-1}(B)$,
and thus $y = f(x) \in f(f^{-1}(B))$.
This shows that $B\subseteq  f(f^{-1}(B))$. 
\end{solution}

\begin{exercise}
Prove Proposition~\ref{p:101108:2}\ref{p:101108:2iii}.
\end{exercise}
\begin{solution}
By Proposition~\ref{p:101108:1}\ref{p:101108:1iii},
$f(A_1\cap A_2)\subseteq f(A_1)\cap f(A_2)$.
Now take $y\in f(A_1)\cap f(A_2)$.
Then there exists $a_1\in A_1$ and $a_2\in A_2$ such
that $f(a_1)=y=f(a_2)$. Since $f$ is injective, it follows that
$a_1=a_2$. Thus $y\in f(A_1\cap A_2)$. 
\end{solution}

\begin{exercise}
Show that the conclusion in each item of Proposition~\ref{p:101108:2}
fails if the hypothesis on $f$ is omitted.
\end{exercise}
\begin{solution}
\ref{p:101108:2i}:
$f\colon \RR\to\left[0,+\infty\right[\colon x\mapsto x^2$
and $A=\{2\}$. Then $f(A) = \{4\}$ and $f^{-1}(f(A)) = \{\pm 2\} \neq
\{2\}$.

\ref{p:101108:2ii}:
$f\colon \RR\to\RR\colon x\mapsto x^2$
and $B=[-4,4]$. Then $f^{-1}(B) = [-2,2]$ and
$f([-2,2]) = [0,4]$. 

\ref{p:101108:2iii}:
$f\colon \RR\to\left[0,+\infty\right[\colon x\mapsto x^2$,
$A_1 = \{-2\}$, and $A_2=\{2\}$.
Then $A_1\cap A_2=\varnothing$ and so $f(A_1\cap A_2)=\varnothing$,
but $f(A_1)\cap f(A_2) = \{4\}$.
\end{solution}

\begin{exercise}
Prove Proposition~\ref{p:quinlan}\ref{p:quinlan:ii}%
\&\ref{p:quinlan:iii}.
\end{exercise}
\begin{solution}
Recall that $g\circ f\colon X\to Z\colon x\mapsto g\big(f(x)\big)$. 

\ref{p:quinlan:ii}:
Suppose $f$ and $g$ are surjective and let $z\in Z$.
On the one hand, since $g$ is surjective, there exists $y\in Y$ such that $g(y)=z$.
On the other hand, since $f$ is surjective, there exists $x\in X$ such that
$f(x) = y$.
Altogether, $(g\circ f)(x) =  g\big(f(x)\big) = g(y) = z$.
Therefore, $g\circ f$ is surjective.

\ref{p:quinlan:iii}:
On the one hand, by \ref{p:quinlan:i}, $g\circ f$ is injective.
On the other hand, by \ref{p:quinlan:ii}, $g\circ f$ is surjective.
Altogether, $g\circ f$ is both injective and surjective, i.e., bijective.
\end{solution}

\begin{exercise}
Show that the set of all polynomials with rational coefficients is
countable. 
\end{exercise}
\begin{solution}
Denote by $P_n$ the set of all polynomials of degree $n$ with rational
coefficients.
It is clear that $P_0 = \QQ$ and that
$P_n = \bigcup_{q\in \QQ} (qx^n + P_{n-1})$.
By induction on $n$, each $P_n$ is countable. 
By Theorem~\ref{t:countunion},
$\bigcup_{\nnn}P_n$ is countable as well.
\end{solution}

\begin{exercise}
Show that for every nonempty subset $X$, we have 
$$\# X < \# \mathcal{P}(X) < \# \mathcal{P}(\mathcal{P}(X)) < \cdots .$$
\end{exercise}
\begin{proof}
This is clear by \emph{Cantor's Theorem} (Theorem~\ref{t:cantor}).
\emph{Comment:} If $X$ is infinite, then Cantor showed us 
that there are infinitely many infinite cardinalities! 
\end{proof}

\enlargethispage{2\baselineskip}

\begin{exercise}
\label{exo:160826c}
Let $f\colon [0,1]\to [0,1]$ be such that 
$(f\circ f)(x) = f(f(x))=1-x$ for every $x\in[0,1]$.
Show that $f$ is bijective. 
\end{exercise}
\begin{solution}
Suppose $f(x)=f(y)$.
Then $1-x=f(f(x))=f(f(y))=1-y$.
Hence $1-x=1-y$ and thus $x=y$.
This shows that $f$ is injective. 

To show that $f$ is surjective, let $y\in[0,1]$
and set $x= f(f(f(y)))$.
Then $f(x) = \big((f\circ f)\circ(f\circ f)\big)(y)
=(f\circ f)(1-y) = 1-(1-y)=y$ as required.

Altogether, we showed that $f$ is injective and surjective,
i.e., bijective.
\end{solution}

\begin{exercise}
\label{exo:160826d}
Show that there is no function $f\colon[0,1]\to[0,1]$ that 
satisfies both 
(i) [either $f$ is strictly increasing or strictly decreasing]
and 
(ii) $(f\circ f)(x) = f(f(x))=1-x$ for every $x\in[0,1]$.
\end{exercise}
\begin{solution}
We argue by contradiction.
Suppose $f$ satisfies both (i) and (ii).
By (ii) and Exercise~\ref{exo:160826c},
$f$ is bijective. 
Since (i) holds, we must have
either $f(0)=0$ (if $f$ is strictly increasing)
or $f(0)=1$ (if $f$ is strictly decreasing).
On the other hand, 
by (ii) and Exercise~\ref{exo:160826b},
$f(0)\notin\{0,1\}$.
Altogether, this is absurd.
\end{solution}

\begin{exercise}
\label{exo:171114}
Show that there is no continuous function $f\colon[0,1]\to[0,1]$ 
such that 
$(f\circ f)(x) = f(f(x))=1-x$ for every $x\in[0,1]$.
\end{exercise}
\begin{solution}
By Exercise~\ref{exo:160826c},
$f$ is bijective.
The \emph{Intermediate Value Theorem} from Calculus/Analysis and a proof by
contradiction shows that $f$ is either strictly increasing or
strictly decreasing. 
Now apply Exercise~\ref{exo:160826d}. 
\end{solution}

\begin{exercise}
In contrast to Exercise~\ref{exo:171114},
there does exist a piecewise continuous function 
$f\colon[0,1]\to[0,1]$ 
such that 
$(f\circ f)(x) = f(f(x))=1-x$ for every $x\in[0,1]$.
Indeed, verify that 
\begin{equation*}
f(x) := \begin{cases}
-x + \tfrac{3}{4} , &\text{if $0\leq x < \tfrac{1}{4}$;}\\
x-\tfrac{1}{4}, &\text{if $\tfrac{1}{4}\leq x < \tfrac{1}{2}$;}\\
\tfrac{1}{2}, &\text{if $x=\tfrac{1}{2}$;}\\
x+\tfrac{1}{4}, &\text{if $\tfrac{1}{2}<x \leq \tfrac{3}{4}$;}\\
-x+\tfrac{5}{4}, &\text{if $\tfrac{3}{4}<x \leq 1$}
\end{cases}
\end{equation*}
does the job. 
\end{exercise}
\begin{solution}
Let $x\in [0,1]$.
Note that $f(f(1/2)) = f(1/2)=1/2=1-1/2$.
We consider the remaining cases.

\emph{Case~1:} $0\leq x<1/4$.\\
Then $f(x) = -x+3/4 \in \left]-1/4,0\right] + 3/4 =
\left]1/2,3/4\right]$ and thus
$f(f(x)) = f(x)+1/4 = (-x+3/4)+1/4 = -x+1$. 

\emph{Case~2:} $1/4\leq x<1/2$.\\
Then $f(x) = x-1/4 \in \left[0,1/4\right[$ 
and thus 
$f(f(x)) = -f(x)+3/4 = -(x-1/4)+3/4 = -x+1$. 

\emph{Case~3:} $1/2< x\leq 3/4$.\\
Then $f(x) = x+1/4 \in \left]3/4,1\right]$ 
and thus 
$f(f(x)) = -f(x)+5/4 = -(x+1/4)+5/4 = -x+1$. 

\emph{Case~4:} $3/4< x\leq 1$.\\
Then $f(x) = -x+5/4 \in \left[-1,-3/4\right[ + 5/4 =
\left[1/4,1/2\right[$ 
and thus 
$f(f(x)) = f(x)-1/4 = (-x+5/4)-1/4 = -x+1$. 

\emph{Remark:}
This function shows up in Quantum Theory.
See Section~17.1 in 
Maria~Luisa Dalla Chiara, Roberto Giuntini, and
Richard Greechie: \emph{Reasoning in Quantum Theory}, Springer. 
\end{solution}

\begin{exercise} 
Let $X$ be the set of all real sequences $(x_n)_{\nnn}$.
Define two operators on $X$, namely
\begin{equation*}
R\colon X\to X\colon (x_0,x_1,x_2,\ldots)
\mapsto (0,x_0,x_1,x_2,\ldots)
\end{equation*}
and
\begin{equation*}
L\colon X\to X\colon (x_0,x_1,x_2,\ldots)
\mapsto (x_1,x_2,x_3,\ldots),
\end{equation*}
which are called the \emph{right shift} and the 
\emph{left shift operator},
respectively.
Prove or disprove:
\begin{enumerate}
    \item $L\circ R = \Id$
    \item $R\circ L = \Id$
\end{enumerate}
\emph{Comment.} These operators are even ``linear''.
The behaviour exhibited by them cannot happen 
in finite-dimensional vector spaces. 
\end{exercise}
\begin{solution}
Write $\mathbf{x} := (x_0,x_1,x_2,\ldots)\in X$.

(i): We have 
\begin{align*}
(L\circ R)(\mathbf{x}) 
&=L(R(\mathbf{x}))
=L\big((0,x_0,x_1,x_2,\ldots)\big)
=(x_0,x_1,\ldots)
=\Id(\mathbf{x})
\end{align*}
and therefore $L\circ R = \Id$.

(ii): This time, 
\begin{align*}
(R\circ L)(\mathbf{x}) 
&=R(L(\mathbf{x}))
=R\big((x_1,x_2,x_3,\ldots)\big)
=(0,x_1,x_2,\ldots)\\
&\neq (x_0,x_1,x_2,\ldots) \qquad\text{[unless $x_0=0$!]}\\
&=\Id(\mathbf{x}).
\end{align*}
So $R\circ L\neq \Id$.
\end{solution}

\begin{exercise} 
Let $f\colon X\to Y$ be a function. 
Show the following:
\begin{enumerate}
\item 
$f$ is injective 
$\Leftrightarrow$
there exists a function $g\colon \ran f \to X$ such that 
$g\circ f = \Id_X$ ($g$ is called a \emph{left inverse} of $f$).
\item 
$f$ is surjective 
$\Leftrightarrow$
there exists a function $h\colon Y\to X$ such that 
$f\circ h = \Id_Y$ ($h$ is called a \emph{right inverse} of $f$).
\end{enumerate}
\end{exercise}
\begin{solution}
(i):
``$\Rightarrow$'':
Suppose that $f$ is injective. 
Let $y\in \ran f$. Then there exists $x\in X$ such that $f(x)=y$.
There is no other element with this property:
indeed, if $x'\in X$ is such that $f(x')=y$,
then $x=x'$ by injectivity of $f$.
Thus we let
\begin{equation*}
    g\colon \ran f \to X \colon y \mapsto 
    \text{the unique $x\in X$ such that $f(x)=y$}.
\end{equation*}
Now let $x\in X$. Then $y := f(x)\in \ran f$.
By definition of $g$, we have $g(y)=x$. 
Altogether, $(g\circ f)(x)=g(f(x))=g(y) = x = \Id_X(x)$.
``$\Leftarrow$'': 
Suppose $g$ is as in the statement. 
Also suppose that $f(x_1)=f(x_2)$.
Then $x_1=\Id_X(x_1)=(g\circ f)(x_1) = (g\circ f)(x_2) 
=\Id_X(x_2)=x_2$ and therefore $f$ is injective.

(ii): 
``$\Rightarrow$'':
Suppose that $f$ is surjective. 
Let $y\in Y$.
Then there exists $x\in X$ such that $f(x)=y$.
Define 
\begin{equation*}
    h\colon Y \to X \colon y \mapsto 
    \text{some $x\in X$ such that $f(x)=y$}.
\end{equation*}
Then clearly $f(h(y))=y$ and so $f\circ h = \Id_Y$. 
``$\Leftarrow$'':
Suppose $h$ is as in the statement.
Let $y\in Y$.
Set $x:= h(y)$.
Then $f(x)=f(h(y)) = (f\circ h)(y)=\Id_Y(y) = y$ and thus 
$f$ is surjective. 
\end{solution}

\begin{exercise} 
Suppose $f\colon A\to B$ and $g\colon C\to D$ are injective.
Prove or disprove:
\begin{equation*}
h\colon A\times C\to B\times D \colon 
(a,c)\mapsto \big(f(a),g(c)\big)
\text{~~is injective.}
\end{equation*}
\end{exercise}
\begin{solution}
The statement is TRUE!
Let $(a_1,c_1)$ and $(a_2,c_2)$ be in $A\times C$ 
such that $h(a_1,c_1)=h(a_2,c_2)$.
Then 
$(f(a_1),g(c_1))=(f(a_2),g(c_2))$.
Comparing the first and the second components gives 
$f(a_1)=f(a_2)$ and $g(c_1)=g(c_2)$, respectively.
Because $f$ and $g$ are injective, it follows that 
$a_1=a_2$ and $c_1=c_2$.
Therefore
$(a_1,c_1)=(a_2,c_2)$ and we are done.
\end{solution}

\begin{exercise}[YOU be the marker!]  
Consider the following statement 
\begin{equation*}
\text{``Every function $f\colon X\to Y$ is injective''}
\end{equation*}
and the following ``proof'':
\begin{quotation}
Let $x_1$ and $x_2$ be in $X$ such that $f(x_1)\neq f(x_2)$.\\
We claim that $x_1\neq x_2$.\\
Indeed, if $x_1=x_2$, then $f(x_1)=f(x_2)$ which is absurd.\\
Hence our claim is true.\\
Therefore, $f$ is injective.
\end{quotation}
Why is this proof wrong?
\end{exercise}
\begin{solution}
All is fine except for the very last line.
All the other lines show that $f(x_1)\neq f(x_2)$
$\Rightarrow$ $x_1\neq x_2$.
But to check injectivity, we need to check 
that $f(x_1)=f(x_2)$ $\Rightarrow$ $x_1=x_2$ 
(or that if $x_1\neq x_2$, then $f(x_1)\neq f(x_2)$).
\end{solution}

\begin{exercise}[TRUE or FALSE?] 
Mark each of the following statements as either true or false. 
Briefly justify your answer.
\begin{enumerate}
\item ``If $f\colon A\to B$ is injective, then $f$ is surjective.''
\item ``If $f\colon A\to B$ is surjective, then $f$ is injective.''
\item ``The set of all \emph{positive} rational numbers is uncountable.''
\item ``If $f\colon A\to B$ is surjective and $g\colon B\to C$ is injective, 
then $g\circ f$ is bijective.''
\end{enumerate}
\end{exercise}
\begin{solution}
(i): FALSE: $f\colon \RP\to\RR\colon x\mapsto x^2$.

(ii): FALSE: $f\colon \RR\to\RP\colon x\mapsto x^2$.

(iii): FALSE: $\QQ$ is countable, and so are all its subsets!

(iv): FALSE: $A=\{1,-1\}$, $C=B=\{1\}$, 
$f(x)=x^2$ and $g(y)=y$.
\end{solution} 
\chapter{Exponential Function and Cauchy Product}

\section{Definition and Absolute Convergence}

\begin{theorem}[exponential series and function]
\label{t:expfunc}
Let $x\in\RR$. 
Then the \textbf{exponential series}
\begin{equation}
\label{e:exp(x)}
\sum_{n=0}^\infty \frac{x^n}{n!} = 1 + x + \frac{x^2}{2} +
\frac{x^3}{6} + \frac{x^4}{24} + \cdots
\end{equation}
is absolutely convergent. The corresponding function
is the \textbf{exponential function}, written $\exp(x)$.
Note that 
\begin{equation}
\label{e:exp(0)=1}
\exp(0) = 1.
\end{equation}
\index{Exponential series}\index{Exponential function}
\end{theorem}
\begin{proof}
Clearly, \eqref{e:exp(0)=1} follows from \eqref{e:exp(x)} when $x=0$,
and the corresponding series is certainly absolutely convergent since
all terms --- except for the first one --- are zero.
So we assume that $x\neq 0$.
Then for every integer $n\in\NN$, we have 
\begin{equation}
\left| \frac{x^{n+1}\big/(n+1)!}{x^n\big/n!}\right|
= \frac{|x|^{n+1}}{(n+1)!}\frac{n!}{|x|^n} 
= \frac{|x|}{n+1}
\to 0. 
\end{equation}
The result thus follows from 
the ratio test (Corollary~\ref{c:utirat}). 
\end{proof}

\emph{Euler's number} $e$ is defined as $\exp(1)$, i.e.,
\begin{equation} \label{e:euler}
e := \exp(1) = \sum_{n=0}^\infty \frac{1}{n!} = 1 + 1 + \frac{1}{2} + \frac{1}{6} +
\frac{1}{24} + \cdots\, .
\end{equation}

The next result gives good information on how many terms
we need to sum before we can be assured of a precise answer.
As so often, the geometric series plays a key role.
We will use it afterwards to estimate $e$.

\begin{theorem}
\label{t:expremainder}
Let $x\in\RR$ and let $N\in\NN$ be such that $N\geq 2|x|-2$.
Then
\begin{equation}
\label{e:expremainder}
\left| \exp(x) - \sum_{n=0}^N \frac{x^n}{n!}\right| \leq
\frac{2|x|^{N+1}}{(N+1)!}.
\end{equation}
\end{theorem}
\begin{proof}
Indeed, since $2|x|\leq N+2 \Leftrightarrow |x|/(N+2)\leq 1/2$
and using the geometric series,
we estimate 
\begin{multline}
\left| \exp(x) - \sum_{n=0}^N \frac{x^n}{n!}\right|
=\left|\sum_{n=N+1}^\infty \frac{x^n}{n!}\right|
\leq \sum_{n=N+1}^\infty \left| \frac{x^n}{n!} \right| \\
= \frac{|x|^{N+1}}{(N+1)!} \bigg(1 + \frac{|x|}{N+2} +
\frac{|x|^2}{(N+2)(N+3)}+\cdots
+ \frac{|x|^k}{(N+2)\cdots(N+k+1)}+\cdots\bigg)\\
\leq \frac{|x|^{N+1}}{(N+1)!} \bigg(1 + \frac{|x|}{N+2} +
\Big(\frac{|x|}{N+2}\Big)^2+\cdots
+ \Big(\frac{|x|}{N+2}\Big)^k+\cdots\bigg)\\
\leq \frac{|x|^{N+1}}{(N+1)!} \bigg(1 + \frac{1}{2} +
\Big(\frac{1}{2}\Big)^2+\cdots
+ \Big(\frac{1}{2}\Big)^k+\cdots\bigg)
= \frac{2|x|^{N+1}}{(N+1)!},
\end{multline}
as claimed.
\end{proof}

\begin{remark}
\label{r:euler}
Since $2/(11+1)!<10^{-8}$,
we can use $x=1$ and $N=11$ in Theorem~\ref{t:expremainder}
to approximate Euler's number $e$ to 7 digits as 
$2.7182818$.
\end{remark}

We now show that Euler's number $e$ is irrational.

\begin{theorem}[$e\notin\QQ$]
Euler's number is irrational, i.e., $e\in\RR\smallsetminus\QQ$. 
\end{theorem}
\begin{proof}
Suppose to the contrary that $e$ is rational, say
\begin{equation}
\label{e:polandout1}
e = \frac{m}{n}
\end{equation}
where $m\in\{1,2,\ldots\}$ and $n\in\{1,2,\ldots\}$. 
Now define
\begin{equation}
\label{e:polandout2}
x := n! \bigg(e - \sum_{k=0}^n \frac{1}{k!} \bigg)
= n!\sum_{k=n+1}^\infty \frac{1}{k!} = \sum_{k=n+1}^\infty \frac{n!}{k!} > 0. 
\end{equation}
Note that if $k\geq n+1$, then 
$k!/n! = (n+1)(n+2)\cdots k \geq (n+1)^{k-n}$ and so 
\begin{equation}
\frac{n!}{k!} \leq \frac{1}{(n+1)^{k-n}}
\end{equation}
and this inequality is \emph{strict} whenever $k\geq n+2$. 
Using the formula for the Geometric Series, it follows that 
\begin{align}
x &= \sum_{k=n+1}^\infty \frac{n!}{k!}
<
\sum_{k=n+1}^\infty \frac{1}{(n+1)^{k-n}}
= \frac{1}{n+1}\sum_{l=0}^\infty \frac{1}{(n+1)^l}
= \frac{1}{n+1} \frac{1}{1-\frac{1}{n+1}} = \frac{1}{n}\leq 1.
\end{align}
To sum up, we have proved so far that 
\begin{equation}
\label{e:polandout3}
0<x < 1. 
\end{equation}
On the other hand, substituting \cref{e:polandout1} into \cref{e:polandout2}, 
we learn that 
\begin{align}
0&< x 
= n!\bigg(\frac{m}{n} - \sum_{k=0}^n \frac{1}{k!} \bigg)
= m(n-1)! - \sum_{k=0}^n \frac{n!}{k!}
= m(n-1)! - \sum_{k=0}^n (k+1)\cdots (n-1)n
\end{align}
is a positive integer! Hence $x\geq 1$ which is absurd in view of 
\cref{e:polandout3}.
\end{proof}

\section{The Cauchy Product of Two Series}

The identity
\begin{equation}
\label{e:slownet}
\sum_{n=0}^\infty a_nb_n = \bigg(\sum_{n=0}^\infty a_n\bigg)
\bigg(\sum_{n=0}^\infty b_n\bigg).
\end{equation}
is \textbf{unfortunately false}
for two series $\sum_{\nnn}a_n$ and $\sum_{\nnn}b_n$ that are absolutely convergent.
In this section, we explore how the right side of \eqref{e:slownet}
is evaluated. 

\begin{definition}
Let $(a_n)_\nnn$ and $(b_n)_\nnn$ be two sequences of real numbers.
Then the \textbf{Cauchy product} of $\sum_{\nnn}a_n$ and $\sum_{\nnn}
b_n$ is defined by $\sum_\nnn c_n$, where\index{Cauchy product}
\begin{equation}
\label{e:cauchyprod}
c_n = \sum_{k=0}^n a_{n-k}b_k, 
\qquad \text{for every $\nnn$.}
\end{equation}
\end{definition}

\begin{theorem}
\label{t:cauchyprod}
Let $\sum_{\nnn} a_n$
and $\sum_{\nnn} b_n$
be two absolutely convergent series, with 
Cauchy product $\sum_\nnn c_n$.
Then $\sum_\nnn c_n$ is absolutely convergent and
\begin{equation}
\sum_{n=0}^\infty c_n = \bigg(\sum_{n=0}^\infty a_n\bigg)
\bigg(\sum_{n=0}^\infty b_n\bigg).
\end{equation}
\end{theorem}

\begin{proof}
Set $A := \sum_{\nnn} a_n$,
$B := \sum_\nnn b_n$, 
and for every $\nnn$,
$C_n := \sum_{k=0}^n c_k$,
$D_n := \big(\sum_{k=0}^n a_k\big)\big(\sum_{k=0}^n b_k\big)$,
and
$E_n := \big(\sum_{k=0}^n |a_k|\big)\big(\sum_{k=0}^n |b_k|\big)$.
Observe that $(E_n)_\nnn$ is a convergent sequence, by 
the product law (Theorem~\ref{t:prodlaw}). 
In particular, it is a Cauchy sequence.
Now let $\varepsilon>0$.
Then there exists $N\in\NN$ such that 
\begin{equation}
\label{e:101107:a}
|E_n-E_N|<\varepsilon,
\qquad\text{for every $n\geq N$.}
\end{equation}
For every $n\in\NN$, we define the set of indices
\begin{equation}
I_n := \menge{(i,j)\in\NN\times\NN}{i\leq n \text{~and~} j\leq n}
\end{equation}
Then \eqref{e:101107:a} turns into
\begin{equation}
\sum_{(i,j)\in I_n\smallsetminus I_N} |a_ib_j| <\varepsilon,
\qquad\text{for every $n\geq N$.}
\end{equation}
On the other hand, again by the product law,
we have $D_n\to AB$.
Now 
\begin{equation}
D_n = \sum_{(i,j)\in I_n} a_ib_j
\end{equation}
while
\begin{equation}
C_n = \sum_{m=0}^{n}c_m = \sum_{m=0}^{n}\sum_{k=0}^ma_{m-k}b_k
= \sum_{(i,j)\in I_n, i+j\leq n} a_ib_j.
\end{equation}
Hence,
\begin{equation}
D_n- C_n = \sum_{(i,j)\in I_n, i+j>n} a_ib_j.
\end{equation}
Let $n\geq 2N$ and let $(i,j)\in I_n$ such that $i+j>n$.
If $(i,j)\in I_N$ we would have 
$i\leq N$ and $j\leq N$, hence $n<i+j\leq 2N\leq n$, which is absurd.
Hence $(i,j)\in I_n\smallsetminus I_N$.
Altogether, we deduce that for $n\geq 2N$, we have
\begin{align}
|D_n-C_n| &= \left|  \sum_{(i,j)\in I_n, i+j>n} a_ib_j\right|
\leq  \sum_{(i,j)\in I_n, i+j>n} |a_ib_j|
\leq  \sum_{(i,j)\in I_n\smallsetminus I_N} |a_ib_j| < \varepsilon.
\end{align}
Hence, $D_n-C_n\to 0$, and since $D_n\to AB$, we obtain 
$C_n = (C_n-D_n)+D_n\to 0 + AB = AB$, i.e.,
the Cauchy product is a convergent series with limit $AB$.

It remains to show the absolute convergence of the Cauchy product. 
To this end, set $\alpha_n = |a_n|$ and $\beta_n=|b_n|$,
for every $\nnn$. 
Then $\sum_\nnn \alpha_n$ and $\sum_\nnn\beta_n$ are two absolutely
convergent series, and denote their Cauchy product by
$\sum_\nnn\gamma_n$. Applying the above part of the proof to these two
series, we deduce that $\sum_\nnn\gamma_n$ is a convergent series.
Since for every $\nnn$, 
\begin{equation}
|c_n| = \left|\sum_{k=0}^n a_{n-k}b_k\right|
\leq \sum_{k=0}^n |a_{n-k}b_k|
= \sum_{k=0}^{n} \alpha_{n-k}\beta_k = \gamma_n,
\end{equation}
it follows from the comparison test (Theorem~\ref{t:compartest})
that $\sum_\nnn c_n$ is absolutely convergent.
\end{proof}

\section{Functional Equation}

\begin{theorem}[Functional equation for the exponential function]
\label{t:superexp}
Let $x$ and $y$ be in $\RR$.
Then \index{Functional equation (for $\exp$)}
\begin{equation}
\exp(x+y)=\exp(x)\exp(y).
\end{equation}
\end{theorem}
\begin{proof}
We have
\begin{equation}
\exp(x) = \sum_{n=0}^\infty \frac{x^n}{n!}
\quad\text{and}\quad
\exp(y) = \sum_{n=0}^\infty \frac{y^n}{n!},
\end{equation}
where both series are absolutely convergent
by Theorem~\ref{t:expfunc}.
Denote their Cauchy product by $\sum_\nnn c_n$.
Then, using 
the Binomial Theorem (Theorem~\ref{t:binomial}) and 
Lemma~\ref{l:choose}\ref{l:choose:ii}, 
we compute that
\begin{equation}
c_n 
= \sum_{k=0}^n \frac{x^{n-k}}{(n-k)!}\frac{y^k}{k!}
= \frac{1}{n!}\sum_{k=0}^n {n\choose k}x^{n-k}y^k
= \frac{(x+y)^n}{n!}.
\end{equation}
In view of Theorem~\ref{t:cauchyprod}, 
we therefore obtain 
\begin{align}
\exp(x)\exp(y) &=  
\bigg(\sum_{n=0}^\infty \frac{x^n}{n!}\bigg)
\bigg(\sum_{n=0}^\infty \frac{y^n}{n!}\bigg)
= \sum_{n=0}^\infty c_n
= \sum_{n=0}^\infty  \frac{(x+y)^n}{n!}
=\exp(x+y),
\end{align}
as claimed.
\end{proof}

\begin{corollary}
\label{c:exp}
The following hold for every $x\in\RR$ and every $m\in\ZZ$.
\begin{enumerate}
\item
\label{c:exp:i}
$\exp(x)>0$.
\item
\label{c:exp:ii}
$\exp(-x) = 1/\exp(x)$.
\item
\label{c:exp:iii}
$\exp(m) = e^m$.
\end{enumerate}
\end{corollary}
\begin{proof}
Since $\exp(0)=1$ (see \eqref{e:exp(0)=1}), 
it follows from Theorem~\ref{t:superexp} that
$1 = \exp(0) = \exp\big(x+(-x)\big) = \exp(x)\exp(-x)$.
In particular, $\exp(x)\neq 0$ and \ref{c:exp:ii} thus holds.
Furthermore, $\exp(x)$ and $\exp(-x)$ are both either positive or
negative. Now if $x\geq 0$, then 
\begin{equation}
\exp(x) =\sum_{n=0}^\infty \frac{x^n}{n!} = 1 + x + \frac{x^2}{2} +
\frac{x^3}{6} +\cdots \geq 1+0+0+0+\cdots = 1;
\end{equation}
consequently, $\exp(x)>0$ in this case, and by \ref{c:exp:ii},
$\exp(-x)>0$ as well. This verifies \ref{c:exp:i}.
The proof of \ref{c:exp:iii} is left as Exercise~\ref{exo:expm}.
\end{proof}

\begin{remark}[Addition Theorems for Sine and Cosine]
What we proved in this chapter actually holds
true for the \emph{complex} exponential function,
with the real absolute value replaced by
the \emph{complex absolute value}
$|x+\ii y| := \sqrt{x^2 + y^2}$ 
to make sense of the notion of absolute convergence of the exponential series. 
The functional equation --- combined with polar coordinates ---
then provides a charming proof of the addition theorems for
$\sin$ and $\cos$. Indeed, let $\alpha$ and $\beta$ be in $\RR$.
Then
\begin{subequations}
\begin{align}
\cos(\alpha+\beta) + \ii\sin(\alpha+\beta)
& = \exp(\ii(\alpha+\beta))
=\exp(\ii\alpha + \ii\beta)
=\exp(\ii\alpha)\exp(\ii\beta)\\
&=(\cos\alpha + \ii\sin\alpha)(\cos\beta+\ii\sin\beta)\\
&= \big(\cos\alpha\cos\beta - \sin\alpha\sin\beta\big)
+\ii\big(\sin\alpha\cos\beta+\cos\alpha\sin\beta\big).
\end{align}
\end{subequations}
Comparing real and imaginary parts, we obtain
\begin{equation}
\cos(\alpha+\beta) = \cos\alpha\cos\beta - \sin\alpha\sin\beta
\end{equation}
and 
\begin{equation}
\sin(\alpha+\beta) = \sin\alpha\cos\beta+\cos\alpha\sin\beta.
\end{equation}
\end{remark}

\section*{Exercises}\markright{Exercises}
\addcontentsline{toc}{section}{Exercises}
\setcounter{theorem}{0}

\begin{exercise}
Illustrate Remark~\ref{r:euler} by using a programming language of
your choice
(e.g., \texttt{Maple}, \texttt{Python}, \texttt{Octave},
\texttt{Java}, etc.). 
Hand in (i) a printout of your computer code and
(ii) a printout of the output of your computer code. 
\end{exercise}
\begin{solution}
\begin{center}
    \includegraphics[
        width=\textwidth,
        height=0.5\textheight,
        keepaspectratio
    ]{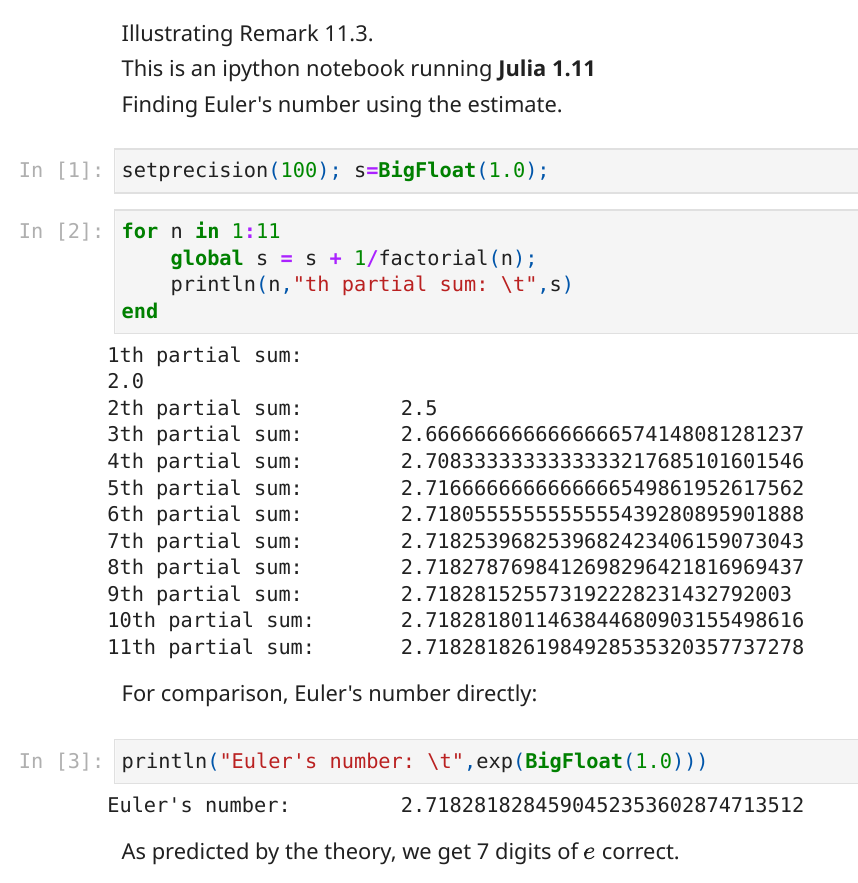}
\end{center}
\end{solution}

\begin{exercise}
Suppose that $N\in\NN$, and 
let $A(z) = a_0+a_1z+a_2z^2 + \cdots + a_Nz^N$ and 
$B(z) = b_0+b_1z+b_2z^2 + \cdots + b_Nz^N$
be polynomials of degree $N$.
Denote the product polynomial $A(z)B(z)$ by $C(z)$,
written as $C(z)=\sum_{k=0}^{2N} c_kz^k$.
If $0\leq n\leq N$, what is $c_n$ in terms of the coefficients of the
original polynomials $A(z)$ and $B(z)$?
Compare to \eqref{e:cauchyprod}.
\end{exercise}
\begin{solution}
We have
\begin{align*}
C(z) &= A(z)B(z)\\
&= \Big(a_0+a_1z+a_2z^2+\cdots +a_Nz^N\Big)
\Big(b_0+b_1z+b_2z^2+\cdots+b_Nz^N\Big)\\
&= (a_0b_0) + (a_1zb_0+a_0b_1z) + (a_2z^2b_0+a_1zb_1z+a_0b_2z^2)\\
&\qquad + \cdots + (a_nz^nb_0 + a_{n-1}z^{n-1}b_1z+\cdots + a_0b_nz^n)
+\cdots +a_Nz^Nb_Nz^N\\
&= (a_0b_0)z^0 + (a_1b_0+a_0b_1)z + (a_2b_0+a_1b_1+a_0b_2)z^2\\
&\qquad + \cdots + (a_nb_0+a_{n-1}b_1+\cdots + a_0b_n)z^n + \cdots 
+ (a_Nb_N)z^{2N}.
\end{align*}
It follows that
\begin{equation*}
c_n = a_nb_0+a_{n-1}b_1+\cdots + a_0b_n = \sum_{k=0}^{n}a_{n-k}b_k,
\end{equation*}
which shows that the definition of the Cauchy product is motivated by
multiplying polynomials or formal power series. 
\end{solution}

\begin{exercise}
Show that the series $\sum_{\nnn} (-1)^n/\sqrt{n+1}$ converges and
then 
show that the Cauchy product of this series with itself
diverges.
\end{exercise}
\begin{solution}
Since $0<1<2<\cdots <n<n+1<\cdots$,
it follows that $(\sqrt{n+1})_\nnn$ is a strictly increasing sequence.
Hence, $(1/\sqrt{n+1})_\nnn$ is a strictly decreasing sequence 
of positive numbers. 
The series thus converges by the Leibniz Alternating Series Test
(Theorem~\ref{t:Leibniz}). 

Denote the Cauchy product of this series with itself by 
$\sum_\nnn c_n$.
Then
\begin{equation*}
c_n = \sum_{k=0}^n
\frac{(-1)^{n-k}}{\sqrt{n-k+1}}\frac{(-1)^k}{\sqrt{k+1}}
= (-1)^n\sum_{k=0}^n\frac{1}{\sqrt{(n-k+1)(k+1)}}.
\end{equation*}
Now if $k\in\{0,1,\ldots,n\}$, then
$(n-k+1)(k+1)\leq (n+1)(n+1)=(n+1)^2$
and so $\sqrt{(n-k+1)(k+1)} \leq n+1$.
It thus follows that
\begin{equation*}
|c_n| = \sum_{k=0}^n\frac{1}{\sqrt{(n-k+1)(k+1)}}
\geq \sum_{k=0}^n \frac{1}{n+1} = (n+1)\frac{1}{n+1}=1.
\end{equation*}
In view of Lemma~\ref{l:convnull},
the series $\sum_\nnn c_n$ cannot converge.
\end{solution}

\begin{exercise}
\label{exo:expm}
Prove Corollary~\ref{c:exp}\ref{c:exp:iii}.
\emph{Hint}:
Prove the result for $m\in\NN$ first, using mathematical induction.
Then use Corollary~\ref{c:exp}\ref{c:exp:ii} to deal with $m<0$.
\end{exercise}
\begin{solution}
When $m=0$, we have $\exp(m)=\exp(0)=1=e^0=e^m$ by \eqref{e:exp(0)=1}. 
Now assume that $\exp(m)=e^m$ for some integer $m\geq 0$.
Then, by Theorem~\ref{t:superexp}, 
$\exp(m+1)=\exp(m)\exp(1)=e^me^1 = e^{m+1}$ and the result
follows for all $m\geq 0$ by the Principle of Mathematical Induction.

Finally, assume that $m<0$.
Then, by Corollary~\ref{c:exp}\ref{c:exp:ii} and what we just
proved about nonnegative integers, we have
$\exp(m) = \exp(-|m|) = 1/\exp(|m|) = 1/e^{|m|} = e^{-|m|} = e^m$.
\end{solution}

\begin{exercise}
Provide two absolutely convergent series such that
\eqref{e:slownet} fails.
\end{exercise}
\begin{solution}
E.g., set $a_n := b_n := 1$ if $n\in\{0,1\}$ and $a_n := b_n := 0$ if
$n\geq 2$. Then
\begin{align*}
\big(1+1+0+0+\cdots\big)
\big(1+1+0+0+\cdots\big)
&= 4
\neq 2 = 1+1 
=\big(1\cdot 1 + 1\cdot 1 + 0\cdot 0 + 0\cdot 0 + \cdots\big),
\end{align*}
which shows that \eqref{e:slownet} fails. 
\end{solution}

\begin{exercise}[logarithm]
Show that $\exp$ is strictly increasing, i.e., 
$x<y$ $\Rightarrow$ $\exp(x)<\exp(y)$, and deduce that 
$\exp$ is injective. Now denote the range of $\exp$ by $R$ so that 
$\exp\colon \RR\to R$ is bijective.
Then the corresponding inverse function is the \textbf{natural logarithm}
$\ln$. \index{logarithm}\index{natural logarithm}
Show that if $u$ and $v$ belong to $R$, then\footnote{The functional
equation for the logarithm is of considerable historical significance
because it converted multiplication into the much simpler addition.}
$$\ln(u\cdot v) = \ln(u) + \ln(v).$$
\end{exercise}
\begin{solution}
Suppose that $x<y$. 
Then $d:= y-x>0$. The definition of $\exp(d)$ gives 
$\exp(d) = \sum_{n=0}^\infty \frac{d^n}{n!} = 1 + d + \frac{d^2}{2} + \cdots > 1$.
Next, $\exp(x)>0$ and $\exp(y)>0$ by Corollary~\ref{c:exp}\ref{c:exp:i}. 
Altogether, 
$\exp(y) = \exp(x+d) = \exp(x)\exp(d) > \exp(x)$, 
and so $\exp$ is strictly increasing, hence injective.

One can show that $R = \RPP$ using Analysis. 
Now let $u$ and $v$ belong to $R$. Then there exist $x$ and $y$ in $\RR$
such that $\exp(x)=u$ and $\exp(y)=v$. 
Thus, $\ln(u)=x$ and $\ln(v)=y$. 
Using the functional equation for the exponential function, we obtain
$u\cdot v = \exp(x)\exp(y)=\exp(x+y)$,
which implies
\begin{equation*}
\ln(u\cdot v) = \ln(\exp(x+y)) = x+y =\ln(u) + \ln(v),
\end{equation*}
as claimed. 
\end{solution}

\begin{exercise}[absolute value in $\CC$]
Given $z=x+\ii y\in \CC$, where $x,y$ are in $\RR$, 
the absolute value is defined by 
$$|z| := \sqrt{x^2+y^2}.$$
Prove the following:
\begin{enumerate}
\item[(i)]
$|z|\geq 0$.
\item[(ii)]
$|z|=0$ $\Leftrightarrow$ $z=0$.
\item[(iii)]
If $z\in\RR$, i.e., $y=0$, then
$|z|$ is the same as the absolute value defined in
Chapter~\ref{ch:of}. 
\end{enumerate}
\end{exercise}
\begin{solution}
(i): This is clear since the square root is nonnegative.

(ii): Using Corollary~\ref{c:130923c}, we have 
$|z|=0$
$\Leftrightarrow$
$\sqrt{x^2+y^2}=0$
$\Leftrightarrow$
${x^2+y^2}=0$
$\Leftrightarrow$
$x=y=0$
$\Leftrightarrow$
$z=0$. 

(iii): Suppose that $y=0$.
Then $|z| = \sqrt{x^2+0^2} = \sqrt{|x|^2} = |x|$. 
\end{solution}

\begin{exercise}[absolute value of the product in $\CC$]
\label{exo:abscompprod}
Show that for $z$ and $w$ in $\CC$, we have
$$|z\cdot w|= |z|\cdot|w|.$$
\end{exercise}
\begin{solution}
Write $z=x+\ii y$ and $w = u+\ii v$,
where $x,y,u,v$ are in $\RR$. 
Using 
Exercise~\ref{exo:151007a}, we obtain that 
\begin{align*}
|z\cdot w|^2 &=|(x+\ii y)(u+\ii v)|^2
=|(xu-yv)+\ii (xv+yu)|^2
=(xu-yv)^2+(xv+yu)^2\\
&=(x^2+y^2)(u^2+v^2)
=|x+\ii y|^2|u+\ii v|^2
=|z|^2|w|^2
=\big(|z|\cdot|w|\big)^2. 
\end{align*}
The result follows by taking square roots. 
\end{solution}

\begin{exercise}[triangle inequality in $\CC$]
Show that for $z$ and $w$ in $\CC$, we have
$$|z+w|\leq |z|+|w|.$$
\emph{Hint:} Show the squared version of this inequality by using
Exercise~\ref{exo:CS}. 
\end{exercise}
\begin{solution}
Write $z=x+\ii y$ and $w = u+\ii v$,
where $x,y,u,v$ are in $\RR$. 
\begin{align*}
|z+w|^2 &=|(x+\ii y)+(u+\ii v)|^2
=|(x+u)+\ii (y+v)|^2
= (x+u)^2 + (y+v)^2 \\
&= x^2+u^2+y^2+v^2 + 2(xu+yv)
\leq x^2+u^2+y^2+v^2 + 2|xu+yv|\\
&= x^2+y^2+u^2+v^2 + 2\sqrt{(xu+yv)^2}
\leq x^2+y^2+u^2+v^2 +2\sqrt{(x^2+y^2)(u^2+v^2)}\\
&= x^2+y^2+u^2+v^2 +2\sqrt{x^2+y^2}\sqrt{u^2+v^2}
=|x+\ii y|^2 + |u+\ii v|^2 + 2|x+\ii y|\cdot|u+\ii v|\\
&=\big(|x+\ii y|+|u+\ii v|\big)^2
=\big(|z|+|w|\big)^2, 
\end{align*}
where we used $\alpha\leq|\alpha|$ 
in the first inequality and 
\cref{exo:CS} in the second inequality. 
The result follows by taking square roots. 
\end{solution}

\begin{exercise}[justifying a limit from Calculus I]
Show that for every $x\in[-3/2,3/2]\smallsetminus\{0\}$ we have
$$\left| \frac{\exp(x)-1}{x}-1\right| \leq |x|.$$
\emph{Hint:} Theorem~\ref{t:expremainder}!
In passing, we note that by letting $x\to 0$, 
we obtain a limit given
(usually without justification) in Calculus~I:
$$\lim_{x\to 0} \frac{\exp(x)-1}{x} = 1.$$
\end{exercise}
\begin{solution}
Consider Theorem~\ref{t:expremainder} with $N=1$.
The inequlity in the hypothesis of that theorem,
$N\geq 2|x|-2$,
translates to
$1\geq 2|x|-2$ 
$\Leftrightarrow$
$2+1\geq 2|x|$
$\Leftrightarrow$
$3\geq 2|x|$
$\Leftrightarrow$
$|x|\leq 3/2$
$\Leftrightarrow$
$x\in[-3/2,3/2]$,
i.e., precisely to our hypothesis on $x$. 
Next, \eqref{e:expremainder} turns into
\begin{equation*}
\left| \exp(x) - \sum_{n=0}^1 \frac{x^n}{n!}\right| \leq
\frac{2|x|^{1+1}}{(1+1)!},
\end{equation*}
which is equivalent to 
\begin{equation*}
\left| \exp(x) - (1+x)\right| \leq
|x|^2. 
\end{equation*}
Now divide by $|x|$ to obtain 
\begin{equation*}
\left| \frac{\exp(x) - 1}{x} - 1\right| \leq
|x|. 
\end{equation*}
The limit statement is clear from the Squeeze Theorem 
(for functions). 
\end{solution}

\begin{exercise}[Cosine Series] 
Show that for every $x\in\RR$, the Maclaurin series 
for the cosine function, 
$$\cos(x) := \sum_{\nnn}(-1)^n\frac{x^{2n}}{(2n)!}= 1 - \frac{x^2}{2!}\pm \cdots,$$
is absolutely convergent.
\end{exercise}
\begin{solution}
The statement is clear when $x=0$. 
So we assume $x\neq 0$ and we set 
\begin{equation*}
a_{n} := \frac{(-1)^nx^{2n}}{(2n)!}, 
\end{equation*}
where $n\in\NN$. 
Then 
\begin{align*}
\left|\frac{a_{n+1}}{a_n}\right|
&= 
\left|\frac{\displaystyle\frac{(-1)^{n+1}x^{2(n+1)}}{(2(n+1))!}}%
{\displaystyle\frac{(-1)^nx^{2n}}{(2n)!}}\right|
= 
\frac{\displaystyle\frac{|x|^{2n+2}}{(2n+2)!}}%
{\displaystyle\frac{|x|^{2n}}{(2n)!}}
=
\frac{|x|^2}{(2n+2)(2n+1)}
\to 0.
\end{align*}
Now apply the Ratio Test (Corollary~\ref{c:utirat} with $q=0$).
\end{solution}

\begin{exercise}[Sine Series] 
Show that for every $x\in\RR$, the Maclaurin series 
for the sine function, 
$$\sin(x) := \sum_{\nnn}(-1)^n\frac{x^{2n+1}}{(2n+1)!}= x - \frac{x^3}{3!}\pm \cdots,$$
is absolutely convergent.
\end{exercise}
\begin{solution}
The statement is clear when $x=0$. 
So we assume $x\neq 0$ and we set 
\begin{equation*}
a_{n} := \frac{(-1)^nx^{2n+1}}{(2n+1)!}.
\end{equation*}
Then 
\begin{align*}
\left|\frac{a_{n+1}}{a_n}\right|
&= 
\left|\frac{\displaystyle\frac{(-1)^{n+1}x^{2(n+1)+1}}{(2(n+1)+1)!}}%
{\displaystyle\frac{(-1)^nx^{2n+1}}{(2n+1)!}}\right|
= 
\frac{\displaystyle\frac{|x|^{2n+3}}{(2n+3)!}}%
{\displaystyle\frac{|x|^{2n+1}}{(2n+1)!}}
=
\frac{|x|^2}{(2n+3)(2n+2)}
\to 0.
\end{align*}
Now apply the Ratio Test (Corollary~\ref{c:utirat} with $q=0$).
\end{solution}

\begin{exercise} 
Let $(c_n)_{n\geq 1}$ be a bounded sequence of real numbers 
(not necessarily convergent).
Prove or disprove:
\begin{equation*}
\sum_{n\geq 1}\frac{c_n}{n^2}
\end{equation*}
is convergent.
\end{exercise}
\begin{solution}
We show that the result is TRUE!

Because $(c_n)_{n\geq 1}$ is a bounded sequence, 
there exists $\gamma\geq 0$ such that 
\begin{equation*}
(\forall n\geq 1)\quad 
|c_n|\leq \gamma.
\end{equation*}
We know that $\sum_{n\geq 1}\frac{1}{n^2}$ converges 
(it is a $p$-series with $p=2$), hence absolutely 
because its terms are positive.
Set, for $n\in\{1,2,\ldots\}$, $b_n := \gamma/n^2 \geq 0$. 
It follows that 
\begin{equation*}
\sum_{n\geq 1} b_n = \sum_{n\geq 1} \frac{\gamma}{n^2}
\end{equation*}
converges absolutely.
Now set 
$a_n := c_n/n^2$, where $n\in\{1,2,\ldots\}$. 
Then 
\begin{equation*}
|a_n| = \frac{|c_n|}{n^2} \leq \frac{\gamma}{n^2} = b_n.
\end{equation*}
By the Comparison Test (Theorem~\ref{t:compartest}),
$\sum_{n\geq 1}a_n$ converges absolutely. 
\end{solution}

\begin{exercise}[Root Test] 
Let $(a_n)_{\nnn}$ be a sequence of 
real numbers and assume that 
\begin{equation*}
q := \lim_{n\to\infty} \sqrt[n]{|a_n|}
\end{equation*}
exists and $q<1$.
(We assume in this exercise the existence of $n$th roots.)
Show that 
$\sum_\nnn a_n$ converges absolutely.
\emph{Hint:} Compare to a geometric series.
\end{exercise}
\begin{solution}
Set $\beta := (1+q)/2<1$. 
Similar to the proof of the Ratio Test, 
there exists $N\in\NN$ such that 
for all $n\geq N$, we have 
\begin{equation*}
\sqrt[n]{|a_n|} \leq \beta; 
\text{~equivalently,~}
|a_n| \leq \beta^n.
\end{equation*}
Comparing with the absolutely convergent series 
$\sum_{n\geq N}\beta^n$, we learn that 
$\sum_{n\geq N}|a_n|$ converges and hence so does 
$\sum_{\nnn}|a_n|$.
\end{solution}

\begin{exercise} 
Define the sequence $(a_n)_{\nnn}$ by 
\begin{equation*}
a_n := \begin{cases}
    \displaystyle \frac{1}{2^n}, &\text{if $n$ is even;}\\[+5mm]
    \displaystyle \frac{3}{2^n}, &\text{if $n$ is odd.}
\end{cases}
\end{equation*}
Prove or disprove: The series 
\begin{equation*}
\sum_{\nnn} a_n = 1 + \frac{3}{2} + \frac{1}{2^2} + \frac{3}{2^3} + \cdots 
\end{equation*}
is convergent.
\end{exercise}
\begin{solution}
We show that the result is TRUE!
(An interesting feature of this example is that the ratio test
is not applicable!)

Set $b_n \equiv 3/2^n$.
Then $\sum_{\nnn} b_n$ converges absolutely 
(it is a multiple of a geometric series!).
Moreover, clearly $|a_n| = a_n \leq b_n$ for all $\nnn$.
By the Comparison Test (Theorem~\ref{t:compartest}),
$\sum_{\nnn}a_n$ converges absolutely. 
\end{solution}

\begin{exercise}[Bessel Function] 
The \emph{Bessel function} (of order $0$) is defined by 
\begin{equation*}
J_0(x) := \sum_{n=0}^\infty \frac{(-1)^n x^{2n}}{2^{2n}(n!)^2}
\end{equation*}
for every $x\in\RR$.
Show that this series is absolutely convergent.
\end{exercise}
\begin{solution}
The statement is clear when $x=0$.
So we assume that $x\neq 0$ and we set
\begin{equation*}
a_n := \frac{(-1)^n x^{2n}}{2^{2n}(n!)^2}.
\end{equation*}
Then 
\begin{align*}
\left|\frac{a_{n+1}}{a_n}\right|
&= 
\left|\frac{\displaystyle\frac{(-1)^{n+1} x^{2(n+1)}}{2^{2(n+1)}((n+1)!)^2}}%
{\displaystyle\frac{(-1)^n x^{2n}}{2^{2n}(n!)^2}}\right|
= \frac{\displaystyle \frac{|x|^{2n+2}}{2^{2n+2}((n+1)!)^2}}%
{\displaystyle \frac{|x|^{2n}}{2^{2n}(n!)^2}}
= \frac{|x|^2}{2^2(n+1)^2}
\to 0. 
\end{align*}
Now apply the Ratio Test (Corollary~\ref{c:utirat} with $q=0$).
\end{solution}

\begin{exercise}[Binomial Series] 
The \emph{binomial series} is defined by 
\begin{equation*}
B(x) := \sum_{\nnn} \frac{p(p-1)\cdots (p-n+1)}{n!}x^n
\end{equation*}
for $x\in\RR$ and $p\in\RR$.\\
(i) Show that this series converges absolutely when $|x|<1$.\\
(ii) Simplify $B(x)$ when $p\in\NN$.
\end{exercise}
\begin{solution}
When $x=0$, the series clearly converges absolutely and $B(0)=1$.

So we assume that $x\neq 0$ and we set
\begin{equation*}
a_n := \frac{p(p-1)\cdots (p-n+1)x^n}{n!}
\end{equation*}
If $p\in\NN$, then $a_n=0$ for $n\geq p+1$,
the series thus converges absolutely and $B(x) = (x+1)^p$
by the Binomial Theorem! (This is also true for any $p$ when $|x|<1$.)

So assume $p\in\RR\smallsetminus\NN$.
Then $a_n\neq 0$ for all $\nnn$ and 
\begin{align*}
\left|\frac{a_{n+1}}{a_n}\right|
&= 
\left|\frac{\displaystyle\frac{p(p-1)\cdots(p-(n+1)+1)x^{n+1}}{(n+1)!}}%
{\displaystyle\frac{p(p-1)\cdots(p-n+1)x^n}{n!}}\right|
= \frac{|x||p-n|}{(n+1)}
= \frac{|x||p/n-1|}{(1+1/n)}
\\
&\to \frac{|x||0-1|}{1+0} = |x|. 
\end{align*}
Now apply the Ratio Test (Corollary~\ref{c:utirat} with $q=|x|$).
\end{solution}

\begin{exercise}[YOU be the marker!]  
Consider the following statement 
\begin{equation*}
\text{``$\exp(1)=0$''}
\end{equation*}
and the following ``proof'':
\begin{quotation}
We have seen that $1=1\cdot 1$.\\
Hence $\exp(1)=\exp(1\cdot 1)$.\\
The functional equation now implies 
$\exp(1) = \exp(1)+\exp(1)$.\\
Therefore, $0=\exp(1)$.
\end{quotation}
Why is this proof wrong?
\end{exercise}
\begin{solution}
The functional equation is for \emph{sums} of values, not \emph{products}.
(The proof could be recycled for $\ln$ to show that $\ln(1)=0$.)
\end{solution}

\begin{exercise}[TRUE or FALSE?] 
Mark each of the following statements as either true or false. 
Briefly justify your answer.
\begin{enumerate}
\item ``If $A := \sum_{\nnn} a_n$ and $B := \sum_{\nnn} b_n$ converge absolutely, 
then $\sum_{\nnn}a_nb_n = A\cdot B$.''
\item ``If $x,y$ are in $\RR$, then $\exp(x\cdot y) = \exp(x)+\exp(y)$.''
\item ``If $x\in\RR$, then $\sqrt{\exp(2x)}=\exp(x)$.''
\item ``$4<\exp(2)<9$.''
\end{enumerate}
\end{exercise}
\begin{solution}
(i): FALSE: $B := A := 2 = 1+1+0+0+\cdots$. 
Then $1\cdot 1 + 1\cdot 1 + 0\cdot 0 + 0 \cdot 0 + \cdots = 2
\neq 4 = A\cdot B$.

(ii): FALSE: $x:=y:=0$ leads to $1=1+1$. 

(iii): TRUE: $\exp(2x)=\exp(x+x)=\exp(x)\exp(x)$. Now take square roots!

(iv): TRUE: We saw that $2<\exp(1)<3$ (see Remark~\ref{r:euler}).
Now $2=1+1$ and so $\exp(2)=\exp(1+1)=\exp(1)\exp(1)$, 
which yields the result.
\end{solution} 
\chapter{Miscellany}

\section{Infimum and Supremum}

\begin{definition}[bounded above and bounded below]
Let $S$ be a subset of $\RR$.
We say that $S$ is \textbf{bounded above}
(resp.\ \textbf{bounded below})
if there exists $\beta\in\RR$ 
such that for all $x\in S$, we have
$x\leq \beta$ (resp.\ $\beta \leq x$). In this case,
the number $\beta$ is called 
an \textbf{upper bound} (resp.\ \textbf{lower bound})
of $S$. 
The set $S$ is called \textbf{bounded}
if it is both bounded above and bounded below.
\index{bounded set}
\index{bounded below (set)}\index{bounded above (set)}
\index{upper bound}\index{lower bound}
\end{definition}

Note that if $(a_n)_\nnn$ is a sequence,
then the set $\{a_n\}_\nnn$ is bounded if and only
if the sequence is bounded in the sense of 
Definition~\ref{d:boundedseq}.

Also, similarly to \eqref{e:bdseq},
we can show that a set $S$ is bounded if and only
if
\begin{equation}
\text{there exists $\gamma\geq 0$ such that
$|x|\leq\gamma$,
for every $x\in S$.}
\end{equation}

\begin{definition}[supremum]
\label{d:sup}
Let $S$ be a nonempty subset of $\RR$ that is bounded above, and
let $\beta\in\RR$. 
Then $\beta$ is 
the \textbf{least upper bound}
or \textbf{supremum}\footnote{The name comes from the Latin
``suppremum'' = ``highest part''.}of $S$, written
\begin{equation}
\beta = \sup S,
\end{equation}
if the following hold.
\begin{enumerate}
\item $\beta$ is an upper bound of $S$.
\item If $\alpha<\beta$, then $\alpha$ is \emph{not} an upper bound of $S$.
\end{enumerate}
If $\sup S\in S$, then $\sup S$ is also called
the \textbf{maximum} of $S$, written $\max S$. 
\index{supremum}\index{maximum}
\end{definition}

It is easy to see that the supremum is unique, provided it exists
(Exercise~\ref{exo:supunique}). 
It is convenient to define $\sup\varnothing := -\infty$, 
and $\sup S :=+\infty$ if $S$ is not bounded above.

The definition of the infimum\footnote{The name is the noun
corresponding to the Latin ``infimus'' = ``lowest''.} is analogous:

\begin{definition}[infimum]\label{d:inf}
Let $S$ be a nonempty subset of $\RR$ that is bounded below, and
let $\beta\in\RR$. 
Then $\beta$ is 
the \textbf{greatest lower bound}
or \textbf{infimum} of $S$, written
\begin{equation}
\beta = \inf S,
\end{equation}
if the following hold.
\begin{enumerate}
\item $\beta$ is a lower bound of $S$.
\item If $\alpha>\beta$, then $\alpha$ is \emph{not} a lower bound of $S$.
\end{enumerate}
If $\inf S\in S$, then $\inf S$ is also called
the \textbf{minimum} of $S$, written $\min S$. 
\end{definition}

The infimum is unqiue provided it exists, and
one sets $\inf \varnothing := +\infty$,
and $\inf S := -\infty$ when $S$ is not bounded below.
With the notation
\begin{equation}
-S := \menge{-s}{s\in S},
\end{equation}
we have (Exercise~\ref{exo:infsup}) 
\begin{equation}
\label{e:infsup}
\inf S = -\sup(-S).
\end{equation}

\begin{example}[intervals]
Let $a$ and $b$ be real numbers such that $a<b$.
Then the following hold.
\begin{enumerate}
\item
$\inf [a,b] = \min [a,b] = a$ and 
$\sup [a,b] = \max [a,b] = b$.
\item
$\inf \left]a,b\right] = a \notin \left]a,b\right]$
and 
$\sup \left]a,b\right] = \max \left]a,b\right] = b$.
\item
$\inf \left[a,b\right[ = \min \left[a,b\right[ = a$
and 
$\sup \left[a,b\right[ = b \notin \left[a,b\right[$.
\item
$\inf \left]a,b\right[ = a \notin \left]a,b\right[$
and 
$\sup \left]a,b\right[ = b \notin \left]a,b\right[$.
\end{enumerate}
\end{example}

The existence of infimum and supremum is a subtle
property of the set of real numbers.

\begin{theorem}
Let $S$ be a nonempty subset of $\RR$ such that
$S$ is bounded above.
Then $\sup S$ exists
and there exists a sequence $(x_n)_\nnn$ in $S$
such that $x_n \to \sup S$.
\end{theorem}
\begin{proof}
Let $\beta_0$ be an upper bound of $S$, take $x_0\in S$,
and set $\delta := \beta_0-x_0 \geq 0$.

We shall construct inductively sequences  $(x_n)_\nnn$ 
and $(\beta_n)_\nnn$ in $S$ and $\RR$ respectively such that 
\begin{enumerate}
\item 
\label{sup1}
$(\forall n\geq 1)$ 
$x_{n-1}\leq x_n\in S$ (which implies that
$(x_n)_\nnn$ is an increasing sequence of points in $S$);
\item
\label{sup2}
$(\forall n\geq 1)$ 
$\beta_{n-1}\geq \beta_n$ and $\beta_n$ is an upper bound of $S$ (which imples that 
$(\beta_n)_\nnn$ is a decreasing sequence of upper bounds of $S$);
\item
\label{sup3}
$(\forall \nnn)$ 
$\beta_n-x_n \leq \delta/2^n$.
\end{enumerate}
The base case is clear.

Now assume that $x_0\leq x_1\leq\cdots\leq x_n$
and $\beta_0\geq\beta_1\geq \cdots\geq\beta_n$ are given
with the prescribed properties, for some integer $\nnn$.
It is clear that $x_n\leq \beta_n$ since $x_n\in S$ and 
$\beta_n$ is an upper bound of $S$.
Set
\begin{equation}
\mu := \frac{x_n + \beta_n}{2},
\end{equation}
i.e., $\mu$ is the midpoint of the interval $[x_n,\beta_n]$.

\emph{Case~1:} $\mu$ is an upper bound of $S$.\\
We then set $x_{n+1} := x_n$ and $\beta_{n+1} := \mu$.

\emph{Case~2:} $\mu$ is not an upper bound of $S$.\\
Then there must exist a point $x_{n+1} \in
\left]\mu,\beta_n\right] \cap S$, and we set $\beta_{n+1} := \beta_n$. 

In either case, 
it is clear that $x_n\leq x_{n+1}\leq \beta_{n+1}\leq \beta_n$,
and that $\beta_{n+1}-x_{n+1} \leq (\beta_n-x_n)/2 \leq 
(\delta/2^n)/2 = \delta/2^{n+1}$, as required.

Having constructed these monotone sequences and utilizing
Theorem~\ref{t:monconv}, we see that they must both converge
to the same limit $\beta$.

It remains to show that $\beta = \sup S$. 
Indeed, if $s\in S$, then $s\leq \beta_n$ for all $\nnn$ and hence
$s\leq\beta$ by Theorem~\ref{t:limleq}.
Now assume that $\alpha < \beta$.
Choose $\nnn$ so large such that $\delta/2^n < \beta-\alpha$.
Then 
$x_n\geq \beta_n-\delta/2^n \geq \beta-\delta/2^n > \alpha$,
so $\alpha$ is definitely not an upper bound of $S$.
\end{proof}

\begin{corollary}
Let $S$ be a nonempty subset of $\RR$ such that $S$ is bounded below.
Then $\inf S$ exists and
there exists a sequence $(x_n)_\nnn$ in $S$ such that
$x_n\to\inf S$.
\end{corollary}

\begin{remark}
Infima and suprema of subsets of $\QQ$ need not necessarily lie in $\QQ$: 
for instance, the set
\begin{equation}
S := \menge{ x\in\QQ}{ x^2<2}
\end{equation}
is bounded above, but $\sup S = \sqrt{2}\in\RR\smallsetminus \QQ$
by Theorem~\ref{t:squareroot2}.
\end{remark}

\section{Limit inferior and limit superior}

The notions of limit inferior and limit superior can
be interpreted as generalizations of the idea of the limit
of a sequence. These generalized limits always exist, even if the
sequence under consideration has no limit.

\begin{definition}\label{d:liminfsup}
Let $(x_n)_\nnn$ be a sequence of real numbers.
For every $\nnn$, set
\begin{equation}
S_n := \{x_n,x_{n+1},\ldots\}.
\end{equation}
Then $S_0 \supseteq S_1 \supseteq \cdots \supseteq S_n \supseteq
S_{n+1}$, which implies that
\begin{equation}
\sup S_n \geq \sup S_{n+1}
\quad\text{and}\quad
\inf S_n \leq \inf S_{n+1},
\qquad \text{for every $\nnn$.}
\end{equation}
Both sequences thus converge (with limits possibly in $\pm\infty$).
One writes
\begin{equation}
\varlimsup_\nnn x_n := \lim_{\nnn} \sup S_n = \inf_{\nnn} \sup S_n
\end{equation}
and 
\begin{equation}
\varliminf_\nnn x_n := \lim_{\nnn} \inf S_n = \sup_{\nnn} \inf S_n,
\end{equation}
and says that 
$\varlimsup_\nnn x_n$ is the \textbf{limit superior}
(or \textbf{upper limit})
of $(x_n)_\nnn$ and
$\varliminf_\nnn x_n$ is the \textbf{limit inferior}
(or \textbf{lower limit}) 
of $(x_n)_\nnn$.
\index{limit inferior}\index{limit superior}
\end{definition}

Not surprisingly, 
\begin{equation}
\label{e:200819a}
\varliminf_\nnn x_n \leq \varlimsup_\nnn x_n.
\end{equation}

\begin{example} \label{ex:minuseins}
$\varliminf_\nnn (-1)^n = -1 < +1 = \varlimsup_\nnn (-1)^n$.
\end{example}

Suppose that $(x_n)_\nnn$ is a bounded sequence.
It may be shown that 
$\varlimsup x_n$ is the
largest cluster point of the sequence $(x_n)_\nnn$,
that $\varliminf x_n$
is the smallest cluster point of the sequence $(x_n)_\nnn$,
and that 
\begin{equation}
\label{e:last}
\text{
$(x_n)_\nnn$ is convergent 
$\Leftrightarrow$
$\varliminf_\nnn x_n = \varlimsup_\nnn x_n$,
}
\end{equation}
in which case
$\lim_\nnn x_n = \varliminf_\nnn x_n = \varlimsup_\nnn x_n$.

\section*{Exercises}\markright{Exercises}
\addcontentsline{toc}{section}{Exercises}
\setcounter{theorem}{0}

\begin{exercise}
\label{exo:supunique}
Let $S$ be a subset of $\RR$. Show that 
the supremum is unique provided it exists.
\end{exercise}
\begin{solution}
Suppose $\beta'$ and $\beta$ are both suprema,
but different, say $\beta'<\beta$.
On the one hand, since $\beta$ is a supremum of $S$ and $\beta'<\beta$,
it follows that $\beta'$ is not an upper bound of $S$.
On the other hand, we assumed that $\beta'$ is a supremum of $S$;
in particular, it is an upper bound.
Altogether, we have reached a contradiction.
\end{solution}

\begin{exercise}
\label{exo:supmono}
Let $A$ and $B$ be nonempty subsets of $\RR$ both of which are bounded 
above and such that $A\subseteq B$.
Show that $\sup A \leq \sup B$.
\end{exercise}
\begin{solution}
Write 
$\alpha := \sup A$ and $\beta := \sup B$.
Then $\beta$ is an upper bound for $B$,
hence $(\forall b\in B)$ $b\leq \beta$ and
thus $(\forall a\in A)$ $a\leq \beta$.
It follows that $\beta$ is an upper bound for $A$.
We claim that $\beta\geq \alpha$.
Suppose to the contrary that $\beta<\alpha$.
By definition of $\alpha$,
we must have that $\beta$ is not an upper bound of $A$
which is absurd.
\end{solution}

\begin{exercise}
\label{exo:infsup}
Prove \eqref{e:infsup}, where we assume (for simplicity) that
$S$ is a nonempty bounded subset of $\RR$. 
\end{exercise}
\begin{solution}
We must show that
\begin{equation*}
\inf S = -\sup(-S).
\end{equation*}
Set $\sigma := \sup(-S)$.
We must show that $-\sigma = \inf S$, i.e.,
(i) $-\sigma$ is a lower bound of $S$ and
(ii) If $\alpha > -\sigma$, then $\alpha$ is not a lower bound of
$S$.

(i): Since $\sigma$ is an upper bound for $-S$, we have
$(\forall s\in S)$ $-s\leq \sigma$;
equivalently, 
$(\forall s\in S)$ $s \geq -\sigma$ as required. 

(ii): Let $\alpha > -\sigma$. 
Then $-\alpha < \sigma$.
By definition of $\sigma$ as the supremum of $-S$,
$-\alpha$ cannot be an upper bound of $-S$.
Hence there exists $s\in S$ such that $-\alpha < -s$.
It follows that $\alpha > s$ and $\alpha$ is not a lower bound
of $S$ as required. 
\end{solution}

\begin{exercise} 
Let $A$ be a nonempty subset of $\RR$ such that 
$A$ is bounded above.
Let $\sigma \in \RR$. 
Show that 
$\sigma = \sup A$
$\Leftrightarrow$
[ 
there exists a sequence $(a_n)_\nnn$ in $A$
such that $a_n\to\sigma$ and 
for every sequence $(b_n)_\nnn$ in $A$ such that 
$b_n\to\beta$ we must have $\beta \leq\sigma$].
\end{exercise}
\begin{solution}
``$\Rightarrow$'':
Suppose $\sigma=\sup A$.
Then $\sigma-\tfrac{1}{2^n}$
is not an upper bound of $A$,
so there exists $a_n\in A$
such that 
$\sigma-\tfrac{1}{2^n}<a_n\leq \sigma$.
The Squeeze Theorem now gives $a_n\to\sigma$.
Next, let $(b_n)_\nnn$ be in $A$ and $b_n\to\beta$.
Then $(\forall\nnn)$ $b_n\leq \sigma$.
Taking limits yields $\beta\leq\sigma$. 

``$\Leftarrow$'':
Let $a\in A$. Setting $b_n\equiv a$ gives $a \leq\sigma$; 
thus, $\sigma$ is an upper bound of $A$. 
On the other hand, let $\lambda<\sigma$ and pick up a sequence $(a_n)_\nnn$
in $A$ such that $a_n\to\sigma$.
Then eventually $a_n>\lambda$ and so $\lambda$ is not an upper bound of $A$.
Altogether, $\sigma = \sup A$.
\end{solution}

\begin{exercise} 
Let $A$ be a nonempty subset of $\RR$ that is bounded above.
Prove or disprove:
\begin{equation*}
\text{If $\sup A$ is not the maximum, then $A$ contains infinitely many elements.}
\end{equation*}
\end{exercise}
\begin{solution}
We show that the result is true.

Suppose to the contrary that $A$ contains only finitely many distinct elements,
say $A =\{a_1,a_2,\ldots,a_n\}$ for some $n\geq 1$. 
Then $A$ has a unique largest element $a_m$, where $1\leq m\leq n$.
But then $a_m=\sup(A)\in A$ and so $a_m$ is the maximum,
which contradicts the assumption. 
\end{solution}

\begin{exercise} 
Verify \eqref{e:200819a}.
\end{exercise}
\begin{solution}
Let $S_n$ be as in Definition~\ref{d:liminfsup}.
Then 
\begin{equation*}
\varliminf_\nnn x_n = \lim_\nnn \inf S_n 
\;\;\text{and}\;\;
\varlimsup_\nnn x_n = \lim_\nnn \sup S_n. 
\end{equation*}
For every $\nnn$, we have 
\begin{equation*}
(\forall s_n\in S_n)\quad 
\inf S_n \leq s_n \leq \sup S_n. 
\end{equation*}
Now taking the limit as $n\to\infty$ gives
$\varliminf_\nnn x_n \leq \varlimsup_\nnn x_n$.
\end{solution}

\begin{exercise}
Provide, if possible, sequences $(x_n)_\nnn$ such that the following hold:
\begin{enumerate}
\item
$-\infty < \varliminf_\nnn x_n < \varlimsup_\nnn x_n=+\infty$;
\item
$-\infty = \varliminf_\nnn x_n < \varlimsup_\nnn x_n=+\infty$;
\item
$-\infty < \varliminf_\nnn x_n < \varlimsup_\nnn x_n<+\infty$ and 
$(x_n)_\nnn$ has exactly 3 cluster points. 
\item
$-\infty < \varliminf_\nnn x_n < \varlimsup_\nnn x_n<+\infty$ and 
$(x_n)_\nnn$ has infinitely many cluster points. 
\item
$-\infty = \varliminf_\nnn x_n < \varlimsup_\nnn x_n=+\infty$ and 
every real number is a cluster point of $(x_n)_\nnn$. 
\end{enumerate}
\end{exercise}
\begin{solution}
(i): $((-1)^n)_\nnn$ has $\varliminf = -1$ and $\varlimsup = +1$.

(ii): $(0,1,0,2,0,3,0,4,0,5,0,6,\ldots)$ has 
$\varliminf = 0$ and $\varlimsup = \pinf$.

(iii): 
$(-1,0,1,-1,0,1,-1,0,1,\ldots)$.

(iv): Give yourself two numbers, say $-10$ and $10$.
Start at $x_0$, then $x_1 = x_0+1/1$,
$x_2=x_1+1/2$, $x_3=x_2+1/3$, $x_{k}=x_{k-1}+1/k$ until you jump over 
your upper hurdle $10$: $x_n<10$, $x_{n+1}\geq 10$.
(This does happen because the harmonic series diverges!).
Now reverse course: $x_{n+2} = x_{n+1}- 1/(n+2)$,
until you reach the lower hurdle $-10$. Then reverse course again.
Every hurdle is crossed infinitely often. And the step sizes 
become smaller and smaller.
The sequence you created will have the entire interval 
$[-10,10]$ as its cluster point --- isn't that wonderfully weird?

(v): 
Argue as in (iv), but change the hurdles:
$10,-10$: Once your crossed below $-10$, go up to $20$,
then down to $-20$, then $30$ and $-30$, and so on.
Again the divergence of the harmonic series makes this possible.
The eventual sequence has \emph{every real numbers as a cluster point}!!
(Math is stranger and more fun than much of fiction!)
\end{solution}

\begin{exercise} 
\label{exo:goodcharoflimsup}
Let $(x_n)_\nnn$ be a bounded sequence of real numbers,
and let $\sigma\in\RR$. 
Show that 
$\sigma = \varlimsup_\nnn x_n$
$\Leftrightarrow$
$(\forall\varepsilon>0)$
$\menge{\nnn}{x_n>\sigma-\varepsilon}$ is infinite and 
$\menge{\nnn}{x_n\geq \sigma+\varepsilon}$ is finite.
\end{exercise}
\begin{solution}
Let $S_n$ be as in Definition~\ref{d:liminfsup} and
set $\sigma_n := \sup S_n\to\varlimsup_\nnn x_n$. 

``$\Rightarrow$'':
Suppose that $\sigma = \varlimsup_\nnn x_n$. 
We know that $\sigma_n \downarrow\sigma$ by assumption. 
Let $\varepsilon>0$.
Entertain the thought that 
$\menge{\nnn}{x_n>\sigma-\varepsilon}$ were finite.
Then there exists $N\in\NN$ such that 
$n\geq N$ $\Rightarrow$
$\menge{\nnn}{x_n>\sigma-\varepsilon}=\varnothing$, i.e., 
$x_n\leq \sigma-\varepsilon$. 
But then $\sigma_n \leq\sigma-\varepsilon$  for all $n\geq N$
and so (taking limits) $\sigma \leq \sigma-\varepsilon$ which is 
absurd. Hence 
$\menge{\nnn}{x_n>\sigma-\varepsilon}$ is indeed infinite.
Now entertain the thought that 
$\menge{\nnn}{x_n\geq \sigma+\varepsilon}$ is infinite.
Then $\sigma_n \geq \sigma+\varepsilon$ and 
(taking limits) $\sigma \geq \sigma+\varepsilon$ which is absurd.

``$\Leftarrow$'':
Because $\menge{\nnn}{x_n\geq \sigma+\varepsilon}$ is finite,
there is an index $N$ such that 
$n\geq N$ $\Rightarrow$ 
$x_n<\sigma+\varepsilon$; 
hence, $\sup S_n \leq \sigma+\varepsilon$.
Taking the limit yields 
$\varlimsup_\nnn x_n \leq \sigma+\varepsilon$.
Since this is true for every $\varepsilon >0$,
it follows that 
$\varlimsup_\nnn x_n \leq \sigma$. 
On the other hand, 
$\menge{\nnn}{x_n>\sigma-\varepsilon}$ is infinite.
Hence $\sigma_n \geq \sigma-\varepsilon$ and 
(taking the limit as $n\to\infty$)
$\varlimsup_\nnn x_n \geq \sigma-\varepsilon$.
Since this is true for every $\varepsilon>0$,
it follows that 
$\varlimsup_\nnn x_n \geq \sigma$.
Altogether, 
$\varlimsup_\nnn x_n = \sigma$.
\end{solution}

\begin{exercise} 
\label{exo:200821a}
Let $(x_n)_\nnn$ be a bounded sequence of real numbers,
and set $\sigma := \varlimsup_\nnn x_n$.
Show that 
(i) $\sigma$ is a cluster point of $(x_n)_\nnn$
and 
(ii) if $\lambda$ is a cluster point of $(x_n)_\nnn$,
then $\lambda \leq\sigma$.
Put differently, $\varlimsup_\nnn x_n$ is the 
largest cluster point of $(x_n)_\nnn$.
\emph{Hint:} Exercise~\ref{exo:goodcharoflimsup}. 
\emph{Comment:} A similar proof then shows that 
$\varliminf_\nnn x_n$ is the 
smallest cluster point of $(x_n)_\nnn$.
\end{exercise}
\begin{solution}
Let $\varepsilon >0$
By Exercise~\ref{exo:goodcharoflimsup}, 
$\menge{\nnn}{x_n>\sigma-\varepsilon}$ is infinite,
and there exists $N\in\NN$ such that 
$\menge{\nnn}{x_n<\sigma+\varepsilon} =
\{N,N+1,N+2,\ldots\}$.
Thus, there exists \emph{infinitely many} indices $n$ such that 
$\sigma-\varepsilon < x_n < \sigma+\varepsilon$, i.e., 
$|x_n-\sigma|<\varepsilon$. 
Letting $\varepsilon = 1/2^0$, 
we find $n_0\in\NN$ such that 
$|x_{n_0}-\sigma|<1/2^0$. 
Letting $\varepsilon = 1/2^1$, 
we find $n_1\in\NN$ such that 
$n_1>n_0$ and 
$|x_{n_1}-\sigma|<1/2^1$. 
Continuing in this fashion, we obtain 
a subsequence $(x_{n_k})_{\kkk}$ such that 
$|x_{n_k}-\sigma| <1/2^k$. 
It follows that $\lim_{\kkk} x_{n_k}\to\sigma$
and therefore $\sigma$ is indeed a cluster point of $(x_n)_\nnn$. 

Now let $\lambda$ be another (possibly different) 
cluster point of $(x_n)_\nnn$. 
The set $\menge{\nnn}{x_n\geq \sigma+\varepsilon}$ is finite;
therefore, $\lambda$ cannot be larger than $\sigma$, 
and so $\lambda \leq\sigma$. 
\end{solution}

\begin{exercise}
Use Exercise~\ref{exo:200821a} to 
quickly verify Example~\ref{ex:minuseins}. 
\end{exercise}
\begin{solution}
The sequence has exactly 
two cluster points, $-1$ and $1$.
(If there was any other cluster point different from these,
there must be terms arbitrarily close to that point but
all terms of the sequence are $1$ or $-1$.)
Now apply Exercise~\ref{exo:200821a}. 
\end{solution}

\begin{exercise}
Let $(x_n)_\nnn$ be a bounded sequence of real numbers.
Verify \eqref{e:last}. 
\end{exercise}
\begin{solution}
``$\Rightarrow$'':
Suppose $x_n\to\lambda$.
Then $(x_n)_\nnn$ has only one cluster point (and it is 
still a bounded sequence).
It follows from Exercise~\ref{exo:200821a} that 
$\lambda = \varlimsup_\nnn x_n = \varliminf x_n$.

``$\Leftarrow$'':
Suppose to the contrary that $(x_n)_\nnn$ does not converge.
Then it has at least two cluster points $\lambda_1<\lambda_2$.
By Exercise~\ref{exo:200821a},
$\varliminf_\nnn x_n \leq \lambda_1 < \lambda_2 \leq \varlimsup_\nnn x_n$ 
which contradicts the assumption. 
\end{solution}

\begin{exercise} 
Prove or disprove:
There exists a sequence $(a_n)_\nnn$ of real numbers such that 
\begin{equation*}
\inf\menge{a_n}{\nnn}
< \varliminf_\nnn a_n 
< \varlimsup_\nnn a_n
< \sup\menge{a_n}{\nnn}.
\end{equation*}
\end{exercise}
\begin{solution}
By way of example, we show that the result is true.
Set 
\begin{equation*}
    a_n\equiv (-1)^n\big(1+\tfrac{1}{2^n})
    = \big(2,-(1+\tfrac{1}{2}),1+\tfrac{1}{4},-(1+\tfrac{1}{8}),\ldots) \big).
\end{equation*}
Then 
$\inf\menge{a_n}{\nnn} = -(1+\tfrac{1}{2}) = -1.5$,
$\varliminf_\nnn a_n = -1$,
$\varlimsup_\nnn a_n = 1$,
and $\sup\menge{a_n}{\nnn} = 2$.
\end{solution}

\begin{exercise} 
Prove or disprove:
If $(a_n)_\nnn$ and $(b_n)_\nnn$ are sequences of real numbers, 
then 
\begin{equation*}
\varliminf_\nnn (a_n+b_n) = \varliminf_\nnn a_n + \varliminf_\nnn b_n.
\end{equation*}
\end{exercise}
\begin{solution}
The result is false, we provide a counterexample. 

Set $a_n\equiv (-1)^n = (1,-1,1,-1,\ldots)$ 
and $b_n\equiv (-1)^{n+1} = (-1,1,-1,1,\ldots)$.
Then $a_n+b_n\equiv 0$ and thus 
$\varliminf_\nnn (a_n+b_n) = 0$.
On the other hand,
$\varliminf_\nnn a_n = \varliminf_\nnn b_n = -1$ and thus 
$\varliminf_\nnn a_n + \varliminf_\nnn b_n = -1 + (-1) = -2\neq 0$.
\end{solution}

\begin{exercise}[YOU be the marker!]  
Consider the following statement 
\begin{equation*}
\text{``$\varliminf_{\nnn}|x_n\cdot y_n| = 
\Big(\varliminf_\nnn|x_n|\Big)\cdot \Big(\varliminf_\nnn|y_n|\Big)$''}
\end{equation*}
concerning two sequences of real numbers and the following ``proof'':
\begin{quotation}
Properties of the absolute value give us $|x_n\cdot y_n| = |x_n|\cdot|y_n|$ for all $\nnn$.\\
Now taking the limit inferior gives 
$\varliminf_{\nnn}|x_n\cdot y_n| = 
\varliminf_\nnn\big(|x_n|\cdot|y_n|\big)$. \\
Therefore, the result follows by using the product law for limits. 
\end{quotation}
Why is this proof wrong?
\end{exercise}
\begin{solution}
The first two lines are correct.
The last line is nonsense, because we only have a product law 
for actual limits, not limit superior or limit inferior.
And indeed, the result is false and here is a counterexample:
\begin{align*}
(x_n)&=\big(1,\tfrac{1}{1},2,\tfrac{1}{2},3,\tfrac{1}{3},\ldots \big),\\
(y_n)&=\big(\tfrac{1}{1},1,\tfrac{1}{2},2,\tfrac{1}{3},3,\ldots \big).
\end{align*}
Then $x_n\cdot y_n\equiv 1>0$ and so the LHS $=1$.
On the other hand, 
$\varliminf_\nnn |x_n| = \varliminf_\nnn |y_n| = 0$ and so the RHS $=0$. 
\end{solution}

\begin{exercise}[TRUE or FALSE?] 
Mark each of the following statements as either true or false. 
Briefly justify your answer.
\begin{enumerate}
\item ``If $A=B$, then  $\inf A = \inf B$ and $\sup A = \sup B$.''
\item ``If $\inf A = \inf B$ and $\sup A = \sup B$, then $A=B$.''
\item ``$\sup(-A)=-\sup(A)$.''
\item ``If $\gamma = \inf A = \sup A$, then $A=\{\gamma\}.$''
\end{enumerate}
\end{exercise}
\begin{solution}
(i): TRUE: If the sets are the same, so are their suprema and infima.

(ii): FALSE: Consider $A=\{0,1\}$ and $B=[0,1]$. 

(iii): FALSE: For $A=[-1,1]=-A$, we have $\sup(-A)=1$ and $-\sup(A)=-1$.

(iv): TRUE: Let $a\in A$.
Then $\gamma = \inf(A)\leq a \leq \sup(A) = \gamma$.
Hence $a=\gamma$.
\end{solution} 
\Extrachap{Symbols and Notation}

\begin{tabular}{lll}

$\lnot p, p\land q, p\lor q$ & p.~\pageref{d:notandor} & not, and,
or\\

$p\Rightarrow q, p\Leftrightarrow q$ & p.~\pageref{d:notandor} &
implies, if and only if\\

$\CC$ & p.~\pageref{sec:fields} & The set of complex numbers\\

$\NN$ & p.~\pageref{def:NN} &
The set of nonnegative integers integers $\{0,1,\dots\}$\\

$\QQ$ & p.~\pageref{sec:fields} & The set of rational numbers\\

$\RR$ & p.~\pageref{cha:fields} & The set of real numbers\\

$\ZZ$ & p.~\pageref{def:ZZ} &
The set of integers integers $\{0,1,-1,2,-2,\dots\}$\\

$n!$ & p.~\pageref{def:factorial} & $1\cdot 2\cdots n$ \\

${n\choose k}$ & p.~\pageref{d:choose} & Binomial coefficients\\ 

$\sum$ & p.~\pageref{def:sigma-notation},\pageref{sec:series} & Sigma notation for sums and
series\\

$\prod$ & p.~\pageref{def:Pi-notation} & Pi notation for products\\

$\in$ & p.~\pageref{cha:sets} & Element\\ 

$\varnothing$ & p.~\pageref{l:emptyset} & Empty set\\ 

$\subseteq$ & p.~\pageref{d:subset} & Subset\\ 

$\cup$ & p.~\pageref{d:uic} & Union\\ 

$\cap$ & p.~\pageref{d:uic} & Intersection\\ 

$\smallsetminus$ & p.~\pageref{d:uic} & Complement in a set\\ 

$\mathcal{P}(X)$ & p.~\pageref{d:powerset} & Power set of $X$ \\ 

$A \times B$ & p.~\pageref{d:productset} & Product set of $A$ and $B$ \\ 

$\Rel$, $\sim$ & p.~\pageref{d:eqrel} & Equivalence
relation\\

$\equiv$ & p.~\pageref{ex:mod5} & Equivalence relation or congruent modulo \\

$[a]$ & p.~\pageref{d:eqclass} & Equivalence class\\

$A/R$ & p.~\pageref{d:eqclass} & Quotient set\\ 

$|x|$ & p.~\pageref{d:absval} & Absolute value of $x$\\

$\max$, $\min$ & p.~\pageref{l:maxmin} & Maximum and minimum\\

$\lfloor x\rfloor$ &p.~\pageref{c:entier} & Floor of $x$ \\

$(a_n)_\nnn$ & p.~\pageref{ex:seq} & Sequence\\ 

$\exp$ & p.~\pageref{t:expfunc} & Exponential function\\

$e = \exp(1)$ & p.~\pageref{e:euler} & Euler's number\\

$\Id$ & p.~\pageref{e:Id} & Identity mapping \\

$f^{-1}$ & p.~\pageref{d:invfunc} & Inverse function of $f$\\

$A\approx B$ & p.~\pageref{d:cardinality} & $A$ and $B$ have the same cardinality\\

$\aleph_0$ & p.~\pageref{d:cardinal} & cardinal number of $\NN$ (``aleph naught'')\\



\end{tabular}

\printindex            

\small
\begin{thebibliography}{999}

\bibitem{Beardon}
A.F.\ Beardon,
\emph{Limits},
Springer, 1997.

\bibitem{BG}
M.\ Beck and R.\ Geoghegan,
\emph{The Art of Proof},
Springer, 2010.

\bibitem{Bloch}
E.D.\ Bloch,
\emph{Proofs and Fundamentals},
Springer, 2011. 

\bibitem{CPZ}
G.\ Chartrand, A.D.\ Polimeni, and P.\ Zhang,
\emph{Mathematical Proofs},
Pearson, 2008. 

\bibitem{Cunning}
D.W.\ Cunningham,
\emph{A Logical Introduction to Proof},
Springer, 2012. 

\bibitem{Forster}   
O.\ Forster, \emph{Analysis 1} (in German), 
Vieweg, 1983. 

\bibitem{Galovich}
S.\ Galovich,
\emph{Doing Mathematics},
Thomson, 2007. 

\bibitem{Gerstein}
L.J.\ Gerstein,
\emph{Introduction to Mathematical Structures and Proofs},
second edition, 
Springer, 2012.

\bibitem{Howie}
J.M.\ Howie, \emph{Real Analysis},
Springer, 2006. 

\bibitem{Lay}
S.R.\ Lay,
\emph{Analysis},
Pearson, 2005. 

\bibitem{Lipschutz}
S.\ Lipschutz,
\emph{Set Theory and Related Topics},
Schaum's Outlines, 1998. 

\bibitem{SES}
D.\ Smith, M.\ Eggen, and R.\ St.\ Andre,
\emph{A Transition to Advanced Mathematics},
Brooks/Cole, 2011. 

\bibitem{Solow}
D.\ Solow,
\emph{How to read and do proofs},
Wiley, 2005. 





\end{thebibliography}
\end{document}